\documentclass[11pt,a4paper,leqno]{article}
\usepackage{a4wide}
\usepackage{latexsym}
\usepackage{amsmath}
\usepackage{amssymb}
\newtheorem{defin}{Definition}
\newtheorem{lemma}{Lemma}
\newtheorem{prop}{Proposition}
\newtheorem{theo}{Theorem}

\newenvironment{proof}{\medskip\par\noindent{\bf Proof}}{\hfill $\Box$
\medskip\par}
\begin{document}
\title{On parametric multisummable formal solutions to some nonlinear initial value Cauchy problems}
\author{{\bf A. Lastra\footnote{The author is partially supported by the project MTM2012-31439 of Ministerio de Ciencia e
Innovacion, Spain}, S. Malek\footnote{The author is partially supported by the french ANR-10-JCJC 0105 project and the PHC
Polonium 2013 project No. 28217SG.}}\\
University of Alcal\'{a}, Departamento de F\'{i}sica y Matem\'{a}ticas,\\
Ap. de Correos 20, E-28871 Alcal\'{a} de Henares (Madrid), Spain,\\
University of Lille 1, Laboratoire Paul Painlev\'e,\\
59655 Villeneuve d'Ascq cedex, France,\\
{\tt alberto.lastra@uah.es}\\
{\tt Stephane.Malek@math.univ-lille1.fr }}
\date{January, 14 2014}
\maketitle
\thispagestyle{empty}
{ \small \begin{center}
{\bf Abstract}
\end{center}
We study a nonlinear initial value Cauchy problem depending upon a complex perturbation parameter
$\epsilon$ whose coefficients depend holomorphically on $(\epsilon,t)$ near the origin in $\mathbb{C}^{2}$
and are bounded holomorphic on some horizontal strip in $\mathbb{C}$ w.r.t the space variable. In our previous
contribution \cite{lama}, we assumed the forcing term of the Cauchy problem to be analytic near 0. Presently, we
consider a family of forcing terms that are holomorphic on a common sector in time $t$ and on sectors w.r.t the
parameter $\epsilon$ whose union form a covering of some neighborhood of 0 in $\mathbb{C}^{\ast}$, which are
asked to share a common formal power series asymptotic expansion of some Gevrey order as $\epsilon$ tends
to 0. We construct a family of actual holomorphic solutions to our Cauchy problem defined on the
sector in time and on the sectors in $\epsilon$ mentioned above. These solutions are achieved by means
of a version of the so-called accelero-summation method in the time variable and by Fourier inverse transform in
space. It appears that these functions share a common formal asymptotic expansion in the perturbation parameter.
Furthermore, this formal series expansion can be written as a sum of two formal series with a corresponding
decomposition for the actual solutions which possess two different asymptotic Gevrey orders, one steming
from the shape of the equation and the other originating from the forcing terms. The special case of multisummability in
$\epsilon$ is also analyzed thoroughly. The proof leans on a version of the so-called Ramis-Sibuya theorem which entails
two distinct Gevrey orders. Finally, we give an application to the study of parametric multi-level Gevrey solutions
for some nonlinear initial value Cauchy problems with holomorphic coefficients and forcing term in $(\epsilon,t)$ near 0 and
bounded holomorphic on a strip in the complex space variable.\medskip

}
\bigskip \bigskip

\section{Introduction}

We consider a family of parameter depending nonlinear initial value Cauchy problems of the form
\begin{multline}
Q(\partial_{z})(\partial_{t}u^{\mathfrak{d}_{p}}(t,z,\epsilon)) =
c_{1,2}(\epsilon)(Q_{1}(\partial_{z})u^{\mathfrak{d}_{p}}(t,z,\epsilon))(Q_{2}(\partial_{z})u^{\mathfrak{d}_{p}}(t,z,\epsilon))\\
+ \epsilon^{(\delta_{D}-1)(k_{2}+1) - \delta_{D} + 1}t^{(\delta_{D}-1)(k_{2}+1)}
\partial_{t}^{\delta_D}R_{D}(\partial_{z})u^{\mathfrak{d}_{p}}(t,z,\epsilon)
+ \sum_{l=1}^{D-1} \epsilon^{\Delta_{l}}t^{d_l}\partial_{t}^{\delta_l}R_{l}(\partial_{z})u^{\mathfrak{d}_{p}}(t,z,\epsilon)\\
+ c_{0}(t,z,\epsilon)R_{0}(\partial_{z})u^{\mathfrak{d}_{p}}(t,z,\epsilon) +
c_{F}(\epsilon)f^{\mathfrak{d}_{p}}(t,z,\epsilon) \label{ICP_main_p_intro}
\end{multline}
for given vanishing initial data $u^{\mathfrak{d}_{p}}(0,z,\epsilon) \equiv 0$, where $D \geq 2$ and
$\delta_{D},k_{2},\Delta_{l},d_{l},\delta_{l}$, $1 \leq l \leq D-1$ are integers which satisfy the inequalities
(\ref{delta_constraints_main_CP}), (\ref{constrain_d_l_delta_l_Delta_l_main_CP}) and
(\ref{constraints_k1_k2_Borel_equation_for_u_p}). Moreover, $Q(X)$,$Q_{1}(X)$,$Q_{2}(X)$,\\$R_{l}(X)$, $0 \leq l \leq D$ are
polynomials submitted to the constraints (\ref{first_constraints_polynomials_Q_R_main_CP}) and which possess the crucial
analytic properties (\ref{quotient_Q_Rl_larger_than_radius_main_CP}), (\ref{quotient_Q_RD_in_S_main_CP}). The coefficient
$c_{0}(t,z,\epsilon)$ is a bounded holomorphic function on a product $D(0,r) \times H_{\beta} \times D(0,\epsilon_{0})$,
where $D(0,r)$ (resp. $D(0,\epsilon_{0})$) denotes a disc centered at 0 with small radius $r>0$ (resp. $\epsilon_{0}>0$) and
$H_{\beta} \subset \mathbb{C}$ is some strip of width $\beta>0$, see (\ref{defin_c_0}). The coefficients
$c_{1,2}(\epsilon)$ and $c_{F}(\epsilon)$ define bounded holomorphic functions on $D(0,\epsilon_{0})$ vanishing
at $\epsilon=0$. The forcing terms $f^{\mathfrak{d}_{p}}(t,z,\epsilon)$, $0 \leq p \leq \varsigma - 1$, form a
family of bounded holomorphic functions on products $\mathcal{T} \times H_{\beta} \times \mathcal{E}_{p}$,
where $\mathcal{T}$ is a small sector centered at 0 contained in $D(0,r)$ and
$\{ \mathcal{E}_{p} \}_{0 \leq p \leq \varsigma - 1}$ is a set of bounded sectors with aperture slightly larger than
$\pi/k_{2}$ covering some neighborhood of 0 in $\mathbb{C}^{\ast}$. We make assumptions in order that all the functions
$\epsilon \mapsto f^{\mathfrak{d}_{p}}(t,z,\epsilon)$, seen as functions from $\mathcal{E}_{p}$ into the Banach
space $\mathbb{F}$ of bounded holomorphic functions on $\mathcal{T} \times H_{\beta}$ endowed with the supremum norm, share a
common asymptotic expansion $\hat{f}(t,z,\epsilon) = \sum_{m \geq 0} f_{m}(t,z) \epsilon^{m}/m! \in \mathbb{F}[[\epsilon]]$
of Gevrey order $1/k_{1}$ on $\mathcal{E}_{p}$, for some integer $1 \leq k_{1} < k_{2}$, see Lemma 11.

Our main purpose is the construction of actual holomorphic solutions $u^{\mathfrak{d}_{p}}(t,z,\epsilon)$ to the
problem (\ref{ICP_main_p_intro}) on the domains $\mathcal{T} \times H_{\beta} \times \mathcal{E}_{p}$ and to analyse their
asymptotic expansions as $\epsilon$ tends to 0.

This work is a continuation of the study initiated in \cite{lama} where the authors have studied initial value problems with
quadratic nonlinearity of the form
\begin{multline}
Q(\partial_{z})(\partial_{t}u(t,z,\epsilon)) =
(Q_{1}(\partial_{z})u(t,z,\epsilon))(Q_{2}(\partial_{z})u(t,z,\epsilon))\\
+ \epsilon^{(\delta_{D}-1)(k+1) - \delta_{D} + 1}t^{(\delta_{D}-1)(k+1)}
\partial_{t}^{\delta_D}R_{D}(\partial_{z})u(t,z,\epsilon)
+ \sum_{l=1}^{D-1} \epsilon^{\Delta_{l}}t^{d_l}\partial_{t}^{\delta_l}R_{l}(\partial_{z})u(t,z,\epsilon)\\
+ c_{0}(t,z,\epsilon)R_{0}(\partial_{z})u(t,z,\epsilon) +
f(t,z,\epsilon) \label{ICP_main_former_work_intro}
\end{multline}
for given vanishing initial data $u(0,z,\epsilon) \equiv 0$, where $D,\Delta_{l},d_{l},\delta_{l}$ are positive integers
and $Q(X),Q_{1}(X),Q_{2}(X),R_{l}(X)$, $0 \leq l \leq D$, are polynomials satisfying similar constraints to the ones envisaged
for the problem (\ref{ICP_main_p_intro}). Under the assumption that the coefficients $c_{0}(t,z,\epsilon)$ and
the forcing term $f(t,z,\epsilon)$ are bounded holomorphic functions on $D(0,r) \times H_{\beta} \times D(0,\epsilon_{0})$, one can
build, using some Borel-Laplace procedure and Fourier inverse transform, a family of holomorphic bounded functions
$u_{p}(t,z,\epsilon)$, $0 \leq p \leq \varsigma - 1$, solutions of (\ref{ICP_main_former_work_intro}),
defined on the products $\mathcal{T} \times H_{\beta} \times \mathcal{E}_{p}$, where $\mathcal{E}_{p}$ has an
aperture slightly larger than $\pi/k$. Moreover, the functions $\epsilon \mapsto u_{p}(t,z,\epsilon)$ share a common
formal power series $\hat{u}(t,z,\epsilon) = \sum_{m \geq 0} h_{m}(t,z) \epsilon^{m}/m!$ as asymptotic expansion of Gevrey
order $1/k$ on $\mathcal{E}_{p}$. In other words, $u_{p}(t,z,\epsilon)$ is the $k-$sum of $\hat{u}(t,z,\epsilon)$ on
$\mathcal{E}_{p}$, see Definition 9.

In this paper, we observe that the asymptotic expansion of the solutions $u^{\mathfrak{d}_{p}}(t,z,\epsilon)$ of
(\ref{ICP_main_p_intro}) w.r.t $\epsilon$ on $\mathcal{E}_{p}$, defined as
$\hat{u}(t,z,\epsilon) = \sum_{m \geq 0} h_{m}(t,z) \epsilon^{m}/m! \in \mathbb{F}[[\epsilon]]$, inherites a finer structure
which involves the two Gevrey orders $1/k_{1}$ and $1/k_{2}$. Namely, the order $1/k_{2}$ originates from the equation
(\ref{ICP_main_p_intro}) itself and its highest order term $\epsilon^{(\delta_{D}-1)(k_{2}+1) - \delta_{D}+1}
t^{(\delta_{D}-1)(k_{2}+1)} \partial_{t}^{\delta_{D}}R_{D}(\partial_{z})$ as it was the case in the work
\cite{lama} mentioned above and the scale $1/k_{1}$ arises, as a new feature, from the asymptotic expansion $\hat{f}$ of the
forcing terms $f^{\mathfrak{d}_{p}}(t,z,\epsilon)$. We can also describe conditions for which
$u^{\mathfrak{d}_{p}}(t,z,\epsilon)$ is the $(k_{2},k_{1})-$sum of $\hat{u}(t,z,\epsilon)$ on $\mathcal{E}_{p}$ for some
$0 \leq p \leq \varsigma - 1$, see Definition 10. More specifically, we can present our two main statements and
its application as follows.\medskip

\noindent {\bf Main results} {\itshape Let $k_{2} > k_{1} \geq 1$ be integers. We choose a family
$\{ \mathcal{E}_{p} \}_{0 \leq p \leq \varsigma - 1}$ of bounded sectors with aperture
slightly larger than $\pi/k_{2}$ which defines a good covering in $\mathbb{C}^{\ast}$ (see Definition 7) and a set of
adequate directions $\mathfrak{d}_{p} \in \mathbb{R}$, $0 \leq p \leq \varsigma - 1$ for which the constraints
(\ref{root_cond_1_in_defin_main_CP}) and (\ref{root_cond_2_in_defin_main_CP}) hold. We also take an open bounded sector
$\mathcal{T}$ centered at 0 such that for every $0 \leq p \leq \varsigma - 1$, the product
$\epsilon t$ belongs to a sector with direction $\mathfrak{d}_{p}$ and aperture slightly larger than $\pi/k_{2}$, for all
$\epsilon\in \mathcal{E}_{p}$, all $t \in \mathcal{T}$. We make the assumption that the coefficient $c_{0}(t,z,\epsilon)$ can be
written as a convergent series of the special form
$$ c_{0}(t,z,\epsilon) = c_{0}(\epsilon) \sum_{n \geq 0}c_{0,n}(z,\epsilon) (\epsilon t)^{n} $$
on a domain $D(0,r) \times H_{\beta'} \times D(0,\epsilon_{0})$, where $H_{\beta'}$ is a strip of width $\beta'$, such that
$\mathcal{T} \subset D(0,r)$, $\cup_{0 \leq p \leq \varsigma - 1} \mathcal{E}_{p} \subset D(0,\epsilon_{0})$ and
$0 < \beta' < \beta$ are given positive real numbers. The coefficients $c_{0,n}(z,\epsilon)$, $n \geq 0$, are supposed to be
inverse Fourier transform of functions $m \mapsto C_{0,n}(m,\epsilon)$ that belong to the Banach space $E_{(\beta,\mu)}$
(see Definition 2) for some $\mu > \max( \mathrm{deg}(Q_{1})+1, \mathrm{deg}(Q_{2})+1 )$ and depend holomorphically on
$\epsilon$ in $D(0,\epsilon_{0})$ and $c_{0}(\epsilon)$ is a holomorphic function on $D(0,\epsilon_{0})$ vanishing at 0.
Since we have in view our principal application (Theorem 3), we choose the forcing term $f^{\mathfrak{d}_{p}}(t,z,\epsilon)$ as
a $m_{k_2}-$Fourier-Laplace transform
$$ f^{\mathfrak{d}_{p}}(t,z,\epsilon) = \frac{k_2}{(2\pi)^{1/2}}
\int_{-\infty}^{+\infty} \int_{L_{\gamma_p}} \psi_{k_2}^{\mathfrak{d}_{p}}(u,m,\epsilon) e^{-(\frac{u}{\epsilon t})^{k_2}}
e^{izm} \frac{du}{u} dm,$$
where the inner integration is made along some halfline $L_{\gamma_p} \subset S_{\mathfrak{d}_{p}}$ and
$S_{\mathfrak{d}_{p}}$ is an unbounded sector with bisecting direction $\mathfrak{d}_{p}$, with small aperture and where
$\psi_{k_2}^{\mathfrak{d}_{p}}(u,m,\epsilon)$ is a holomorphic function w.r.t $u$ on
$S_{\mathfrak{d}_{p}}$, defined as an integral transform called acceleration operator with indices $m_{k_2}$ and $m_{k_1}$,
$$ \psi_{k_2}^{\mathfrak{d}_{p}}(u,m,\epsilon) = \int_{L_{\gamma_{p}^{1}}}
\psi_{k_1}^{\mathfrak{d}_{p}}(h,m,\epsilon) G(u,h) \frac{dh}{h} $$
where $G(u,h)$ is a kernel function with exponential decay of order $\kappa = (\frac{1}{k_1} - \frac{1}{k_{2}})^{-1}$, see
(\ref{G_xi_h_exp_growth_order_kappa}). The integration path $L_{\gamma_{p}^{1}}$ is a halfline in an unbounded sector
$U_{\mathfrak{d}_{p}}$ with bisecting direction $\mathfrak{d}_{p}$ and $\psi_{k_1}^{\mathfrak{d}_{p}}(h,m,\epsilon)$ is a
function with exponential growth of order $k_1$ w.r.t $h$ on $U_{\mathfrak{d}_{p}} \cup D(0,\rho)$ and exponential decay
w.r.t $m$ on $\mathbb{R}$, satisfying the bounds (\ref{psi_k1_bounded_norm_k1_k1_main_CP}). The function
$f^{\mathfrak{d}_{p}}(t,z,\epsilon)$ represents a bounded holomorphic function on
$\mathcal{T} \times H_{\beta'} \times \mathcal{E}_{p}$. Actually, it turns out that $f^{\mathfrak{d}_{p}}(t,z,\epsilon)$ can
be simply written as a $m_{k_1}-$Fourier-Laplace transform of $\psi_{k_1}^{\mathfrak{d}_p}(h,m,\epsilon)$,
$$ f^{\mathfrak{d}_{p}}(t,z,\epsilon) = \frac{k_1}{(2\pi)^{1/2}} \int_{-\infty}^{+\infty}
\int_{L_{\gamma_{p}}} \psi_{k_1}^{\mathfrak{d}_{p}}(u,m,\epsilon) e^{-(\frac{u}{\epsilon t})^{k_1}} e^{izm} \frac{du}{u} dm,$$
see Lemma 13.

Our first result stated in Theorem 1 claims that if the sup norms of the coefficients $c_{1,2}(\epsilon)/\epsilon$,
$c_{0}(\epsilon)/\epsilon$ and $c_{F}(\epsilon)/\epsilon$ on $D(0,\epsilon_{0})$ are chosen small enough, then we can construct
a family of holomorphic bounded functions $u^{\mathfrak{d}_{p}}(t,z,\epsilon)$, $0 \leq p \leq \varsigma -1$, defined on the
products $\mathcal{T} \times H_{\beta'} \times \mathcal{E}_{p}$, which solves the problem
(\ref{ICP_main_p_intro}) with initial data $u^{\mathfrak{d}_{p}}(0,z,\epsilon) \equiv 0$. Similarly to the forcing term,
$u^{\mathfrak{d}_{p}}(t,z,\epsilon)$ can be written as a $m_{k_2}-$Fourier-Laplace transform
$$ u^{\mathfrak{d}_{p}}(t,z,\epsilon) = \frac{k_2}{(2\pi)^{1/2}} \int_{-\infty}^{+\infty} \int_{L_{\gamma_p}}
\omega_{k_2}^{\mathfrak{d}_{p}}(u,m,\epsilon) e^{-(\frac{u}{\epsilon t})^{k_2}} e^{izm} \frac{du}{u} dm $$
where $\omega_{k_2}^{\mathfrak{d}_{p}}(u,m,\epsilon)$ denotes a function with at most exponential growth of order $k_{2}$ in
$u$ on $S_{\mathfrak{d}_{p}}$ and exponential decay in $m \in \mathbb{R}$, satisfying (\ref{omega_k2_frak_d_p_exp_growth}).
The function $\omega_{k_2}^{\mathfrak{d}_{p}}(u,m,\epsilon)$ is shown to be the analytic continuation of a function
$\mathrm{Acc}_{k_{2},k_{1}}^{\mathfrak{d}_{p}}(\omega_{k_1}^{\mathfrak{d}_{p}})(u,m,\epsilon)$ defined only on a
bounded sector $S_{\mathfrak{d}_{p}}^{b}$ with aperture slightly larger than $\pi/\kappa$ w.r.t $u$, for all $m \in \mathbb{R}$,
with the help of an acceleration operator with indices $m_{k_2}$ and $m_{k_1}$,
$$ \mathrm{Acc}_{k_{2},k_{1}}^{\mathfrak{d}_{p}}(\omega_{k_1}^{\mathfrak{d}_{p}})(u,m,\epsilon) =
\int_{L_{\gamma_{p}^{1}}} \omega_{k_1}^{\mathfrak{d}_{p}}(h,m,\epsilon) G(u,h) \frac{dh}{h} .$$
We show that, in general, $\omega_{k_1}^{\mathfrak{d}_{p}}(h,m,\epsilon)$ suffers an exponential growth of order larger than
$k_{1}$ (and actually less than $\kappa$) w.r.t $h$ on $U_{\mathfrak{d}_{p}} \cup D(0,\rho)$, and obeys the estimates
(\ref{omega_k1_frak_d_p_exp_growth}). At this point $u^{\mathfrak{d}_{p}}(t,z,\epsilon)$ cannot be merely expressed as a
$m_{k_1}-$Fourier-Laplace transform of $\omega_{k_1}^{\mathfrak{d}_p}$ and is obtained by a version of the so-called
accelero-summation procedure, as described in \cite{ba}, Chapter 5.

Our second main result, described in Theorem 2, asserts that the functions $u^{\mathfrak{d}_{p}}$, seen as maps
from $\mathcal{E}_{p}$ into $\mathbb{F}$, for $0 \leq p \leq \varsigma - 1$, turn out to share on $\mathcal{E}_{p}$ a common
formal power series $\hat{u}(\epsilon) = \sum_{m \geq 0} h_{m} \epsilon^{m}/m! \in \mathbb{F}[[\epsilon]]$ as asymptotic
expansion of Gevrey order $1/k_{1}$. The formal series $\hat{u}(\epsilon)$ formally solves the equation
(\ref{ICP_main_p_intro}) where the analytic forcing term $f^{\mathfrak{d}_{p}}(t,z,\epsilon)$ is replaced by its
asymptotic expansion $\hat{f}(t,z,\epsilon) \in \mathbb{F}[[\epsilon]]$ of Gevrey order $1/k_{1}$ (see Lemma 11).
Furthermore, the functions $u^{\mathfrak{d}_{p}}$ and the formal series $\hat{u}$ own a fine structure which actually involves
two different Gevrey orders of asymptotics. Namely, $u^{\mathfrak{d}_{p}}$ and $\hat{u}$ can be written as sums
$$ \hat{u}(\epsilon) = a(\epsilon) + \hat{u}_{1}(\epsilon) + \hat{u}_{2}(\epsilon) \ \ , \ \
u^{\mathfrak{d}_{p}}(t,z,\epsilon) = a(\epsilon) + u_{1}^{\mathfrak{d}_{p}}(\epsilon) + u_{2}^{\mathfrak{d}_{p}}(\epsilon) $$
where $a(\epsilon)$ is a convergent series near $\epsilon=0$ with coefficients in $\mathbb{F}$ and
$\hat{u}_{1}(\epsilon)$ (resp. $\hat{u}_{2}(\epsilon)$) belongs to $\mathbb{F}[[\epsilon]]$ and is the asymptotic expansion
of Gevrey order $1/k_{1}$ (resp. $1/k_{2}$) of the $\mathbb{F}-$valued function $u_{1}^{\mathfrak{d}_{p}}(\epsilon)$
(resp. $u_{2}^{\mathfrak{d}_{p}}(\epsilon)$) on $\mathcal{E}_{p}$. Besides, under a more restrictive assumption on the
covering $\{ \mathcal{E}_{p} \}_{0 \leq p \leq \varsigma -1}$ and the unbounded sectors
$\{U_{\mathfrak{d}_{p}} \}_{0 \leq p \leq \varsigma -1}$ (see Assumption 5 in Theorem 2), one gets that
$u^{\mathfrak{d}_{p_0}}(t,z,\epsilon)$ is even the $(k_{2},k_{1})-$sum of $\hat{u}(\epsilon)$ on some sector
$\mathcal{E}_{p_0}$, with $0 \leq p_{0} \leq \varsigma - 1$, meaning that $u_{1}^{\mathfrak{d}_{p_0}}(\epsilon)$ can
be analytically continued on a larger sector $S_{\pi/k_{1}}$, containing $\mathcal{E}_{p_0}$, with aperture
slightly larger than $\pi/k_{1}$ where it becomes the $k_{1}-$sum of $\hat{u}_{1}(\epsilon)$ and by construction
$u_{2}^{\mathfrak{d}_{p_0}}(\epsilon)$ is already the $k_{2}-$sum of $\hat{u}_{2}(\epsilon)$ on $\mathcal{E}_{p_0}$, see
Definition 10.

As an important application (Theorem 3), we deal with the special case when the forcing terms
$f^{\mathfrak{d}_{p}}(t,z,\epsilon)$ themselves solve a linear partial differential equation with a similar shape
as (\ref{ICP_main_former_work_intro}), see (\ref{ICP_main_fp_bf}), whose coefficients are holomorphic functions
on $D(0,r) \times H_{\beta} \times D(0,\epsilon_{0})$. When this holds, it turns out that
$u^{\mathfrak{d}_{p}}(t,z,\epsilon)$ and its asymptotic expansion $\hat{u}(t,z,\epsilon)$ solves a nonlinear
singularly perturbed PDE with analytic coefficients and forcing term on $D(0,r) \times H_{\beta} \times D(0,\epsilon_{0})$,
see (\ref{nlpde_holcoef_near_origin_u_dp}).}\medskip

We stress the fact that our application (Theorem 3) relies on the factorization of some nonlinear differential operator
which is an approach that belongs to an active domain of research in symbolic computation these last years,
see for instance \cite{beka}, \cite{ber}, \cite{hosi}, \cite{schw}, \cite{shwi}, \cite{ts}.

We mention that a similar result has been recently obtained by H. Tahara and H. Yamazawa, see \cite{taya},
for the multisummability
of formal series $\hat{u}(t,x) = \sum_{n \geq 0} u_{n}(x)t^{n} \in \mathcal{O}(\mathbb{C}^{N})[[t]]$ with
entire coefficients on $\mathbb{C}^{N}$, $N \geq 1$, solutions of general
non-homogeneous time depending linear PDEs of the form
$$ \partial_{t}^{m}u + \sum_{j + |\alpha| \leq L} a_{j,\alpha}(t) \partial_{t}^{j}\partial_{x}^{\alpha}u = f(t,x) $$
for given initial data $(\partial_{t}^{j}u)(0,x) = \varphi_{j}(x)$, $0 \leq j \leq m-1$ (where $1 \leq m \leq L$),
provided that the coefficients $a_{j,\alpha}(t)$ together with $t \mapsto f(t,x)$ are analytic near 0 and that
$\varphi_{j}(x)$ with the forcing term $x \mapsto f(t,x)$ belong to a suitable class of
entire functions of finite exponential order on $\mathbb{C}^{N}$. The different levels of multisummability are related
to the slopes of a Newton polygon attached to the main equation and analytic acceleration procedures as described above
are heavily used in their proof.

It is worthwhile noticing that the multisummable structure of formal solutions to linear and nonlinear meromorphic
ODEs has been discovered two decades ago, see for instance \cite{ba1},
\cite{babrrasi}, \cite{br}, \cite{lori}, \cite{malra}, \cite{rasi}, but in the framework of PDEs very few results are known.
In the linear case in two complex variables with constant coefficients, we mention the important contributions
of W. Balser, \cite{ba4} and S. Michalik, \cite{mi}, \cite{mi1}. Their strategy consists in the construction of
a multisummable formal solution written as a sum of formal series, each of them associated to a root of
the symbol attached to the PDE using the so-called Puiseux expansion for the roots of polynomial with holomorphic coefficients.
In the linear and nonlinear context of PDEs that come from a perturbation of ordinary differential equations, we refer to
the works of S. Ouchi, \cite{ou}, \cite{ou1}, which are based on a Newton polygon approach and
accelero-summation technics as in \cite{taya}. Our result concerns more peculiarly multisummability and multiple scale
analysis in the complex parameter $\epsilon$. Also from this point of view, only few advances have been performed. Among them,
we must mention two recent works by K. Suzuki and Y. Takei, \cite{sutak} and Y. Takei, \cite{tak}, for WKB solutions of the
Schr\"{o}dinger equation
$$ \epsilon^{2}\psi''(z) = (z - \epsilon^{2}z^{2})\psi(z) $$
which possesses 0 as fixed turning point and $z_{\epsilon}=\epsilon^{-2}$ as movable turning point tending to infinity
as $\epsilon$ tends to 0.\medskip

In the sequel, we describe our main intermediate results and the sketch of the arguments needed in their proofs. In a first part,
we depart from an auxiliary parameter depending initial value differential and convolution equation which is regularly
perturbed in its parameter $\epsilon$, see (\ref{SCP}). This equation is formally constructed by making the change of variable
$T=\epsilon t$ in the equation (\ref{ICP_main_p_intro}) and by taking the Fourier transform w.r.t the variable $z$
(as done in our previous contribution \cite{lama}). We construct a formal power series
$\hat{U}(T,m,\epsilon)=\sum_{n \geq 1}U_{n}(m,\epsilon)T^{n}$ solution of (\ref{SCP}) whose coefficients
$m \mapsto U_{n}(m,\epsilon)$ depend holomorphically on $\epsilon$ near 0 and belong to a Banach space $E_{(\beta,\mu)}$
of continuous functions with exponential decay on $\mathbb{R}$ introduced by O. Costin and S. Tanveer in \cite{cota2}.

As a first step, we follow the strategy recently developped by H. Tahara and H. Yamazawa in \cite{taya}, namely we
multiply each hand side of (\ref{SCP}) by the power $T^{k_{1}+1}$ which transforms it into an equation
(\ref{SCP_irregular_k1}) which involves only differential operators in $T$ of irregular type at $T=0$ of the form
$T^{\beta}\partial_{T}$ with $\beta \geq k_{1}+1$ due to the assumption (\ref{constraint_dl_deltal_k1}) on the shape of
the equation (\ref{SCP}). Then, we apply a formal Borel transform of order $k_{1}$, that we call $m_{k_1}-$Borel transform
in Definition 4, to the formal series $\hat{U}$ with respect to $T$, denoted
$$ \omega_{k_1}(\tau,m,\epsilon) = \sum_{n \geq 1} U_{n}(m,\epsilon) \frac{\tau^{n}}{\Gamma(n/k_{1})}. $$
Then, we show that $\omega_{k_1}(\tau,m,\epsilon)$ formally solves a convolution equation in both variables $\tau$ and $m$,
see (\ref{k_1_Borel_equation}). Under some size constraints on the sup norm of the coefficients
$c_{1,2}(\epsilon)/\epsilon$, $c_{0}(\epsilon)/\epsilon$ and $c_{F}(\epsilon)/\epsilon$ near 0, we show that
$\omega_{k_1}(\tau,m,\epsilon)$ is actually convergent for $\tau$ on some fixed neighborhood of 0 and can be extended
to a holomorphic function $\omega_{k_1}^{d}(\tau,m,\epsilon)$ on unbounded sectors $U_{d}$ centered at 0 with
bisecting direction $d$ and tiny aperture, provided that the $m_{k_1}-$Borel transform of the formal forcing term
$F(T,m,\epsilon)$, denoted $\psi_{k_1}(\tau,m,\epsilon)$ is convergent near $\tau=0$ and can be extended on $U_{d}$ w.r.t $\tau$
as a holomorphic function $\psi_{k_1}^{d}(\tau,m,\epsilon)$ with exponential growth of order less than $k_{1}$. Besides,
the function $\omega_{k_1}^{d}(\tau,m,\epsilon)$ satisfies estimates of the form: there exist constants $\nu>0$ and
$\varpi_{d}>0$ with
$$ |\omega_{k_1}^{d}(\tau,m,\epsilon)| \leq \varpi_{d}(1+|m|)^{-\mu}e^{-\beta |m|} \frac{|\tau|}{1 + |\tau|^{2k_{1}}}
e^{\nu |\tau|^{\kappa}} $$
for all $\tau \in U_{d}$, all $m \in \mathbb{R}$, all $\epsilon \in D(0,\epsilon_{0})$, see Proposition 11. The proof leans on
a fixed point argument in a Banach space of holomorphic functions $F_{(\nu,\beta,\mu,k_{1},\kappa)}^{d}$ studied in
Section 2.1. Since the exponential growth order $\kappa$ of $\omega_{k_1}^{d}$ is larger than $k_{1}$, we cannot take
a $m_{k_1}-$Laplace transform of it in direction $d$. We need to use a version of what is called an
accelero-summation procedure as described in \cite{ba}, Chapter 5, which is explained in Section 4.3.

In a second step, we go back to our seminal convolution equation (\ref{SCP}) and we multiply each handside by the power
$T^{k_{2}+1}$ which transforms it into the equation (\ref{SCP_irregular_k2}). Then, we apply a $m_{k_2}-$Borel
transform to the formal series $\hat{U}$ w.r.t $T$, denoted $\hat{\omega}_{k_2}(\tau,m,\epsilon)$. We show that
$\hat{\omega}_{k_2}(\tau,m,\epsilon)$ formally solves a convolution equation in both variables $\tau$ and $m$, see
(\ref{k2_Borel_equation}), where the formal $m_{k_2}-$Borel transform of the forcing term is set as
$\hat{\psi}_{k_2}(\tau,m,\epsilon)$. Now, we observe that a version of the analytic acceleration transform
with indices $k_{2}$ and $k_{1}$ constructed in Proposition 13 applied to $\psi_{k_1}^{d}(\tau,m,\epsilon)$, standing for
$\psi_{k_2}^{d}(\tau,m,\epsilon)$, is the $\kappa-$sum of $\hat{\psi}_{k_2}(\tau,m,\epsilon)$ w.r.t $\tau$ on
some bounded sector $S_{d,\kappa}^{b}$ with aperture slightly larger than $\pi/\kappa$, viewed as a function with values in
$E_{(\beta,\mu)}$. Furthermore, $\psi_{k_2}^{d}(\tau,m,\epsilon)$ can be extended as an analytic function on an
unbounded sector $S_{d,\kappa}$ with aperture slightly larger than $\pi/\kappa$ where it possesses an exponential growth of
order less than $k_{2}$, see Lemma 4. In the sequel, we focus on the solution $\omega_{k_2}^{d}(\tau,m,\epsilon)$ of the
convolution problem (\ref{k2_Borel_equation_analytic}) which is similar to (\ref{k2_Borel_equation}) but with the formal
forcing term $\hat{\psi}_{k_2}(\tau,m,\epsilon)$ replaced by $\psi_{k_2}^{d}(\tau,m,\epsilon)$. Under some size restriction
on the sup norm of the coefficients $c_{1,2}(\epsilon)/\epsilon$, $c_{0}(\epsilon)/\epsilon$ and
$c_{F}(\epsilon)/\epsilon$ near 0, we show that $\omega_{k_2}^{d}(\tau,m,\epsilon)$ defines a bounded holomorphic function
for $\tau$ on the bounded sector $S_{d,\kappa}^{b}$ and can be extended to a holomorphic function on unbounded sectors
$S_{d}$ with direction $d$ and tiny aperture, provided that $S_{d}$ stays away from the roots of some
polynomial $P_{m}(\tau)$ constructed with the help of $Q(X)$ and $R_{D}(X)$ in (\ref{ICP_main_p_intro}),
see (\ref{factor_P_m}). Moreover, the function $\omega_{k_2}^{d}(\tau,m,\epsilon)$ satisfies estimates of the form: there
exist constants $\nu'>0$ and $\upsilon_{d}>0$ with
$$ |\omega_{k_2}^{d}(\tau,m,\epsilon)| \leq \upsilon_{d}(1+|m|)^{-\mu}e^{-\beta|m|}
\frac{|\tau|}{1 + |\tau|^{2k_{2}}} e^{\nu' |\tau|^{k_2}} $$
for all $\tau \in S_{d}$, all $m \in \mathbb{R}$, all $\epsilon \in D(0,\epsilon_{0})$, see Proposition 14. Again, the
proof rests on a fixed point argument in a Banach space of holomorphic functions $F_{(\nu',\beta,\mu,k_{2})}^{d}$ outlined in
Section 2.2. In Proposition 15, we show that $\omega_{k_2}^{d}(\tau,m,\epsilon)$ actually coincides with the
analytic acceleration transform with indices $m_{k_2}$ and $m_{k_1}$ applied to $\omega_{k_1}^{d}(\tau,m,\epsilon)$,
denoted $\mathrm{Acc}_{k_{2},k_{1}}^{d}(\omega_{k_1}^{d})(\tau,m,\epsilon)$, as long as $\tau$ lies in the
bounded sector $S_{d,\kappa}^{b}$. As a result, some $m_{k_2}-$Laplace transform of the analytic continuation of
$\mathrm{Acc}_{k_{2},k_{1}}^{d}(\omega_{k_1}^{d})(\tau,m,\epsilon)$, set as $U^{d}(T,m,\epsilon)$, can be considered for all
$T$ belonging to a sector $S_{d,\theta_{k_2},h}$ with bisecting direction $d$, aperture $\theta_{k_2}$ slightly larger than
$\pi/k_{2}$ and radius $h>0$. Following the terminology of \cite{ba}, Section 6.1, $U^{d}(T,m,\epsilon)$ can be called
the $(m_{k_2},m_{k_1})-$sum of the formal series $\hat{U}(T,m,\epsilon)$ in direction $d$. Additionally,
$U^{d}(T,m,\epsilon)$ solves our primary convolution equation (\ref{SCP}), where the formal forcing term
$\hat{F}(T,m,\epsilon)$ is interchanged with $F^{d}(T,m,\epsilon)$ which denotes the $(m_{k_2},m_{k_1})-$sum of
$\hat{F}$ in direction $d$.

In Theorem 1, we construct a family of actual bounded holomorphic solutions
$u^{\mathfrak{d}_{p}}(t,z,\epsilon)$, $0 \leq p \leq \varsigma - 1$, of our original problem (\ref{ICP_main_p_intro}) on
domains of the form $\mathcal{T} \times H_{\beta'} \times \mathcal{E}_{p}$ described in the main results above. Namely, the
functions $u^{\mathfrak{d}_{p}}(t,z,\epsilon)$ (resp. $f^{\mathfrak{d}_{p}}(t,z,\epsilon)$) are set as Fourier inverse
transforms of $U^{\mathfrak{d}_{p}}$,
$$ u^{\mathfrak{d}_{p}}(t,z,\epsilon) = \mathcal{F}^{-1}( m \mapsto U^{\mathfrak{d}_{p}}(\epsilon t,m,\epsilon) )(z) \ \ , \ \
f^{\mathfrak{d}_{p}}(t,z,\epsilon) = \mathcal{F}^{-1}( m \mapsto F^{\mathfrak{d}_{p}}(\epsilon t,m,\epsilon) )(z) $$
where the definition of $\mathcal{F}^{-1}$ is pointed out in Proposition 9. One proves the crucial property that the difference
of any two neighboring functions
$u^{\mathfrak{d}_{p+1}}(t,z,\epsilon) - u^{\mathfrak{d}_{p}}(t,z,\epsilon)$ tends to zero as $\epsilon \rightarrow 0$ on
$\mathcal{E}_{p+1} \cap \mathcal{E}_{p}$ faster than a function with exponential decay of order $k$, uniformly w.r.t
$t \in \mathcal{T}$, $z \in H_{\beta'}$, with $k=k_{2}$ when the intersection
$U_{\mathfrak{d}_{p+1}} \cap U_{\mathfrak{d}_{p}}$ is not empty and with $k=k_{1}$, when this intersection is empty. The same
estimates hold for the difference $f^{\mathfrak{d}_{p+1}}(t,z,\epsilon) - f^{\mathfrak{d}_{p}}(t,z,\epsilon)$.

The whole section 6 is devoted to the study of the asymptotic behaviour of $u^{\mathfrak{d}_{p}}(t,z,\epsilon)$ as
$\epsilon$ tends to zero. Using the decay estimates on the differences of the functions $u^{\mathfrak{d}_{p}}$ and
$f^{\mathfrak{d}_{p}}$, we show the existence of a common asymptotic expansion
$\hat{u}(\epsilon) = \sum_{m \geq 0} h_{m} \epsilon^{m}/m! \in \mathbb{F}[[\epsilon]]$ (resp.
$\hat{f}(\epsilon) = \sum_{m \geq 0} f_{m} \epsilon^{m}/m! \in \mathbb{F}[[\epsilon]]$) of Gevrey order
$1/k_{1}$ for all functions $u^{\mathfrak{d}_{p}}(t,z,\epsilon)$ (resp. $f^{\mathfrak{d}_{p}}(t,z,\epsilon)$)
as $\epsilon$ tends to 0 on $\mathcal{E}_{p}$. We obtain also a double scale asymptotics for
$u^{\mathfrak{d}_{p}}$ as explained in the main results above. The key tool in proving the result is a version of the
Ramis-Sibuya theorem which entails two Gevrey asymptotics orders, described in Section 6.1. It is worthwhile noting that a
similar version was recently brought into play by Y. Takei and K. Suzuki in \cite{sutak}, \cite{tak}, in order to study
parametric multisummability for the complex Schr\"{o}dinger equation.

In the last section, we study the particular situation when the formal forcing term $F(T,m,\epsilon)$ solves a linear
differential and convolution initial value problem, see (\ref{SCP_bf}). We multiply each handside of this equation by the
power $T^{k_{1}+1}$ which transforms it into the equation (\ref{SCP_irregular_bf}). Then, we show that the
$m_{k_1}-$Borel transform $\psi_{k_1}(\tau,m,\epsilon)$ formally solves a convolution equation in both variables
$\tau$ and $m$, see (\ref{k1_Borel_equation_analytic_bf}). Under a size control of the sup norm of the coefficients
$\mathbf{c}_{0}(\epsilon)/\epsilon$ and $\mathbf{c}_{\mathbf{F}}(\epsilon)/\epsilon$ near 0, we show that
$\psi_{k_1}(\tau,m,\epsilon)$ is actually convergent near 0 w.r.t $\tau$ and can be holomorphically extended as a function
$\psi_{k_1}^{\mathfrak{d}_{p}}(\tau,m,\epsilon)$ on any unbounded sectors $U_{\mathfrak{d}_{p}}$ with direction
$\mathfrak{d}_{p}$ and small aperture, provided that $U_{\mathfrak{d}_{p}}$ stays away from the roots of some polynomial
$\mathbf{P}_{m}(\tau)$ constructed with the help of $\mathbf{Q}(X)$ and $\mathbf{R}_{\mathbf{D}}(X)$ in (\ref{SCP_bf}).
Additionally, the function $\psi_{k_1}^{\mathfrak{d}_{p}}(\tau,m,\epsilon)$ satisfies estimates of the form: there exists
a constant $\boldsymbol{\upsilon}>0$ with
$$ |\psi_{k_1}^{\mathfrak{d}_{p}}(\tau,m,\epsilon)| \leq
\boldsymbol{\upsilon} (1+|m|)^{-\mu}e^{-\beta |m|} \frac{|\tau|}{1 + |\tau|^{2k_{1}}} e^{\nu |\tau|^{k_1}} $$
for all $\tau \in U_{\mathfrak{d}_{p}}$, all $m \in \mathbb{R}$, all $\epsilon \in D(0,\epsilon_{0})$, see Proposition 18. The
proof is once more based upon a fixed point argument in a Banach space of holomorphic functions
$F_{(\nu,\beta,\mu,k_{1},k_{1})}^{d}$ defined in Section 2.1. These latter properties on
$\psi_{k_1}^{\mathfrak{d}_{p}}(\tau,m,\epsilon)$ legitimize all the assumptions made above on the forcing term $F(T,m,\epsilon)$.
Now, we can take the $m_{k_1}-$Laplace transform
$\mathcal{L}_{m_{k_1}}^{\mathfrak{d}_{p}}(\psi_{k_1}^{\mathfrak{d}_{p}})(T)$ of $\psi_{k_1}^{\mathfrak{d}_{p}}(\tau,m,\epsilon)$
w.r.t $\tau$ in direction $\mathfrak{d}_{p}$, which yields an analytic solution of the initial linear equation (\ref{SCP_bf})
on some bounded sector $S_{\mathfrak{d}_{p},\theta_{k_1},h}$ with aperture $\theta_{k_1}$ slightly larger than $\pi/k_{1}$.
It comes to light in Lemma 13, that $\mathcal{L}_{m_{k_1}}^{\mathfrak{d}_{p}}(\psi_{k_1}^{\mathfrak{d}_{p}})(T)$ coincides with
the analytic $(m_{k_{2}},m_{k_{1}})-$sum $F^{\mathfrak{d}_{p}}(T,m,\epsilon)$ of $\hat{F}$ in
direction $\mathfrak{d}_{p}$ on the smaller sector $S_{\mathfrak{d}_{p},\theta_{k_2},h}$ with aperture slightly larger than
$\pi/k_{2}$. We deduce consequently that the analytic forcing term $f^{\mathfrak{d}_{p}}(t,z,\epsilon)$ solves the linear
PDE (\ref{ICP_main_fp_bf}) with analytic coefficients on $D(0,r) \times H_{\beta'} \times D(0,\epsilon_{0})$, for all
$t \in \mathcal{T}$, $z \in H_{\beta'}$, $\epsilon \in \mathcal{E}_{p}$. In our last main result (Theorem 3), we see that this
latter issue implies that $u^{\mathfrak{d}_{p}}(t,z,\epsilon)$ itself solves a nonlinear PDE
(\ref{nlpde_holcoef_near_origin_u_dp}) with analytic coefficients and forcing term on
$D(0,r) \times H_{\beta'} \times D(0,\epsilon_{0})$,
for all $t \in \mathcal{T}$, $z \in H_{\beta'}$, $\epsilon \in \mathcal{E}_{p}$.\medskip

\noindent The paper is organized as follows.\\
In Section 2, we define some weighted Banach spaces of continuous functions on $(D(0,\rho) \cup U) \times \mathbb{R}$ with exponential
growths of different orders on unbounded sectors $U$ w.r.t the first variable and exponential decay on $\mathbb{R}$ w.r.t the second one.
We study the continuity properties of several kind of linear and nonlinear operators acting on these spaces that will be useful in
Sections 4.2, 4.4 and 7.2.\\
In Section 3, we recall the definition and the main analytic and algebraic properties of the $m_{k}-$summability.\\
In Section 4.1, we introduce an auxiliary differential and convolution problem (\ref{SCP}) for which we construct a formal
solution.\\
In Section 4.2, we show that the $m_{k_1}-$Borel transform of this formal solution satisfies a convolution problem
(\ref{k_1_Borel_equation}) that we can uniquely solve within the Banach spaces described in Section 2.\\
In Section 4.3, we describe the properties of a variant of the formal and analytic acceleration operators associated to the
$m_{k}-$Borel and $m_{k}-$Laplace transforms.\\
In Section 4.4, we see that the $m_{k_2}-$Borel transform of the formal solution of (\ref{SCP}) satisfies a convolution
problem (\ref{k2_Borel_equation}). We show that its formal forcing term is $\kappa-$summable and that its $\kappa-$sum is
an acceleration of the $m_{k_1}-$Borel transform of the above formal forcing term. Then, we construct an actual solution
to the corresponding problem with the analytic continuation of this $\kappa-$sum as nonhomogeneous term, within the
Banach spaces defined in Section 2. We recognize that this actual solution is the analytic continuation of the acceleration of
the $m_{k_1}-$Borel transform of the formal solution of (\ref{SCP}). Finally, we take its $m_{k_2}-$Laplace transform in order
to get an actual solution of (\ref{SCP_analytic_d}).\\
In Section 5, with the help of Section 4, we build a family of actual holomorphic solutions to our initial Cauchy problem
(\ref{ICP_main_p_intro}). We show that the difference of any two neighboring solutions is exponentially flat for some
integer order in $\epsilon$ (Theorem 1).\\
In Section 6, we show that the actual solutions constructed in Section 5 share a common formal series as Gevrey asymptotic
expansion as $\epsilon$ tends to 0 on sectors (Theorem 2). The result is buit on a version of the Ramis-Sibuya theorem with
two Gevrey orders stated in Section 6.1.\\
In Section 7, we inspect the special case when the forcing term itself solves a linear PDE. Then, we notice that the solutions
of (\ref{ICP_main_p_intro}) constructed in Section 5 actually solve a nonlinear PDE with holomorphic coefficients and forcing term
near the origin (Theorem 3).

\section{Banach spaces of functions with exponential growth and decay}

The Banach spaces introduced in the next subsection 2.1 (resp. subsection 2.2) will be crucial in the construction of analytic solutions of a
convolution problem investigated in the forthcoming subsection 4.2 (resp. subsection 4.4).

\subsection{Banach spaces of functions with exponential growth $\kappa$ and decay of exponential order 1}

We denote $D(0,r)$ the open disc centered at $0$ with radius $r>0$ in $\mathbb{C}$ and $\bar{D}(0,r)$ its closure. Let 
$U_{d}$ be an open unbounded sector in direction $d \in \mathbb{R}$ centered at $0$ in $\mathbb{C}$. By
convention, the sectors we consider do not contain the origin in $\mathbb{C}$.

\begin{defin} Let $\nu,\beta,\mu>0$ and $\rho>0$ be positive real numbers. Let $k \geq 1$, $\kappa \geq 1$ be integers and
$d \in \mathbb{R}$. We denote
$F_{(\nu,\beta,\mu,k,\kappa)}^{d}$ the vector space of continuous functions $(\tau,m) \mapsto h(\tau,m)$ on
$(\bar{D}(0,\rho) \cup U_{d}) \times \mathbb{R}$, which are holomorphic with respect to $\tau$ on $D(0,\rho) \cup U_{d}$ and such
that
$$ ||h(\tau,m)||_{(\nu,\beta,\mu,k,\kappa)} =
\sup_{\tau \in \bar{D}(0,\rho) \cup U_{d},m \in \mathbb{R}} (1+|m|)^{\mu}
\frac{1 + |\tau|^{2k}}{|\tau|}\exp( \beta|m| - \nu|\tau|^{\kappa} ) |h(\tau,m)|$$
is finite. One can check that the normed space
$(F_{(\nu,\beta,\mu,k,\kappa)}^{d},||.||_{(\nu,\beta,\mu,k,\kappa)})$ is a Banach space.
\end{defin}
{\bf Remark:} These norms are appropriate modifications of those introduced in the work \cite{lama}, Section 2.\medskip

Throughout the whole subsection, we assume $\mu,\beta,\nu,\rho>0$, $k,\kappa \geq 1$ and $d \in \mathbb{R}$ are fixed. In the next
lemma, we check the continuity property under multiplication operation with bounded functions.

\begin{lemma} Let $(\tau,m) \mapsto a(\tau,m)$ be a bounded continuous function on
$(\bar{D}(0,\rho) \cup U_{d}) \times \mathbb{R}$ by a constant $C_{1}>0$. We assume that $a(\tau,m)$ is holomorphic with respect
to $\tau$ on $D(0,\rho) \cup U_{d}$. Then, we have
\begin{equation}
|| a(\tau,m) h(\tau,m) ||_{(\nu,\beta,\mu,k,\kappa)} \leq
C_{1} ||h(\tau,m)||_{(\nu,\beta,\mu,k,\kappa)}
\end{equation}
for all $h(\tau,m) \in F_{(\nu,\beta,\mu,k,\kappa)}^{d}$.
\end{lemma}

In the next proposition, we study the continuity property of some convolution operators acting on the latter Banach spaces.

\begin{prop} Let $\chi_{2}>-1$ be a real number. Let $\nu_{2} \geq -1$ be an integer. We assume that
$1 + \chi_{2} + \nu_{2} \geq 0$.\\

\noindent If $\kappa \geq k(\frac{\nu_{2}}{\chi_{2}+1} + 1)$, then there exists a constant $C_{2}>0$ (depending on
$\nu,\nu_{2},\chi_{2}$) such that
\begin{equation}
|| \int_{0}^{\tau^k} (\tau^{k}-s)^{\chi_2}s^{\nu_{2}}f(s^{1/k},m) ds ||_{(\nu,\beta,\mu,k,\kappa)}
\leq C_{2}
||f(\tau,m)||_{(\nu,\beta,\mu,k,\kappa)}
\label{conv_op_prod_s_continuity_1}
\end{equation}
for all $f(\tau,m) \in F_{(\nu,\beta,\mu,k,\kappa)}^{d}$.\medskip
\end{prop}
\begin{proof} Let
$f(\tau,m) \in F_{(\nu,\beta,\mu,k,\kappa)}^{d}$. By definition, we have
\begin{multline}
||\int_{0}^{\tau^k} (\tau^{k}-s)^{\chi_2}s^{\nu_2}f(s^{1/k},m) ds
||_{(\nu,\beta,\mu,k,\kappa)} \\ = \sup_{\tau \in \bar{D}(0,\rho) \cup U_{d},m \in \mathbb{R}} (1+|m|)^{\mu}
\frac{1 + |\tau|^{2k}}{|\tau|}\exp( \beta|m| - \nu|\tau|^{\kappa} )\\
\times |\int_{0}^{\tau^k} \{ (1+|m|)^{\mu} e^{\beta |m|}
\exp( -\nu |s|^{\kappa/k} )
\frac{1 + |s|^2}{ |s|^{1/k} }f(s^{1/k},m) \}\\
\times \mathcal{B}(\tau,s,m) ds|
\end{multline}
where
$$
\mathcal{B}(\tau,s,m) = \frac{1}{(1+|m|)^{\mu}} e^{-\beta |m|}
\frac{\exp( \nu |s|^{\kappa/k} )}{1 + |s|^2} |s|^{1/k}
(\tau^{k}-s)^{\chi_2}s^{\nu_2}.
$$
Therefore,
\begin{equation}
||\int_{0}^{\tau^k} (\tau^{k}-s)^{\chi_2}s^{\nu_2}f(s^{1/k},m) ds
||_{(\nu,\beta,\mu,k,\kappa)} \leq C_{2} ||f(\tau,m)||_{(\nu,\beta,\mu,k,\kappa)} \label{continuity_op_prod_s}
\end{equation}
where
\begin{equation}
 C_{2} = \sup_{\tau \in \bar{D}(0,\rho) \cup U_{d}}
\frac{1 + |\tau|^{2k}}{|\tau|} \exp( -\nu |\tau|^{\kappa} ) \int_{0}^{|\tau|^k}
\frac{\exp( \nu h^{\kappa/k} )}{1+h^2}
h^{\frac{1}{k}} (|\tau|^{k}-h)^{\chi_2} h^{\nu_2} dh
= \sup_{x \geq 0} B(x) \label{C_2}
\end{equation}
where
$$ B(x) = \frac{1 + x^2}{x^{1/k}} \exp( -\nu x^{\kappa/k} )
\int_{0}^{x} \frac{\exp( \nu h^{\kappa/k} )}{1 + h^2} h^{\frac{1}{k} + \nu_{2}} (x-h)^{\chi_2} dh.
$$
We write $B(x)=B_{1}(x) + B_{2}(x)$, where
\begin{multline*}
B_{1}(x) = \frac{1+x^2}{x^{1/k}}
\exp(-\nu x^{\kappa/k}) \int_{0}^{x/2} \frac{\exp( \nu h^{\kappa/k} )}{1 + h^2} h^{\frac{1}{k}+\nu_2} (x-h)^{\chi_{2}} dh,\\
B_{2}(x) = 
\frac{1+x^2}{x^{1/k}}
\exp(-\nu x^{\kappa/k}) \int_{x/2}^{x} \frac{\exp( \nu h^{\kappa/k} )}{1 + h^2} h^{\frac{1}{k}+\nu_2} (x-h)^{\chi_{2}} dh
\end{multline*}
Now, we study the function $B_{1}(x)$. We first assume that $-1 < \chi_{2} < 0$. In that case, we have that
$(x-h)^{\chi_2} \leq (x/2)^{\chi_2}$ for all $0 \leq h \leq x/2$ with $x>0$. Since $\nu_{2} \geq -1$, we deduce that
\begin{multline}
B_{1}(x) \leq \frac{1 + x^2}{x^{1/k}} (\frac{x}{2})^{\chi_2}e^{-\nu x^{\kappa/k}} \int_{0}^{x/2}
\frac{e^{\nu h^{\kappa/k}}}{1+h^2}
h^{\frac{1}{k}+\nu_2} dh\\
\leq (1 + x^2)\frac{1}{2^{1/k}(\frac{1}{k}+\nu_{2}+1)} (\frac{x}{2})^{1 + \chi_{2} + \nu_{2}}
\exp( -\nu (1 - \frac{1}{2^{\kappa/k}}) x^{\kappa/k} )
\end{multline}
for all $x > 0$. Since $\kappa \geq k$ and $1 + \chi_{2} + \nu_{2} \geq 0$, we deduce that there exists a constant $K_{1}>0$ with 
\begin{equation}
\sup_{x \geq 0} B_{1}(x) \leq K_{1}. \label{B_1_bounded_chi2_negative}
\end{equation}
We assume now that $\chi_{2} \geq 0$. In this situation, we know that $(x-h)^{\chi_2} \leq x^{\chi_2}$ for all $0 \leq h \leq x/2$, with
$x \geq 0$. Hence, since $\nu_{2} \geq -1$,
\begin{equation}
B_{1}(x) \leq (1+x^2)\frac{1}{2^{1/k}(\frac{1}{k}+\nu_{2}+1)}
x^{\chi_2}(x/2)^{\nu_{2}+1} \exp( -\nu (1 - \frac{1}{2^{\kappa/k}}) x^{\kappa/k} )
\end{equation}
for all $x \geq 0$. Again, we deduce that there exists a constant $K_{1.1}>0$ with 
\begin{equation}
\sup_{x \geq 0} B_{1}(x) \leq K_{1.1}. \label{B_1_bounded_chi2_positive}
\end{equation}
In the next step, we focus on the function $B_{2}(x)$. First, we observe that $1 + h^{2} \geq 1 + (x/2)^2$ for all
$x/2 \leq h \leq x$. Therefore, there exists a constant $K_{2}>0$ such that
\begin{multline}
B_{2}(x) \leq \frac{1+x^2}{1 + (\frac{x}{2})^2}\frac{1}{x^{1/k}}
\exp(-\nu x^{\kappa/k}) \int_{x/2}^{x} \exp( \nu h^{\kappa/k} ) h^{\frac{1}{k}+\nu_2} (x-h)^{\chi_{2}} dh\\
\leq K_{2} \frac{1}{x^{1/k}}\exp(-\nu x^{\kappa/k}) \int_{0}^{x} \exp( \nu h^{\kappa/k} )
h^{\frac{1}{k}+\nu_2} (x-h)^{\chi_{2}} dh
\label{B_2_maj}
\end{multline}
for all $x > 0$. It remains to study the function
$$ B_{2.1}(x) = \int_{0}^{x} \exp( \nu h^{\kappa/k} ) h^{\frac{1}{k}+\nu_2} (x-h)^{\chi_{2}} dh $$
for $x \geq 0$. By the uniform expansion $e^{\nu h^{\kappa/k}} = \sum_{n \geq 0} (\nu h^{\kappa/k})^{n}/n!$ on every compact
interval $[0,x]$, $x \geq 0$, we can write
\begin{equation}
B_{2.1}(x) = \sum_{n \geq 0} \frac{\nu^{n}}{n!} \int_{0}^{x} h^{\frac{n \kappa}{k} + \frac{1}{k}+\nu_{2}} (x-h)^{\chi_{2}} dh
\end{equation}
Using the Beta integral formula (see \cite{ba2}, Appendix B3) and since $\chi_{2} > -1$, $\frac{1}{k}+\nu_{2} > -1$, we can write
\begin{equation}
B_{2.1}(x) = \sum_{n \geq 0} \frac{\nu^{n}}{n!}
\frac{\Gamma(\chi_{2}+1)\Gamma( \frac{n \kappa}{k} + \frac{1}{k}+\nu_{2}+1 )}{\Gamma( \frac{n \kappa}{k} +
\frac{1}{k}+\nu_{2} + \chi_{2} + 2)}
x^{\frac{n \kappa}{k} + \frac{1}{k}+\nu_{2} + \chi_{2} + 1} \label{B_21_gamma}
\end{equation}
for all $x \geq 0$. Bearing in mind that
\begin{equation}
\Gamma(x)/\Gamma(x+a) \sim 1/x^{a} \label{quotient_gamma}
\end{equation}
as $x \rightarrow +\infty$, for any $a>0$ (see for instance, \cite{ba2}, Appendix B3), from (\ref{B_21_gamma}), we get
a constant $K_{2.1}>0$ such that
\begin{equation}
B_{2.1}(x) \leq K_{2.1} x^{\frac{1}{k}+\nu_{2}+\chi_{2}+1} \sum_{n \geq 0} \frac{1}{(n+1)^{\chi_{2}+1}n!}
(\nu x^{\kappa/k})^{n} \label{B_21_n_power}
\end{equation}
for all $x \geq 0$. Using again (\ref{quotient_gamma}), we know that $1/(n+1)^{\chi_{2}+1} \sim \Gamma(n+1)/\Gamma(n+\chi_{2}+2)$
as $n \rightarrow +\infty$. Hence, from (\ref{B_21_n_power}), there exists a constant $K_{2.2}>0$ such that
\begin{equation}
B_{2.1}(x) \leq K_{2.2}x^{\frac{1}{k}+\nu_{2} + \chi_{2} + 1} \sum_{n \geq 0}
\frac{1}{\Gamma(n + \chi_{2} +2)} (\nu x^{\kappa/k})^{n} \label{B_21_wiman}
\end{equation}
for all $x \geq 0$.

Remembering the asymptotic properties of the generalized Mittag-Leffler function (known as Wiman function in the literature)
$E_{\alpha,\beta}(z) = \sum_{n \geq 0} z^{n}/\Gamma(\beta + \alpha n)$, for any $\alpha,\beta >0$ (see \cite{ba2}, Appendix B4
or \cite{erd}, expansion (22) p. 210), we get from (\ref{B_21_wiman}) a constant $K_{2.3}>0$ such that
\begin{equation}
B_{2.1}(x) \leq K_{2.3}x^{\frac{1}{k}+\nu_{2} + \chi_{2}+1}x^{-\frac{\kappa}{k}(\chi_{2}+1)}e^{\nu x^{\kappa/k}}
\label{B_21_exp_growth}
\end{equation}
for all $x \geq 1$. Under the assumption that $\nu_{2} + \chi_{2} + 1 \leq \frac{\kappa}{k}(\chi_{2}+1)$ and gathering
(\ref{B_2_maj}), (\ref{B_21_exp_growth}), we get a constant $K_{2.4}>0$ such that
\begin{equation}
\sup_{x \geq 0} B_{2}(x) \leq K_{2.4}. \label{B_2_bounded_on_R}
\end{equation}
Finally, taking into account the estimates (\ref{continuity_op_prod_s}), (\ref{C_2}),
(\ref{B_1_bounded_chi2_negative}), (\ref{B_1_bounded_chi2_positive}), (\ref{B_2_bounded_on_R}), the
inequality (\ref{conv_op_prod_s_continuity_1}) follows.
\end{proof}
 
\begin{prop} Let $k,\kappa \geq 1$ be integers such that $\kappa \geq k$. Let $Q_{1}(X),Q_{2}(X),R(X) \in \mathbb{C}[X]$ such that
\begin{equation}
\mathrm{deg}(R) \geq \mathrm{deg}(Q_{1}) \ \ , \ \ \mathrm{deg}(R) \geq \mathrm{deg}(Q_{2}) \ \ , \ \ R(im) \neq 0
\label{R>Q1_R>Q2_R_nonzero_first}
\end{equation}
for all $m \in \mathbb{R}$. Assume that $\mu > \max( \mathrm{deg}(Q_{1})+1, \mathrm{deg}(Q_{2})+1 )$. Let
$m \mapsto b(m)$ be a continuous function on $\mathbb{R}$ such that
$$
|b(m)| \leq \frac{1}{|R(im)|}
$$
for all $m \in \mathbb{R}$. Then, there exists a constant $C_{3}>0$ (depending on $Q_{1},Q_{2},R,\mu,k,\kappa,\nu$) such that
\begin{multline}
|| b(m) \int_{0}^{\tau^{k}} (\tau^{k}-s)^{\frac{1}{k}} ( \int_{0}^{s} \int_{-\infty}^{+\infty} Q_{1}(i(m-m_{1}))
f((s-x)^{1/k},m-m_{1}) \\
\times Q_{2}(im_{1}) g(x^{1/k},m_{1}) \frac{1}{(s-x)x} dx dm_{1} ) ds ||_{(\nu,\beta,\mu,k,\kappa)} \\
\leq C_{3} ||f(\tau,m)||_{(\nu,\beta,\mu,k,\kappa)} ||g(\tau,m)||_{(\nu,\beta,\mu,k,\kappa)}
\label{norm_conv_f_g<norm_f_times_norm_g}
\end{multline}
for all $f(\tau,m), g(\tau,m) \in F_{(\nu,\beta,\mu,k,\kappa)}^{d}$.
\end{prop}
\begin{proof} Let $f(\tau,m), g(\tau,m) \in F_{(\nu,\beta,\mu,k,\kappa)}^{d}$. For any $\tau \in \bar{D}(0,\rho) \cup U_{d}$, the
segment $[0,\tau^{k}]$ is such that for any $s \in [0,\tau^{k}]$, any $x \in [0,s]$, the expressions $f((s-x)^{1/k},m-m_{1})$
and $g(x^{1/k},m_{1})$ are well defined, provided that $m,m_{1} \in \mathbb{R}$. By definition, we can write
\begin{multline*}
|| b(m) \int_{0}^{\tau^{k}} (\tau^{k}-s)^{\frac{1}{k}} ( \int_{0}^{s} \int_{-\infty}^{+\infty} Q_{1}(i(m-m_{1}))
f((s-x)^{1/k},m-m_{1}) \\
\times Q_{2}(im_{1}) g(x^{1/k},m_{1}) \frac{1}{(s-x)x} dx dm_{1} ) ds ||_{(\nu,\beta,\mu,k,\kappa)}\\
= \sup_{\tau \in \bar{D}(0,\rho) \cup U_{d},m \in \mathbb{R}} (1+|m|)^{\mu}
\frac{1 + |\tau|^{2k}}{|\tau|}\exp( \beta|m| - \nu|\tau|^{\kappa} )\\
\times | \int_{0}^{\tau^{k}} (\tau^{k}-s)^{1/k} ( \int_{0}^{s}
\int_{-\infty}^{+\infty} \{ (1+|m-m_{1}|)^{\mu} e^{\beta|m-m_{1}|}
\frac{1 + |s-x|^{2}}{|s-x|^{1/k}}
\exp( -\nu |s-x|^{\kappa/k} )\\
\times f( (s-x)^{1/k},m-m_{1}) \} \times \{ (1 + |m_{1}|)^{\mu}e^{\beta|m_{1}|}
\frac{1 + |x|^2 }{|x|^{1/k} } \exp( -\nu |x|^{\kappa/k} )
g(x^{1/k},m_{1}) \}\\
\times \mathcal{C}(s,x,m,m_{1}) dxdm_{1} ) ds |
\end{multline*}
where
\begin{multline*}
\mathcal{C}(s,x,m,m_{1}) = \frac{ \exp(-\beta|m_1|) \exp(-\beta|m-m_{1}|) }{(1 + |m-m_{1}|)^{\mu}(1+|m_{1}|)^{\mu}}
b(m)Q_{1}(i(m-m_{1}))Q_{2}(im_{1})\\
\times \frac{ |s-x|^{1/k}|x|^{1/k} }{(1 + |s-x|^2)(1 + |x|^2)}
\times 
\exp( \nu |s-x|^{\kappa/k} ) \exp( \nu |x|^{\kappa/k} ) \frac{1}{(s-x)x}
\end{multline*}
Now, we know that there exist $\mathfrak{Q}_{1},\mathfrak{Q}_{2},\mathfrak{R}>0$ with
\begin{multline}
|Q_{1}(i(m-m_{1}))| \leq \mathfrak{Q}_{1}(1 + |m-m_{1}|)^{\mathrm{deg}(Q_{1})} \ \ , \ \
|Q_{2}(im_{1})| \leq \mathfrak{Q}_{2}(1 + |m_{1}|)^{\mathrm{deg}(Q_{2})},\\
|R(im)| \geq \mathfrak{R}(1+|m|)^{\mathrm{deg}(R)} \label{Q1_Q2_R_deg_order}
\end{multline}
for all $m,m_{1} \in \mathbb{R}$. Therefore,
\begin{multline}
|| b(m) \int_{0}^{\tau^{k}} (\tau^{k}-s)^{\frac{1}{k}} ( \int_{0}^{s} \int_{-\infty}^{+\infty} Q_{1}(i(m-m_{1}))
f((s-x)^{1/k},m-m_{1})\\
\times Q_{2}(im_{1})g(x^{1/k},m_{1}) \frac{1}{(s-x)x} dx dm_{1} ) ds ||_{(\nu,\beta,\mu,k,\kappa)}\\
\leq C_{3}
||f(\tau,m)||_{(\nu,\beta,\mu,k,\kappa)} ||g(\tau,m)||_{(\nu,\beta,\mu,k,\kappa)}
\end{multline}
where
\begin{multline}
C_{3} = \sup_{\tau \in \bar{D}(0,\rho) \cup U_{d},m \in \mathbb{R}} (1+|m|)^{\mu}
\frac{1 + |\tau|^{2k}}{|\tau|}\exp( \beta|m| - \nu|\tau|^{\kappa} )
\frac{1}{\mathfrak{R}(1+|m|)^{\mathrm{deg}(R)}}\\
\times \int_{0}^{|\tau|^{k}} (|\tau|^{k}-h)^{1/k} (\int_{0}^{h}
\int_{-\infty}^{+\infty} \frac{ \exp(-\beta|m_1|) \exp(-\beta|m-m_{1}|) }{(1 + |m-m_{1}|)^{\mu}(1+|m_{1}|)^{\mu}}\\
\times
\mathfrak{Q}_{1}\mathfrak{Q}_{2}(1+|m-m_{1}|)^{\mathrm{deg}(Q_{1})}(1+|m_{1}|)^{\mathrm{deg}(Q_{2})}
\frac{ (h-x)^{1/k}x^{1/k} }{(1 + (h-x)^2)(1 + x^2)}\\
\times 
\exp( \nu (h-x)^{\kappa/k} ) \exp( \nu x^{\kappa/k} ) \frac{1}{(h-x)x} dx dm_{1}) dh
\label{defin_C_3}
\end{multline}
Now, since $\kappa \geq k$, we have that
\begin{equation}
h^{\kappa/k} \geq (h-x)^{\kappa/k} + x^{\kappa/k} \label{convex_power_h}
\end{equation}
for all $h \geq 0$, all $x \in [0,h]$. Indeed, let $x=hu$ where $u \in [0,1]$. Then, the inequality
(\ref{convex_power_h}) is equivalent to show that
\begin{equation}
1 \geq (1-u)^{\kappa/k} + u^{\kappa/k} \label{convex_power_u_unit}
\end{equation}
for all $u \in [0,1]$. Let $\varphi(u) = (1-u)^{\kappa/k} + u^{\kappa/k}$ on $[0,1]$. We have
$\varphi'(u) = \frac{\kappa}{k}( u^{\frac{\kappa}{k}-1} - (1-u)^{\frac{\kappa}{k}-1})$. Since, $\kappa \geq k$, we know that
the function $\psi(z) = z^{\frac{\kappa}{k}-1}$ is increasing on $[0,1]$, and therefore we get that $\varphi'(u) < 0$ if
$0 \leq u < 1/2$, $\varphi'(u)=0$, if $u=1/2$ and $\varphi'(u) > 0$ if $1/2 < u \leq 1$. Since $\varphi(0)=\varphi(1)=1$, we get
that $\varphi(u) \leq 1$ for all $u \in [0,1]$. Therefore, (\ref{convex_power_u_unit}) holds and (\ref{convex_power_h}) is proved.\medskip

Using the triangular inequality $|m| \leq |m_{1}| + |m-m_{1}|$, for all $m,m_{1} \in \mathbb{R}$, we get that
$C_{3} \leq C_{3.1}C_{3.2}$ where
\begin{equation}
C_{3.1} = \frac{\mathfrak{Q}_{1}\mathfrak{Q}_{2}}{\mathfrak{R}}
\sup_{m \in \mathbb{R}} (1+|m|)^{\mu - \mathrm{deg}(R)} \int_{-\infty}^{+\infty}
\frac{1}{(1+|m-m_{1}|)^{\mu - \mathrm{deg}(Q_{1})}(1+|m_{1}|)^{\mu - \mathrm{deg}(Q_{2})}} dm_{1} \label{defin_C_3.1} 
\end{equation}
which is finite whenever $\mu>\max( \mathrm{deg}(Q_{1})+1, \mathrm{deg}(Q_{2})+1)$ under the assumption
(\ref{R>Q1_R>Q2_R_nonzero_first}) using the same estimates as in Lemma 4 of \cite{ma2} (see also Lemma 2.2 from \cite{cota2}),
and where
\begin{multline}
C_{3.2} = \sup_{\tau \in \bar{D}(0,\rho) \cup U_{d}}
\frac{1 + |\tau|^{2k}}{|\tau|} \exp( -\nu |\tau|^{\kappa} )\\
\times 
\int_{0}^{|\tau|^k} (|\tau|^{k}-h)^{1/k} \exp( \nu h^{\kappa/k} ) \int_{0}^{h}
\frac{ (h-x)^{1/k}x^{1/k} }{(1 + (h-x)^2)(1 + x^2)}
\frac{1}{(h-x)x} dx dh. \label{defin_C3.2}
\end{multline}
From (\ref{defin_C3.2}) we get that $C_{3.2} \leq C_{3.3}$, where
\begin{multline}
C_{3.3} = \sup_{x \geq 0} \frac{1+x^2}{x^{1/k}} \exp( -\nu x^{\kappa/k} )
\int_{0}^{x} (x-h')^{1/k} \exp( \nu h'^{\kappa/k} )\\
\times (\int_{0}^{h'} \frac{1}{(1 + (h'-x')^2)(1+x'^{2})} \frac{1}{(h'-x')^{1-\frac{1}{k}}x'^{1-\frac{1}{k}}} dx') dh'
\label{defin_C3.3}
\end{multline}
By the change of variable $x'=h'u$, for $u \in [0,1]$, we can write
\begin{multline}
\int_{0}^{h'} \frac{1}{(1 + (h'-x')^2)(1+x'^{2})} \frac{1}{(h'-x')^{1-\frac{1}{k}}x'^{1-\frac{1}{k}}} dx'\\
= \frac{1}{h'^{1 - \frac{2}{k}}} \int_{0}^{1}
\frac{1}{(1 + h'^{2}(1-u)^{2})(1 + h'^{2}u^2)(1 - u)^{1 - \frac{1}{k}}u^{1 - \frac{1}{k}}} du = J_{k}(h')
\end{multline}
Using a partial fraction decomposition, we can split $J_{k}(h') = J_{1,k}(h') + J_{2,k}(h')$, where
\begin{multline}
J_{1,k}(h') = \frac{1}{h'^{1 - \frac{2}{k}}(h'^{2}+4)} \int_{0}^{1}
\frac{3-2u}{(1 + h'^{2}(1-u)^{2})(1-u)^{1-\frac{1}{k}}u^{1-\frac{1}{k}}} du \\
J_{2,k}(h') = \frac{1}{h'^{1 - \frac{2}{k}}(h'^{2}+4)} \int_{0}^{1}
\frac{2u+1}{(1 + h'^{2}u^{2})(1-u)^{1-\frac{1}{k}}u^{1-\frac{1}{k}}} du
\end{multline}
From now on, we assume that $k \geq 2$. By construction of $J_{1,k}(h')$ and $J_{2,k}(h')$, we see that there exists a constant
$j_{k}>0$ such that
\begin{equation}
J_{k}(h') \leq \frac{j_{k}}{h'^{1 - \frac{2}{k}}(h'^{2}+4)} \label{J_k_maj_frac}
\end{equation}
for all $h' > 0$. From (\ref{defin_C3.3}) and (\ref{J_k_maj_frac}), we deduce that $C_{3.3} \leq \sup_{x \geq 0} \tilde{C}_{3.3}(x)$,
where
\begin{equation}
\tilde{C}_{3.3}(x) = (1+x^2) \exp( -\nu x^{\kappa/k} )\int_{0}^{x} \frac{j_{k}
\exp( \nu h'^{\kappa/k} )}{h'^{1 - \frac{2}{k}}(h'^{2}+4)} dh'. \label{defin_tilde_C3.3}
\end{equation}
From L'Hospital rule, we know that
$$
\lim_{x \rightarrow +\infty} \tilde{C}_{3.3}(x) = \lim_{x \rightarrow +\infty} \frac{j_{k}}{x^{1-\frac{2}{k}}}
\frac{ \frac{(1+x^2)^2}{x^2+4}}{\nu \frac{\kappa}{k}x^{\frac{\kappa}{k}-1}(1+x^2) - 2x }
$$
which is finite if $\kappa \geq k$ and when $k \geq 2$. Therefore, we get a constant $\tilde{C}_{3.3}>0$ such that
\begin{equation}
\sup_{x \geq 0} \tilde{C}_{3.3}(x) \leq \tilde{C}_{3.3}. \label{maj_C_3.3}
\end{equation}
Taking into account the estimates for (\ref{defin_C_3}), (\ref{defin_C_3.1}), (\ref{defin_C3.2}), (\ref{defin_C3.3}),
(\ref{defin_tilde_C3.3}) and (\ref{maj_C_3.3}), we obtain the result (\ref{norm_conv_f_g<norm_f_times_norm_g}).\medskip

\noindent It remains to consider the case $k=1$. In that case, we know from Corollary 4.9 of \cite{cota} that there exists a
constant $j_{1}>0$ such that
\begin{equation}
J_{1}(h') \leq \frac{j_{1}}{h'^{2}+1} \label{J_1_maj_frac}
\end{equation}
for all $h' \geq 0$. From (\ref{defin_C3.3}) and (\ref{J_1_maj_frac}), we deduce that $C_{3.3} \leq \sup_{x \geq 0} \tilde{C}_{3.3.1}(x)$,
where
\begin{equation}
\tilde{C}_{3.3.1}(x) = (1+x^2) \exp( -\nu x^{\kappa} )\int_{0}^{x} \frac{j_{1}
\exp( \nu h'^{\kappa} )}{h'^{2}+1} dh'. \label{defin_tilde_C3.3.1}
\end{equation}
From L'Hospital rule, we know that
$$
\lim_{x \rightarrow +\infty}\tilde{C}_{3.3.1}(x) =
\lim_{x \rightarrow +\infty} \frac{ (1+x^2)j_{1} }{ \nu \kappa x^{\kappa-1}(1+x^2) -2x }
$$
which is finite whenever $\kappa \geq 1$. Therefore, we get a constant $\tilde{C}_{3.3.1}>0$ such that
\begin{equation}
\sup_{x \geq 0} \tilde{C}_{3.3.1}(x) \leq \tilde{C}_{3.3.1}. \label{maj_C_3.3.1}
\end{equation}
Taking into account the estimates for (\ref{defin_C_3}), (\ref{defin_C_3.1}), (\ref{defin_C3.2}), (\ref{defin_C3.3}),
(\ref{defin_tilde_C3.3.1}) and (\ref{maj_C_3.3.1}), we obtain the result (\ref{norm_conv_f_g<norm_f_times_norm_g}) for $k=1$.
\end{proof}

\begin{defin} Let $\beta, \mu \in \mathbb{R}$. We denote
$E_{(\beta,\mu)}$ the vector space of continuous functions $h : \mathbb{R} \rightarrow \mathbb{C}$ such that
$$ ||h(m)||_{(\beta,\mu)} = \sup_{m \in \mathbb{R}} (1+|m|)^{\mu} \exp( \beta |m|) |h(m)| $$
is finite. The space $E_{(\beta,\mu)}$ equipped with the norm $||.||_{(\beta,\mu)}$ is a Banach space.
\end{defin}

\begin{prop} Let $k,\kappa \geq 1$ be integers such that $\kappa \geq k$. Let $Q(X),R(X) \in \mathbb{C}[X]$ be polynomials such that
\begin{equation}
\mathrm{deg}(R) \geq \mathrm{deg}(Q) \ \ , \ \ R(im) \neq 0 \label{cond_R_Q_first}
\end{equation}
for all $m \in \mathbb{R}$. Assume that $\mu > \mathrm{deg}(Q) + 1$. Let $m \mapsto b(m)$ be a continuous function such that
$$
|b(m)| \leq \frac{1}{|R(im)|}
$$
for all $m \in \mathbb{R}$. Then, there exists a constant $C_{4}>0$ (depending on $Q,R,\mu,k,\kappa,\nu$) such that
\begin{multline}
|| b(m) \int_{0}^{\tau^{k}} (\tau^{k}-s)^{\frac{1}{k}}\int_{-\infty}^{+\infty} f(m-m_{1})Q(im_{1})g(s^{1/k},m_{1})dm_{1}
\frac{ds}{s}||_{(\nu,\beta,\mu,k,\kappa)}\\
\leq C_{4} ||f(m)||_{(\beta,\mu)} ||g(\tau,m)||_{(\nu,\beta,\mu,k,\kappa)}
\label{norm_conv_f_g<norm_f_beta_mu_times_norm_g}
\end{multline}
for all $f(m) \in E_{(\beta,\mu)}$, all $g(\tau,m) \in F_{(\nu,\beta,\mu,k,\kappa)}^{d}$.
\end{prop}
\begin{proof} The proof follows the same lines of arguments as those of Propositions 1 and 2. Let
$f(m) \in E_{(\beta,\mu)}$, $g(\tau,m) \in F_{(\nu,\beta,\mu,k,\kappa)}^{d}$. We can write
\begin{multline}
N_{2} := || b(m) \int_{0}^{\tau^{k}} (\tau^{k}-s)^{\frac{1}{k}}\int_{-\infty}^{+\infty} f(m-m_{1})Q(im_{1})g(s^{1/k},m_{1})dm_{1}
\frac{ds}{s}||_{(\nu,\beta,\mu,k,\kappa)} \\
= \sup_{\tau \in \bar{D}(0,\rho) \cup U_{d},m \in \mathbb{R}} (1+|m|)^{\mu}
\frac{1 + |\tau|^{2k}}{|\tau|}\exp( \beta|m| - \nu|\tau|^{\kappa} )\\
\times|b(m) \int_{0}^{\tau^k}\int_{-\infty}^{+\infty} \{ (1 + |m-m_{1}|)^{\mu}
\exp( \beta|m-m_{1}| )f(m-m_{1}) \} \\
\times \{ (1+|m_{1}|)^{\mu} \exp( \beta |m_{1}|) 
\exp( -\nu|s|^{\kappa/k} ) \frac{1 + |s|^2}{|s|^{1/k}}
g(s^{1/k},m_{1}) \} \times \mathcal{D}(\tau,s,m,m_{1}) dm_{1}ds |
\end{multline}
where
$$
\mathcal{D}(\tau,s,m,m_{1}) =
\frac{ Q(im_{1})e^{-\beta |m_{1}|} e ^{-\beta|m-m_{1}|} }{(1 + |m-m_{1}|)^{\mu} (1 + |m_{1}|)^{\mu}} \times
\frac{ \exp( \nu |s|^{\kappa/k} ) }{1 + |s|^2} |s|^{1/k}
(\tau^{k}-s)^{1/k} \frac{1}{s}
$$
Again, we know that there exist constants $\mathfrak{Q},\mathfrak{R}>0$ such that
$$ |Q(im_{1})| \leq \mathfrak{Q}(1+|m_{1}|)^{\mathrm{deg}(Q)} \ \ , \ \ |R(im)| \geq \mathfrak{R}(1+|m|)^{\mathrm{deg}(R)} $$
for all $m,m_{1} \in \mathbb{R}$. By means of the triangular inequality $|m| \leq |m_{1}|+|m-m_{1}|$, we get that
\begin{equation}
N_{2} \leq C_{4.1}C_{4.2}||f(m)||_{(\beta,\mu)}||g(\tau,m)||_{(\nu,\beta,\mu,k,\kappa)} \label{norm_b_int_conv_f_Q_g<}
\end{equation}
where
$$ C_{4.1} = \sup_{\tau \in \bar{D}(0,\rho) \cup U_{d}} \frac{1 + |\tau|^{2k}}{|\tau|}
\exp( - \nu|\tau|^{\kappa} ) \int_{0}^{|\tau|^k}
\frac{ \exp( \nu h^{\kappa/k}) }{1 + h^2}
h^{\frac{1}{k}-1} (|\tau|^{k} - h)^{1/k} dh
$$
and
$$
C_{4.2} = \frac{\mathfrak{Q}}{\mathfrak{R}} \sup_{m \in \mathbb{R}} (1 + |m|)^{\mu - \mathrm{deg}(R)}
\int_{-\infty}^{+\infty} \frac{1}{(1 + |m-m_{1}|)^{\mu}(1+|m_{1}|)^{\mu - \mathrm{deg}(Q)}} dm_{1}.
$$
Under the hypothesis $\kappa \geq k$ and from the estimates
(\ref{C_2}), (\ref{B_1_bounded_chi2_positive}) and (\ref{B_2_bounded_on_R}) in the special case
$\chi_{2}=1/k$ and $\nu_{2}=-1$, we know that $C_{4.1}$ is finite.

From the estimates for (\ref{defin_C_3.1}), we know that $C_{4.2}$ is finite under the assumption (\ref{cond_R_Q_first})
provided that
$\mu > \mathrm{deg}(Q)+1$. Finally, gathering these latter bound estimates together with (\ref{norm_b_int_conv_f_Q_g<})
yields the result (\ref{norm_conv_f_g<norm_f_beta_mu_times_norm_g}).
\end{proof}

In the next proposition, we recall from \cite{lama}, Proposition 5, that $(E_{(\beta,\mu)},||.||_{(\beta,\mu)})$ is a
Banach algebra for some noncommutative product $\star$ introduced below.

\begin{prop} Let $Q_{1}(X),Q_{2}(X),R(X) \in \mathbb{C}[X]$ be polynomials such that
\begin{equation}
\mathrm{deg}(R) \geq \mathrm{deg}(Q_{1}) \ \ , \ \ \mathrm{deg}(R) \geq \mathrm{deg}(Q_{2}) \ \ , \ \ R(im) \neq 0,
\label{cond_R_Q1_Q2}
\end{equation}
for all $m \in \mathbb{R}$. Assume that $\mu > \max( \mathrm{deg}(Q_{1})+1, \mathrm{deg}(Q_{2})+1 )$. Then, there exists a
constant $C_{5}>0$ (depending on $Q_{1},Q_{2},R,\mu$) such that
\begin{multline}
|| \frac{1}{R(im)} \int_{-\infty}^{+\infty} Q_{1}(i(m-m_{1})) f(m-m_{1}) Q_{2}(im_{1})g(m_{1}) dm_{1} ||_{(\beta,\mu)}\\
\leq C_{5} ||f(m)||_{(\beta,\mu)}||g(m)||_{(\beta,\mu)}
\end{multline}
for all $f(m),g(m) \in E_{(\beta,\mu)}$. Therefore, $(E_{(\beta,\mu)},||.||_{(\beta,\mu)})$ becomes a Banach algebra for the product
$\star$ defined by
$$ f \star g (m) = \frac{1}{R(im)} \int_{-\infty}^{+\infty} Q_{1}(i(m-m_{1})) f(m-m_{1}) Q_{2}(im_{1})g(m_{1}) dm_{1}.$$
As a particular case, when $f,g \in E_{(\beta,\mu)}$ with $\beta>0$ and $\mu>1$, the classical convolution
product
$$ f \ast g (m) = \int_{-\infty}^{+\infty} f(m-m_{1})g(m_{1}) dm_{1} $$
belongs to $E_{(\beta,\mu)}$.
\end{prop}

\subsection{Banach spaces of functions with exponential growth $k$ and decay of exponential order 1}

In this subsection, we mainly recall some functional properties of the Banach spaces already introduced in the work
\cite{lama}, Section 2. The Banach spaces we consider here coincide with those introduced in \cite{lama} except the fact
that they are not depending on a complex parameter $\epsilon$ and that the functions living in these spaces are not
holomorphic on a disc centered at 0 but only on a bounded
sector centered at 0. For this reason, all the propositions are given without proof except Proposition 5 which is an
improved version of Propositions 1 and 2 of \cite{lama}.

We denote $S_{d}^{b}$ an open bounded sector centered at $0$ in direction $d \in \mathbb{R}$ and $\bar{S}_{d}^{b}$
its closure. Let $S_{d}$ be an open unbounded sector in direction $d$. By convention, we recall that the sectors
we consider throughout the paper do not contain the origin in $\mathbb{C}$.

\begin{defin} Let $\nu,\beta,\mu>0$ be positive real numbers. Let $k \geq 1$ be an integer and let $d \in \mathbb{R}$. We denote
$F_{(\nu,\beta,\mu,k)}^{d}$ the vector space of continuous functions $(\tau,m) \mapsto h(\tau,m)$ on
$(\bar{S}_{d}^{b} \cup S_{d}) \times \mathbb{R}$, which are holomorphic with respect to $\tau$ on $S_{d}^{b} \cup S_{d}$ and
such that
$$ ||h(\tau,m)||_{(\nu,\beta,\mu,k)} =
\sup_{\tau \in \bar{S}_{d}^{b} \cup S_{d},m \in \mathbb{R}} (1+|m|)^{\mu}
\frac{1 + |\tau|^{2k}}{|\tau|}\exp( \beta|m| - \nu|\tau|^{k} ) |h(\tau,m)|$$
is finite. One can check that the normed space
$(F_{(\nu,\beta,\mu,k)}^{d},||.||_{(\nu,\beta,\mu,k)})$ is a Banach space.
\end{defin}

Throughout the whole subsection, we assume that $\mu,\beta,\nu>0$ and $k \geq 1$, $d \in \mathbb{R}$ are fixed. In the next lemma, we check
the continuity property by multiplication operation with bounded functions.

\begin{lemma} Let $(\tau,m) \mapsto a(\tau,m)$ be a bounded continuous function on
$(\bar{S}_{d}^{b} \cup S_{d}) \times \mathbb{R}$, which is holomorphic with respect to $\tau$ on $S_{d}^{b} \cup S_{d}$. Then, we have
\begin{equation}
|| a(\tau,m) h(\tau,m) ||_{(\nu,\beta,\mu,k)} \leq
\left( \sup_{\tau \in \bar{S}_{d}^{b} \cup S_{d},m \in \mathbb{R}} |a(\tau,m)| \right)
||h(\tau,m)||_{(\nu,\beta,\mu,k)}
\end{equation}
for all $h(\tau,m) \in F_{(\nu,\beta,\mu,k)}^{d}$.
\end{lemma}

In the next proposition, we study the continuity property of some convolution operators acting on the latter Banach spaces.

\begin{prop} Let $\gamma_{1} \geq 0$ and $\chi_{2}>-1$ be real numbers. Let $\nu_{2} \geq -1$ be an integer.
We consider a holomorphic function $a_{\gamma_{1},k}(\tau)$ on $S_{d}^{b} \cup S_{d}$, continuous on
$\bar{S}_{d}^{b} \cup S_{d}$, such that
$$ |a_{\gamma_{1},k}(\tau)| \leq \frac{1}{(1+|\tau|^{k})^{\gamma_1}} $$
for all $\tau \in S_{d}^{b} \cup S_{d}$.\medskip

\noindent If $1 + \chi_{2} + \nu_{2} \geq 0$ and $\gamma_{1} \geq \nu_{2}$, then there exists a constant $C_{6}>0$ (depending on
$\nu,\nu_{2},\chi_{2},\gamma_{1}$) such that
\begin{equation}
|| a_{\gamma_{1},k}(\tau) \int_{0}^{\tau^k} (\tau^{k}-s)^{\chi_2}s^{\nu_{2}}f(s^{1/k},m) ds ||_{(\nu,\beta,\mu,k)} \\
\leq C_{6}||f(\tau,m)||_{(\nu,\beta,\mu,k)}
\label{conv_op_prod_s_continuity_1_a_gamma_1}
\end{equation}
for all $f(\tau,m) \in F_{(\nu,\beta,\mu,k)}^{d}$.\medskip
\end{prop}
\begin{proof} The proof follows similar arguments to those in Proposition 1. Indeed, let
$f(\tau,m) \in F_{(\nu,\beta,\mu,k)}^{d}$. By definition, we have
\begin{multline}
|| a_{\gamma_{1},k}(\tau) \int_{0}^{\tau^k} (\tau^{k}-s)^{\chi_2}s^{\nu_2}f(s^{1/k},m) ds
||_{(\nu,\beta,\mu,k)} \\ = \sup_{\tau \in \bar{S}_{d}^{b} \cup S_{d},m \in \mathbb{R}} (1+|m|)^{\mu}
\frac{1 + |\tau|^{2k}}{|\tau|}\exp( \beta|m| - \nu|\tau|^{k} )\\
\times | a_{\gamma_{1},k}(\tau) \int_{0}^{\tau^k} \{ (1+|m|)^{\mu} e^{\beta |m|}
\exp( -\nu |s| )
\frac{ 1 + |s|^2}{ |s|^{1/k} }f(s^{1/k},m) \}\\
\times \mathcal{F}(\tau,s,m) ds|
\end{multline}
where
$$
\mathcal{F}(\tau,s,m) = \frac{1}{(1+|m|)^{\mu}} e^{-\beta |m|}
\frac{\exp( \nu |s| )}{1 + |s|^2 } |s|^{1/k}
(\tau^{k}-s)^{\chi_2}s^{\nu_2}.
$$
Therefore,
\begin{equation}
|| a_{\gamma_{1},k}(\tau) \int_{0}^{\tau^k} (\tau^{k}-s)^{\chi_2}s^{\nu_2}f(s^{1/k},m) ds
||_{(\nu,\beta,\mu,k)} \leq C_{6} ||f(\tau,m)||_{(\nu,\beta,\mu,k)} \label{continuity_op_prod_s_gamma_1}
\end{equation}
where
\begin{multline*}
 C_{6} = \sup_{\tau \in \bar{S}_{d}^{b} \cup S_{d}}
\frac{1 + |\tau|^{2k}}{|\tau|} \exp( -\nu |\tau|^{k} )\\
\times \frac{1}{(1+|\tau|^{k})^{\gamma_1}} \int_{0}^{|\tau|^k}
\frac{\exp( \nu h )}{1+h^2}
h^{\frac{1}{k}} (|\tau|^{k}-h)^{\chi_2} h^{\nu_2} dh = \sup_{x \geq 0} F(x)
\end{multline*}
where
$$ F(x) = \frac{1 + x^2}{x^{1/k}} \exp( -\nu x ) \frac{1}{(1+x)^{\gamma_1}}
\int_{0}^{x} \frac{\exp( \nu h )}{1 + h^2} h^{\frac{1}{k} + \nu_{2}} (x-h)^{\chi_2} dh.
$$
We write $F(x)=F_{1}(x) + F_{2}(x)$, where
\begin{multline*}
F_{1}(x) = \frac{1+x^2}{x^{1/k}}
\exp(-\nu x) \frac{1}{(1+x)^{\gamma_1}}
\int_{0}^{x/2} \frac{\exp( \nu h )}{1 + h^2} h^{\frac{1}{k}+\nu_2} (x-h)^{\chi_{2}} dh,\\
F_{2}(x) = 
\frac{1+x^2}{x^{1/k}}
\exp(-\nu x) \frac{1}{(1+x)^{\gamma_1}}
\int_{x/2}^{x} \frac{\exp( \nu h )}{1 + h^2} h^{\frac{1}{k}+\nu_2} (x-h)^{\chi_{2}} dh
\end{multline*}
Now, we study the function $F_{1}(x)$. We first assume that $-1 < \chi_{2} < 0$. In that case, we have that
$(x-h)^{\chi_2} \leq (x/2)^{\chi_2}$ for all $0 \leq h \leq x/2$ with $x>0$. We deduce that
\begin{multline}
F_{1}(x) \leq \frac{1 + x^2}{x^{1/k}} (\frac{x}{2})^{\chi_2} e^{-\nu x} \frac{1}{(1+x)^{\gamma_1}} \int_{0}^{x/2}
\frac{e^{\nu h}}{1+h^2}
h^{\frac{1}{k}+\nu_2} dh\\
\leq (1 + x^2)\frac{1}{2^{1/k}(\frac{1}{k}+\nu_{2}+1)} (\frac{x}{2})^{1 + \chi_{2} + \nu_{2}}
\frac{1}{(1+x)^{\gamma_1}} \exp(-\frac{\nu x}{2})
\end{multline}
for all $x > 0$. Bearing in mind that $1 + \chi_{2} + \nu_{2} \geq 0$ and since
$1 + x \geq 1$ for all $x \geq 0$, we deduce that there exists a constant $K_{1}>0$ with 
\begin{equation}
\sup_{x \geq 0} F_{1}(x) \leq K_{1}. \label{F_1_bounded_chi2_negative}
\end{equation}
We assume now that $\chi_{2} \geq 0$. In this situation, we know that $(x-h)^{\chi_2} \leq x^{\chi_2}$ for all $0 \leq h \leq x/2$,
with $x \geq 0$. Hence,
\begin{equation}
F_{1}(x) \leq (1+x^2)\frac{1}{2^{1/k}(\frac{1}{k}+\nu_{2}+1)}x^{\chi_2}
(x/2)^{\nu_{2}+1} \frac{1}{(1+x)^{\gamma_1}} \exp( -\frac{\nu x}{2} )
\end{equation}
for all $x \geq 0$. Again, we deduce that there exists a constant $K_{1.1}>0$ with 
\begin{equation}
\sup_{x \geq 0} F_{1}(x) \leq K_{1.1}. \label{F_1_bounded_chi2_positive}
\end{equation}
In the next step, we focus on the function $F_{2}(x)$. First, we observe that $1 + h^{2} \geq 1 + (x/2)^2$ for all
$x/2 \leq h \leq x$. Therefore, there exists a constant $K_{2}>0$ such that
\begin{multline}
F_{2}(x) \leq \frac{1+x^2}{1 + (\frac{x}{2})^2}\frac{1}{x^{1/k}}
\exp(-\nu x) \frac{1}{(1+x)^{\gamma_1}}
\int_{x/2}^{x} \exp( \nu h ) h^{\frac{1}{k}+\nu_2} (x-h)^{\chi_{2}} dh\\
\leq K_{2} \frac{1}{x^{1/k}} \frac{1}{(1+x)^{\gamma_1}} \exp(-\nu x) \int_{0}^{x} \exp( \nu h )
h^{\frac{1}{k}+\nu_2} (x-h)^{\chi_{2}} dh
\label{F_2_maj}
\end{multline}
for all $x > 0$. Now, from the estimates (\ref{B_21_exp_growth}), we know that there exists a constant $K_{2.3}>0$ such that
\begin{equation}
F_{2.1}(x) = \int_{0}^{x} \exp( \nu h )
h^{\frac{1}{k}+\nu_2} (x-h)^{\chi_{2}} dh \leq K_{2.3}x^{\frac{1}{k} + \nu_{2}} e^{\nu x} \label{F_2_1_maj}
\end{equation}
for all $x \geq 1$. From (\ref{F_2_maj}) we get the existence of a constant $\tilde{F}_{2}>0$ with
\begin{equation}
\sup_{x \in [0,1]} F_{2}(x) \leq \tilde{F}_{2}. \label{F_2_maj_1}
\end{equation}
On the other hand, we also have that $1 + x \geq x$ for all $x \geq 1$. Since $\gamma_{1} \geq \nu_{2}$ and due
to (\ref{F_2_maj}) with (\ref{F_2_1_maj}), we get a constant $\check{F}_{2}>0$ with
\begin{equation}
\sup_{x \geq 1} F_{2}(x) \leq \check{F}_{2} \label{F_2_maj_2}
\end{equation}
Gathering the estimates (\ref{continuity_op_prod_s_gamma_1}), (\ref{F_1_bounded_chi2_negative}),
(\ref{F_1_bounded_chi2_positive}), (\ref{F_2_maj_1}) and (\ref{F_2_maj_2}), we finally obtain
(\ref{conv_op_prod_s_continuity_1_a_gamma_1}).
\end{proof}
The next two propositions are already stated as Propositions 3 and 4 in \cite{lama}.
\begin{prop} Let $k \geq 1$ be an integer. Let $Q_{1}(X),Q_{2}(X),R(X) \in \mathbb{C}[X]$ such that
\begin{equation}
\mathrm{deg}(R) \geq \mathrm{deg}(Q_{1}) \ \ , \ \ \mathrm{deg}(R) \geq \mathrm{deg}(Q_{2}) \ \ , \ \ R(im) \neq 0
\label{R>Q1_R>Q2_R_nonzero}
\end{equation}
for all $m \in \mathbb{R}$. Assume that $\mu > \max( \mathrm{deg}(Q_{1})+1, \mathrm{deg}(Q_{2})+1 )$. Let
$m \mapsto b(m)$ be a continuous function on $\mathbb{R}$ such that
$$
|b(m)| \leq \frac{1}{|R(im)|}
$$
for all $m \in \mathbb{R}$. Then, there exists a constant $C_{7}>0$ (depending on $Q_{1},Q_{2},R,\mu,k,\nu$) such that
\begin{multline}
|| b(m) \int_{0}^{\tau^{k}} (\tau^{k}-s)^{\frac{1}{k}} ( \int_{0}^{s} \int_{-\infty}^{+\infty} Q_{1}(i(m-m_{1}))
f((s-x)^{1/k},m-m_{1}) \\
\times Q_{2}(im_{1}) g(x^{1/k},m_{1}) \frac{1}{(s-x)x} dx dm_{1} ) ds ||_{(\nu,\beta,\mu,k)} \\
\leq C_{7} ||f(\tau,m)||_{(\nu,\beta,\mu,k)} ||g(\tau,m)||_{(\nu,\beta,\mu,k)}
\label{norm_k_conv_f_g<norm_k_f_times_norm_k_g}
\end{multline}
for all $f(\tau,m), g(\tau,m) \in F_{(\nu,\beta,\mu,k)}^{d}$.
\end{prop}

\begin{prop} Let $k \geq 1$ be an integer. Let $Q(X),R(X) \in \mathbb{C}[X]$ be polynomials such that
\begin{equation}
\mathrm{deg}(R) \geq \mathrm{deg}(Q) \ \ , \ \ R(im) \neq 0 \label{cond_R_Q}
\end{equation}
for all $m \in \mathbb{R}$. Assume that $\mu > \mathrm{deg}(Q) + 1$. Let $m \mapsto b(m)$ be a continuous function such that
$$
|b(m)| \leq \frac{1}{|R(im)|}
$$
for all $m \in \mathbb{R}$. Then, there exists a constant $C_{8}>0$ (depending on $Q,R,\mu,k,\nu$) such that
\begin{multline}
|| b(m) \int_{0}^{\tau^{k}} (\tau^{k}-s)^{\frac{1}{k}}\int_{-\infty}^{+\infty} f(m-m_{1})Q(im_{1})g(s^{1/k},m_{1})dm_{1}
\frac{ds}{s}||_{(\nu,\beta,\mu,k)}\\
\leq C_{8} ||f(m)||_{(\beta,\mu)} ||g(\tau,m)||_{(\nu,\beta,\mu,k)}
\label{norm_k_conv_f_g<norm_f_beta_mu_times_norm_k_g}
\end{multline}
for all $f(m) \in E_{(\beta,\mu)}$, all $g(\tau,m) \in F_{(\nu,\beta,\mu,k)}^{d}$.
\end{prop}

\section{Laplace transform, asymptotic expansions and Fourier transform}

We recall a definition of $k-$Borel summability of formal series with coefficients in a Banach space which is a slightly modified version of
the one given in \cite{ba}, Section 3.2, that was introduced in \cite{lama}. All the properties stated in this section are
already contained in our previous work \cite{lama}.

\begin{defin} Let $k \geq 1$ be an integer. Let $m_{k}(n)$ be the sequence defined by
$$ m_{k}(n) = \Gamma(\frac{n}{k}) =  \int_{0}^{+\infty} t^{\frac{n}{k}-1} e^{-t} dt $$
for all $n \geq 1$. A formal series
$$\hat{X}(T) = \sum_{n=1}^{\infty}  a_{n}T^{n} \in T\mathbb{E}[[T]]$$
with coefficients in a Banach space $( \mathbb{E}, ||.||_{\mathbb{E}} )$ is said to be $m_{k}-$summable
with respect to $T$ in the direction $d \in [0,2\pi)$ if \medskip

{\bf i)} there exists $\rho \in \mathbb{R}_{+}$ such that the following formal series, called a formal $m_{k}-$Borel transform of
$\hat{X}$ 
$$ \mathcal{B}_{m_k}(\hat{X})(\tau) = \sum_{n=1}^{\infty} \frac{ a_{n} }{ \Gamma(\frac{n}{k}) } \tau^{n}
\in \tau\mathbb{E}[[\tau]],$$
is absolutely convergent for $|\tau| < \rho$. \medskip

{\bf ii)} there exists $\delta > 0$ such that the series $\mathcal{B}_{m_k}(\hat{X})(\tau)$ can be analytically continued with
respect to $\tau$ in a sector
$S_{d,\delta} = \{ \tau \in \mathbb{C}^{\ast} : |d - \mathrm{arg}(\tau) | < \delta \} $. Moreover, there exist $C >0$ and $K >0$
such that
$$ ||\mathcal{B}_{m_k}(\hat{X})(\tau)||_{\mathbb{E}}
\leq C e^{ K|\tau|^{k} } $$
for all $\tau \in S_{d, \delta}$.
\end{defin}
If this is so, the vector valued $m_{k}-$Laplace transform of $\mathcal{B}_{m_{k}}(\hat{X})(\tau)$ in the direction $d$ is
defined by
$$ \mathcal{L}^{d}_{m_k}(\mathcal{B}_{m_k}(\hat{X}))(T) = k \int_{L_{\gamma}}
\mathcal{B}_{m_k}(\hat{X})(u) e^{ - ( u/T )^{k} } \frac{d u}{u},$$
along a half-line $L_{\gamma} = \mathbb{R}_{+}e^{i\gamma} \subset S_{d,\delta} \cup \{ 0 \}$, where $\gamma$ depends on
$T$ and is chosen in such a way that $\cos(k(\gamma - \mathrm{arg}(T))) \geq \delta_{1} > 0$, for some fixed $\delta_{1}$.
The function $\mathcal{L}^{d}_{m_k}(\mathcal{B}_{m_{k}}(\hat{X}))(T)$ is well defined, holomorphic and bounded in any sector
$$ S_{d,\theta,R^{1/k}} = \{ T \in \mathbb{C}^{\ast} : |T| < R^{1/k} \ \ , \ \ |d - \mathrm{arg}(T) | < \theta/2 \},$$
where $\frac{\pi}{k} < \theta < \frac{\pi}{k} + 2\delta$ and
$0 < R < \delta_{1}/K$. This function is called the $m_{k}-$sum of the formal series $\hat{X}(T)$ in the direction $d$.\medskip

\noindent We now state some elementary properties concerning the $m_{k}-$sums of formal power series.\\

\noindent 1) The function $\mathcal{L}^{d}_{m_k}(\mathcal{B}_{m_k}(\hat{X}))(T)$ has the formal series $\hat{X}(T)$ as
Gevrey asymptotic
expansion of order $1/k$ with respect to $t$ on $S_{d,\theta,R^{1/k}}$. This means that for all
$\frac{\pi}{k} < \theta_{1} < \theta$, there exist $C,M > 0$
such that
\begin{equation}
 ||\mathcal{L}^{d}_{m_k}(\mathcal{B}_{m_k}(\hat{X}))(T) - \sum_{p=1}^{n-1} a_p T^{p}||_{\mathbb{E}} \leq
CM^{n}\Gamma(1+\frac{n}{k})|T|^{n} \label{Laplace_k_Gevrey_ae}
\end{equation}
for all $n \geq 2$, all $T \in S_{d,\theta_{1},R^{1/k}}$. Moreover, from Watson's lemma (see Proposition 11 p. 75 in \cite{ba2}), we get
that $\mathcal{L}^{d}_{m_k}(\mathcal{B}_{m_k}(\hat{X}))(T)$ is the unique holomorphic function that satisfies the estimates
(\ref{Laplace_k_Gevrey_ae}) on the sectors $S_{d,\theta_{1},R^{1/k}}$ with large aperture $\theta_{1} > \frac{\pi}{k}$.\medskip

\noindent 2) Let us assume that $( \mathbb{E}, ||.||_{\mathbb{E}} )$ also has the structure of a Banach algebra for a product $\star$.
Let $\hat{X}_{1}(T),\hat{X}_{2}(T) \in T\mathbb{E}[[T]]$ be $m_{k}-$summable formal power series in direction
$d$. Let $q_{1} \geq q_{2} \geq 1$ be integers. We assume that 
$\hat{X}_{1}(T)+\hat{X}_{2}(T)$, $\hat{X}_{1}(T) \star \hat{X}_{2}(T)$ and
$T^{q_1}\partial_{T}^{q_2}\hat{X}_{1}(T)$, which are elements of $T\mathbb{E}[[T]]$, are $m_{k}-$summable in direction $d$.
Then, the following equalities
\begin{multline}
\mathcal{L}^{d}_{m_k}(\mathcal{B}_{m_k}(\hat{X}_{1}))(T) +
\mathcal{L}^{d}_{m_k}(\mathcal{B}_{m_k}(\hat{X}_{2}))(T) =
\mathcal{L}^{d}_{m_k}(\mathcal{B}_{m_k}(\hat{X}_{1} + \hat{X}_{2}))(T),\\
\mathcal{L}^{d}_{m_k}(\mathcal{B}_{m_k}(\hat{X}_{1}))(T) \star
\mathcal{L}^{d}_{m_k}(\mathcal{B}_{m_k}(\hat{X}_{2}))(T) =
\mathcal{L}^{d}_{m_k}(\mathcal{B}_{m_k}(\hat{X}_{1} \star \hat{X}_{2}))(T)\\
T^{q_1}\partial_{T}^{q_2}\mathcal{L}^{d}_{m_k}(\mathcal{B}_{m_k}(\hat{X}_{1}))(T) =
\mathcal{L}^{d}_{m_k}(\mathcal{B}_{m_k}(T^{q_1}\partial_{T}^{q_2}\hat{X}_{1}))(T) \label{sum_prod_deriv_m_k_sum}
\end{multline}
hold for all $T \in S_{d,\theta,R^{1/k}}$. These equalities are consequence of the unicity of the function having a given
Gevrey expansion of order $1/k$ in large sectors as stated above in 1) and from the fact that the set of holomorphic functions having
Gevrey asymptotic expansion of order $1/k$ on a sector with values in the Banach algebra $\mathbb{E}$ form a
differential algebra (meaning that this set is stable with respect to the sum and product of functions and derivation in the variable
$T$) (see Theorems 18,19 and 20 in \cite{ba2}).\medskip

In the next proposition, we give some identities for the $m_{k}-$Borel transform that will be useful in the sequel.
\begin{prop} Let $\hat{f}(t) = \sum_{ n \geq 1} f_{n}t^{n}$, $\hat{g}(t) = \sum_{n \geq 1} g_{n}t^{n}$ be formal series
whose coefficients $f_{n},g_{n}$ belong to some Banach space $(\mathbb{E},||.||_{\mathbb{E}})$. We assume that
$(\mathbb{E},||.||_{\mathbb{E}})$ is a Banach algebra for some product $\star$. Let $k,m \geq 1$ be integers.
The following formal identities hold.
\begin{equation}
\mathcal{B}_{m_k}(t^{k+1}\partial_{t}\hat{f}(t))(\tau) = k \tau^{k} \mathcal{B}_{m_k}(\hat{f}(t))(\tau) \label{Borel_diff}
\end{equation}
\begin{equation}
\mathcal{B}_{m_k}(t^{m}\hat{f}(t))(\tau) = \frac{\tau^{k}}{\Gamma(\frac{m}{k})}
\int_{0}^{\tau^{k}} (\tau^{k} - s)^{\frac{m}{k}-1} \mathcal{B}_{m_k}(\hat{f}(t))(s^{1/k}) \frac{ds}{s} \label{Borel_mult_monom}
\end{equation}
and
\begin{equation}
\mathcal{B}_{m_k}( \hat{f}(t) \star \hat{g}(t) )(\tau) = \tau^{k}\int_{0}^{\tau^{k}}
\mathcal{B}_{m_k}(\hat{f}(t))((\tau^{k}-s)^{1/k}) \star \mathcal{B}_{m_k}(\hat{g}(t))(s^{1/k}) \frac{1}{(\tau^{k}-s)s} ds
\label{Borel_product}
\end{equation}
\end{prop}
In the following proposition, we recall some properties of the inverse Fourier transform
\begin{prop}
Let $f \in E_{(\beta,\mu)}$ with $\beta > 0$, $\mu > 1$. The inverse Fourier transform of $f$ is defined by
$$ \mathcal{F}^{-1}(f)(x) = \frac{1}{ (2\pi)^{1/2} } \int_{-\infty}^{+\infty} f(m) \exp( ixm ) dm $$
for all $x \in \mathbb{R}$. The function $\mathcal{F}^{-1}(f)$ extends to an analytic function on the strip
\begin{equation}
H_{\beta} = \{ z \in \mathbb{C} / |\mathrm{Im}(z)| < \beta \}. \label{strip_H_beta}
\end{equation}
Let $\phi(m) = im f(m) \in E_{(\beta,\mu - 1)}$. Then, we have
\begin{equation}
\partial_{z} \mathcal{F}^{-1}(f)(z) = \mathcal{F}^{-1}(\phi)(z) \label{dz_fourier}
\end{equation}
for all $z \in H_{\beta}$.\\
Let $g \in E_{(\beta,\mu)}$ and let $\psi(m) = \frac{1}{(2\pi)^{1/2}}f \ast g(m)$, the convolution product of $f$ and
$g$, for all $m \in \mathbb{R}$.
From Proposition 4, we know that $\psi \in E_{(\beta,\mu)}$. Moreover, we have
\begin{equation}
\mathcal{F}^{-1}(f)(z)\mathcal{F}^{-1}(g)(z) = \mathcal{F}^{-1}(\psi)(z) \label{prod_fourier}
\end{equation}
for all $z \in H_{\beta}$.
\end{prop}

\section{Formal and analytic solutions of convolution initial value problems with complex parameters}

\subsection{Formal solutions of the main convolution initial value problem}

Let $k_{1},k_{2} \geq 1$, $D \geq 2$ be integers such that $k_{2} > k_{1}$. Let $\delta_{l} \geq 1$ be integers such that
\begin{equation}
1 = \delta_{1} \ \ , \ \ \delta_{l} < \delta_{l+1}, \label{delta_constraints}
\end{equation}
for all $1 \leq l \leq D-1$. For all $1 \leq l \leq D-1$, let
$d_{l},\Delta_{l} \geq 0$ be nonnegative integers such that
\begin{equation}
d_{l} > \delta_{l} \ \ , \ \ \Delta_{l} - d_{l} + \delta_{l} - 1 \geq 0. \label{constrain_d_l_delta_l_Delta_l}
\end{equation}
Let
$Q(X),Q_{1}(X),Q_{2}(X),R_{l}(X) \in \mathbb{C}[X]$, $0 \leq l \leq D$, be polynomials such that
\begin{multline}
\mathrm{deg}(Q) \geq \mathrm{deg}(R_{D}) \geq \mathrm{deg}(R_{l}) \ \ , \ \ \mathrm{deg}(R_{D}) \geq \mathrm{deg}(Q_{1}) \ \ , \ \
\mathrm{deg}(R_{D}) \geq \mathrm{deg}(Q_{2}), \\
Q(im) \neq 0 \ \ , \ \ R_{l}(im) \neq 0 \ \ , \ \ R_{D}(im) \neq 0 \label{first_constraints_polynomials_Q_R}
\end{multline}
for all $m \in \mathbb{R}$, all $0 \leq l \leq D-1$. We consider sequences of functions
$m \mapsto C_{0,n}(m,\epsilon)$, for all $n \geq 0$ and $m \mapsto F_{n}(m,\epsilon)$, for all $n \geq 1$,
that belong to the Banach space $E_{(\beta,\mu)}$ for some $\beta > 0$ and
$\mu > \max( \mathrm{deg}(Q_{1})+1, \mathrm{deg}(Q_{2})+1)$ and which
depend holomorphically on $\epsilon \in D(0,\epsilon_{0})$ for some $\epsilon_{0}>0$. We assume that there exist
constants $K_{0},T_{0}>0$ such that
\begin{equation}
||C_{0,n}(m,\epsilon)|| _{(\beta,\mu)} \leq K_{0} (\frac{1}{T_{0}})^{n}  \label{norm_beta_mu_C0_n}
\end{equation}
for all $n \geq 1$, for all $\epsilon \in D(0,\epsilon_{0})$. We define $C_{0}(T,m,\epsilon) =
\sum_{n \geq 1} C_{0,n}(m,\epsilon) T^{n}$
which is a convergent series on $D(0,T_{0}/2)$ with values in $E_{(\beta,\mu)}$ and
$F(T,m,\epsilon) = \sum_{n \geq 1}F_{n}(m,\epsilon)T^{n}$ which is a formal series with coefficients in $E_{(\beta,\mu)}$.
Let $c_{1,2}(\epsilon),c_{0}(\epsilon),c_{0,0}(\epsilon)$ and $c_{F}(\epsilon)$ be bounded holomorphic functions
on $D(0,\epsilon_{0})$ which vanish at the origin $\epsilon=0$. We consider the following initial value problem
\begin{multline}
Q(im)(\partial_{T}U(T,m,\epsilon) ) - T^{(\delta_{D}-1)(k_{2}+1)}\partial_{T}^{\delta_{D}}R_{D}(im)U(T,m,\epsilon) \\
= \epsilon^{-1}\frac{c_{1,2}(\epsilon)}{(2\pi)^{1/2}}\int_{-\infty}^{+\infty}Q_{1}(i(m-m_{1}))U(T,m-m_{1},\epsilon)
Q_{2}(im_{1})U(T,m_{1},\epsilon) dm_{1}\\
+ \sum_{l=1}^{D-1} R_{l}(im) \epsilon^{\Delta_{l} - d_{l} + \delta_{l} - 1} T^{d_{l}} \partial_{T}^{\delta_l}U(T,m,\epsilon)\\
+ \epsilon^{-1}\frac{c_{0}(\epsilon)}{(2\pi)^{1/2}}\int_{-\infty}^{+\infty}
C_{0}(T,m-m_{1},\epsilon)R_{0}(im_{1})U(T,m_{1},\epsilon) dm_{1}\\
+  \epsilon^{-1}\frac{c_{0,0}(\epsilon)}{(2\pi)^{1/2}}\int_{-\infty}^{+\infty}C_{0,0}(m-m_{1},\epsilon)
R_{0}(im_{1})U(T,m_{1},\epsilon) dm_{1}
+ \epsilon^{-1}c_{F}(\epsilon)F(T,m,\epsilon)
\label{SCP}
\end{multline}
for given initial data $U(0,m,\epsilon) \equiv 0$.\medskip

\begin{prop} There exists a unique formal series
$$ \hat{U}(T,m,\epsilon) = \sum_{n \geq 1} U_{n}(m,\epsilon) T^{n} $$
solution of (\ref{SCP}) with initial data $U(0,m,\epsilon) \equiv 0$, where the coefficients
$m \mapsto U_{n}(m,\epsilon)$ belong to $E_{(\beta,\mu)}$ for $\beta>0$ and
$\mu>\max( \mathrm{deg}(Q_{1}) + 1, \mathrm{deg}(Q_{2})+1)$ given above and depend
holomorphically on $\epsilon$ in $D(0,\epsilon_{0})$. 
\end{prop}
\begin{proof} From Proposition 4 and the conditions stated above, we get that the coefficients
$U_{n}(m,\epsilon)$ of $\hat{U}(T,m,\epsilon)$ are well defined, belong to $E_{(\beta,\mu)}$ for all
$\epsilon \in D(0,\epsilon_{0})$, all $n \geq 1$ and satisfy the following recursion relation
\begin{multline}
(n+1)U_{n+1}(m,\epsilon)\\
= \frac{R_{D}(im)}{Q(im)} \Pi_{j=0}^{\delta_{D}-1} (n + \delta_{D} - (\delta_{D}-1)(k_{2}+1) - j)
U_{n + \delta_{D} - (\delta_{D}-1)(k_{2}+1)}(m,\epsilon)\\
+ \frac{\epsilon^{-1}}{Q(im)}
\sum_{n_{1}+n_{2}=n,n_{1} \geq 1,n_{2} \geq 1}
\frac{c_{1,2}(\epsilon)}{(2\pi)^{1/2}} \int_{-\infty}^{+\infty} Q_{1}(i(m-m_{1}))U_{n_1}(m-m_{1},\epsilon)
Q_{2}(im_{1})U_{n_2}(m_{1},\epsilon) dm_{1}\\
+ \sum_{l=1}^{D-1} \frac{R_{l}(im)}{Q(im)}\left( \epsilon^{\Delta_{l} - d_{l} + \delta_{l} - 1}
\Pi_{j=0}^{\delta_{l}-1} (n+\delta_{l}-d_{l}-j) \right)
U_{n+\delta_{l}-d_{l}}(m,\epsilon)\\
+ \frac{\epsilon^{-1}}{Q(im)}
\sum_{n_{1}+n_{2}=n,n_{1} \geq 1,n_{2} \geq 1} \frac{c_{0}(\epsilon)}{(2\pi)^{1/2}} \int_{-\infty}^{+\infty}
C_{0,n_{1}}(m-m_{1},\epsilon)R_{0}(im_{1})U_{n_2}(m_{1},\epsilon) dm_{1}\\
+ \frac{\epsilon^{-1}c_{0,0}(\epsilon)}{(2\pi)^{1/2}Q(im)} \int_{-\infty}^{+\infty}
C_{0,0}(m-m_{1},\epsilon)R_{0}(im_{1})U_{n}(m_{1},\epsilon) dm_{1}
+ \frac{\epsilon^{-1}c_{F}(\epsilon)}{Q(im)}F_{n}(m,\epsilon)
\end{multline}
for all $n \geq \max( \max_{1 \leq l \leq D-1}d_{l}, (\delta_{D}-1)(k_{2}+1) )$.
\end{proof}

\subsection{Analytic solutions for an auxiliary convolution problem resulting from a $m_{k_{1}}-$Borel transform applied
to the main convolution initial value problem}

We make the additional assumption that
\begin{equation}
d_{l} > (\delta_{l}-1)(k_{1}+1) \label{constraint_dl_deltal_k1}
\end{equation}
for all $1 \leq l \leq D-1$. Using the formula (8.7) from \cite{taya}, p. 3630, we can expand the operators
$T^{\delta_{l}(k_{1}+1)} \partial_{T}^{\delta_l}$ in the form
\begin{equation}
T^{\delta_{l}(k_{1}+1)} \partial_{T}^{\delta_l} = (T^{k_{1}+1}\partial_{T})^{\delta_l} +
\sum_{1 \leq p \leq \delta_{l}-1} A_{\delta_{l},p} T^{k_{1}(\delta_{l}-p)} (T^{k_{1}+1}\partial_{T})^{p} \label{expand_op_diff_k1}
\end{equation}
where $A_{\delta_{l},p}$, $p=1,\ldots,\delta_{l}-1$ are real numbers, for all $1 \leq l \leq D$.
We define integers $d_{l,k_{1}}^{1} > 0$ in order to satisfy
\begin{equation}
d_{l} + k_{1} + 1 = \delta_{l}(k_{1}+1) + d_{l,k_{1}}^{1} \label{defin_d_l_k1_1}
\end{equation}
for all $1 \leq l \leq D-1$. We also rewrite $(\delta_{D}-1)(k_{2}+1) = (\delta_{D}-1)(k_{1}+1) + (\delta_{D}-1)(k_{2}-k_{1})$.

Multiplying the equation (\ref{SCP}) by $T^{k_{1}+1}$ and using
(\ref{expand_op_diff_k1}), we can rewrite the equation (\ref{SCP}) in the form
\begin{multline}
Q(im)( T^{k_{1}+1}\partial_{T}U(T,m,\epsilon) ) \\
= R_{D}(im)T^{(\delta_{D}-1)(k_{2}-k_{1})}(T^{k_{1}+1}\partial_{T})^{\delta_{D}}U(T,m,\epsilon)\\
+ R_{D}(im)\sum_{1 \leq p \leq \delta_{D}-1} A_{\delta_{D},p} T^{(\delta_{D}-1)(k_{2}-k_{1})} T^{k_{1}(\delta_{D}-p)}
(T^{k_{1}+1}\partial_{T})^{p}U(T,m,\epsilon)\\
 + \epsilon^{-1}T^{k_{1}+1}
\frac{c_{1,2}(\epsilon)}{(2\pi)^{1/2}}\int_{-\infty}^{+\infty} Q_{1}(i(m-m_{1}))U(T,m-m_{1},\epsilon)
Q_{2}(im_{1})U(T,m_{1},\epsilon) dm_{1} \\
+ \sum_{l=1}^{D-1} R_{l}(im)\left( \epsilon^{\Delta_{l} - d_{l} + \delta_{l} - 1}
T^{d_{l,k_{1}}^{1}}(T^{k_{1}+1}\partial_{T})^{\delta_l}
U(T,m,\epsilon) \right.
\\+ \sum_{1 \leq p \leq \delta_{l}-1} A_{\delta_{l},p}
\left. \epsilon^{\Delta_{l}-d_{l}+\delta_{l}-1} T^{k_{1}(\delta_{l}-p) + d_{l,k_{1}}^{1}}
(T^{k_{1}+1}\partial_{T})^{p}U(T,m,\epsilon) \right)\\
+ \epsilon^{-1}T^{k_{1}+1}
\frac{c_{0}(\epsilon)}{(2\pi)^{1/2}}\int_{-\infty}^{+\infty} C_{0}(T,m-m_{1},\epsilon) R_{0}(im_{1})U(T,m_{1},\epsilon) dm_{1}\\
+ \epsilon^{-1}T^{k_{1}+1}
\frac{c_{0,0}(\epsilon)}{(2\pi)^{1/2}}\int_{-\infty}^{+\infty} C_{0,0}(m-m_{1},\epsilon) R_{0}(im_{1})U(T,m_{1},\epsilon) dm_{1}
+ \epsilon^{-1}c_{F}(\epsilon)T^{k_{1}+1}F(T,m,\epsilon)
\label{SCP_irregular_k1}
\end{multline}
We denote
$\omega_{k_1}(\tau,m,\epsilon)$ the formal $m_{k_1}-$Borel transform of
$\hat{U}(T,m,\epsilon)$ with respect to $T$, $\varphi_{k_1}(\tau,m,\epsilon)$ the formal $m_{k_1}-$Borel transform of
$C_{0}(T,m,\epsilon)$ with respect to $T$ and $\psi_{k_1}(\tau,m,\epsilon)$ the formal $m_{k_1}-$Borel transform of
$F(T,m,\epsilon)$ with respect to $T$. More precisely,
\begin{multline*}
 \omega_{k_1}(\tau,m,\epsilon) = \sum_{n \geq 1} U_{n}(m,\epsilon) \frac{\tau^n}{\Gamma(\frac{n}{k_1})} \ \ , \ \
\varphi_{k_1}(\tau,m,\epsilon) = \sum_{n \geq 1} C_{0,n}(m,\epsilon) \frac{\tau^n}{\Gamma(\frac{n}{k_1})}\\
\psi_{k_1}(\tau,m,\epsilon) = \sum_{n \geq 1} F_{n}(m,\epsilon) \frac{\tau^n}{\Gamma(\frac{n}{k_1})}
\end{multline*}
Using (\ref{norm_beta_mu_C0_n}) we get that for any $\kappa \geq k_{1}$, the function
$\varphi_{k_1}(\tau,m,\epsilon)$ belongs to $F_{(\nu,\beta,\mu,k_{1},\kappa)}^{d}$ for
all $\epsilon \in D(0,\epsilon_{0})$, any unbounded sector $U_{d}$ centered at 0 with bisecting direction
$d \in \mathbb{R}$, for some $\nu>0$. Indeed, we have that
\begin{multline}
||\varphi_{k_1}(\tau,m,\epsilon)||_{(\nu,\beta,\mu,k_{1},\kappa)} \\
\leq \sum_{n \geq 1}
||C_{0,n}(m,\epsilon)||_{(\beta,\mu)} (\sup_{\tau \in \bar{D}(0,\rho) \cup U_{d}}
\frac{1 + |\tau|^{2k_1}}{|\tau|} \exp(-\nu |\tau|^{\kappa})
\frac{|\tau|^n}{\Gamma(\frac{n}{k_{1}})}) \label{maj_norm_varphi_k_1}
\end{multline}
By using the classical estimates
\begin{equation}
\sup_{x \geq 0} x^{m_1}\exp(-m_{2}x) = (\frac{m_1}{m_2})^{m_1}e^{-m_1} \label{x_m_exp_x<}
\end{equation}
for any real numbers $m_{1} \geq 0$, $m_{2}>0$ and Stirling formula
$\Gamma(n/k_1) \sim (2\pi)^{1/2}(n/k_1)^{\frac{n}{k_1}-\frac{1}{2}}e^{-n/k_1}$ as $n$ tends to $+\infty$, we get two constants
$A_{1},A_{2}>0$ depending on $\nu,k_{1},\kappa$ such that
\begin{multline}
\sup_{\tau \in \bar{D}(0,\rho) \cup U_{d}}
\frac{1 + |\tau|^{2k_1}}{|\tau|} \exp(-\nu |\tau|^{\kappa})
\frac{|\tau|^n}{\Gamma(\frac{n}{k_1})}
= \sup_{x \geq 0} (1+x^{2k_{1}/\kappa})x^{\frac{n-1}{\kappa}}
\frac{e^{-\nu x}}{\Gamma(\frac{n}{k_1})} \\
\leq \left( (\frac{n-1}{\nu \kappa})^{\frac{n-1}{\kappa}}
e^{-\frac{n-1}{\kappa}} +
( \frac{n-1}{\nu \kappa} + \frac{2k_{1}}{\nu \kappa})^{\frac{n-1}{\kappa}+\frac{2k_{1}}{\kappa}}
e^{-(\frac{n-1}{\kappa}+\frac{2k_{1}}{\kappa})} \right)/ \Gamma(n/k_{1})\\
\leq A_{1}(A_{2})^{n} \label{sup_Stirling}
\end{multline}
for all $n \geq 1$. Therefore, if the inequality
$A_{2} < T_{0}$ holds, we get the estimates
\begin{equation}
||\varphi_{k_1}(\tau,m,\epsilon)||_{(\nu,\beta,\mu,k_{1},\kappa)} \leq A_{1} \sum_{n \geq 1}
||C_{0,n}(m,\epsilon)||_{(\beta,\mu)}
(A_{2})^{n} \leq \frac{A_{1}A_{2}K_{0}}{T_0} \frac{1}{ 1 - \frac{A_{2}}{T_0} }
\label{norm_F_varphi_k1_epsilon_0}
\end{equation}

On the other hand, we make the assumption that $\psi_{k_1}(\tau,m,\epsilon) \in F_{(\nu,\beta,\mu,k_{1},k_{1})}^{d}$, for
all $\epsilon \in D(0,\epsilon_{0})$, for some unbounded sector $U_{d}$ with bisecting direction
$d \in \mathbb{R}$, where $\nu$ is chosen above. We will make the convention to denote $\psi_{k_1}^{d}$ the analytic continuation of
the convergent power series $\psi_{k_1}$ on the domain $U_{d} \cup D(0,\rho)$. In particular, we get that
$\psi_{k_1}^{d}(\tau,m,\epsilon) \in F_{(\nu,\beta,\mu,k_{1},\kappa)}^{d}$ for any $\kappa \geq k_{1}$. We also assume that
there exists a constant $\zeta_{\psi_{k_1}}>0$ such that
\begin{equation}
||\psi_{k_1}^{d}(\tau,m,\epsilon) ||_{(\nu,\beta,\mu,k_{1},k_{1})} \leq \zeta_{\psi_{k_1}} \label{psi_k1_bounded_norm_k1_k1}
\end{equation}
for all $\epsilon \in D(0,\epsilon_{0})$. In particular, we notice that
\begin{equation}
||\psi_{k_1}^{d}(\tau,m,\epsilon) ||_{(\nu,\beta,\mu,k_{1},\kappa)} \leq \zeta_{\psi_{k_1}} \label{psi_k1_bounded_norm_k1_kappa}
\end{equation}
for any $\kappa \geq k_{1}$. We require that there exists a constant $r_{Q,R_{l}}>0$ such that
\begin{equation}
|\frac{Q(im)}{R_{l}(im)}| \geq r_{Q,R_{l}} \label{quotient_Q_Rl_larger_than_radius}
\end{equation} 
for all $m \in \mathbb{R}$, all $1 \leq l \leq D$.

Using the computation rules for the formal $m_{k_1}-$Borel transform in
Proposition 8, we deduce the following equation satisfied by $\omega_{k_1}(\tau,m,\epsilon)$,
\begin{multline}
Q(im)( k_{1} \tau^{k_1} \omega_{k_1}(\tau,m,\epsilon) )\\
= R_{D}(im) \frac{\tau^{k_1}}{\Gamma( \frac{(\delta_{D}-1)(k_{2}-k_{1})}{k_1} )} \int_{0}^{\tau^{k_1}}
(\tau^{k_1}-s)^{ \frac{(\delta_{D}-1)(k_{2}-k_{1})}{k_1} - 1 } k_{1}^{\delta_D} s^{\delta_D}
\omega_{k_1}(s^{1/k_{1}},m,\epsilon) \frac{ds}{s}\\
+ R_{D}(im) \sum_{1 \leq p \leq \delta_{D}-1} A_{\delta_{D},p}
\frac{\tau^{k_1}}{\Gamma( \frac{(\delta_{D}-1)(k_{2}-k_{1}) + k_{1}(\delta_{D}-p)}{k_1} )}\\
\times \int_{0}^{\tau^{k_1}}
(\tau^{k_1}-s)^{ \frac{(\delta_{D}-1)(k_{2}-k_{1})+k_{1}(\delta_{D}-p)}{k_1} - 1 } k_{1}^{p} s^{p}
\omega_{k_1}(s^{1/k_{1}},m,\epsilon) \frac{ds}{s}\\
+ \epsilon^{-1}
\frac{\tau^{k_1}}{\Gamma(1 + \frac{1}{k_{1}})} \int_{0}^{\tau^{k_1}}
(\tau^{k_1}-s)^{1/k_1}\\
\times \left( \frac{c_{1,2}(\epsilon)}{(2\pi)^{1/2}} s\int_{0}^{s} \int_{-\infty}^{+\infty} \right.
Q_{1}(i(m-m_{1}))\omega_{k_1}((s-x)^{1/k_1},m-m_{1},\epsilon)\\
\left. \times  Q_{2}(im_{1})
\omega_{k_1}(x^{1/k_1},m_{1},\epsilon) \frac{1}{(s-x)x} dxdm_{1} \right) \frac{ds}{s}\\
+ \sum_{l=1}^{D-1} R_{l}(im) \left( \epsilon^{\Delta_{l}-d_{l}+\delta_{l}-1}
\frac{\tau^{k_1}}{\Gamma( \frac{d_{l,k_1}^1}{k_1} )} \right.
\int_{0}^{\tau^{k_1}} (\tau^{k_1}-s)^{\frac{d_{l,k_1}^{1}}{k_1}-1}({k_1}^{\delta_l}s^{\delta_l}
\omega_{k_1}(s^{1/k_1},m,\epsilon)) \frac{ds}{s}\\
+ \sum_{1 \leq p \leq \delta_{l}-1} A_{\delta_{l},p}\epsilon^{\Delta_{l}-d_{l}+\delta_{l}-1}
\frac{\tau^{k_1}}{\Gamma( \frac{d_{l,k_{1}}^{1}}{k_1} + \delta_{l}-p)} \int_{0}^{\tau^{k_1}}
\left. (\tau^{k_1}-s)^{\frac{d_{l,k_{1}}^{1}}{k_1}+\delta_{l}-p-1}(k_{1}^{p}s^{p}\omega_{k_1}(s^{1/k_1},m,\epsilon))
\frac{ds}{s} \right)\\
 + \epsilon^{-1}
\frac{\tau^{k_1}}{\Gamma(1 + \frac{1}{k_1})} \int_{0}^{\tau^{k_1}}
(\tau^{k_1}-s)^{1/k_1}\\
\times \left( \frac{c_{0}(\epsilon)}{(2\pi)^{1/2}} s\int_{0}^{s} \int_{-\infty}^{+\infty} \right.
\left. \varphi_{k_1}((s-x)^{1/k_1},m-m_{1},\epsilon) R_{0}(im_{1}) \omega_{k_1}(x^{1/k_1},m_{1},\epsilon) \frac{1}{(s-x)x}
dxdm_{1} \right) \frac{ds}{s}\\
+ \epsilon^{-1}\frac{\tau^{k_1}}{\Gamma(1 + \frac{1}{k_1})} \int_{0}^{\tau^{k_1}}
(\tau^{k_1}-s)^{1/k_1} \frac{c_{0,0}(\epsilon)}{(2\pi)^{1/2}} ( \int_{-\infty}^{+\infty} C_{0,0}(m-m_{1},\epsilon)R_{0}(im_{1})
\omega_{k_1}(s^{1/k_1},m_{1},\epsilon) dm_{1} )\frac{ds}{s}\\
+ \epsilon^{-1}c_{F}(\epsilon) \frac{\tau^{k_1}}{\Gamma(1 + \frac{1}{k_1})}\int_{0}^{\tau^{k_1}}
(\tau^{k_1}-s)^{1/k_1} \psi_{k_1}^{d}(s^{1/k_1},m,\epsilon) \frac{ds}{s} \label{k_1_Borel_equation}
\end{multline}
In the next proposition, we give sufficient conditions under which the equation (\ref{k_1_Borel_equation}) has a solution
$\omega_{k_1}^{d}(\tau,m,\epsilon)$ in the Banach space $F_{(\nu,\beta,\mu,k_{1},\kappa)}^{d}$ where $\beta,\mu$ are
defined above and for well chosen $\kappa > k_{1}$.

\begin{prop} Under the assumption that
\begin{equation}
\frac{1}{\kappa} = \frac{1}{k_{1}} - \frac{1}{k_{2}} \ \ , \ \ \frac{k_2}{k_{2} - k_{1}} \geq
\frac{d_{l} + (1-\delta_{l})}{d_{l} + (1-\delta_{l})(k_{1}+1)}   \label{constraints_k_1_Borel_equation}
\end{equation}
for all $1 \leq l \leq D-1$, there exist radii $r_{Q,R_{l}}>0$, $1 \leq l \leq D$, a constant $\varpi>0$ and constants
$\zeta_{1,2},\zeta_{0,0},\zeta_{0},\zeta_{1},\zeta_{1,0},\zeta_{F},\zeta_{2}>0$ (depending on
$Q_{1},Q_{2},k_{1},\mu,\nu,\epsilon_{0},R_{l},\Delta_{l},\delta_{l},d_{l}$ for
$1 \leq l \leq D-1$) such that if
\begin{multline}
\sup_{\epsilon \in D(0,\epsilon_{0})}|\frac{c_{1,2}(\epsilon)}{\epsilon}| \leq \zeta_{1,2} \ \ , \ \
\sup_{\epsilon \in D(0,\epsilon_{0})}|\frac{c_{0}(\epsilon)}{\epsilon}| \leq \zeta_{1,0} \ \ , \ \
||\varphi_{k_1}(\tau,m,\epsilon)||_{(\nu,\beta,\mu,k_{1},\kappa)} \leq \zeta_{1},\\
\sup_{\epsilon \in D(0,\epsilon_{0})}|\frac{c_{0,0}(\epsilon)}{\epsilon}| \leq \zeta_{0,0} \ \ , \ \
||C_{0,0}(m,\epsilon)||_{(\beta,\mu)} \leq \zeta_{0},\\
\sup_{\epsilon \in D(0,\epsilon_{0})}|\frac{c_{F}(\epsilon)}{\epsilon}| \leq \zeta_{F} \ \ , \ \
||\psi_{k_1}^{d}(\tau,m,\epsilon)||_{(\nu,\beta,\mu,k_{1},\kappa)} \leq \zeta_{2} \label{norm_F_varphi_k_psi_k_1_small}
\end{multline}
for all $\epsilon \in D(0,\epsilon_{0})$, the equation (\ref{k_1_Borel_equation}) has a unique solution
$\omega_{k_1}^{d}(\tau,m,\epsilon)$ in the space $F_{(\nu,\beta,\mu,k_{1},\kappa)}^{d}$ where $\beta,\mu>0$ are defined in
Proposition 10 which verifies $||\omega_{k_1}^{d}(\tau,m,\epsilon)||_{(\nu,\beta,\mu,k_{1},\kappa)} \leq \varpi$, for all
$\epsilon \in D(0,\epsilon_{0})$.
\end{prop}
\begin{proof} We start the proof with a lemma which provides appropriate conditions in order to apply a fixed point theorem.
\begin{lemma} One can choose the constants $r_{Q,R_{l}}>0$, for $1 \leq l \leq D$, a small enough constant $\varpi$ and
constants $\zeta_{1,2},\zeta_{0},\zeta_{0,0},\zeta_{1},\zeta_{1,0},\zeta_{F},\zeta_{2}>0$
(depending on $Q_{1},Q_{2},k_{1},\mu,\nu,\epsilon_{0},R_{l},\Delta_{l},\delta_{l},d_{l}$ for
$1 \leq l \leq D-1$) such that if (\ref{norm_F_varphi_k_psi_k_1_small}) holds for all $\epsilon \in D(0,\epsilon_{0})$, the map $\mathcal{H}_{\epsilon}^{k_1}$ defined by
\begin{multline}
\mathcal{H}_{\epsilon}^{k_1}(w(\tau,m))\\
= \frac{R_{D}(im)}{Q(im)} \frac{1}{k_{1}\Gamma( \frac{(\delta_{D}-1)(k_{2}-k_{1})}{k_1} )} \int_{0}^{\tau^{k_1}}
(\tau^{k_1}-s)^{ \frac{(\delta_{D}-1)(k_{2}-k_{1})}{k_1} - 1 } k_{1}^{\delta_D} s^{\delta_D}
w(s^{1/k_{1}},m) \frac{ds}{s}\\
+ \frac{R_{D}(im)}{Q(im)} \sum_{1 \leq p \leq \delta_{D}-1} A_{\delta_{D},p}
\frac{1}{k_{1}\Gamma( \frac{(\delta_{D}-1)(k_{2}-k_{1}) + k_{1}(\delta_{D}-p)}{k_1} )}\\
\times \int_{0}^{\tau^{k_1}}
(\tau^{k_1}-s)^{ \frac{(\delta_{D}-1)(k_{2}-k_{1})+k_{1}(\delta_{D}-p)}{k_1} - 1 } k_{1}^{p} s^{p}
w(s^{1/k_{1}},m) \frac{ds}{s}\\
+ \epsilon^{-1}
\frac{1}{Q(im)k_{1}\Gamma(1 + \frac{1}{k_{1}})} \int_{0}^{\tau^{k_1}}
(\tau^{k_1}-s)^{1/k_1}\\
\times \left( \frac{c_{1,2}(\epsilon)}{(2\pi)^{1/2}} s\int_{0}^{s} \int_{-\infty}^{+\infty} \right.
Q_{1}(i(m-m_{1}))w((s-x)^{1/k_1},m-m_{1})\\
\left. \times  Q_{2}(im_{1})
w(x^{1/k_1},m_{1}) \frac{1}{(s-x)x} dxdm_{1} \right) \frac{ds}{s}\\
+ \sum_{l=1}^{D-1} \frac{R_{l}(im)}{Q(im)} \left( \epsilon^{\Delta_{l}-d_{l}+\delta_{l}-1}
\frac{1}{k_{1}\Gamma( \frac{d_{l,k_1}^1}{k_1} )} \right.
\int_{0}^{\tau^{k_1}} (\tau^{k_1}-s)^{\frac{d_{l,k_1}^{1}}{k_1}-1}({k_1}^{\delta_l}s^{\delta_l}
w(s^{1/k_1},m)) \frac{ds}{s}\\
+ \sum_{1 \leq p \leq \delta_{l}-1} A_{\delta_{l},p}\epsilon^{\Delta_{l}-d_{l}+\delta_{l}-1}
\frac{1}{k_{1}\Gamma( \frac{d_{l,k_{1}}^{1}}{k_1} + \delta_{l}-p)} \int_{0}^{\tau^{k_1}}
\left. (\tau^{k_1}-s)^{\frac{d_{l,k_{1}}^{1}}{k_1}+\delta_{l}-p-1}(k_{1}^{p}s^{p}w(s^{1/k_1},m))
\frac{ds}{s} \right)\\
 + \epsilon^{-1}
\frac{c_{0}(\epsilon)}{Q(im)k_{1}\Gamma(1 + \frac{1}{k_1})} \int_{0}^{\tau^{k_1}}
(\tau^{k_1}-s)^{1/k_1}\\
\times \left( \frac{1}{(2\pi)^{1/2}} s\int_{0}^{s} \int_{-\infty}^{+\infty} \right.
\left. \varphi_{k_1}((s-x)^{1/k_1},m-m_{1},\epsilon) R_{0}(im_{1}) w(x^{1/k_1},m_{1}) \frac{1}{(s-x)x}
dxdm_{1} \right) \frac{ds}{s}\\
+ \epsilon^{-1}\frac{c_{0,0}(\epsilon)}{Q(im)k_{1}\Gamma(1 + \frac{1}{k_1})} \int_{0}^{\tau^{k_1}}
(\tau^{k_1}-s)^{1/k_1} \frac{1}{(2\pi)^{1/2}} ( \int_{-\infty}^{+\infty} C_{0,0}(m-m_{1},\epsilon)R_{0}(im_{1})
w(s^{1/k_1},m_{1}) dm_{1} )\frac{ds}{s}\\
+ \epsilon^{-1} \frac{c_{F}(\epsilon)}{Q(im)k_{1}\Gamma(1 + \frac{1}{k_1})}\int_{0}^{\tau^{k_1}}
(\tau^{k_1}-s)^{1/k_1} \psi_{k_1}^{d}(s^{1/k_1},m,\epsilon) \frac{ds}{s}
\end{multline}
satisfies the next properties.\\
{\bf i)} The following inclusion holds
\begin{equation}
\mathcal{H}_{\epsilon}^{k_1}(\bar{B}(0,\varpi)) \subset \bar{B}(0,\varpi) \label{H_k1_inclusion}
\end{equation}
where $\bar{B}(0,\varpi)$ is the closed ball of radius $\varpi>0$ centered at 0 in $F_{(\nu,\beta,\mu,k_{1},\kappa)}^{d}$,
for all $\epsilon \in D(0,\epsilon_{0})$.\\
{\bf ii)} We have
\begin{equation}
|| \mathcal{H}_{\epsilon}^{k_1}(w_{1}) - \mathcal{H}_{\epsilon}^{k_1}(w_{2})||_{(\nu,\beta,\mu,k_{1},\kappa)}
\leq \frac{1}{2} ||w_{1} - w_{2}||_{(\nu,\beta,\mu,k_{1},\kappa)}
\label{H_k1_shrink}
\end{equation}
for all $w_{1},w_{2} \in \bar{B}(0,\varpi)$, for all $\epsilon \in D(0,\epsilon_{0})$.
\end{lemma}
\begin{proof}
We first check the property (\ref{H_k1_inclusion}). Let $\epsilon \in D(0,\epsilon_{0})$ and
$w(\tau,m)$ be in $F_{(\nu,\beta,\mu,k_{1},\kappa)}^{d}$. We take $\zeta_{1,2},\zeta_{0},\zeta_{0,0},\zeta_{1},\zeta_{1,0},
\zeta_{2},\zeta_{F},\varpi > 0$
such that if (\ref{norm_F_varphi_k_psi_k_1_small}) holds and $||w(\tau,m)||_{(\nu,\beta,\mu,k_{1},\kappa)} \leq \varpi$
for all $\epsilon \in D(0,\epsilon_{0})$.

Since $\kappa \geq k_{1}$ and (\ref{first_constraints_polynomials_Q_R}) hold, using Proposition 2, we get that
\begin{multline}
||\epsilon^{-1}
\frac{c_{1,2}(\epsilon)}{Q(im)k_{1}\Gamma(1 + \frac{1}{k_{1}})} \int_{0}^{\tau^{k_1}}
(\tau^{k_1}-s)^{1/k_1}\\
\times \left( \frac{1}{(2\pi)^{1/2}} s\int_{0}^{s} \int_{-\infty}^{+\infty} \right.
Q_{1}(i(m-m_{1}))w((s-x)^{1/k_1},m-m_{1})\\
\left. \times  Q_{2}(im_{1})
w(x^{1/k_1},m_{1}) \frac{1}{(s-x)x} dxdm_{1} \right) \frac{ds}{s}||_{(\nu,\beta,\mu,k_{1},\kappa)}\\
\leq \frac{C_{3} \zeta_{1,2}}{(2\pi)^{1/2}k_{1}\Gamma(1 + \frac{1}{k_1})} ||w(\tau,m)||_{(\nu,\beta,\mu,k_{1},\kappa)}^{2}
\leq \frac{C_{3}\zeta_{1,2}\varpi^{2}}{(2\pi)^{1/2}k_{1}\Gamma(1 + \frac{1}{k_1})}. \label{H_k1_incl_Q1_Q2}
\end{multline}

Due to the lower bound assumption (\ref{quotient_Q_Rl_larger_than_radius}) and taking into account the definition of $\kappa$
in (\ref{constraints_k_1_Borel_equation}), we get from Lemma 1 and Proposition 1
\begin{multline}
||\frac{R_{D}(im)}{Q(im)} \frac{1}{k_{1}\Gamma( \frac{(\delta_{D}-1)(k_{2}-k_{1})}{k_1} )} \int_{0}^{\tau^{k_1}}
(\tau^{k_1}-s)^{ \frac{(\delta_{D}-1)(k_{2}-k_{1})}{k_1} - 1 } k_{1}^{\delta_D} s^{\delta_D}
w(s^{1/k_{1}},m) \frac{ds}{s}||_{(\nu,\beta,\mu,k_{1},\kappa)}\\
\leq \frac{C_{2} k_{1}^{\delta_D}}{r_{Q,R_{D}} k_{1} \Gamma( \frac{(\delta_{D}-1)(k_{2}-k_{1})}{k_1} ) }
||w(\tau,m)||_{(\nu,\beta,\mu,k_{1},\kappa)}\\
\leq \frac{C_{2} k_{1}^{\delta_D}}{r_{Q,R_{D}} k_{1} \Gamma( \frac{(\delta_{D}-1)(k_{2}-k_{1})}{k_1} ) }
\varpi \label{H_k1_incl_deltaD_part1}
\end{multline}
and that
\begin{multline}
||\frac{R_{D}(im)}{Q(im)}A_{\delta_{D},p}
\frac{1}{k_{1}\Gamma( \frac{(\delta_{D}-1)(k_{2}-k_{1}) + k_{1}(\delta_{D}-p)}{k_1} )}\\
\times \int_{0}^{\tau^{k_1}}
(\tau^{k_1}-s)^{ \frac{(\delta_{D}-1)(k_{2}-k_{1})+k_{1}(\delta_{D}-p)}{k_1} - 1 } k_{1}^{p} s^{p}
w(s^{1/k_{1}},m) \frac{ds}{s}||_{(\nu,\beta,\mu,k_{1},\kappa)}\\
\leq \frac{|A_{\delta_{D},p}|C_{2}k_{1}^{p}}{r_{Q,R_{D}}k_{1}
\Gamma( \frac{(\delta_{D}-1)(k_{2}-k_{1}) + k_{1}(\delta_{D}-p)}{k_1} ) }
||w(\tau,m)||_{(\nu,\beta,\mu,k_{1},\kappa)}\\
\leq \frac{|A_{\delta_{D},p}|C_{2}k_{1}^{p}}{r_{Q,R_{D}}k_{1}
\Gamma( \frac{(\delta_{D}-1)(k_{2}-k_{1}) + k_{1}(\delta_{D}-p)}{k_1} ) } \varpi
\label{H_k1_incl_deltaD_part2}
\end{multline}
for all $1 \leq p \leq \delta_{D}-1$.

From assumption (\ref {first_constraints_polynomials_Q_R}) and due to the second constraint in
(\ref{constraints_k_1_Borel_equation}), we get from Lemma 1 and Proposition 1
\begin{multline}
|| \frac{R_{l}(im)}{Q(im)}\epsilon^{\Delta_{l}-d_{l}+\delta_{l}-1}
\frac{1}{k_{1}\Gamma( \frac{d_{l,k_1}^1}{k_1} )}
\int_{0}^{\tau^{k_1}} (\tau^{k_1}-s)^{\frac{d_{l,k_1}^{1}}{k_1}-1}({k}_{1}^{\delta_l}s^{\delta_l}
w(s^{1/k_1},m)) \frac{ds}{s}||_{(\nu,\beta,\mu,k_{1},\kappa)}\\
\leq |\epsilon|^{\Delta_{l} - d_{l} + \delta_{l} - 1}\frac{1}{r_{Q,R_{l}}}
\frac{C_{2}k_{1}^{\delta_{l}}}{k_{1}\Gamma( \frac{d^{1}_{l,k_{1}}}{k_1} ) }||w(\tau,m)||_{(\nu,\beta,\mu,k_{1},\kappa)} \leq
|\epsilon|^{\Delta_{l}-d_{l}+\delta_{l}-1} \frac{1}{r_{Q,R_{l}}}
\frac{C_{2}k_{1}^{\delta_{l}}}{k_{1}\Gamma( \frac{d^{1}_{l,k_{1}}}{k_1} ) } \varpi
\label{H_k1_incl_delta_l_part1}
\end{multline}
for all $1 \leq l \leq D-1$ and
\begin{multline}
|| \frac{R_{l}(im)}{Q(im)}A_{\delta_{l},p}\epsilon^{\Delta_{l}-d_{l}+\delta_{l}-1}
\frac{1}{k_{1}\Gamma( \frac{d_{l,k_{1}}^{1}}{k_1} + \delta_{l}-p)}\\
\times \int_{0}^{\tau^{k_1}}
 (\tau^{k_1}-s)^{\frac{d_{l,k_{1}}^{1}}{k_1}+\delta_{l}-p-1}(k_{1}^{p}s^{p}w(s^{1/k_1},m))
\frac{ds}{s}||_{(\nu,\beta,\mu,k_{1},\kappa)}\\
\leq |\epsilon|^{\Delta_{l}-d_{l}+\delta_{l}-1}\frac{1}{r_{Q,R_{l}}}|A_{\delta_{l},p}|
\frac{C_{2}k_{1}^{p}}{k_{1}\Gamma( \frac{d_{l,k_{1}}^{1}}{k_{1}} + \delta_{l} -p)}
||w(\tau,m)||_{(\nu,\beta,\mu,k_{1},\kappa)}\\
\leq |\epsilon|^{\Delta_{l}-d_{l}+\delta_{l}-1}\frac{1}{r_{Q,R_{l}}}|A_{\delta_{l},p}|
\frac{C_{2}k_{1}^{p}}{k_{1}\Gamma( \frac{d_{l,k_{1}}^{1}}{k_{1}} + \delta_{l} -p)} \varpi
\label{H_k1_incl_delta_l_part2}
\end{multline}
for all $1 \leq p \leq \delta_{l}-1$. Since $\kappa \geq k_{1}$ and (\ref {first_constraints_polynomials_Q_R})
we get from Proposition 2 that
\begin{multline}
||\epsilon^{-1}
\frac{c_{0}(\epsilon)}{Q(im)k_{1}\Gamma(1 + \frac{1}{k_1})} \int_{0}^{\tau^{k_1}}
(\tau^{k_1}-s)^{1/k_1} \left( \frac{1}{(2\pi)^{1/2}} s\int_{0}^{s} \int_{-\infty}^{+\infty}
 \varphi_{k_1}((s-x)^{1/k_1},m-m_{1},\epsilon) \right. \\
\times \left. R_{0}(im_{1}) w(x^{1/k_1},m_{1}) \frac{1}{(s-x)x}
dxdm_{1} \right) \frac{ds}{s}||_{(\nu,\beta,\mu,k_{1},\kappa)}\\
\leq \frac{C_{3}\zeta_{1,0}}{(2\pi)^{1/2}k_{1}\Gamma(1 + \frac{1}{k_1})}
||\varphi_{k_1}(\tau,m,\epsilon)||_{(\nu,\beta,\mu,k_{1},\kappa)}||w(\tau,m)||_{(\nu,\beta,\mu,k_{1},\kappa)} 
\leq \frac{C_{3}\zeta_{1,0}}{(2\pi)^{1/2}k_{1}\Gamma(1 + \frac{1}{k_1})} \zeta_{1} \varpi.
\label{H_k1_incl_varphik1_R0}
\end{multline}
Since $\kappa \geq k_{1}$ and (\ref {first_constraints_polynomials_Q_R}) we deduce from Proposition 3 that
\begin{multline}
||\epsilon^{-1}\frac{c_{0,0}(\epsilon)}{Q(im)k_{1}\Gamma(1 + \frac{1}{k_1})} \int_{0}^{\tau^{k_1}}
(\tau^{k_1}-s)^{1/k_1} \frac{1}{(2\pi)^{1/2}} ( \int_{-\infty}^{+\infty} C_{0,0}(m-m_{1},\epsilon)\\
\times R_{0}(im_{1})
w(s^{1/k_1},m_{1}) dm_{1} )\frac{ds}{s}||_{(\nu,\beta,\mu,k_{1},\kappa)} \leq
\frac{C_{4}\zeta_{0,0}}{(2\pi)^{1/2}k_{1}\Gamma(1 + \frac{1}{k_1})}
||C_{0,0}(m,\epsilon)||_{(\beta,\mu)}\\
\times ||w(\tau,m)||_{(\nu,\beta,\mu,k_{1},\kappa)} \leq
\frac{C_{4}\zeta_{0,0}}{(2\pi)^{1/2}k_{1}\Gamma(1 + \frac{1}{k_1})} \zeta_{0} \varpi.
\label{H_k1_incl_C00}
\end{multline}
and finally bearing in mind Proposition 1 we get that
\begin{multline}
|| \epsilon^{-1} \frac{c_{F}(\epsilon)}{Q(im)k_{1}\Gamma(1 + \frac{1}{k_1})}\int_{0}^{\tau^{k_1}}
(\tau^{k_1}-s)^{1/k_1} \psi_{k_1}^{d}(s^{1/k_1},m,\epsilon) \frac{ds}{s} ||_{(\nu,\beta,\mu,k_{1},\kappa)}\\
\leq \sup_{m \in \mathbb{R}} \frac{1}{|Q(im)|} \frac{C_{2}\zeta_{F}}{k_{1}\Gamma(1 + \frac{1}{k_1})}
||\psi_{k_1}^{d}(\tau,m,\epsilon)||_{(\nu,\beta,\mu,k_{1},\kappa)}
\leq \sup_{m \in \mathbb{R}} \frac{1}{|Q(im)|} \frac{C_{2}\zeta_{F}}{k_{1}\Gamma(1 + \frac{1}{k_1})} \zeta_{2}.
\label{H_k1_incl_psik1}
\end{multline}
Now, we choose $r_{Q,R_{l}}>0$, for $1 \leq l \leq D$, $\zeta_{1,2},\zeta_{0,0},\zeta_{0},\zeta_{F},
\zeta_{1,0},\zeta_{1},\zeta_{2}>0$ and $\varpi>0$ such that
\begin{multline}
\frac{C_{3}\zeta_{1,2}\varpi^{2}}{(2\pi)^{1/2}k_{1}\Gamma(1 + \frac{1}{k_1})} +
 \frac{C_{2} k_{1}^{\delta_D}}{r_{Q,R_{D}} k_{1} \Gamma( \frac{(\delta_{D}-1)(k_{2}-k_{1})}{k_1} ) }\varpi \\
+ \sum_{1 \leq p \leq \delta_{D}-1} \frac{|A_{\delta_{D},p}|C_{2}k_{1}^{p}}{r_{Q,R_{D}}k_{1}
\Gamma( \frac{(\delta_{D}-1)(k_{2}-k_{1}) + k_{1}(\delta_{D}-p)}{k_1} ) }\varpi
+ \sum_{l=1}^{D-1} \epsilon_{0}^{\Delta_{l}-d_{l}+\delta_{l}-1} \frac{1}{r_{Q,R_{l}}}
\frac{C_{2}k_{1}^{\delta_{l}}}{k_{1}\Gamma( \frac{d^{1}_{l,k_{1}}}{k_1} ) } \varpi\\
+ \sum_{1 \leq p \leq \delta_{l}-1} \epsilon_{0}^{\Delta_{l} - d_{l} + \delta_{l} - 1}\frac{1}{r_{Q,R_{l}}}|A_{\delta_{l},p}|
\frac{C_{2}k_{1}^{p}}{k_{1}\Gamma( \frac{d_{l,k_{1}}^{1}}{k_{1}} + \delta_{l} -p)} \varpi +
\frac{C_{3}\zeta_{1,0}}{(2\pi)^{1/2}k_{1}\Gamma(1 + \frac{1}{k_1})} \zeta_{1} \varpi \\
+ \frac{C_{4}\zeta_{0,0}}{(2\pi)^{1/2}k_{1}\Gamma(1 + \frac{1}{k_1})} \zeta_{0} \varpi +
\sup_{m \in \mathbb{R}} \frac{1}{|Q(im)|} \frac{C_{2}\zeta_{F}}{k_{1}\Gamma(1 + \frac{1}{k_1})} \zeta_{2} \leq \varpi
\label{H_k1_incl_constr}
\end{multline}
Gathering all the norm estimates (\ref{H_k1_incl_Q1_Q2}), (\ref{H_k1_incl_deltaD_part1}), (\ref{H_k1_incl_deltaD_part2}),
(\ref{H_k1_incl_delta_l_part1}), (\ref{H_k1_incl_delta_l_part2}), (\ref{H_k1_incl_varphik1_R0}), (\ref{H_k1_incl_C00}),
(\ref{H_k1_incl_psik1}) together with the constraint (\ref{H_k1_incl_constr}), one gets (\ref{H_k1_inclusion}).\bigskip

Now, we check the second property (\ref{H_k1_shrink}). Let $w_{1}(\tau,m),w_{2}(\tau,m)$ be in
$F^{d}_{(\nu,\beta,\mu,k_{1},\kappa)}$.
We take $\varpi > 0$ such that
$$ ||w_{l}(\tau,m)||_{(\nu,\beta,\mu,k_{1},\kappa)} \leq \varpi,$$
for $l=1,2$, for all $\epsilon \in D(0,\epsilon_{0})$. One can write
\begin{multline}
Q_{1}(i(m-m_{1}))w_{1}((s-x)^{1/k_{1}},m-m_{1})Q_{2}(im_{1})w_{1}(x^{1/k_{1}},m_{1})\\
-Q_{1}(i(m-m_{1}))w_{2}((s-x)^{1/k_{1}},m-m_{1})Q_{2}(im_{1})w_{2}(x^{1/k_{1}},m_{1})\\
= Q_{1}(i(m-m_{1}))\left(w_{1}((s-x)^{1/k_{1}},m-m_{1}) - w_{2}((s-x)^{1/k_{1}},m-m_{1})\right)Q_{2}(im_{1})
w_{1}(x^{1/k_{1}},m_{1})\\
+ Q_{1}(i(m-m_{1}))w_{2}((s-x)^{1/k_{1}},m-m_{1})Q_{2}(im_{1})\left(w_{1}(x^{1/k_{1}},m_{1}) - w_{2}(x^{1/k_{1}},m_{1})\right)
\label{conv_product_w1_w2}
\end{multline}
and taking into account that $\kappa \geq k_{1}$, (\ref{first_constraints_polynomials_Q_R}), (\ref{conv_product_w1_w2})
and using Proposition 2, we get that
\begin{multline}
||\epsilon^{-1}
\frac{c_{1,2}(\epsilon)}{Q(im)k_{1}\Gamma(1 + \frac{1}{k_{1}})} \int_{0}^{\tau^{k_1}}
(\tau^{k_1}-s)^{1/k_1}\\
\times \left( \frac{1}{(2\pi)^{1/2}} s\int_{0}^{s} \int_{-\infty}^{+\infty} \right.
(Q_{1}(i(m-m_{1}))w_{1}((s-x)^{1/k_1},m-m_{1})\\
\times  Q_{2}(im_{1})
w_{1}(x^{1/k_1},m_{1}) - Q_{1}(i(m-m_{1}))w_{2}((s-x)^{1/k_1},m-m_{1})\\
\left. \times  Q_{2}(im_{1})
w_{2}(x^{1/k_1},m_{1}))\frac{1}{(s-x)x} dxdm_{1} \right) \frac{ds}{s}||_{(\nu,\beta,\mu,k_{1},\kappa)}\\
\leq \frac{C_{3} \zeta_{1,2}}{(2\pi)^{1/2}k_{1}\Gamma(1 + \frac{1}{k_1})} 
(||w_{1}(\tau,m) - w_{2}(\tau,m)||_{(\nu,\beta,\mu,k_{1},\kappa)}(
||w_{1}(\tau,m)||_{(\nu,\beta,\mu,k_{1},\kappa)} \\
+ ||w_{2}(\tau,m)||_{(\nu,\beta,\mu,k_{1},\kappa)}  )
\leq \frac{C_{3}\zeta_{1,2}2\varpi}{(2\pi)^{1/2}k_{1}\Gamma(1 + \frac{1}{k_1})}
||w_{1}(\tau,m) - w_{2}(\tau,m)||_{(\nu,\beta,\mu,k_{1},\kappa)}. \label{H_k1_shrink_Q1_Q2}
\end{multline}
On the other hand, from the estimates (\ref{H_k1_incl_deltaD_part1}), (\ref{H_k1_incl_deltaD_part2}), (\ref{H_k1_incl_delta_l_part1}),
(\ref{H_k1_incl_delta_l_part2}),
(\ref{H_k1_incl_varphik1_R0}), (\ref{H_k1_incl_C00}) and under the constraints (\ref {first_constraints_polynomials_Q_R}),
(\ref{constraints_k_1_Borel_equation}) and the lower bound assumption (\ref{quotient_Q_Rl_larger_than_radius}),
we deduce that
\begin{multline}
||\frac{R_{D}(im)}{Q(im)} \frac{1}{k_{1}\Gamma( \frac{(\delta_{D}-1)(k_{2}-k_{1})}{k_1} )} \int_{0}^{\tau^{k_1}}
(\tau^{k_1}-s)^{ \frac{(\delta_{D}-1)(k_{2}-k_{1})}{k_1} - 1 } k_{1}^{\delta_D} s^{\delta_D}\\
\times 
(w_{1}(s^{1/k_{1}},m) - w_{2}(s^{1/k_{1}},m)) \frac{ds}{s}||_{(\nu,\beta,\mu,k_{1},\kappa)}\\
\leq \frac{C_{2} k_{1}^{\delta_D}}{r_{Q,R_{D}} k_{1} \Gamma( \frac{(\delta_{D}-1)(k_{2}-k_{1})}{k_1} ) }
||w_{1}(\tau,m) - w_{2}(\tau,m)||_{(\nu,\beta,\mu,k_{1},\kappa)}
\label{H_k1_shrink_deltaD_part1}
\end{multline}
and that
\begin{multline}
||\frac{R_{D}(im)}{Q(im)}A_{\delta_{D},p}
\frac{1}{k_{1}\Gamma( \frac{(\delta_{D}-1)(k_{2}-k_{1}) + k_{1}(\delta_{D}-p)}{k_1} )}\\
\times \int_{0}^{\tau^{k_1}}
(\tau^{k_1}-s)^{ \frac{(\delta_{D}-1)(k_{2}-k_{1})+k_{1}(\delta_{D}-p)}{k_1} - 1 } k_{1}^{p} s^{p}
(w_{1}(s^{1/k_{1}},m) - w_{2}(s^{1/k_{1}},m)) \frac{ds}{s}||_{(\nu,\beta,\mu,k_{1},\kappa)}\\
\leq \frac{|A_{\delta_{D},p}|C_{2}k_{1}^{p}}{r_{Q,R_{D}}k_{1}
\Gamma( \frac{(\delta_{D}-1)(k_{2}-k_{1}) + k_{1}(\delta_{D}-p)}{k_1} ) }
||w_{1}(\tau,m) - w_{2}(\tau,m)||_{(\nu,\beta,\mu,k_{1},\kappa)}
\label{H_k1_shrink_deltaD_part2}
\end{multline}
for all $1 \leq p \leq \delta_{D}-1$ and also
\begin{multline}
|| \frac{R_{l}(im)}{Q(im)}\epsilon^{\Delta_{l}-d_{l}+\delta_{l}-1}
\frac{1}{k_{1}\Gamma( \frac{d_{l,k_1}^1}{k_1} )}
\int_{0}^{\tau^{k_1}} (\tau^{k_1}-s)^{\frac{d_{l,k_1}^{1}}{k_1}-1}\\
\times ({k}_{1}^{\delta_l}s^{\delta_l}
(w_{1}(s^{1/k_1},m) - w_{2}(s^{1/k_1},m))) \frac{ds}{s}||_{(\nu,\beta,\mu,k_{1},\kappa)}\\
\leq |\epsilon|^{\Delta_{l}-d_{l}+\delta_{l}-1}\frac{1}{r_{Q,R_{l}}}
\frac{C_{2}k_{1}^{\delta_{l}}}{k_{1}\Gamma( \frac{d^{1}_{l,k_{1}}}{k_1} ) }
||w_{1}(\tau,m) - w_{2}(\tau,m)||_{(\nu,\beta,\mu,k_{1},\kappa)}
\label{H_k1_shrink_delta_l_part1}
\end{multline}
for all $1 \leq l \leq D-1$ together with
\begin{multline}
|| \frac{R_{l}(im)}{Q(im)}A_{\delta_{l},p}\epsilon^{\Delta_{l}-d_{l}+\delta_{l}-1}
\frac{1}{k_{1}\Gamma( \frac{d_{l,k_{1}}^{1}}{k_1} + \delta_{l}-p)}\\
\times \int_{0}^{\tau^{k_1}}
 (\tau^{k_1}-s)^{\frac{d_{l,k_{1}}^{1}}{k_1}+\delta_{l}-p-1}(k_{1}^{p}s^{p}(w_{1}(s^{1/k_1},m) - w_{2}(s^{1/k_1},m) ))
\frac{ds}{s}||_{(\nu,\beta,\mu,k_{1},\kappa)}\\
\leq |\epsilon|^{\Delta_{l}-d_{l}+\delta_{l}-1}\frac{1}{r_{Q,R_{l}}}|A_{\delta_{l},p}|
\frac{C_{2}k_{1}^{p}}{k_{1}\Gamma( \frac{d_{l,k_{1}}^{1}}{k_{1}} + \delta_{l} -p)}
||w_{1}(\tau,m) - w_{2}(\tau,m)||_{(\nu,\beta,\mu,k_{1},\kappa)}
\label{H_k1_shrink_delta_l_part2}
\end{multline}
for all $1 \leq p \leq \delta_{l}-1$. Finally, we also obtain
\begin{multline}
||\epsilon^{-1}
\frac{c_{0}(\epsilon)}{Q(im)k_{1}\Gamma(1 + \frac{1}{k_1})} \int_{0}^{\tau^{k_1}}
(\tau^{k_1}-s)^{1/k_1} \left( \frac{1}{(2\pi)^{1/2}} s\int_{0}^{s} \int_{-\infty}^{+\infty}
 \varphi_{k_1}((s-x)^{1/k_1},m-m_{1},\epsilon) \right. \\
\times \left. R_{0}(im_{1}) (w_{1}(x^{1/k_1},m_{1}) - w_{2}(x^{1/k_1},m_{1})) \frac{1}{(s-x)x}
dxdm_{1} \right) \frac{ds}{s}||_{(\nu,\beta,\mu,k_{1},\kappa)}\\
\leq \frac{C_{3}\zeta_{1,0}}{(2\pi)^{1/2}k_{1}\Gamma(1 + \frac{1}{k_1})} \zeta_{1}
||w_{1}(\tau,m) - w_{2}(\tau,m)||_{(\nu,\beta,\mu,k_{1},\kappa)}\label{H_k1_shrink_varphik1_R0}
\end{multline}
and
\begin{multline}
||\epsilon^{-1}\frac{c_{0,0}(\epsilon)}{Q(im)k_{1}\Gamma(1 + \frac{1}{k_1})} \int_{0}^{\tau^{k_1}}
(\tau^{k_1}-s)^{1/k_1} \frac{1}{(2\pi)^{1/2}} ( \int_{-\infty}^{+\infty} C_{0,0}(m-m_{1},\epsilon)\\
\times R_{0}(im_{1})
(w_{1}(s^{1/k_1},m_{1}) - w_{2}(s^{1/k_1},m_{1}))  dm_{1} )\frac{ds}{s}||_{(\nu,\beta,\mu,k_{1},\kappa)} \leq
\frac{C_{4}\zeta_{0,0}}{(2\pi)^{1/2}k_{1}\Gamma(1 + \frac{1}{k_1})}
\zeta_{0}\\
\times ||w_{1}(\tau,m) - w_{2}(\tau,m)||_{(\nu,\beta,\mu,k_{1},\kappa)} 
\label{H_k1_shrink_C00}
\end{multline}
Now, we take $\varpi$, $r_{Q,R_{l}}>0$, for $1 \leq l \leq D$ and $\zeta_{1,2},\zeta_{0,0},\zeta_{0},\zeta_{1,0},\zeta_{1}>0$
such that
\begin{multline}
 \frac{C_{3}\zeta_{1,2}2\varpi}{(2\pi)^{1/2}k_{1}\Gamma(1 + \frac{1}{k_1})} +
\frac{C_{2} k_{1}^{\delta_D}}{r_{Q,R_{D}} k_{1} \Gamma( \frac{(\delta_{D}-1)(k_{2}-k_{1})}{k_1} ) } + \\
\sum_{1 \leq p \leq \delta_{D}-1}  \frac{|A_{\delta_{D},p}|C_{2}k_{1}^{p}}{r_{Q,R_{D}}k_{1}
\Gamma( \frac{(\delta_{D}-1)(k_{2}-k_{1}) + k_{1}(\delta_{D}-p)}{k_1} ) }
 + \sum_{1 \leq l \leq D-1} \epsilon_{0}^{\Delta_{l}-d_{l}+\delta_{l}-1}\frac{1}{r_{Q,R_{l}}}
\frac{C_{2}k_{1}^{\delta_{l}}}{k_{1}\Gamma( \frac{d^{1}_{l,k_{1}}}{k_1} ) }\\
+ \sum_{1 \leq p \leq \delta_{l}-1} \epsilon_{0}^{\Delta_{l}-d_{l}+\delta_{l}-1}\frac{1}{r_{Q,R_{l}}}|A_{\delta_{l},p}|
\frac{C_{2}k_{1}^{p}}{k_{1}\Gamma( \frac{d_{l,k_{1}}^{1}}{k_{1}} + \delta_{l} -p)}
+  \frac{C_{3}\zeta_{1,0}}{(2\pi)^{1/2}k_{1}\Gamma(1 + \frac{1}{k_1})} \zeta_{1}\\
+ \frac{C_{4}\zeta_{0,0}}{(2\pi)^{1/2}k_{1}\Gamma(1 + \frac{1}{k_1})}
\zeta_{0} \leq 1/2. \label{H_k1_shrink_constr}
\end{multline}
Bearing in mind the estimates (\ref{H_k1_shrink_Q1_Q2}), (\ref{H_k1_shrink_deltaD_part1}), (\ref{H_k1_shrink_deltaD_part2}),
(\ref{H_k1_shrink_delta_l_part1}), (\ref{H_k1_shrink_delta_l_part2}), (\ref{H_k1_shrink_varphik1_R0}), (\ref{H_k1_shrink_C00}) with the
constraint (\ref{H_k1_shrink_constr}), one gets (\ref{H_k1_shrink}).
Finally, we choose $r_{Q,R_{l}}>0$, for $1 \leq l \leq D$,
$\zeta_{1,2},\zeta_{0,0},\zeta_{0},\zeta_{F},\zeta_{1,0},\zeta_{1},\zeta_{2}>0$ and $\varpi>0$ such that
both (\ref{H_k1_incl_constr}), (\ref{H_k1_shrink_constr}) are fulfilled. This yields our lemma. 
\end{proof}
We consider the ball $\bar{B}(0,\varpi) \subset F_{(\nu,\beta,\mu,k_{1},\kappa)}^{d}$ constructed in Lemma 3 which is a complete
metric space for the norm $||.||_{(\nu,\beta,\mu,k_{1},\kappa)}$. From the lemma above, we get that
$\mathcal{H}_{\epsilon}^{k_1}$ is a
contractive map from $\bar{B}(0,\varpi)$ into itself. Due to the classical contractive mapping theorem, we deduce that
the map $\mathcal{H}_{\epsilon}^{k_1}$ has a unique fixed point denoted by $\omega_{k_1}^{d}(\tau,m,\epsilon)$ (i.e
$\mathcal{H}_{\epsilon}^{k_1}(\omega_{k_1}^{d}(\tau,m,\epsilon))= \omega_{k_1}^{d}(\tau,m,\epsilon)$) in
$\bar{B}(0,\varpi)$, for all $\epsilon \in D(0,\epsilon_{0})$. Moreover, the function
$\omega_{k_1}^{d}(\tau,m,\epsilon)$ depends holomorphically on $\epsilon$ in $D(0,\epsilon_{0})$. By construction,
$\omega_{k_1}^{d}(\tau,m,\epsilon)$ defines a solution of the equation (\ref{k_1_Borel_equation}). This yields Proposition 11.
\end{proof}

\subsection{Formal and analytic acceleration operators}

In this section, we give a definition of the formal and analytic acceleration operator which is a slightly modified version of the one
given in \cite{ba}, Chapter 5, adapted to our definitions of $m_{k}-$Laplace and $m_{k}-$Borel transforms. First we
give a definition for the formal transform.

\begin{defin} Let $\tilde{k} > k > 0$ be real numbers. Let $\hat{f}(\tau) = \sum_{n \geq 1} f_{n}\tau^{n} \in \tau\mathbb{C}[[\tau]]$
be a formal series. We define the formal acceleration operator with indices $m_{\tilde{k}}$, $m_{k}$ by
$$ \hat{\mathcal{A}}_{m_{\tilde{k}},m_{k}} \hat{f} (\xi) = \sum_{n \geq 1} f_{n}
\frac{\Gamma( \frac{n}{k} )}{\Gamma( \frac{n}{\tilde{k}} )} \xi^{n} \in \xi \mathbb{C}[[ \xi ]]. $$
Notice that if one defines the formal $m_{k}-$Laplace transform $\hat{\mathcal{L}}_{m_k}(\hat{f})$ and the formal
$m_{\tilde{k}}-$Borel
transform $\hat{\mathcal{B}}_{m_{\tilde{k}}}(\hat{f})$ of $\hat{f}(\tau)$ by
$$\hat{\mathcal{L}}_{m_k}(\hat{f})(T) = \sum_{n \geq 1} f_{n} \Gamma( \frac{n}{k} ) T^{n} \ \ , \ \ 
\hat{\mathcal{B}}_{m_{\tilde{k}}}(\hat{f})(Z) = \sum_{n \geq 1} \frac{f_{n}}{\Gamma( \frac{n}{\tilde{k}} )} Z^{n},$$
then the formal acceleration operator $\hat{\mathcal{A}}_{m_{\tilde{k}},m_{k}}$ can also be defined as
$$ \hat{\mathcal{A}}_{m_{\tilde{k}},m_{k}} \hat{f} (\xi) = ( \hat{\mathcal{B}}_{m_{\tilde{k}}} \circ
\hat{\mathcal{L}}_{m_k} )(\hat{f}) (\xi).$$
\end{defin}

\noindent In the next definition, we define the analytic transforms.

\begin{prop} Let $\tilde{k} > k > 0$ be real numbers. Let $S(d,\frac{\pi}{\tilde{k}}+\delta,\rho)$ be a bounded sector
of radius $\rho$ with aperture $\frac{\pi}{\tilde{k}}+\delta$, for some $\delta>0$ and with direction $d$. Let
$F : S(d,\frac{\pi}{\tilde{k}}+\delta,\rho) \rightarrow \mathbb{C}$ be a bounded analytic function such that there exist a
formal series $\hat{F}(z) = \sum_{n \geq 1} F_{n}z^{n} \in \mathbb{C}[[z]]$ and two constants $C,K>0$ with
\begin{equation}
|F(z) - \sum_{n=1}^{N-1} F_{n}z^{n}| \leq CK^{N}\Gamma(1 + N/k)|z|^{N} \label{k_Gevrey_ae_F}
\end{equation}
for all $z \in S(d,\frac{\pi}{\tilde{k}}+\delta,\rho)$, all $N \geq 2$. The analytic $m_{\tilde{k}}-$Borel transform of $F$ in direction
$d$ is defined as
\begin{equation}
(\mathcal{B}_{m_{\tilde{k}}}^{d}F)(Z) = \frac{-\tilde{k}}{2i\pi} \int_{\gamma_{\tilde{k}}} F(u) 
\exp( (\frac{Z}{u})^{\tilde{k}} ) \frac{Z^{\tilde{k}}}{u^{\tilde{k}+1}} du \label{analytic_Borel_defin}
\end{equation}
where $\gamma_{\tilde{k}}$ is the closed Hankel path starting from the origin along the segment
$[0, (\rho/2)e^{i(d + \frac{\pi}{2\tilde{k}} + \frac{\delta'}{2})}]$, following the arc of circle
$[(\rho/2)e^{i(d + \frac{\pi}{2\tilde{k}} + \frac{\delta'}{2})}, (\rho/2)e^{i(d - \frac{\pi}{2\tilde{k}} - \frac{\delta'}{2})}]$ and going
back to the origin along the segment $[(\rho/2)e^{i(d - \frac{\pi}{2\tilde{k}} - \frac{\delta'}{2})},0]$ where
$0< \delta' < \delta$ that can be chosen as close to $\delta$ as needed. Then, the function
$(\mathcal{B}_{m_{\tilde{k}}}^{d}F)(Z)$ is analytic on the unbounded sector $S(d,\delta'')$ with direction $d$ and aperture
$\delta''$ where $0< \delta'' < \delta'$ which can be chosen as close to $\delta'$ as needed. Moreover, if
$(\hat{\mathcal{B}}_{m_{\tilde{k}}}\hat{F})(Z) = \sum_{n \geq 1} F_{n}Z^{n}/\Gamma(n/\tilde{k})$ denotes the
formal $m_{\tilde{k}}-$Borel transform
of $\hat{F}$, then for any given $\rho'>0$, there exists two constants $C,K>0$ with
\begin{equation}
|(\mathcal{B}_{m_{\tilde{k}}}^{d}F)(Z) - \sum_{n=1}^{N-1} \frac{F_n}{\Gamma(\frac{n}{\tilde{k}})}Z^{n}| \leq
CK^{N}\Gamma(1 + N/\kappa)|Z|^{N}  \label{kappa_Gevrey_ae_analytic_Borel}
\end{equation}
for all $Z \in S(d,\delta'') \cap D(0,\rho')$, all $N \geq 2$, where $\kappa$ is defined as $1/\kappa = 1/k - 1/\tilde{k}$.
Finally, the $m_{\tilde{k}}-$Borel operator $\mathcal{B}_{m_{\tilde{k}}}^{d}$ is the right inverse operator of the
$m_{\tilde{k}}-$Laplace transform, namely we have that
\begin{equation}
\mathcal{L}_{m_{\tilde{k}}}^{d}(v \mapsto (\mathcal{B}_{m_{\tilde{k}}}^{d}F)(v) )(T) = F(T), \label{Laplace_Borel_inverse} 
\end{equation}
for all $T \in S(d,\frac{\pi}{\tilde{k}} + \delta', \rho/2)$.
\end{prop}

\begin{proof} The proof follows the same lines of arguments as Theorem 2, Section 2.3 in \cite{ba}. Namely, one can check that
if $F(z) = z^{n}$, for an integer $n \geq 0$, then
\begin{equation}
\mathcal{B}_{m_{\tilde{k}}}^{d}F(Z) = Z^{n}/\Gamma(n/\tilde{k}) \label{analytic_borel_monomial}
\end{equation}
for all $Z \in S(d,\delta'')$ by using the change of variable $u=z/w^{1/\tilde{k}}$ in the integral (\ref{analytic_Borel_defin}) and a path
deformation, bearing in mind the Hankel formula
$$ \frac{1}{\Gamma(\frac{n}{\tilde{k}})} = \frac{1}{2i\pi} \int_{\gamma} w^{-\frac{n}{\tilde{k}}} e^{w} dw $$
where $\gamma$ is the path of integration from infinity along the the ray $\mathrm{arg}(w)=-\pi$ to the unit disc, then around the
circle and back to infinity along the ray $\mathrm{arg}(w)=\pi$.
From the asymptotic expansion (\ref{k_Gevrey_ae_F}) and using the same integrals estimates as in Theorem 2, Section 2.3 in \cite{ba},
together with the Stirling formula, for any given $\rho'>0$, we get two constants $\check{C},\check{K}>0$ such that
$$
|\mathcal{B}_{m_{\tilde{k}}}^{d}F(Z) - \sum_{n=1}^{N-1} \frac{F_n}{\Gamma(\frac{n}{\tilde{k}})}Z^{n}| =
|\mathcal{B}_{m_{\tilde{k}}}^{d}(R_{N-1}F)(Z)| \leq \check{C} \check{K}^{N}
\frac{\Gamma(1 + N/k)}{\Gamma(1 + N/\tilde{k})}|Z|^{N}
$$
where $R_{N-1}F(u) = F(u) - \sum_{n=1}^{N-1}F_{n}u^{n}$, for all $N \geq 2$, for all
$Z \in S(d,\delta'') \cap D(0,\rho')$. Therefore
(\ref{kappa_Gevrey_ae_analytic_Borel}) follows.

In the last part of the proof, we show the identity (\ref{Laplace_Borel_inverse}). We follow the same lines of
arguments as Theorem 3 in Section 2.4 from \cite{ba}. Using Fubini's theorem, we can write
\begin{multline}
\mathcal{L}_{m_{\tilde{k}}}^{d}(v \mapsto (\mathcal{B}_{m_{\tilde{k}}}^{d}F)(v) )(T) =
\tilde{k}\int_{L_d} \left( -\frac{\tilde{k}}{2i \pi} \int_{\gamma_{\tilde{k}}} F(u) e^{(\frac{v}{u})^{\tilde{k}}}
\frac{v^{\tilde{k}}}{u^{\tilde{k}+1}} du \right) e^{-(\frac{v}{T})^{\tilde{k}}} \frac{dv}{v}\\
= -\frac{\tilde{k}}{2i \pi} \int_{\gamma_{\tilde{k}}} \frac{F(u)}{u^{\tilde{k}+1}}
\left( \int_{L_d} \exp( v^{\tilde{k}}( \frac{1}{u^{\tilde{k}}} - \frac{1}{T^{\tilde{k}}} ) ) \tilde{k} v^{\tilde{k}-1} dv \right) du
\end{multline}
Therefore, by direct integration, we deduce that
\begin{equation}
\mathcal{L}_{m_{\tilde{k}}}^{d}(v \mapsto (\mathcal{B}_{m_{\tilde{k}}}^{d}F)(v) )(T) =
\frac{\tilde{k}}{2i \pi} \int_{\gamma_{\tilde{k}}} \frac{F(u)}{u} \frac{T^{\tilde{k}}}{T^{\tilde{k}} - u^{\tilde{k}}} du.
\end{equation}
Now, the function $u \mapsto \frac{F(u)}{u} \frac{T^{\tilde{k}}}{T^{\tilde{k}} - u^{\tilde{k}}}$ has in the interior of
$\gamma_{\tilde{k}}$ exactly one singularity at $u=T$ (since $T$ is assumed to belong to
$S(d,\frac{\pi}{\tilde{k}} + \delta', \rho/2)$), this being a pole of order one, with residue $-F(T)/\tilde{k}$.
The residue theorem completes the proof of (\ref{Laplace_Borel_inverse}).
\end{proof}

\begin{prop} Let $S(d,\alpha)$ be an unbounded sector with direction $d \in \mathbb{R}$ and aperture $\alpha$. Let
$\tilde{k} > k > 0$ be given real numbers and let $\kappa>0$ be the real number defined by $1/\kappa = 1/k - 1/\tilde{k}$.
Let $f : S(d,\alpha) \cup D(0,r) \rightarrow \mathbb{C}$ be an analytic function with $f(0)=0$ and such that there exist $C,M>0$ with
$$ |f(h)| \leq Ce^{M|h|^{\kappa}} $$
for all $h \in S(d,\alpha) \cup D(0,r)$.

For all $0 < \delta' <\pi/\kappa$ (which can be chosen close to $\pi/\kappa$), we define the kernel function
$$ G(\xi,h) = -\frac{\tilde{k} k}{2 i \pi} \xi^{\tilde{k}} \int_{V_{d,\tilde{k},\delta'}}
\exp( -(\frac{h}{u})^{k} + (\frac{\xi}{u})^{\tilde{k}} ) \frac{du}{u^{\tilde{k}+1}} $$
where $V_{d,\tilde{k},\delta'}$ is the path starting from 0 along the halfline
$\mathbb{R}_{+}e^{i(d + \frac{\pi}{2\tilde{k}} + \frac{\delta'}{2})}$ and back to the origin
along the halfline $\mathbb{R}_{+}e^{i(d - \frac{\pi}{2\tilde{k}} - \frac{\delta'}{2})}$. The function $G(\xi,h)$ is well
defined and satisfies the following estimates : there exist $c_{1},c_{2}>0$ such that
\begin{equation}
|G(\xi,h)| \leq c_{1} \exp( -c_{2} (\frac{|h|}{|\xi|})^{\kappa} ) \label{G_xi_h_exp_growth_order_kappa}
\end{equation}
for all $h \in L_{d} = \mathbb{R}_{+}e^{id}$ and all $\xi \in S(d, \delta'')$ for $0< \delta'' < \delta'$ (that can be chosen close to
$\delta'$).

Then, for any $0 < \rho < (c_{2}/M)^{1/\kappa}$, the function
$$ \mathcal{A}_{m_{\tilde{k}},m_{k}}^{d}f(\xi) = \int_{L_{d}} f(h) G(\xi,h) \frac{dh}{h} = g(\xi) $$
defines an analytic function on the bounded sector $S_{d,\kappa,\delta,\rho}$ with aperture $\frac{\pi}{\kappa} + \delta$,
for any $0 < \delta < \alpha$, in direction $d$ and radius $\rho$ and which satisfies that there exist $C,K>0$ with
\begin{equation}
|g(\xi) - \sum_{n=1}^{N-1} f_{n} \frac{\Gamma(n/k)}{\Gamma(n/\tilde{k})} \xi^{n}| \leq
CK^{N}\Gamma(1 + N/\kappa) |\xi|^{N} \label{kappa_Gevrey_ae_analytic_acceleration}
\end{equation}
for all $\xi \in S_{d,\kappa,\delta,\rho}$, all $N \geq 2$, where
$\hat{g}(\xi) = \sum_{n \geq 1} f_{n} \frac{\Gamma(n/k)}{\Gamma(n/\tilde{k})} \xi^{n}$ is
the formal acceleration $\hat{\mathcal{A}}_{m_{\tilde{k}},m_{k}} \hat{f} (\xi)$
where $\hat{f}(h) = \sum_{n \geq 1} f_{n}h^{n}$ is the (convergent) Taylor expansion at $h=0$ of $f$ on $D(0,r)$.

In other words, $g(\xi)$ is the $\kappa-$sum of $\hat{g}(\xi)$ on $S_{d,\kappa,\delta,\rho}$
in the sense of the definition \cite{ba} from Section 3.2.
\end{prop}

\begin{proof} We first show the estimates (\ref{G_xi_h_exp_growth_order_kappa}). We follow the idea of
proof of Lemma 1, Section 5.1 in \cite{ba}. We make the change of variable $u=h\tilde{u}$
in the integral $G(\xi,h)$ and we deform the path of integration in order to get the expression
$$ G(\xi,h) = - \frac{\tilde{k}k}{2 i \pi} (\frac{\xi}{h})^{\tilde{k}} \int_{\gamma_{\tilde{k}}} e^{-(1/\tilde{u})^{k}}
e^{(\xi/h)^{\tilde{k}}(\frac{1}{\tilde{u}})^{\tilde{k}}} \frac{1}{\tilde{u}^{\tilde{k}+1}} d\tilde{u}
$$
where $\gamma_{\tilde{k}}$ is the closed Hankel path defined in Proposition 12 with the direction $d=0$. Hence, we
recognize that $G(\xi,h)$ can be
written as an analytic Borel transform $G(\xi,h) = k (\mathcal{B}_{m_{\tilde{k}}}^{0}e_{k})(\xi/h)$ where
$e_{k}(u) = e^{-(1/u)^{k}}$. 
From Exercise 1 in Section 2.2 from \cite{ba}, we know that $e_{k}(u)$ has $\hat{0}$ as formal power series expansion of Gevrey order
$k$ on any sector $S_{0,\frac{\pi}{\tilde{k}} + \delta}$ with direction 0 for any $0<\delta<\pi/\kappa$. From Proposition 12,
we deduce that
$(\mathcal{B}_{m_{\tilde{k}}}^{0}e_{k})(Z)$ has $\hat{0}$ as formal series expansion of Gevrey order $\kappa$ on any unbounded
sector $S_{0,\delta''}$ where $0 < \delta'' < \delta' < \delta < \pi/\kappa$ (where $\delta''$ can be chosen close to $\pi/\kappa$.
Finally, using Exercice 3 in Section 2.2 from \cite{ba}, we get two constants $c_{1},c_{2}>0$ such that
$$ |(\mathcal{B}_{m_{\tilde{k}}}^{0}e_{k})(Z)| \leq c_{1} e^{-c_{2}|Z|^{-\kappa}} $$
for all $Z \in S_{0,\delta''}$. The estimates (\ref{G_xi_h_exp_growth_order_kappa}) follow.

In order to show the asymptotic expansion with bound estimates (\ref{kappa_Gevrey_ae_analytic_acceleration}), we first check that
if $f(h)=h^{n}$, for an integer $n \geq 0$, then
\begin{equation}
\mathcal{A}_{m_{\tilde{k}},m_{k}}^{d}f(\xi) =
\frac{\Gamma(n/k)}{\Gamma(n/\tilde{k})}\xi^{n} \label{analytic_acceleration_monomial}
\end{equation}
on $S_{d,\kappa,\delta,\rho}$. Indeed using Fubini's theorem, we can write
$$
\mathcal{A}_{m_{\tilde{k}},m_{k}}^{d}f(\xi) = -\frac{\tilde{k}}{2i\pi}
\int_{V_{d,\tilde{k},\delta'}} \left( k \int_{L_d} h^{n} e^{-(\frac{h}{u})^{k}}  \frac{dh}{h} \right)
e^{(\frac{\xi}{u})^{\tilde{k}}} \frac{\xi^{\tilde{k}}}{u^{\tilde{k}+1}} du.
$$ 
From the definition of the Gamma function we know that
$$ 
k \int_{L_d} h^{n} e^{-(\frac{h}{u})^{k}}  \frac{dh}{h} = \mathcal{L}_{m_{k}}^{d} (h^{n})(u) = \Gamma(\frac{n}{k}) u^{n},
$$
and bearing in mind (\ref{analytic_borel_monomial}), after a path deformation, we recognize that
$$
\mathcal{A}_{m_{\tilde{k}},m_{k}}^{d}f(\xi) = \Gamma( \frac{n}{k} ) \mathcal{B}_{m_{\tilde{k}}}^{d}(u^{n})(\xi) =  
\frac{\Gamma(n/k)}{\Gamma(n/\tilde{k})}\xi^{n}.
$$
Since the Taylor expansion of $f$ at $h=0$ is convergent, there exist two constants $C_{f},K_{f}>0$ such that
\begin{equation}
|f(h) - \sum_{n=1}^{N-1} f_{n} h^{n}| \leq C_{f}K_{f}^{N} |h|^{N} \label{cv_Taylor_f}
\end{equation}
for all $h \in D(0,r)$, all $N \geq 2$. Taking the expansion (\ref{cv_Taylor_f}) and the exponential growth estimates 
(\ref{G_xi_h_exp_growth_order_kappa}), using the same integrals estimates as in Exercice 3 in Section 2.1 of \cite{ba}, we get two constants
$C,K>0$ such that
$$
| \mathcal{A}_{m_{\tilde{k}},m_{k}}^{d}f(\xi) - \sum_{n=1}^{N-1} f_{n} \frac{ \Gamma( \frac{n}{k} )}{ \Gamma( \frac{n}{\tilde{k}} )}
\xi^{n} | = | \mathcal{A}_{m_{\tilde{k}},m_{k}}^{d} (R_{N-1}f)(\xi) | \leq C K^{N} \Gamma( 1 + N/\kappa )|\xi|^{N} 
$$
where $R_{N-1}f(h) = f(h) - \sum_{n=1}^{N-1} f_{n} h^{n}$, for all $N \geq 2$, all $\xi \in S_{d,\kappa,\delta,\rho}$.
\end{proof}

\subsection{Analytic solutions for an auxiliary convolution problem resulting from a $m_{k_{2}}-$Borel transform applied
to the main convolution initial value problem}

We keep the notations of Sections 4.1 and 4.2. For the integers $d_{l},\delta_{l}$, for $1 \leq l \leq D-1$ that satisfy the
constraints (\ref{delta_constraints}), (\ref{constrain_d_l_delta_l_Delta_l}) and
(\ref{constraint_dl_deltal_k1}), we make the additional assumption that there
exist integers $d_{l,k_{2}}^{2} > 0$ such that 
\begin{equation}
d_{l} + k_{2} + 1 = \delta_{l}(k_{2}+1) + d_{l,k_{2}}^{2}  \label{constraint_dl_deltal_k2} 
\end{equation}
for all $1 \leq l \leq D-1$. In order the ensure the positivity of the integers $d_{l,k_{2}}^{2}$, we impose the following
assumption on the integers $d_{l,k_{1}}^{1}$,
\begin{equation}
d_{l,k_{1}}^{1} > (\delta_{l}-1)(k_{2}-k_{1}),   \label{constraint_d_l_k1_with_k2}
\end{equation}
for all $1 \leq l \leq D-1$. Indeed, by Definition of $d_{l,k_{1}}^{1}$ in (\ref{defin_d_l_k1_1}), the constraint
(\ref{constraint_dl_deltal_k2}) rewrites $d_{l,k_{2}}^{2} = d_{l,k_{1}}^{1} + k_{2} - k_{1} - \delta_{l}(k_{2}-k_{1})$.
Using the formula (8.7) from \cite{taya}, p. 3630, we can expand the operators
$T^{\delta_{l}(k_{2}+1)} \partial_{T}^{\delta_l}$ in the form
\begin{equation}
T^{\delta_{l}(k_{2}+1)} \partial_{T}^{\delta_l} = (T^{k_{2}+1}\partial_{T})^{\delta_l} +
\sum_{1 \leq p \leq \delta_{l}-1} A_{\delta_{l},p} T^{k_{2}(\delta_{l}-p)} (T^{k_{2}+1}\partial_{T})^{p} \label{expand_op_diff_k2}
\end{equation}
where $A_{\delta_{l},p}$, $p=1,\ldots,\delta_{l}-1$ are real numbers, for all $1 \leq l \leq D$.

Multiplying the equation (\ref{SCP}) by $T^{k_{2}+1}$ and using
(\ref{expand_op_diff_k2}), we can rewrite the equation (\ref{SCP}) in the form
\begin{multline}
Q(im)( T^{k_{2}+1}\partial_{T}U(T,m,\epsilon) ) - R_{D}(im)(T^{k_{2}+1}\partial_{T})^{\delta_D}U(T,m,\epsilon) \\
= R_{D}(im) \sum_{1 \leq p \leq \delta_{D}-1} A_{\delta_{D},p} T^{k_{2}(\delta_{D}-p)}(T^{k_{2}+1}\partial_{T})^{p}U(T,m,\epsilon)\\
 + \epsilon^{-1}T^{k_{2}+1}
\frac{c_{1,2}(\epsilon)}{(2\pi)^{1/2}}\int_{-\infty}^{+\infty} Q_{1}(i(m-m_{1}))U(T,m-m_{1},\epsilon)
Q_{2}(im_{1})U(T,m_{1},\epsilon) dm_{1} \\
+ \sum_{l=1}^{D-1} R_{l}(im)\left( \epsilon^{\Delta_{l} - d_{l} + \delta_{l} - 1}
T^{d_{l,k_{2}}^{2}}(T^{k_{2}+1}\partial_{T})^{\delta_l}
U(T,m,\epsilon) \right.
\\+ \sum_{1 \leq p \leq \delta_{l}-1} A_{\delta_{l},p}
\left. \epsilon^{\Delta_{l}-d_{l}+\delta_{l}-1} T^{k_{2}(\delta_{l}-p) + d_{l,k_{2}}^{2}}
(T^{k_{2}+1}\partial_{T})^{p}U(T,m,\epsilon) \right)\\
+ \epsilon^{-1}T^{k_{2}+1}
\frac{c_{0}(\epsilon)}{(2\pi)^{1/2}}\int_{-\infty}^{+\infty} C_{0}(T,m-m_{1},\epsilon) R_{0}(im_{1})U(T,m_{1},\epsilon) dm_{1}\\
+ \epsilon^{-1}T^{k_{2}+1}
\frac{c_{0,0}(\epsilon)}{(2\pi)^{1/2}}\int_{-\infty}^{+\infty} C_{0,0}(m-m_{1},\epsilon) R_{0}(im_{1})U(T,m_{1},\epsilon) dm_{1}
+ \epsilon^{-1}c_{F}(\epsilon)T^{k_{2}+1}F(T,m,\epsilon)
\label{SCP_irregular_k2}
\end{multline}
We denote
$\hat{\omega}_{k_2}(\tau,m,\epsilon)$ the formal $m_{k_2}-$Borel transform of
$\hat{U}(T,m,\epsilon)$ with respect to $T$, $\varphi_{k_2}(\tau,m,\epsilon)$ the formal $m_{k_2}-$Borel transform of
$C_{0}(T,m,\epsilon)$ with respect to $T$ and $\hat{\psi}_{k_2}(\tau,m,\epsilon)$ the formal $m_{k_2}-$Borel transform of
$F(T,m,\epsilon)$ with respect to $T$,
\begin{multline}
\hat{\omega}_{k_2}(\tau,m,\epsilon) = \sum_{n \geq 1} U_{n}(m,\epsilon) \frac{\tau^n}{\Gamma(\frac{n}{k_2})} \ \ , \ \
\varphi_{k_2}(\tau,m,\epsilon) = \sum_{n \geq 1} C_{0,n}(m,\epsilon) \frac{\tau^n}{\Gamma(\frac{n}{k_2})}\\
\hat{\psi}_{k_2}(\tau,m,\epsilon) = \sum_{n \geq 1} F_{n}(m,\epsilon) \frac{\tau^n}{\Gamma(\frac{n}{k_2})}
\label{defin_borel_k_2_omega_varphi_psi}
\end{multline}
Using the computation rules for the formal $m_{k_2}-$Borel transform in Proposition 8, we deduce the following equation satisfied by
$\hat{\omega}_{k_2}(\tau,m,\epsilon)$,
\begin{multline}
Q(im)( k_{2} \tau^{k_2} \hat{\omega}_{k_2}(\tau,m,\epsilon) ) - (k_{2} \tau^{k_2})^{\delta_D}R_{D}(im)
\hat{\omega}_{k_2}(\tau,m,\epsilon)\\
= R_{D}(im) \sum_{1 \leq p \leq \delta_{D}-1} A_{\delta_{D},p}
\frac{\tau^{k_2}}{\Gamma(\delta_{D}-p)}\int_{0}^{\tau^{k_2}} (\tau^{k_2}-s)^{\delta_{D}-p-1}
 (k_2^{p} s^{p} \hat{\omega}_{k_2}(s^{1/k_2},m,\epsilon)) \frac{ds}{s}\\
+ \epsilon^{-1}
\frac{\tau^{k_2}}{\Gamma(1 + \frac{1}{k_2})} \int_{0}^{\tau^{k_2}}
(\tau^{k_2}-s)^{1/k_2}\\
\times \left( \frac{c_{1,2}(\epsilon)}{(2\pi)^{1/2}} s\int_{0}^{s} \int_{-\infty}^{+\infty} \right.
Q_{1}(i(m-m_{1}))\hat{\omega}_{k_2}((s-x)^{1/k_2},m-m_{1},\epsilon)\\
\left. \times  Q_{2}(im_{1})
\hat{\omega}_{k_2}(x^{1/k_2},m_{1},\epsilon) \frac{1}{(s-x)x} dxdm_{1} \right) \frac{ds}{s}\\
+ \sum_{l=1}^{D-1} R_{l}(im) \left( \epsilon^{\Delta_{l}-d_{l}+\delta_{l}-1}
\frac{\tau^{k_2}}{\Gamma( \frac{d_{l,k_{2}}^{2}}{k_2} )} \right.
\int_{0}^{\tau^{k_2}} (\tau^{k_2}-s)^{\frac{d_{l,k_{2}}^{2}}{k_2}-1}({k_2}^{\delta_l}s^{\delta_l}
\hat{\omega}_{k_2}(s^{1/k_2},m,\epsilon)) \frac{ds}{s}\\
+ \sum_{1 \leq p \leq \delta_{l}-1} A_{\delta_{l},p}\epsilon^{\Delta_{l}-d_{l}+\delta_{l}-1}
\frac{\tau^{k_2}}{\Gamma( \frac{d_{l,k_{2}}^{2}}{k_2} + \delta_{l}-p)} \int_{0}^{\tau^{k_2}}
\left. (\tau^{k_2}-s)^{\frac{d_{l,k_{2}}^{2}}{k_2}+\delta_{l}-p-1}({k_2}^{p}s^{p}
\hat{\omega}_{k_2}(s^{1/k_2},m,\epsilon)) \frac{ds}{s} \right)\\
+ \epsilon^{-1}
\frac{\tau^{k_2}}{\Gamma(1 + \frac{1}{k_2})} \int_{0}^{\tau^{k_2}}
(\tau^{k_2}-s)^{1/k_2}\\
\times \left( \frac{c_{0}(\epsilon)}{(2\pi)^{1/2}} s\int_{0}^{s} \int_{-\infty}^{+\infty} \right.
\left. \varphi_{k_2}((s-x)^{1/k_2},m-m_{1},\epsilon) R_{0}(im_{1}) \hat{\omega}_{k_2}(x^{1/k_2},m_{1},\epsilon) \frac{1}{(s-x)x}
dxdm_{1} \right) \frac{ds}{s}\\
+ \epsilon^{-1}\frac{\tau^{k_2}}{\Gamma(1 + \frac{1}{k_2})} \int_{0}^{\tau^{k_2}}
(\tau^{k_2}-s)^{1/k_2} \frac{c_{0,0}(\epsilon)}{(2\pi)^{1/2}} ( \int_{-\infty}^{+\infty} C_{0,0}(m-m_{1},\epsilon)R_{0}(im_{1})
\hat{\omega}_{k_2}(s^{1/k_2},m_{1},\epsilon) dm_{1} )\frac{ds}{s}\\
+ \epsilon^{-1} c_{F}(\epsilon) \frac{\tau^{k_2}}{\Gamma(1 + \frac{1}{k_2})}\int_{0}^{\tau^{k_2}}
(\tau^{k_2}-s)^{1/k_2} \hat{\psi}_{k_2}(s^{1/k_2},m,\epsilon) \frac{ds}{s} \label{k2_Borel_equation}
\end{multline}
We recall from \cite{lama} that
$\varphi_{k_2}(\tau,m,\epsilon) \in F_{(\nu,\beta,\mu,k_{2})}^{d}$ for
all $\epsilon \in D(0,\epsilon_{0})$, any unbounded sector $S_{d}$ and any bounded
sector $S_{d}^{b}$ centered at 0 with bisecting direction $d \in \mathbb{R}$, for some $\nu>0$. Indeed, we have that
\begin{multline}
||\varphi_{k_2}(\tau,m,\epsilon)||_{(\nu,\beta,\mu,k_{2})} \leq \sum_{n \geq 1}
||C_{0,n}(m,\epsilon)||_{(\beta,\mu)} (\sup_{\tau \in \bar{S}_{d}^{b} \cup S_{d}}
\frac{1 + |\tau|^{2k_{2}}}{|\tau|} \exp(-\nu |\tau|^{k_{2}})
\frac{|\tau|^n}{\Gamma(\frac{n}{k_{2}})}) \label{maj_norm_varphi_k_2}
\end{multline}
By using the classical estimates (\ref{x_m_exp_x<}) and Stirling formula
$\Gamma(n/k_{2}) \sim (2\pi)^{1/2}(n/k_{2})^{\frac{n}{k_2}-\frac{1}{2}}e^{-n/k_{2}}$ as $n$ tends to $+\infty$, we get two constants
$A_{1},A_{2}>0$ depending on $\nu,k_{2}$ such that
\begin{multline}
\sup_{\tau \in \bar{S}_{d}^{b} \cup S_{d}}
\frac{1 + |\tau|^{2k_{2}}}{|\tau|} \exp(-\nu |\tau|^{k_2})
\frac{|\tau|^n}{\Gamma(\frac{n}{k_{2}})} 
= \sup_{x \geq 0} (1+x^{2})x^{\frac{n-1}{k_{2}}}
\frac{e^{-\nu x}}{\Gamma(\frac{n}{k_{2}})}\\
\leq \left( (\frac{n-1}{\nu k_{2}})^{\frac{n-1}{k_{2}}}
e^{-\frac{n-1}{k_{2}}} +
( \frac{n-1}{\nu k_{2}} + \frac{2}{\nu})^{\frac{n-1}{k_{2}}+2} e^{-(\frac{n-1}{k_{2}}+2)} \right)/ \Gamma(n/k_{2})
\leq A_{1}(A_{2})^{n} \label{sup_Stirling}
\end{multline}
for all $n \geq 1$, all $\epsilon \in D(0,\epsilon_{0})$. Therefore, if
$A_{2} < T_{0}$ holds, we get the estimates
\begin{multline}
||\varphi_{k_2}(\tau,m,\epsilon)||_{(\nu,\beta,\mu,k_{2})} \leq A_{1} \sum_{n \geq 1} ||C_{0,n}(m,\epsilon)||_{(\beta,\mu)}
(A_{2})^{n} \leq \frac{A_{1}A_{2}K_{0}}{T_0} \frac{1}{ 1 - \frac{A_{2}}{T_0} },\\
\label{norm_F_varphi_k2_epsilon_0}
\end{multline}
for all $\epsilon \in D(0,\epsilon_{0})$.

From Section 4.2, we recall that $\psi_{k_1}^{d}(\tau,m,\epsilon) \in F_{(\nu,\beta,\mu,k_{1},k_{1})}^{d}$, for
all $\epsilon \in D(0,\epsilon_{0})$, for some unbounded sector $U_{d}$ with bisecting direction
$d \in \mathbb{R}$, where $\nu$ is chosen in that section.
\begin{lemma} The function
$$
\psi_{k_2}^{d}(\tau,m,\epsilon) := \mathcal{A}_{m_{k_2},m_{k_1}}^{d}( h \mapsto \psi_{k_1}^{d}(h,m,\epsilon)) (\tau) =
\int_{L_{d}} \psi_{k_1}^{d}(h,m,\epsilon) G(\tau,h)
\frac{dh}{h}
$$
is analytic on an unbounded sector $S_{d,\kappa,\delta}$ with aperture $\frac{\pi}{\kappa} + \delta$ in direction $d$, for any
$0 < \delta < \mathrm{ap}(U_{d})$ where $\mathrm{ap}(U_{d})$ denotes the aperture of the sector $U_{d}$, and has estimates
of the form : there exist constants $C_{\psi_{k_2}}>0$ and $\nu'>0$ such that 
\begin{equation}
|\psi_{k_2}^{d}(\tau,m,\epsilon)| \leq C_{\psi_{k_2}} (1+|m|)^{-\mu} e^{-\beta |m|}
\frac{ |\tau| }{1 + |\tau|^{2k_{2}}} \exp( \nu' |\tau|^{k_2} )
\label{psi_k2_exp_growth}
\end{equation}
for all $\tau \in S_{d,\kappa,\delta}$, all $m \in \mathbb{R}$, all $\epsilon \in D(0,\epsilon_{0})$. In particular,
we get that $\mathcal{A}_{m_{k_2},m_{k_1}}^{d}( h \mapsto \psi_{k_1}^{d}(h,m,\epsilon)) (\tau) \in
F_{(\nu',\beta,\mu,k_{2})}^{d}$
for any unbounded sector $S_{d}$ and bounded sector $S_{d}^{b}$ with aperture $\frac{\pi}{\kappa} + \delta$, with $\delta$ as above,
and we carry a constant $\zeta_{\psi_{k_2}}>0$ with 
\begin{equation}
||\psi_{k_2}^{d}(\tau,m,\epsilon)||_{(\nu',\beta,\mu,k_{2})} \leq \zeta_{\psi_{k_2}}
\end{equation}
for all $\epsilon \in D(0,\epsilon_{0})$.
\end{lemma}
\begin{proof} Bearing in mind the inclusion (\ref{psi_k1_bounded_norm_k1_kappa}) we already know from Proposition 13 that
the function $\tau \mapsto \psi_{k_2}^{d}(\tau,m,\epsilon)$ defines a holomorphic and bounded function (with bound independent of
$\epsilon \in D(0,\epsilon_{0})$) on a sector
$S_{d,\kappa,\delta,(c_{2}/\nu)^{1/\kappa}/2}$ with direction $d$, aperture
$\frac{\pi}{\kappa} + \delta$ and radius $(c_{2}/\nu)^{1/\kappa}/2$, for some
$\delta>0$ and the constant $c_{2}$ introduced in (\ref{G_xi_h_exp_growth_order_kappa}), for all $m \in \mathbb{R}$, all
$\epsilon \in D(0,\epsilon_{0})$.

From the assumption that the function $\psi_{k_1}^{d}(\tau,m,\epsilon)$ belongs to the space
$F_{(\nu,\beta,\mu,k_{1},k_{1})}^{d}$, see
(\ref{psi_k1_bounded_norm_k1_k1}), we know that the $m_{k_1}-$Laplace transform
$$
\mathcal{L}_{m_{k_1}}^{d}(h \mapsto \psi_{k_1}^{d}(h,m,\epsilon))(u) = k_{1} \int_{L_d} \psi_{k_1}^{d}(h,m,\epsilon)
\exp(-(\frac{h}{u})^{k_1}) \frac{dh}{h}
$$
defines a holomorphic and bounded function (by a constant that does not depend on
$\epsilon \in D(0,\epsilon_{0})$) on a sector
$S_{d,\theta,\sigma'}$ in direction $d$, with radius
$\sigma'$ and aperture $\theta$ which satisfies $\frac{\pi}{k_2} + \frac{\pi}{\kappa} < \theta <
\frac{\pi}{k_2} + \frac{\pi}{\kappa} + \mathrm{ap}(U_{d})$, where $\mathrm{ap}(U_{d})$ is the aperture of $U_{d}$, for
some $\sigma'>0$.

Hence, by using a path deformation and the Fubini theorem, we can rewrite the function $\psi_{k_2}^{d}(\tau,m,\epsilon)$ in the form
\begin{multline}
\psi_{k_2}^{d}(\tau,m,\epsilon) = -\frac{k_2}{2i \pi} \int_{V_{d,k_{2},\delta',\sigma'/2}}
\mathcal{L}_{m_{k_1}}^{d}(h \mapsto \psi_{k_1}^{d}(h,m,\epsilon))(u) e^{(\frac{\tau}{u})^{k_2}}
\frac{\tau^{k_2}}{u^{k_{2}+1}} du \\
= \mathcal{B}_{m_{k_2}}^{d}( \mathcal{L}_{m_{k_1}}^{d}(h \mapsto \psi_{k_1}^{d}(h,m,\epsilon))(u) ) (\tau) \label{psi_k2_Borel_form}
\end{multline}
where $V_{d,k_{2},\delta',\sigma'/2}$ is the closed Hankel path starting from the origin along the segment
$$ [0,(\sigma'/2)e^{i(d + \frac{\pi}{2k_{2}} + \frac{\delta'}{2})}]$$
following the arc of circle
$[ (\sigma'/2)e^{i(d + \frac{\pi}{2k_{2}} + \frac{\delta'}{2})},
(\sigma'/2)e^{i(d - \frac{\pi}{2k_{2}} - \frac{\delta'}{2})} ]$ and going back to the origin along the segment
$[(\sigma'/2)e^{i(d - \frac{\pi}{2k_{2}} - \frac{\delta'}{2})},0]$, where $0 < \delta' < \frac{\pi}{\kappa} + \mathrm{ap}(U_{d})$
that can be chosen close to $\frac{\pi}{\kappa} + \mathrm{ap}(U_{d})$.

Therefore, from Proposition 12, we know that $\tau \mapsto \psi_{k_2}^{d}(\tau,m,\epsilon)$ defines a holomorphic function
on the unbounded sector $S(d,\delta'')$ where $0 < \delta'' < \delta'$, which can be chosen close to $\delta'$, for all
$m \in \mathbb{R}$, all $\epsilon \in D(0,\epsilon_{0})$. Now, we turn to the estimates (\ref{psi_k2_exp_growth}).
From the representation (\ref{psi_k2_Borel_form}), we get the following estimates : there exist constants $E_{1},E_{2},E_{3}>0$
such that
\begin{multline}
|\psi_{k_2}^{d}(\tau,m,\epsilon)| \leq \frac{E_{1}e^{-\beta|m|}}{(1 + |m|)^{\mu}} \left( e^{E_{2}|\tau|^{k_2}}
|\tau|^{k_2} + \int_{0}^{\frac{\sigma'}{2}} e^{-E_{3}(\frac{|\tau|}{s})^{k_2}}
\frac{|\tau|^{k_2}}{s^{k_{2}+1}} ds \right)\\
\leq \frac{E_{1}e^{-\beta|m|}}{(1 + |m|)^{\mu}} \left( e^{E_{2}|\tau|^{k_2}}
|\tau|^{k_2}  + \frac{1}{E_{3}k_{2}} e^{-E_{3}(\frac{2}{\sigma'})^{k_2}|\tau|^{k_2}} \right) \label{psi_k2_epsilontau_growth_large_tau}
\end{multline}
for all $\tau \in S(d,\delta'')$, all $m \in \mathbb{R}$, all $\epsilon \in D(0,\epsilon_{0})$. Besides, from the asymptotic
expansion (\ref{kappa_Gevrey_ae_analytic_Borel}), we get in particular the existence of a constant $E_{0}>0$ such that
\begin{equation}
|\psi_{k_2}^{d}(\tau,m,\epsilon)| \leq \frac{E_{0}e^{-\beta|m|}}{(1 + |m|)^{\mu}}|\tau| \label{psi_k2_epsilontau_growth_small_tau}
\end{equation}
for all $\tau \in S(d,\delta'') \cap D(0,\rho')$ and some $\rho'>0$. Finally, combining the estimates
(\ref{psi_k2_epsilontau_growth_large_tau}) and (\ref{psi_k2_epsilontau_growth_small_tau}) yields (\ref{psi_k2_exp_growth}).
\end{proof}
We consider now the following problem
\begin{multline}
Q(im)( k_{2} \tau^{k_2} \omega_{k_2}(\tau,m,\epsilon) ) - (k_{2} \tau^{k_2})^{\delta_D}R_{D}(im)
\omega_{k_2}(\tau,m,\epsilon)\\
= R_{D}(im) \sum_{1 \leq p \leq \delta_{D}-1} A_{\delta_{D},p}
\frac{\tau^{k_2}}{\Gamma(\delta_{D}-p)}\int_{0}^{\tau^{k_2}} (\tau^{k_2}-s)^{\delta_{D}-p-1}
 (k_2^{p} s^{p} \omega_{k_2}(s^{1/k_2},m,\epsilon)) \frac{ds}{s}\\
+ \epsilon^{-1}
\frac{\tau^{k_2}}{\Gamma(1 + \frac{1}{k_2})} \int_{0}^{\tau^{k_2}}
(\tau^{k_2}-s)^{1/k_2}\\
\times \left( \frac{c_{1,2}(\epsilon)}{(2\pi)^{1/2}} s\int_{0}^{s} \int_{-\infty}^{+\infty} \right.
Q_{1}(i(m-m_{1}))\omega_{k_2}((s-x)^{1/k_2},m-m_{1},\epsilon)\\
\left. \times  Q_{2}(im_{1})
\omega_{k_2}(x^{1/k_2},m_{1},\epsilon) \frac{1}{(s-x)x} dxdm_{1} \right) \frac{ds}{s}\\
+ \sum_{l=1}^{D-1} R_{l}(im) \left( \epsilon^{\Delta_{l}-d_{l}+\delta_{l}-1}
\frac{\tau^{k_2}}{\Gamma( \frac{d_{l,k_{2}}^{2}}{k_2} )} \right.
\int_{0}^{\tau^{k_2}} (\tau^{k_2}-s)^{\frac{d_{l,k_{2}}^{2}}{k_2}-1}({k_2}^{\delta_l}s^{\delta_l}
\omega_{k_2}(s^{1/k_2},m,\epsilon)) \frac{ds}{s}\\
+ \sum_{1 \leq p \leq \delta_{l}-1} A_{\delta_{l},p}\epsilon^{\Delta_{l}-d_{l}+\delta_{l}-1}
\frac{\tau^{k_2}}{\Gamma( \frac{d_{l,k_{2}}^{2}}{k_2} + \delta_{l}-p)} \int_{0}^{\tau^{k_2}}
\left. (\tau^{k_2}-s)^{\frac{d_{l,k_{2}}^{2}}{k_2}+\delta_{l}-p-1}({k_2}^{p}s^{p}
\omega_{k_2}(s^{1/k_2},m,\epsilon)) \frac{ds}{s} \right)\\
+ \epsilon^{-1}
\frac{\tau^{k_2}}{\Gamma(1 + \frac{1}{k_2})} \int_{0}^{\tau^{k_2}}
(\tau^{k_2}-s)^{1/k_2}\\
\times \left( \frac{c_{0}(\epsilon)}{(2\pi)^{1/2}} s\int_{0}^{s} \int_{-\infty}^{+\infty} \right.
\left. \varphi_{k_2}((s-x)^{1/k_2},m-m_{1},\epsilon) R_{0}(im_{1}) \omega_{k_2}(x^{1/k_2},m_{1},\epsilon) \frac{1}{(s-x)x}
dxdm_{1} \right) \frac{ds}{s}\\
+ \epsilon^{-1}\frac{\tau^{k_2}}{\Gamma(1 + \frac{1}{k_2})} \int_{0}^{\tau^{k_2}}
(\tau^{k_2}-s)^{1/k_2} \frac{c_{0,0}(\epsilon)}{(2\pi)^{1/2}} ( \int_{-\infty}^{+\infty} C_{0,0}(m-m_{1},\epsilon)R_{0}(im_{1})
\omega_{k_2}(s^{1/k_2},m_{1},\epsilon) dm_{1} )\frac{ds}{s}\\
+ \epsilon^{-1} c_{F}(\epsilon)\frac{\tau^{k_2}}{\Gamma(1 + \frac{1}{k_2})}\int_{0}^{\tau^{k_2}}
(\tau^{k_2}-s)^{1/k_2} \psi_{k_2}^{d}(s^{1/k_2},m,\epsilon) \frac{ds}{s} \label{k2_Borel_equation_analytic}
\end{multline}
for vanishing initial data $\omega_{k_2}(0,m,\epsilon) \equiv 0$, where $\psi_{k_2}^{d}(\tau,m,\epsilon)$ has been constructed in Lemma
4.

We make the additional assumption that there exists an unbounded sector
$$ S_{Q,R_{D}} = \{ z \in \mathbb{C} / |z| \geq r_{Q,R_{D}} \ \ , \ \ |\mathrm{arg}(z) - d_{Q,R_{D}}| \leq \eta_{Q,R_{D}} \} $$
with direction $d_{Q,R_{D}} \in \mathbb{R}$, aperture $\eta_{Q,R_{D}}>0$ for some radius $r_{Q,R_{D}}>0$ such that
\begin{equation}
\frac{Q(im)}{R_{D}(im)} \in S_{Q,R_{D}} \label{quotient_Q_RD_in_S}
\end{equation} 
for all $m \in \mathbb{R}$. We factorize the polynomial $P_{m}(\tau) = Q(im)k_{2} - R_{D}(im)k_{2}^{\delta_D}
\tau^{(\delta_{D}-1)k_{2}}$ in the form
\begin{equation}
 P_{m}(\tau) = -R_{D}(im)k_{2}^{\delta_D}\Pi_{l=0}^{(\delta_{D}-1)k_{2}-1} (\tau - q_{l}(m)) \label{factor_P_m}
\end{equation}
where
\begin{multline}
q_{l}(m) = (\frac{|Q(im)|}{|R_{D}(im)|k_{2}^{\delta_{D}-1}})^{\frac{1}{(\delta_{D}-1)k_{2}}}\\
\times \exp( \sqrt{-1}( \mathrm{arg}( \frac{Q(im)}{R_{D}(im)k_{2}^{\delta_{D}-1}}) \frac{1}{(\delta_{D}-1)k_{2}} +
\frac{2\pi l}{(\delta_{D}-1)k_{2}} ) ) \label{defin_roots}
\end{multline}
for all $0 \leq l \leq (\delta_{D}-1)k_{2}-1$, all $m \in \mathbb{R}$.

We choose an unbounded sector $S_{d}$ centered at 0, a small closed disc $\bar{D}(0,\rho)$ and we prescribe the sector
$S_{Q,R_{D}}$ in such a way that the following conditions hold.\medskip

\noindent 1) There exists a constant $M_{1}>0$ such that
\begin{equation}
|\tau - q_{l}(m)| \geq M_{1}(1 + |\tau|) \label{root_cond_1}
\end{equation}
for all $0 \leq l \leq (\delta_{D}-1)k_{2}-1$, all $m \in \mathbb{R}$, all $\tau \in S_{d} \cup \bar{D}(0,\rho)$. Indeed,
from (\ref{quotient_Q_RD_in_S}) and the explicit expression (\ref{defin_roots}) of $q_{l}(m)$, we first observe that
$|q_{l}(m)| > 2\rho$ for every $m \in \mathbb{R}$, all $0 \leq l \leq (\delta_{D}-1)k_{2}-1$ for an appropriate choice of $r_{Q,R_{D}}$
and of $\rho>0$. We also see that for all $m \in \mathbb{R}$, all $0 \leq l \leq (\delta_{D}-1)k_{2}-1$, the roots
$q_{l}(m)$ remain in a union $\mathcal{U}$ of unbounded sectors centered at 0 that do not cover a full neighborhood of
the origin in $\mathbb{C}^{\ast}$ provided that $\eta_{Q,R_{D}}$ is small enough. Therefore, one can choose an adequate
sector $S_{d}$ such that $S_{d} \cap \mathcal{U} = \emptyset$ with the property that for all
$0 \leq l \leq (\delta_{D}-1)k_{2}-1$ the quotients $q_{l}(m)/\tau$ lay outside
some small disc centered at 1 in $\mathbb{C}$ for all $\tau \in S_{d}$, all $m \in \mathbb{R}$. This yields (\ref{root_cond_1})
for some small constant $M_{1}>0$.\medskip

\noindent 2) There exists a constant $M_{2}>0$ such that
\begin{equation}
|\tau - q_{l_0}(m)| \geq M_{2}|q_{l_0}(m)| \label{root_cond_2}
\end{equation}
for some $l_{0} \in \{0,\ldots,(\delta_{D}-1)k_{2}-1 \}$, all $m \in \mathbb{R}$, all $\tau \in S_{d} \cup \bar{D}(0,\rho)$. Indeed, for the
sector $S_{d}$ and the disc $\bar{D}(0,\rho)$ chosen as above in 1), we notice that for any fixed
$0 \leq l_{0} \leq (\delta_{D}-1)k_{2}-1$, the quotient $\tau/q_{l_0}(m)$ stays outside a small disc centered at 1 in $\mathbb{C}$
for all $\tau \in S_{d} \cup \bar{D}(0,\rho)$, all $m \in \mathbb{R}$. Hence (\ref{root_cond_2}) must hold for some small
constant $M_{2}>0$.\medskip

By construction
of the roots (\ref{defin_roots}) in the factorization (\ref{factor_P_m}) and using the lower bound estimates
(\ref{root_cond_1}), (\ref{root_cond_2}), we get a constant $C_{P}>0$ such that
\begin{multline}
|P_{m}(\tau)| \geq M_{1}^{(\delta_{D}-1)k_{2}-1}M_{2}|R_{D}(im)|k_{2}^{\delta_D}
(\frac{|Q(im)|}{|R_{D}(im)|k_{2}^{\delta_{D}-1}})^{\frac{1}{(\delta_{D}-1)k_{2}}} (1+|\tau|)^{(\delta_{D}-1)k_{2}-1}\\
\geq M_{1}^{(\delta_{D}-1)k_{2}-1}M_{2}\frac{k_{2}^{\delta_D}}{(k_{2}^{\delta_{D}-1})^{\frac{1}{(\delta_{D}-1)k_{2}}}}
(r_{Q,R_{D}})^{\frac{1}{(\delta_{D}-1)k_{2}}} |R_{D}(im)| \\
\times (\min_{x \geq 0}
\frac{(1+x)^{(\delta_{D}-1)k_{2}-1}}{(1+x^{k_2})^{(\delta_{D}-1) - \frac{1}{k_2}}}) (1 + |\tau|^{k_2})^{(\delta_{D}-1)-\frac{1}{k_2}}\\
= C_{P} (r_{Q,R_{D}})^{\frac{1}{(\delta_{D}-1)k_{2}}} |R_{D}(im)| (1+|\tau|^{k_2})^{(\delta_{D}-1)-\frac{1}{k_2}}
\label{low_bounds_P_m}
\end{multline}
for all $\tau \in S_{d} \cup \bar{D}(0,\rho)$, all $m \in \mathbb{R}$.

In the next proposition, we give sufficient conditions under which the equation (\ref{k2_Borel_equation_analytic}) has a solution
$\omega_{k_2}^{d}(\tau,m,\epsilon)$ in the Banach space $F_{(\nu',\beta,\mu,k_{2})}^{d}$ where $\nu',\beta,\mu$ are defined above.
\begin{prop} Under the assumption that
\begin{equation}
\delta_{D} \geq \delta_{l} + \frac{1}{k_{2}}
\label{constraints_k2_Borel_equation}
\end{equation}
for all $1 \leq l \leq D-1$, there exist a radius $r_{Q,R_{D}}>0$, a constant $\upsilon>0$ and constants
$\varsigma_{1,2},\varsigma_{0,0},\varsigma_{0},\varsigma_{1},\varsigma_{1,0},\varsigma_{F},\varsigma_{2}>0$ (depending on
$Q_{1},Q_{2},k_{2},C_{P},\mu,\nu,\epsilon_{0},R_{l},\Delta_{l},\delta_{l},d_{l}$ for
$1 \leq l \leq D-1$) such that if
\begin{multline}
\sup_{\epsilon \in D(0,\epsilon_{0})}|\frac{c_{1,2}(\epsilon)}{\epsilon}| \leq \varsigma_{1,2} \ \ , \ \
\sup_{\epsilon \in D(0,\epsilon_{0})}|\frac{c_{0}(\epsilon)}{\epsilon}| \leq \varsigma_{1,0} \ \ , \ \
||\varphi_{k_2}(\tau,m,\epsilon)||_{(\nu',\beta,\mu,k_{2})} \leq \varsigma_{1},\\
\sup_{\epsilon \in D(0,\epsilon_{0})}|\frac{c_{0,0}(\epsilon)}{\epsilon}| \leq \varsigma_{0,0} \ \ , \ \
||C_{0,0}(m,\epsilon)||_{(\beta,\mu)} \leq \varsigma_{0},\\
\sup_{\epsilon \in D(0,\epsilon_{0})}|\frac{c_{F}(\epsilon)}{\epsilon}| \leq \varsigma_{F} \ \ , \ \
||\psi_{k_2}^{d}(\tau,m,\epsilon)||_{(\nu',\beta,\mu,k_{2})} \leq \varsigma_{2} \label{norm_F_varphi_k2_psi_k_2_small}
\end{multline}
for all $\epsilon \in D(0,\epsilon_{0})$, the equation (\ref{k2_Borel_equation_analytic}) has a unique solution
$\omega_{k_2}^{d}(\tau,m,\epsilon)$ in the space $F_{(\nu',\beta,\mu,k_{2})}^{d}$ with the property that
$||\omega_{k_2}^{d}(\tau,m,\epsilon)||_{(\nu',\beta,\mu,k_{2})} \leq \upsilon$, for all
$\epsilon \in D(0,\epsilon_{0})$, where $\beta,\mu>0$ are defined above,
for any unbounded sector $S_{d}$ that satisfies the constraints (\ref{root_cond_1}), (\ref{root_cond_2}) and for any
bounded sector $S_{d}^{b}$ with aperture strictly larger than $\frac{\pi}{\kappa}$ such that
\begin{equation}
S_{d}^{b} \subset D(0,\rho) \ \ , \ \ S_{d}^{b} \subset S_{d,\kappa,\delta} \label{constraint_S_d_b}
\end{equation}
where $D(0,\rho)$ fulfills the constraints (\ref{root_cond_1}), (\ref{root_cond_2}) and where the sector
$S_{d,\kappa,\delta}$ with aperture $\frac{\pi}{\kappa} + \delta$ is defined in Lemma 4, where $0 < \delta < \mathrm{ap}(U_{d})$. 
\end{prop}
\begin{proof} We start the proof with a lemma which provides appropriate conditions in order to apply a fixed point theorem.
\begin{lemma} One can choose the constant $r_{Q,R_{D}}>0$, a constant $\upsilon$ small enough and constants
$\varsigma_{1,2},\varsigma_{0,0},\varsigma_{0},\varsigma_{1},\varsigma_{1,0},\varsigma_{F},\varsigma_{2}>0$
(depending on
$Q_{1},Q_{2},k_{2},C_{P},\mu,\nu,\epsilon_{0},R_{l},\Delta_{l},\delta_{l},d_{l}$ for
$1 \leq l \leq D-1$) such that if (\ref{norm_F_varphi_k2_psi_k_2_small}) holds
for all $\epsilon \in D(0,\epsilon_{0})$, the map $\mathcal{H}_{\epsilon}^{k_2}$ defined by
\begin{multline}
\mathcal{H}_{\epsilon}^{k_2}(w(\tau,m))\\ =
 \frac{R_{D}(im)}{P_{m}(\tau)} \sum_{1 \leq p \leq \delta_{D}-1} A_{\delta_{D},p}
\frac{1}{\Gamma(\delta_{D}-p)}\int_{0}^{\tau^{k_2}} (\tau^{k_2}-s)^{\delta_{D}-p-1}
 (k_2^{p} s^{p} w(s^{1/k_2},m)) \frac{ds}{s}\\
+ \epsilon^{-1}
\frac{1}{P_{m}(\tau)\Gamma(1 + \frac{1}{k_2})} \int_{0}^{\tau^{k_2}}
(\tau^{k_2}-s)^{1/k_2}\\
\times \left( \frac{c_{1,2}(\epsilon)}{(2\pi)^{1/2}} s\int_{0}^{s} \int_{-\infty}^{+\infty} \right.
Q_{1}(i(m-m_{1}))w((s-x)^{1/k_2},m-m_{1})\\
\left. \times  Q_{2}(im_{1})
w(x^{1/k_2},m_{1}) \frac{1}{(s-x)x} dxdm_{1} \right) \frac{ds}{s}\\
+ \sum_{l=1}^{D-1} \frac{R_{l}(im)}{P_{m}(\tau)} \left( \epsilon^{\Delta_{l}-d_{l}+\delta_{l}-1}
\frac{1}{\Gamma( \frac{d_{l,k_{2}}^{2}}{k_2} )} \right.
\int_{0}^{\tau^{k_2}} (\tau^{k_2}-s)^{\frac{d_{l,k_{2}}^{2}}{k_2}-1}({k_2}^{\delta_l}s^{\delta_l}
w(s^{1/k_2},m)) \frac{ds}{s}\\
+ \sum_{1 \leq p \leq \delta_{l}-1} A_{\delta_{l},p}\epsilon^{\Delta_{l}-d_{l}+\delta_{l}-1}
\frac{1}{\Gamma( \frac{d_{l,k_{2}}^{2}}{k_2} + \delta_{l}-p)} \int_{0}^{\tau^{k_2}}
\left. (\tau^{k_2}-s)^{\frac{d_{l,k_{2}}^{2}}{k_2}+\delta_{l}-p-1}({k_2}^{p}s^{p}
w(s^{1/k_2},m)) \frac{ds}{s} \right)\\
+ \epsilon^{-1}
\frac{1}{P_{m}(\tau)\Gamma(1 + \frac{1}{k_2})} \int_{0}^{\tau^{k_2}}
(\tau^{k_2}-s)^{1/k_2}\\
\times \left( \frac{c_{0}(\epsilon)}{(2\pi)^{1/2}} s\int_{0}^{s} \int_{-\infty}^{+\infty} \right.
\left. \varphi_{k_2}((s-x)^{1/k_2},m-m_{1},\epsilon) R_{0}(im_{1}) w(x^{1/k_2},m_{1}) \frac{1}{(s-x)x}
dxdm_{1} \right) \frac{ds}{s}\\
+ \epsilon^{-1}\frac{1}{P_{m}(\tau)\Gamma(1 + \frac{1}{k_2})} \int_{0}^{\tau^{k_2}}
(\tau^{k_2}-s)^{1/k_2} \frac{c_{0,0}(\epsilon)}{(2\pi)^{1/2}} ( \int_{-\infty}^{+\infty} C_{0,0}(m-m_{1},\epsilon)R_{0}(im_{1})
w(s^{1/k_2},m_{1}) dm_{1} )\frac{ds}{s}\\
+ \epsilon^{-1}c_{F}(\epsilon) \frac{1}{P_{m}(\tau)\Gamma(1 + \frac{1}{k_2})}\int_{0}^{\tau^{k_2}}
(\tau^{k_2}-s)^{1/k_2} \psi_{k_2}^{d}(s^{1/k_2},m,\epsilon) \frac{ds}{s}
\end{multline}
satisfies the next properties.\\
{\bf i)} The following inclusion holds
\begin{equation}
\mathcal{H}_{\epsilon}^{k_2}(\bar{B}(0,\upsilon)) \subset \bar{B}(0,\upsilon) \label{H_k2_inclusion}
\end{equation}
where $\bar{B}(0,\upsilon)$ is the closed ball of radius $\upsilon>0$ centered at 0 in $F_{(\nu',\beta,\mu,k_{2})}^{d}$,
for all $\epsilon \in D(0,\epsilon_{0})$.\\
{\bf ii)} We have
\begin{equation}
|| \mathcal{H}_{\epsilon}^{k_2}(w_{1}) - \mathcal{H}_{\epsilon}^{k_2}(w_{2})||_{(\nu',\beta,\mu,k_{2})}
\leq \frac{1}{2} ||w_{1} - w_{2}||_{(\nu',\beta,\mu,k_{2})}
\label{H_k2_shrink}
\end{equation}
for all $w_{1},w_{2} \in \bar{B}(0,\upsilon)$, for all $\epsilon \in D(0,\epsilon_{0})$.
\end{lemma}
\begin{proof} We first check the property (\ref{H_k2_inclusion}). Let $\epsilon \in D(0,\epsilon_{0})$ and
$w(\tau,m)$ be in $F_{(\nu',\beta,\mu,k_{2})}^{d}$. We take
$\varsigma_{1,2},\varsigma_{0,0},\varsigma_{0},\varsigma_{1},\varsigma_{1,0},\varsigma_{F},\varsigma_{2}>0$ and
$\upsilon > 0$ such that if (\ref{norm_F_varphi_k2_psi_k_2_small}) holds and
$||w(\tau,m)||_{(\nu',\beta,\mu,k_{2})} \leq \upsilon$ for all $\epsilon \in D(0,\epsilon_{0})$. Due to (\ref{low_bounds_P_m})
and by using Lemma 2 with Proposition 6, we get that
\begin{multline}
||\epsilon^{-1}
\frac{c_{1,2}(\epsilon)}{P_{m}(\tau)\Gamma(1 + \frac{1}{k_2})} \int_{0}^{\tau^{k_2}}
(\tau^{k_2}-s)^{1/k_2}\\
\times \left( \frac{1}{(2\pi)^{1/2}} s\int_{0}^{s} \int_{-\infty}^{+\infty} \right.
Q_{1}(i(m-m_{1}))w((s-x)^{1/k_2},m-m_{1})\\
\left. \times  Q_{2}(im_{1})
w(x^{1/k_2},m_{1}) \frac{1}{(s-x)x} dxdm_{1} \right) \frac{ds}{s}||_{(\nu',\beta,\mu,k_{2})} \\
\leq \frac{\varsigma_{1,2}}{\Gamma(1 + \frac{1}{k_2})(2\pi)^{1/2}}
\frac{C_{7}||w(\tau,m)||_{(\nu',\beta,\mu,k_{2})}^{2}}{C_{P}(r_{Q,R_{D}})^{\frac{1}{(\delta_{D}-1)k_{2}}}} \\
\leq \frac{\varsigma_{1,2}}{\Gamma(1 + \frac{1}{k_2})(2\pi)^{1/2}}
\frac{C_{7}\upsilon^{2}}{C_{P}(r_{Q,R_{D}})^{\frac{1}{(\delta_{D}-1)k_{2}}}}. \label{fix_point_norm_2_estim_1}
\end{multline}
Moreover, for $1 \leq p \leq \delta_{D}-1$ and by means of Proposition 5, we deduce
\begin{multline}
||\frac{R_{D}(im)}{P_{m}(\tau)}A_{\delta_{D},p}
\frac{1}{\Gamma(\delta_{D}-p)}\int_{0}^{\tau^{k_2}} (\tau^{k_2}-s)^{\delta_{D}-p-1}
 (k_2^{p} s^{p} w(s^{1/k_2},m)) \frac{ds}{s}||_{(\nu',\beta,\mu,k_{2})}\\
 \leq \frac{|A_{\delta_{D},p}|k_{2}^{p}C_{6}}{\Gamma(\delta_{D}-p)C_{P}(r_{Q,R_{D}})^{\frac{1}{(\delta_{D}-1)k_{2}}}}
 ||w(\tau,m)||_{(\nu',\beta,\mu,k_{2})}\\
 \leq \frac{|A_{\delta_{D},p}|k_{2}^{p}C_{6}}{\Gamma(\delta_{D}-p)C_{P}(r_{Q,R_{D}})^{\frac{1}{(\delta_{D}-1)k_{2}}}}
\upsilon. \label{fix_point_norm_2_estim_2}
\end{multline}
For all $1 \leq l \leq D-1$, from the assumption (\ref{constraints_k2_Borel_equation}) and by means of Proposition 5, we have
\begin{multline}
||\frac{R_{l}(im)}{P_{m}(\tau)} \epsilon^{\Delta_{l}-d_{l}+\delta_{l}-1}
\frac{1}{\Gamma( \frac{d_{l,k_{2}}^{2}}{k_2} )}
\int_{0}^{\tau^{k_2}} (\tau^{k_2}-s)^{\frac{d_{l,k_{2}}^{2}}{k_2}-1}({k_2}^{\delta_l}s^{\delta_l}
w(s^{1/k_2},m)) \frac{ds}{s}||_{(\nu',\beta,\mu,k_{2})}\\
\leq \frac{k_{2}^{\delta_l}C_{6}}{\Gamma( \frac{d_{l,k_{2}}^{2}}{k_2}) C_{P}(r_{Q,R_{D}})^{\frac{1}{(\delta_{D}-1)k_{2}}} }
|\epsilon|^{\Delta_{l} - d_{l} + \delta_{l} -1}
\sup_{m \in \mathbb{R}} \frac{|R_{l}(im)|}{|R_{D}(im)|} ||w(\tau,m)||_{(\nu',\beta,\mu,k_{2})}\\
\leq \frac{k_{2}^{\delta_l}C_{6}}{\Gamma( \frac{d_{l,k_{2}}^{2}}{k_2}) C_{P}(r_{Q,R_{D}})^{\frac{1}{(\delta_{D}-1)k_{2}}} }
\epsilon_{0}^{\Delta_{l} - d_{l} + \delta_{l} -1 }
\sup_{m \in \mathbb{R}} \frac{|R_{l}(im)|}{|R_{D}(im)|} \upsilon \label{fix_point_norm_2_estim_3}
\end{multline}
and
\begin{multline}
|| \frac{R_{l}(im)}{P_{m}(\tau)} \frac{ A_{\delta_{l},p}\epsilon^{\Delta_{l}-d_{l}+\delta_{l}-1} }
{\Gamma( \frac{d_{l,k_{2}}^{2}}{k_2} + \delta_{l}-p)} \int_{0}^{\tau^{k_2}}
(\tau^{k_2}-s)^{\frac{d_{l,k_{2}}^{2}}{k_2}+\delta_{l}-p-1}(k_{2}^{p}s^{p}
w(s^{1/k_2},m)) \frac{ds}{s} ||_{(\nu',\beta,\mu,k_{2})}\\
\leq \frac{|A_{\delta_{l},p}|k_{2}^{p}C_{6}}{\Gamma( \frac{d_{l,k_{2}}^{2}}{k_{2}} + \delta_{l} - p)
C_{P}(r_{Q,R_{D}})^{\frac{1}{(\delta_{D}-1)k_{2}}} }|\epsilon|^{\Delta_{l}-d_{l}+\delta_{l} -1}\\
\times \sup_{m \in \mathbb{R}} \frac{|R_{l}(im)|}{|R_{D}(im)|} ||w(\tau,m)||_{(\nu',\beta,\mu,k_{2})}\\
\leq \frac{|A_{\delta_{l},p}|k_{2}^{p}C_{6}}{\Gamma( \frac{d_{l,k_{2}}^{2}}{k_{2}} + \delta_{l} - p)
C_{P}(r_{Q,R_{D}})^{\frac{1}{(\delta_{D}-1)k_{2}}} } \epsilon_{0}^{\Delta_{l}-d_{l}+\delta_{l}-1}
\sup_{m \in \mathbb{R}} \frac{|R_{l}(im)|}{|R_{D}(im)|}\upsilon, \label{fix_point_norm_2_estim_4}
\end{multline}
for all $1 \leq p \leq \delta_{l}-1$. Taking into account Lemma 2 and Proposition 6, we also get that
\begin{multline}
|| \epsilon^{-1}
\frac{c_{0}(\epsilon)}{P_{m}(\tau) \Gamma(1 + \frac{1}{k_2})} \int_{0}^{\tau^{k_2}}
(\tau^{k_2}-s)^{1/k_2}\\
\times \left( \frac{1}{(2\pi)^{1/2}} s\int_{0}^{s} \int_{-\infty}^{+\infty} \right.
\varphi_{k_2}((s-x)^{1/k_2},m-m_{1},\epsilon) \\
\times \left. R_{0}(im_{1}) w(x^{1/k_2},m_{1}) \frac{1}{(s-x)x}
dxdm_{1} \right) \frac{ds}{s} ||_{(\nu',\beta,\mu,k_{2})}\\
\leq \frac{\varsigma_{1,0}}{\Gamma(1 + \frac{1}{k_2})(2\pi)^{1/2}} \frac{C_{7}
||\varphi_{k_2}(\tau,m,\epsilon)||_{(\nu',\beta,\mu,k_{2})} ||w(\tau,m)||_{(\nu',\beta,\mu,k_{2})} }{
C_{P}(r_{Q,R_{D}})^{\frac{1}{(\delta_{D}-1)k_{2}}} }\\
\leq \frac{\varsigma_{1,0}}{\Gamma(1 + \frac{1}{k_2})(2\pi)^{1/2}} \frac{C_{7}
\varsigma_{1} \upsilon}{ C_{P}(r_{Q,R_{D}})^{\frac{1}{(\delta_{D}-1)k_{2}}} }. \label{fix_point_norm_2_estim_5}
\end{multline}
Due to Lemma 2 and Proposition 7,
\begin{multline}
|| \epsilon^{-1}\frac{c_{0,0}(\epsilon)}{P_{m}(\tau)\Gamma(1 + \frac{1}{k_2})} \int_{0}^{\tau^{k_2}}
(\tau^{k_2}-s)^{1/k_2} \frac{1}{(2\pi)^{1/2}} ( \int_{-\infty}^{+\infty} C_{0,0}(m-m_{1},\epsilon)\\
\times R_{0}(im_{1})
w(s^{1/k_2},m_{1}) dm_{1} )\frac{ds}{s} ||_{(\nu',\beta,\mu,k_{2})}\\
\leq \frac{\varsigma_{0,0}}{\Gamma(1 + \frac{1}{k_2})(2\pi)^{1/2}} \frac{C_{8} \varsigma_{0} \upsilon }{
C_{P}(r_{Q,R_{D}})^{\frac{1}{(\delta_{D}-1)k_{2}}} } \label{fix_point_norm_2_estim_6}
\end{multline}
holds. Finally, from Lemma 2 and Proposition 5, we get
\begin{multline}
|| \epsilon^{-1} \frac{c_{F}(\epsilon)}{P_{m}(\tau)\Gamma(1 + \frac{1}{k_2})}\int_{0}^{\tau^{k_2}}
(\tau^{k_2}-s)^{1/k_2} \psi_{k_2}^{d}(s^{1/k_2},m,\epsilon) \frac{ds}{s} ||_{(\nu',\beta,\mu,k_{2})}\\
\leq \frac{\varsigma_{F}C_{6}}{\Gamma(1 + \frac{1}{k_2})C_{P}(r_{Q,R_{D}})^{\frac{1}{(\delta_{D}-1)k_{2}}}
\min_{m \in \mathbb{R}} |R_{D}(im)| }
||\psi_{k_2}^{d}(\tau,m,\epsilon)||_{(\nu',\beta,\mu,k_{2})}\\
\leq \frac{\varsigma_{F}C_{6}}{\Gamma(1 + \frac{1}{k_2})C_{P}(r_{Q,R_{D}})^{\frac{1}{(\delta_{D}-1)k_{2}}}
\min_{m \in \mathbb{R}} |R_{D}(im)| } \varsigma_{2}. \label{fix_point_norm_2_estim_7}
\end{multline}
Now, we choose $\upsilon$,
$\varsigma_{1,2},\varsigma_{0,0},\varsigma_{0},\varsigma_{1},\varsigma_{1,0},\varsigma_{F},\varsigma_{2}>0$
and $r_{Q,R_{D}}>0$ such that
\begin{multline}
\frac{\varsigma_{1,2}}{\Gamma(1 + \frac{1}{k_2})(2\pi)^{1/2}}
\frac{C_{7}\upsilon^{2}}{C_{P}(r_{Q,R_{D}})^{\frac{1}{(\delta_{D}-1)k_{2}}}} + \sum_{1 \leq p \leq \delta_{D}-1}
\frac{|A_{\delta_{D},p}|k_{2}^{p}C_{6}}{\Gamma(\delta_{D}-p)C_{P}(r_{Q,R_{D}})^{\frac{1}{(\delta_{D}-1)k_{2}}}}
\upsilon\\
+ \sum_{1 \leq l \leq D-1}\frac{k_{2}^{\delta_l}C_{6}}{\Gamma( \frac{d_{l,k_{2}}^{2}}{k_2})
C_{P}(r_{Q,R_{D}})^{\frac{1}{(\delta_{D}-1)k_{2}}} }
\epsilon_{0}^{\Delta_{l} - d_{l} + \delta_{l} -1 }
\sup_{m \in \mathbb{R}} \frac{|R_{l}(im)|}{|R_{D}(im)|} \upsilon\\
+ \sum_{1 \leq p \leq \delta_{l}-1} \frac{|A_{\delta_{l},p}|k_{2}^{p}C_{6}}{\Gamma( \frac{d_{l,k_{2}}^{2}}{k_{2}} + \delta_{l} - p)
C_{P}(r_{Q,R_{D}})^{\frac{1}{(\delta_{D}-1)k_{2}}} } \epsilon_{0}^{\Delta_{l}-d_{l}+\delta_{l}-1}
\sup_{m \in \mathbb{R}} \frac{|R_{l}(im)|}{|R_{D}(im)|} \upsilon\\
+ \frac{\varsigma_{1,0}}{\Gamma(1 + \frac{1}{k_2})(2\pi)^{1/2}} \frac{C_{7}
\varsigma_{1} \upsilon}{ C_{P}(r_{Q,R_{D}})^{\frac{1}{(\delta_{D}-1)k_{2}}} } +
\frac{\varsigma_{0,0}}{\Gamma(1 + \frac{1}{k_2})(2\pi)^{1/2}} \frac{C_{8} \varsigma_{0} \upsilon }{
C_{P}(r_{Q,R_{D}})^{\frac{1}{(\delta_{D}-1)k_{2}}} }\\
+ \frac{\varsigma_{F}C_{6}}{\Gamma(1 + \frac{1}{k_2})C_{P}(r_{Q,R_{D}})^{\frac{1}{(\delta_{D}-1)k_{2}}}
\min_{m \in \mathbb{R}} |R_{D}(im)| } \varsigma_{2} \leq \upsilon \label{fix_point_sum<upsilon}
\end{multline}
Gathering all the norm estimates (\ref{fix_point_norm_2_estim_1}), (\ref{fix_point_norm_2_estim_2}),
(\ref{fix_point_norm_2_estim_3}), (\ref{fix_point_norm_2_estim_4}), (\ref{fix_point_norm_2_estim_5}),
(\ref{fix_point_norm_2_estim_6}), (\ref{fix_point_norm_2_estim_7}) under the constraint (\ref{fix_point_sum<upsilon}), one gets
(\ref{H_k2_inclusion}).\bigskip

Now, we check the second property (\ref{H_k2_shrink}). Let $w_{1}(\tau,m),w_{2}(\tau,m)$ be in
$F^{d}_{(\nu',\beta,\mu,k_{2})}$.
We take $\upsilon > 0$ such that
$$ ||w_{l}(\tau,m)||_{(\nu',\beta,\mu,k_{2})} \leq \upsilon,$$
for $l=1,2$. One can write
\begin{multline}
Q_{1}(i(m-m_{1}))w_{1}((s-x)^{1/k_{2}},m-m_{1})Q_{2}(im_{1})w_{1}(x^{1/k_{2}},m_{1})\\
-Q_{1}(i(m-m_{1}))w_{2}((s-x)^{1/k_{2}},m-m_{1})Q_{2}(im_{1})w_{2}(x^{1/k_{2}},m_{1})\\
= Q_{1}(i(m-m_{1}))\left(w_{1}((s-x)^{1/k_{2}},m-m_{1}) - w_{2}((s-x)^{1/k_{2}},m-m_{1})\right)
Q_{2}(im_{1})w_{1}(x^{1/k_{2}},m_{1})\\
+ Q_{1}(i(m-m_{1}))w_{2}((s-x)^{1/k_{2}},m-m_{1})Q_{2}(im_{1})\left(w_{1}(x^{1/k_{2}},m_{1}) - w_{2}(x^{1/k_{2}},m_{1})\right)
\label{conv_product_w1_w2_for_k_2}
\end{multline}
and using Lemma 2 with Proposition 6, we obtain
\begin{multline}
||\epsilon^{-1}
\frac{c_{1,2}(\epsilon)}{P_{m}(\tau)\Gamma(1 + \frac{1}{k_2})} \int_{0}^{\tau^{k_2}}
(\tau^{k_2}-s)^{1/k_2}\\
\times \left( \frac{1}{(2\pi)^{1/2}} s\int_{0}^{s} \int_{-\infty}^{+\infty} \right.
(Q_{1}(i(m-m_{1}))w_{1}((s-x)^{1/k_2},m-m_{1})\\
\times  Q_{2}(im_{1})
w_{1}(x^{1/k_2},m_{1}) - Q_{1}(i(m-m_{1}))w_{2}((s-x)^{1/k_2},m-m_{1})\\
\left. \times  Q_{2}(im_{1})
w_{2}(x^{1/k_2},m_{1})) \frac{1}{(s-x)x} dxdm_{1} \right) \frac{ds}{s}||_{(\nu',\beta,\mu,k_{2})} \\
\leq \frac{\varsigma_{1,2}}{\Gamma(1 + \frac{1}{k_2})(2\pi)^{1/2}}
\frac{C_{7}}{C_{P}(r_{Q,R_{D}})^{\frac{1}{(\delta_{D}-1)k_{2}}}}
||w_{1}(\tau,m) - w_{2}(\tau,m)||_{(\nu',\beta,\mu,k_{2})}\\
\times (||w_{1}(\tau,m)||_{(\nu',\beta,\mu,k_{2})} + ||w_{2}(\tau,m)||_{(\nu',\beta,\mu,k_{2})})
\\
\leq \frac{\varsigma_{1,2}}{\Gamma(1 + \frac{1}{k_2})(2\pi)^{1/2}}
\frac{C_{7}2\upsilon}{C_{P}(r_{Q,R_{D}})^{\frac{1}{(\delta_{D}-1)k_{2}}}}
||w_{1}(\tau,m) - w_{2}(\tau,m)||_{(\nu',\beta,\mu,k_{2})}
\label{fix_point_norm_2_estim_1_shrink}
\end{multline}
From the estimates (\ref{fix_point_norm_2_estim_2}), (\ref{fix_point_norm_2_estim_3}), 
(\ref{fix_point_norm_2_estim_4}), (\ref{fix_point_norm_2_estim_5}), (\ref{fix_point_norm_2_estim_6}) and under the constraint
(\ref{constraints_k2_Borel_equation}), we find that for $1 \leq p \leq \delta_{D}-1$, 
\begin{multline}
||\frac{R_{D}(im)}{P_{m}(\tau)}A_{\delta_{D},p}
\frac{1}{\Gamma(\delta_{D}-p)}\int_{0}^{\tau^{k_2}} (\tau^{k_2}-s)^{\delta_{D}-p-1}\\
\times 
 (k_2^{p} s^{p} (w_{1}(s^{1/k_2},m) - w_{2}(s^{1/k_2},m))) \frac{ds}{s}||_{(\nu',\beta,\mu,k_{2})}\\
 \leq \frac{|A_{\delta_{D},p}|k_{2}^{p}C_{6}}{\Gamma(\delta_{D}-p)C_{P}(r_{Q,R_{D}})^{\frac{1}{(\delta_{D}-1)k_{2}}}}
 ||w_{1}(\tau,m) - w_{2}(\tau,m)||_{(\nu',\beta,\mu,k_{2})} \label{fix_point_norm_2_estim_2_shrink}
\end{multline}
holds and that for $1 \leq l \leq D-1$, 
\begin{multline}
||\frac{R_{l}(im)}{P_{m}(\tau)} \epsilon^{\Delta_{l}-d_{l}+\delta_{l}-1}
\frac{1}{\Gamma( \frac{d_{l,k_{2}}^{2}}{k_2} )}
\int_{0}^{\tau^{k_2}} (\tau^{k_2}-s)^{\frac{d_{l,k_{2}}^{2}}{k_2}-1}\\
\times (k_{2}^{\delta_l}s^{\delta_l}
(w_{1}(s^{1/k_2},m) - w_{2}(s^{1/k_2},m))) \frac{ds}{s}||_{(\nu',\beta,\mu,k_{2})}\\
\leq \frac{k_{2}^{\delta_l}C_{6}}{\Gamma( \frac{d_{l,k_{2}}^{2}}{k_2}) C_{P}(r_{Q,R_{D}})^{\frac{1}{(\delta_{D}-1)k_{2}}} }
\epsilon_{0}^{\Delta_{l} - d_{l} + \delta_{l} - 1}\\
\times \sup_{m \in \mathbb{R}} \frac{|R_{l}(im)|}{|R_{D}(im)|} ||w_{1}(\tau,m) - w_{2}(\tau,m)||_{(\nu',\beta,\mu,k_{2})}
\label{fix_point_norm_2_estim_3_shrink}
\end{multline}
and
\begin{multline}
|| \frac{R_{l}(im)}{P_{m}(\tau)} \frac{ A_{\delta_{l},p}\epsilon^{\Delta_{l}-d_{l}+\delta_{l}-1} }
{\Gamma( \frac{d_{l,k_{2}}^{2}}{k_2} + \delta_{l}-p)} \int_{0}^{\tau^{k_2}}
(\tau^{k_2}-s)^{\frac{d_{l,k_{2}}^{2}}{k_2}+\delta_{l}-p-1}\\
\times (k_{2}^{p}s^{p}(w_{1}(s^{1/k_2},m) - w_{2}(s^{1/k_2},m))) \frac{ds}{s} ||_{(\nu',\beta,\mu,k_{2})}\\
\leq \frac{|A_{\delta_{l},p}|k_{2}^{p}C_{6}}{\Gamma( \frac{d_{l,k_{2}}^{2}}{k_{2}} + \delta_{l} - p)
C_{P}(r_{Q,R_{D}})^{\frac{1}{(\delta_{D}-1)k_{2}}} } \epsilon_{0}^{\Delta_{l}-d_{l}+\delta_{l} - 1} \\
\times \sup_{m \in \mathbb{R}} \frac{|R_{l}(im)|}{|R_{D}(im)|} ||w_{1}(\tau,m) - w_{2}(\tau,m)||_{(\nu',\beta,\mu,k_{2})}
\label{fix_point_norm_2_estim_4_shrink}
\end{multline}
arise for all $1 \leq p \leq \delta_{l}-1$, together with
\begin{multline}
|| \epsilon^{-1}
\frac{c_{0}(\epsilon)}{P_{m}(\tau) \Gamma(1 + \frac{1}{k_2})} \int_{0}^{\tau^{k_2}}
(\tau^{k_2}-s)^{1/k_2}\\
\times \left( \frac{1}{(2\pi)^{1/2}} s\int_{0}^{s} \int_{-\infty}^{+\infty} \right.
\varphi_{k_2}((s-x)^{1/k_2},m-m_{1},\epsilon) \\
\times \left. R_{0}(im_{1}) (w_{1}(x^{1/k_2},m_{1}) - w_{2}(x^{1/k_2},m_{1})) \frac{1}{(s-x)x}
dxdm_{1} \right) \frac{ds}{s} ||_{(\nu',\beta,\mu,k_{2})}\\
\leq \frac{\varsigma_{1,0}}{\Gamma(1 + \frac{1}{k_2})(2\pi)^{1/2}} \frac{C_{7}
||\varphi_{k_2}(\tau,m,\epsilon)||_{(\nu',\beta,\mu,k_{2})} ||w_{1}(\tau,m) - w_{2}(\tau,m)||_{(\nu',\beta,\mu,k_{2})} }{
C_{P}(r_{Q,R_{D}})^{\frac{1}{(\delta_{D}-1)k_{2}}} }\\
\leq \frac{\varsigma_{1,0}}{\Gamma(1 + \frac{1}{k_2})(2\pi)^{1/2}} \frac{C_{7}
\varsigma_{1}
||w_{1}(\tau,m) - w_{2}(\tau,m)||_{(\nu',\beta,\mu,k_{2})}}{ C_{P}(r_{Q,R_{D}})^{\frac{1}{(\delta_{D}-1)k_{2}}} }
\label{fix_point_norm_2_estim_5_shrink}
\end{multline}
and finally
\begin{multline}
|| \epsilon^{-1}\frac{c_{0,0}(\epsilon)}{P_{m}(\tau)\Gamma(1 + \frac{1}{k_2})} \int_{0}^{\tau^{k_2}}
(\tau^{k_2}-s)^{1/k_2} \frac{1}{(2\pi)^{1/2}} ( \int_{-\infty}^{+\infty} C_{0,0}(m-m_{1},\epsilon)\\
\times R_{0}(im_{1})
(w_{1}(s^{1/k_2},m_{1}) - w_{2}(s^{1/k_2},m_{1})) dm_{1} )\frac{ds}{s} ||_{(\nu',\beta,\mu,k_{2})}\\
\leq \frac{\varsigma_{0,0}}{\Gamma(1 + \frac{1}{k_2})(2\pi)^{1/2}}
\frac{C_{8} \varsigma_{0} ||w_{1}(\tau,m) - w_{2}(\tau,m)||_{(\nu',\beta,\mu,k_{2})} }{
C_{P}(r_{Q,R_{D}})^{\frac{1}{(\delta_{D}-1)k_{2}}} }. \label{fix_point_norm_2_estim_6_shrink}
\end{multline}
Now, we take $\upsilon$ and $r_{Q,R_{D}}$ such that
\begin{multline}
\frac{ \varsigma_{1,2}}{\Gamma(1 + \frac{1}{k_2})(2\pi)^{1/2}}
\frac{C_{7}2\upsilon}{C_{P}(r_{Q,R_{D}})^{\frac{1}{(\delta_{D}-1)k_{2}}}} + \sum_{1 \leq p \leq \delta_{D}-1}
\frac{|A_{\delta_{D},p}|k_{2}^{p}C_{6}}{\Gamma(\delta_{D}-p)C_{P}(r_{Q,R_{D}})^{\frac{1}{(\delta_{D}-1)k_{2}}}}\\
+ \sum_{1 \leq l \leq D-1} \frac{k_{2}^{\delta_l}C_{6}}{\Gamma( \frac{d_{l,k_{2}}^{2}}{k_2})
C_{P}(r_{Q,R_{D}})^{\frac{1}{(\delta_{D}-1)k_{2}}} }
\epsilon_{0}^{\Delta_{l} - d_{l} + \delta_{l} -1}
\sup_{m \in \mathbb{R}} \frac{|R_{l}(im)|}{|R_{D}(im)|}\\
+ \sum_{1 \leq p \leq \delta_{l}-1} \frac{|A_{\delta_{l},p}|k_{2}^{p}C_{6}}{\Gamma( \frac{d_{l,k_{2}}^{2}}{k_{2}} + \delta_{l} - p)
C_{P}(r_{Q,R_{D}})^{\frac{1}{(\delta_{D}-1)k_{2}}} } \epsilon_{0}^{\Delta_{l}-d_{l}+\delta_{l}-1}
\sup_{m \in \mathbb{R}} \frac{|R_{l}(im)|}{|R_{D}(im)|}\\
+ \frac{1}{\Gamma(1 + \frac{1}{k_2})(2\pi)^{1/2}} \frac{\varsigma_{1,0}C_{7}
\varsigma_{1} + \varsigma_{0,0}C_{8} \varsigma_{0}}{ C_{P}(r_{Q,R_{D}})^{\frac{1}{(\delta_{D}-1)k_{2}}} } \leq \frac{1}{2}
\label{fix_point_sum<1/2}
\end{multline}
Bearing in mind the estimates (\ref{fix_point_norm_2_estim_1_shrink}),
(\ref{fix_point_norm_2_estim_2_shrink}), (\ref{fix_point_norm_2_estim_3_shrink}), (\ref{fix_point_norm_2_estim_4_shrink}),
(\ref{fix_point_norm_2_estim_5_shrink}), (\ref{fix_point_norm_2_estim_6_shrink}) under the constraint (\ref{fix_point_sum<1/2}), one
gets (\ref{H_k2_shrink}).

Finally, we choose $\upsilon$ and $r_{Q,R_{D}}$ such that both (\ref{fix_point_sum<upsilon}) and (\ref{fix_point_sum<1/2}) are
fulfilled. This yields our lemma.
\end{proof}
We consider the ball $\bar{B}(0,\upsilon) \subset F_{(\nu',\beta,\mu,k_{2})}^{d}$ constructed in Lemma 5 which is a complete
metric space for the norm $||.||_{(\nu',\beta,\mu,k_{2})}$. From the lemma above, we get that
$\mathcal{H}_{\epsilon}^{k_2}$ is a
contractive map from $\bar{B}(0,\upsilon)$ into itself. Due to the classical contractive mapping theorem, we deduce that
the map $\mathcal{H}_{\epsilon}^{k_2}$ has a unique fixed point denoted $\omega_{k_2}^{d}(\tau,m,\epsilon)$ (i.e
$\mathcal{H}_{\epsilon}^{k_2}(\omega_{k_2}^{d}(\tau,m,\epsilon))= \omega_{k_2}^{d}(\tau,m,\epsilon)$) in
$\bar{B}(0,\upsilon)$, for all $\epsilon \in D(0,\epsilon_{0})$. Moreover, the function
$\omega_{k_2}^{d}(\tau,m,\epsilon)$ depends holomorphically on $\epsilon$ in $D(0,\epsilon_{0})$. By construction,
$\omega_{k_2}^{d}(\tau,m,\epsilon)$ defines a solution of the equation (\ref{k2_Borel_equation_analytic}). This yields
the proposition.
\end{proof}

In the next proposition, we present the link, by means of the analytic acceleration operator defined in Proposition 13,
between the holomorphic solution of the problem (\ref{k_1_Borel_equation}) constructed in Proposition 11 and the solution
of the problem (\ref{k2_Borel_equation_analytic}) found in Proposition 14.  

\begin{prop} Let us consider the function $\omega_{k_1}^{d}(\tau,m,\epsilon)$ constructed in Proposition 11 and which
solves the equation (\ref{k_1_Borel_equation}). The function
$$
\tau \mapsto \mathrm{Acc}_{k_{2},k_{1}}^{d}(\omega_{k_1}^{d})(\tau,m,\epsilon) :=
\mathcal{A}_{m_{k_2},m_{k_1}}^{d}( h \mapsto \omega_{k_1}^{d}(h,m,\epsilon)) (\tau) =
\int_{L_{d}} \omega_{k_1}^{d}(h,m,\epsilon) G(\tau,h)
\frac{dh}{h}
$$
defines an analytic function on a sector $S_{d,\kappa,\delta,(c_{2}/\nu)^{1/\kappa}/2}$ with direction
$d$, aperture $\frac{\pi}{\kappa} + \delta$ and radius $(c_{2}/\nu)^{1/\kappa}/2$, for any $0 < \delta < \mathrm{ap}(U_d)$
and for a constant $c_{2}$ introduced in (\ref{G_xi_h_exp_growth_order_kappa}), with the property that
$\mathrm{Acc}_{k_{2},k_{1}}^{d}(\omega_{k_1}^{d})(0,m,\epsilon) \equiv 0$, for all
$m \in \mathbb{R}$, all $\epsilon \in D(0,\epsilon_{0})$.

Moreover, for all fixed $\epsilon \in D(0,\epsilon_{0})$, the following identity
\begin{equation}
\mathrm{Acc}_{k_{2},k_{1}}^{d}(\omega_{k_1}^{d})(\tau,m,\epsilon) = \omega_{k_2}^{d}(\tau,m,\epsilon) \label{acc_omega_k_1=omega_k_2}
\end{equation}
holds for all $\tau \in S_{d,\kappa,\delta,(c_{2}/\nu)^{1/\kappa}/2}$, all
$m \in \mathbb{R}$, provided that $\nu>0$ is chosen in a way that
$S_{d,\kappa,\delta,(c_{2}/\nu)^{1/\kappa}/2} \subset S_{d}^{b}$ holds where $S_{d}^{b}$ is the bounded sector introduced in
Proposition 14.

As a consequence of Proposition 14, the function
$\tau \mapsto \mathrm{Acc}_{k_{2},k_{1}}^{d}(\omega_{k_1}^{d})(\tau,m,\epsilon)$ has an analytic continuation on the union
$S_{d}^{b} \cup S_{d}$, where the sector $S_{d}$ has been described in Proposition 14,
denoted again
$\mathrm{Acc}_{k_{2},k_{1}}^{d}(\omega_{k_1}^{d})(\tau,m,\epsilon)$ which satisfies estimates of the form : there exists a constant
$C_{\omega_{k_2}}>0$ with
\begin{equation}
|\mathrm{Acc}_{k_{2},k_{1}}^{d}(\omega_{k_1}^{d})(\tau,m,\epsilon)| \leq C_{\omega_{k_2}} (1+|m|)^{-\mu} e^{-\beta |m|}
\frac{ |\tau| }{1 + |\tau|^{2k_{2}}} \exp( \nu' |\tau|^{k_2} )
\label{acc_k_2_k_1_omega_k_1_exp_growth}
\end{equation}
for all $\tau \in S_{d}^{b} \cup S_{d}$, all $m \in \mathbb{R}$, all $\epsilon \in D(0,\epsilon_{0})$.
\end{prop}
\begin{proof} From Proposition 11, we point out that $\omega_{k_1}^{d}(\tau,m,\epsilon)$ belongs to the space
$F_{(\nu,\beta,\mu,k_{1},\kappa)}^{d}$ and that $||\omega_{k_1}^{d}||_{(\nu,\beta,\mu,k_{1},\kappa)} \leq
\varpi$ for all $\epsilon \in D(0,\epsilon_{0})$. Due to Proposition 13, we deduce that the function
$\tau \mapsto \mathrm{Acc}_{k_{2},k_{1}}^{d}(\omega_{k_1}^{d})(\tau,m,\epsilon)$ defines a holomorphic and bounded function 
with values in the Banach space $E_{(\beta,\mu)}$ (with bound independent of $\epsilon$) on a sector
$S_{d,\kappa,\delta,(c_{2}/\nu)^{1/\kappa}/2}$ with direction
$d$, aperture $\frac{\pi}{\kappa} + \delta$ and radius $(c_{2}/\nu)^{1/\kappa}/2$, for any $0< \delta < \mathrm{ap}(U_d)$
and for a constant $c_{2}$ introduced in (\ref{G_xi_h_exp_growth_order_kappa}), for all
$\epsilon \in D(0,\epsilon_{0})$.

Now, as a result of Proposition 13, we also know that the function
$\tau \mapsto \mathrm{Acc}_{k_{2},k_{1}}^{d}(\omega_{k_1}^{d})(\tau,m,\epsilon)$ is the $\kappa-$sum of the formal series
$$
\hat{\mathcal{A}}_{m_{k_2},m_{k_1}}( h \mapsto \omega_{k_1}(h,m,\epsilon)) (\tau) = \hat{\omega}_{k_2}(\tau,m,\epsilon)
$$
viewed as formal series in $\tau$ with coefficients in the Banach space
$E_{(\beta,\mu)}$, on $S_{d,\kappa,\delta,(c_{2}/\nu)^{1/\kappa}/2}$,
for all $\epsilon \in D(0,\epsilon_{0})$. In particular, one sees that
$\mathrm{Acc}_{k_{2},k_{1}}^{d}(\omega_{k_1}^{d})(0,m,\epsilon) \equiv 0$, for all
$\epsilon \in D(0,\epsilon_{0})$.

Likewise, we notice from Lemma 4, that the function
$\tau \mapsto \psi_{k_2}^{d}(\tau,m,\epsilon)$ is the $\kappa-$sum on
$S_{d,\kappa,\delta,(c_{2}/\nu)^{1/\kappa}/2}$ of the formal series $\hat{\psi}_{k_2}(\tau,m,\epsilon)$
defined in (\ref{defin_borel_k_2_omega_varphi_psi}), viewed
as formal series in $\tau$ with coefficients in the Banach space
$E_{(\beta,\mu)}$, for all $\epsilon \in D(0,\epsilon_{0})$. We recall that 
$\hat{\omega}_{k_2}(\tau,m,\epsilon)$ formally solves the equation (\ref{k2_Borel_equation}) for vanishing
initial data $\hat{\omega}_{k_2}(0,m,\epsilon) \equiv 0$. Using standard stability properties of
the $\kappa-$sums of formal series with respect to algebraic operations and integration (see \cite{ba}, Section 3.3 Theorem 2 p. 28),
we deduce that the function $\mathrm{Acc}_{k_{2},k_{1}}^{d}(\omega_{k_1}^{d})(\tau,m,\epsilon)$ satisfies the
equation (\ref{k2_Borel_equation_analytic}) for all $\tau \in S_{d,\kappa,\delta,(c_{2}/\nu)^{1/\kappa}/2}$, all
$m \in \mathbb{R}$, all $\epsilon \in D(0,\epsilon_{0})$, for vanishing initial data
$\mathrm{Acc}_{k_{2},k_{1}}^{d}(\omega_{k_1}^{d})(0,m,\epsilon) \equiv 0$.

In order to justify the identity (\ref{acc_omega_k_1=omega_k_2}), we need to define some additional Banach space. 
We keep the aforementioned notations.

\begin{defin} Let $h'=(c_{2}/\nu)^{1/\kappa}/2$. We
denote $H_{(\nu',\beta,\mu,k_{2},h')}$ the vector space of continuous functions $(\tau,m) \mapsto h(\tau,m)$ on
$\bar{S}_{d,\kappa,\delta,h'} \times \mathbb{R}$, holomorphic with respect to $\tau$
on $S_{d,\kappa,\delta,h'}$ such that
\begin{equation}
||h(\tau,m)||_{(\nu',\beta,\mu,k_{2},h')} =
\sup_{\tau \in \bar{S}_{d,\kappa,\delta,h'}, m \in \mathbb{R}} (1+|m|)^{\mu}
\frac{1 + |\tau|^{2k_{2}}}{|\tau|} \exp( \beta|m| - \nu'|\tau|^{k_2} ) |h(\tau,m)|
\end{equation}
is finite. One can check that $H_{(\nu',\beta,\mu,k_{2},h')}$ endowed with the norm
$||.||_{(\nu',\beta,\mu,k_{2},h')}$ is a Banach space.
\end{defin}

\noindent {\bf Remark:} Notice that if a function $h(\tau,m)$ belongs to the space $F_{(\nu',\beta,\mu,k_{2})}^{d}$
for the sectors $S_{d}$ and $S_{d}^{b}$ described in Proposition 14, then it belongs to the space
$H_{(\nu',\beta,\mu,k_{2},h')}$ (provided that $\nu>0$ is chosen such that
$S_{d,\kappa,\delta,h'} \subset S_{d}^{b}$) and moreover
$$ ||h(\tau,m)||_{(\nu',\beta,\mu,k_{2},h')} \leq ||h(\tau,m)||_{(\nu',\beta,\mu,k_{2})} $$
holds.\medskip

\noindent From the remark above, one deduces that
the functions $\varphi_{k_2}(\tau,m,\epsilon)$ and $\psi_{k_2}^{d}(\tau,m,\epsilon)$ belong to the space
$H_{(\nu',\beta,\mu,k_{2},h')}$.

In the following, one can reproduce the same lines of arguments as in the proof of Proposition 14 just by replacing the Banach
space $F_{(\nu',\beta,\mu,k_{2})}^{d}$ by $H_{(\nu',\beta,\mu,k_{2},h')}$, one gets the next

\begin{lemma} Under the assumption that (\ref{constraints_k2_Borel_equation}) holds, for the radius $r_{Q,R_{D}}>0$,
the constants $\upsilon$ and
$\varsigma_{1,2},\varsigma_{0,0},\varsigma_{0},\varsigma_{1},\varsigma_{1,0},\varsigma_{F},\varsigma_{2}$ given in
Proposition 14 for which the constraints (\ref{norm_F_varphi_k2_psi_k_2_small}) hold,
the equation (\ref{k2_Borel_equation_analytic}) has a unique solution
$\omega_{k_{2},h'}(\tau,m,\epsilon)$ in the space $H_{(\nu',\beta,\mu,k_{2},h')}$ with the property that
$||\omega_{k_{2},h'}(\tau,m,\epsilon)||_{(\nu',\beta,\mu,k_{2},h')} \leq \upsilon$, for all
$\epsilon \in D(0,\epsilon_{0})$.
\end{lemma}
Taking into account Proposition 14, since $\omega_{k_2}^{d}(\tau,m,\epsilon)$ belongs to $F_{(\nu',\beta,\mu,k_{2})}^{d}$, it
also belongs to the space $H_{(\nu',\beta,\mu,k_{2},h')}$. Moreover, since  
$\tau \mapsto \mathrm{Acc}_{k_{2},k_{1}}^{d}(\omega_{k_1}^{d})(\tau,m,\epsilon)$ defines a holomorphic and bounded function 
with values in the Banach space $E_{(\beta,\mu)}$ (with bound independent of $\epsilon$) on 
$S_{d,\kappa,\delta,h'}$ that vanishes at $\tau=0$, we also get that
$\mathrm{Acc}_{k_{2},k_{1}}^{d}(\omega_{k_1}^{d})(\tau,m,\epsilon)$ belongs to $H_{(\nu',\beta,\mu,k_{2},h')}$.

As a summary, we have seen that both $\omega_{k_2}^{d}(\tau,m,\epsilon)$ and
$\mathrm{Acc}_{k_{2},k_{1}}^{d}(\omega_{k_1}^{d})(\tau,m,\epsilon)$ solve the same equation
(\ref{k2_Borel_equation_analytic}) for vanishing initial data and belong to $H_{(\nu',\beta,\mu,k_{2},h')}$.
Moreover, one can check that the constant $\upsilon>0$ in Lemma 6 and Proposition 14 can be chosen sufficiently large such that
$||\mathrm{Acc}_{k_{2},k_{1}}^{d}(\omega_{k_1}^{d})(\tau,m,\epsilon)||_{(\nu',\beta,\mu,k_{2},h')} \leq \upsilon$ holds,
if the constants $\varsigma_{1,2},\varsigma_{0,0},\varsigma_{1,0},\varsigma_{F}>0$ are chosen small enough and
$r_{Q,R_{D}}>0$ is taken large enough. By construction, we already know that
$||\omega_{k_{2}}^{d}(\tau,m,\epsilon)||_{(\nu',\beta,\mu,k_{2},h')} \leq \upsilon$. Therefore, from
Lemma 6, we get that they must be equal. Proposition 15 follows.
\end{proof}

Now, we define the $m_{k_2}-$Laplace transforms
\begin{multline}
F^{d}(T,m,\epsilon) := k_{2}\int_{L_{d}} \psi_{k_2}^{d}(u,m,\epsilon) e^{-(\frac{u}{T})^{k_2}} \frac{du}{u},\\
U^{d}(T,m,\epsilon) := k_{2}\int_{L_{d}} \omega_{k_2}^{d}(u,m,\epsilon) e^{-(\frac{u}{T})^{k_2}} \frac{du}{u} \label{defin_Fd_Ud}
\end{multline}
which, according to the estimates (\ref{psi_k2_exp_growth}) and (\ref{acc_k_2_k_1_omega_k_1_exp_growth}),
are $E_{(\beta,\mu)}-$valued bounded holomorphic functions
on the sector $S_{d,\theta,h'}$ with bisecting direction $d$, aperture $\frac{\pi}{k_2} < \theta < \frac{\pi}{k_2} +
\mathrm{ap}(S_{d})$ and radius $h'$, where $h'>0$ is some positive real number, for all
$\epsilon \in D(0,\epsilon_{0})$.\medskip

\noindent {\bf Remark:} The analytic functions $F^{d}(T,m,\epsilon)$ (resp. $U^{d}(T,m,\epsilon)$) can be called the\\
$(m_{k_2},m_{k_1})-$sums in the direction $d$ of the formal series $F(T,m,\epsilon)$ (resp. $U(T,m,\epsilon)$)
introduced in the Section 4.1, following the terminology of \cite{ba}, Section 6.1.\medskip

\noindent In the next proposition, we construct analytic solutions to the problem (\ref{SCP}) with analytic forcing term and
for vanishing initial data.

\begin{prop} The function $U^{d}(T,m,\epsilon)$ solves the following equation
\begin{multline}
Q(im)(\partial_{T}U^{d}(T,m,\epsilon) ) - T^{(\delta_{D}-1)(k_{2}+1)}\partial_{T}^{\delta_{D}}R_{D}(im)U^{d}(T,m,\epsilon) \\
= \epsilon^{-1}\frac{c_{1,2}(\epsilon)}{(2\pi)^{1/2}}\int_{-\infty}^{+\infty}Q_{1}(i(m-m_{1}))U^{d}(T,m-m_{1},\epsilon)
Q_{2}(im_{1})U^{d}(T,m_{1},\epsilon) dm_{1}\\
+ \sum_{l=1}^{D-1} R_{l}(im) \epsilon^{\Delta_{l} - d_{l} + \delta_{l} - 1} T^{d_{l}} \partial_{T}^{\delta_l}U^{d}(T,m,\epsilon)\\
+ \epsilon^{-1}\frac{c_{0}(\epsilon)}{(2\pi)^{1/2}}\int_{-\infty}^{+\infty}
C_{0}(T,m-m_{1},\epsilon)R_{0}(im_{1})U^{d}(T,m_{1},\epsilon) dm_{1}\\
+  \epsilon^{-1}\frac{c_{0,0}(\epsilon)}{(2\pi)^{1/2}}\int_{-\infty}^{+\infty}
C_{0,0}(m-m_{1},\epsilon)R_{0}(im_{1})U^{d}(T,m_{1},\epsilon) dm_{1}
+ \epsilon^{-1}c_{F}(\epsilon)F^{d}(T,m,\epsilon)
\label{SCP_analytic_d}
\end{multline}
for given initial data $U^{d}(0,m,\epsilon) = 0$, for all $T \in S_{d,\theta,h'}$, $m\in \mathbb{R}$, all
$\epsilon \in D(0,\epsilon_{0})$.
\end{prop}
\begin{proof} Since the function $\omega_{k_2}^{d}(u,m,\epsilon)$ solves the integral equation (\ref{k2_Borel_equation_analytic}),
one can check by direct computations similar to those described in Proposition 8, using the integral representations
(\ref{defin_Fd_Ud}) that $U^{d}(T,m,\epsilon)$ solves the equation (\ref{SCP_irregular_k2}) where the formal series
$F(T,m,\epsilon)$ is replaced by $F^{d}(T,m,\epsilon)$ and hence solves the equation (\ref{SCP}) where $F^{d}(T,m,\epsilon)$ must be put in
place of $F(T,m,\epsilon)$.
\end{proof}

\section{Analytic solutions of a nonlinear initial value Cauchy problem with analytic forcing term on sectors and with complex
parameter}

Let $k_{1},k_{2} \geq 1$, $D \geq 2$ be integers such that $k_{2} > k_{1}$. Let $\delta_{l} \geq 1$ be integers such that
\begin{equation}
1 = \delta_{1} \ \ , \ \ \delta_{l} < \delta_{l+1}, \label{delta_constraints_main_CP}
\end{equation}
for all $1 \leq l \leq D-1$. For all $1 \leq l \leq D-1$, let
$d_{l},\Delta_{l} \geq 0$ be nonnegative integers such that 
\begin{equation}
d_{l} > \delta_{l} \ \ , \ \ \Delta_{l} - d_{l} + \delta_{l} - 1 \geq 0. \label{constrain_d_l_delta_l_Delta_l_main_CP}
\end{equation}
Let
$Q(X),Q_{1}(X),Q_{2}(X),R_{l}(X) \in \mathbb{C}[X]$, $0 \leq l \leq D$, be polynomials such that
\begin{multline}
\mathrm{deg}(Q) \geq \mathrm{deg}(R_{D}) \geq \mathrm{deg}(R_{l}) \ \ , \ \ \mathrm{deg}(R_{D}) \geq \mathrm{deg}(Q_{1}) \ \ , \ \
\mathrm{deg}(R_{D}) \geq \mathrm{deg}(Q_{2}), \\
Q(im) \neq 0 \ \ , \ \ R_{l}(im) \neq 0 \ \ , \ \ R_{D}(im) \neq 0 \label{first_constraints_polynomials_Q_R_main_CP}
\end{multline}
for all $m \in \mathbb{R}$, all $0 \leq l \leq D-1$.

We require that there exists a constant $r_{Q,R_{l}}>0$ such that
\begin{equation}
|\frac{Q(im)}{R_{l}(im)}| \geq r_{Q,R_{l}} \label{quotient_Q_Rl_larger_than_radius_main_CP}
\end{equation} 
for all $m \in \mathbb{R}$, all $1 \leq l \leq D$.  We make the additional assumption that there exists an unbounded sector
$$ S_{Q,R_{D}} = \{ z \in \mathbb{C} / |z| \geq r_{Q,R_{D}} \ \ , \ \ |\mathrm{arg}(z) - d_{Q,R_{D}}| \leq \eta_{Q,R_{D}} \} $$
with direction $d_{Q,R_{D}} \in \mathbb{R}$, aperture $\eta_{Q,R_{D}}>0$ for the radius $r_{Q,R_{D}}>0$ given above, such that
\begin{equation}
\frac{Q(im)}{R_{D}(im)} \in S_{Q,R_{D}} \label{quotient_Q_RD_in_S_main_CP}
\end{equation} 
for all $m \in \mathbb{R}$.

\begin{defin} Let $\varsigma \geq 2$ be an integer. For all $0 \leq p \leq \varsigma-1$, we consider open sectors
$\mathcal{E}_{p}$ centered at $0$, with radius $\epsilon_{0}$ and opening
$\frac{\pi}{k_2}+\kappa_{p}$, with $\kappa_{p}>0$ small enough such that
$\mathcal{E}_{p} \cap \mathcal{E}_{p+1} \neq \emptyset$, for all
$0 \leq p \leq \varsigma-1$ (with the convention that $\mathcal{E}_{\varsigma} = \mathcal{E}_{0})$. Moreover, we assume that
the intersection of any three different elements in $\{\mathcal{E}_{p}\}_{0 \leq p \leq \varsigma-1}$ is empty and that
$\cup_{p=0}^{\varsigma - 1} \mathcal{E}_{p} = \mathcal{U} \setminus \{ 0 \}$,
where $\mathcal{U}$ is some neighborhood of 0 in $\mathbb{C}$. Such a set of sectors
$\{ \mathcal{E}_{p} \}_{0 \leq p \leq \varsigma - 1}$ is called a good covering in $\mathbb{C}^{\ast}$.
\end{defin}

\begin{defin} Let $\{ \mathcal{E}_{p} \}_{0 \leq p \leq \varsigma - 1}$ be a good covering in $\mathbb{C}^{\ast}$. Let
$\mathcal{T}$ be an open bounded sector centered at 0 with radius $r_{\mathcal{T}}$ and consider a family of open sectors
$$ S_{\mathfrak{d}_{p},\theta,\epsilon_{0}r_{\mathcal{T}}} =
\{ T \in \mathbb{C}^{\ast} / |T| < \epsilon_{0}r_{\mathcal{T}} \ \ , \ \ |\mathfrak{d}_{p} - \mathrm{arg}(T)| < \theta/2 \} $$
with aperture $\theta > \pi/k_{2}$ and where $\mathfrak{d}_{p} \in \mathbb{R}$, for all $0 \leq p \leq \varsigma-1$, are directions
which satisfy
the following constraints: Let $q_{l}(m)$ be the roots of the polynomials (\ref{factor_P_m}) defined by (\ref{defin_roots}) and
$S_{\mathfrak{d}_p}$, $0 \leq p \leq \varsigma -1$ be unbounded sectors centered at 0 with directions
$\mathfrak{d}_{p}$ and with small aperture. Let $\rho>0$ be a positive real number. We assume
that\\
1) There exists a constant $M_{1}>0$ such that
\begin{equation}
|\tau - q_{l}(m)| \geq M_{1}(1 + |\tau|) \label{root_cond_1_in_defin_main_CP}
\end{equation}
for all $0 \leq l \leq (\delta_{D}-1)k_{2}-1$, all $m \in \mathbb{R}$, all $\tau \in S_{\mathfrak{d}_p} \cup \bar{D}(0,\rho)$, for all
$0 \leq p \leq \varsigma-1$.\\
2) There exists a constant $M_{2}>0$ such that
\begin{equation}
|\tau - q_{l_0}(m)| \geq M_{2}|q_{l_0}(m)| \label{root_cond_2_in_defin_main_CP}
\end{equation}
for some $l_{0} \in \{0,\ldots,(\delta_{D}-1)k_{2}-1 \}$, all $m \in \mathbb{R}$, all $\tau \in S_{\mathfrak{d}_p} \cup \bar{D}(0,\rho)$, for
all $0 \leq p \leq \varsigma - 1$.\\
3) There exist a family of unbounded sectors $U_{\mathfrak{d}_{p}}$ with bisecting direction $\mathfrak{d}_{p}$ and bounded
sectors $S_{\mathfrak{d}_{p}}^{b}$ with bisecting direction $\mathfrak{d}_{p}$, with radius less than $\rho$, with aperture
$\frac{\pi}{\kappa} + \delta_{p}$, with $0 < \delta_{p} < \mathrm{ap}(U_{\mathfrak{d}_{p}})$, for all $0 \leq p  \leq \varsigma - 1$,
with the property that $S_{\mathfrak{d}_{p}}^{b} \cap S_{\mathfrak{d}_{p+1}}^{b} \neq \emptyset$ for all $0 \leq p  \leq
\varsigma - 1$ (with
the convention that $\mathfrak{d}_{\varsigma} = \mathfrak{d}_{0}$).\\
4) For all $0 \leq p \leq \varsigma - 1$, for all $t \in \mathcal{T}$, all $\epsilon \in \mathcal{E}_{p}$, we have that
$\epsilon t \in S_{\mathfrak{d}_{p},\theta,\epsilon_{0}r_{\mathcal{T}}}$.\medskip

\noindent We say that the family
$\{ (S_{\mathfrak{d}_{p},\theta,\epsilon_{0}r_{\mathcal{T}}})_{0 \leq p \leq \varsigma-1},\mathcal{T} \}$
is associated to the good covering $\{ \mathcal{E}_{p} \}_{0 \leq p \leq \varsigma - 1}$.
\end{defin}

We consider a good covering $\{ \mathcal{E}_{p} \}_{0 \leq p \leq \varsigma - 1}$ and a family of sectors 
$\{ (S_{\mathfrak{d}_{p},\theta,\epsilon_{0}r_{\mathcal{T}}})_{0 \leq p \leq \varsigma-1},\mathcal{T} \}$ associated to it.
For all $0 \leq p \leq \varsigma-1$, we consider the following nonlinear initial value problem with forcing term
\begin{multline}
Q(\partial_{z})(\partial_{t}u^{\mathfrak{d}_{p}}(t,z,\epsilon)) =
c_{1,2}(\epsilon)(Q_{1}(\partial_{z})u^{\mathfrak{d}_{p}}(t,z,\epsilon))(Q_{2}(\partial_{z})u^{\mathfrak{d}_{p}}(t,z,\epsilon))\\
+ \epsilon^{(\delta_{D}-1)(k_{2}+1) - \delta_{D} + 1}t^{(\delta_{D}-1)(k_{2}+1)}
\partial_{t}^{\delta_D}R_{D}(\partial_{z})u^{\mathfrak{d}_{p}}(t,z,\epsilon)
+ \sum_{l=1}^{D-1} \epsilon^{\Delta_{l}}t^{d_l}\partial_{t}^{\delta_l}R_{l}(\partial_{z})u^{\mathfrak{d}_{p}}(t,z,\epsilon)\\
+ c_{0}(t,z,\epsilon)R_{0}(\partial_{z})u^{\mathfrak{d}_{p}}(t,z,\epsilon) +
c_{F}(\epsilon)f^{\mathfrak{d}_{p}}(t,z,\epsilon) \label{ICP_main_p}
\end{multline}
for given initial data $u^{\mathfrak{d}_{p}}(0,z,\epsilon) \equiv 0$.

The functions $c_{1,2}(\epsilon)$ and $c_{F}(\epsilon)$ are holomorphic and
bounded on the disc $D(0,\epsilon_{0})$ and are such that $c_{1,2}(0)=c_{F}(0)=0$.
The coefficient $c_{0}(t,z,\epsilon)$ and the forcing term $f^{\mathfrak{d}_{p}}(t,z,\epsilon)$ are constructed as follows.
Let $c_{0}(\epsilon)$ and $c_{0,0}(\epsilon)$ be holomorphic and bounded functions on the disc $D(0,\epsilon_{0})$ which
satisfy $c_{0}(0)=c_{0,0}(0)=0$. We consider sequences of functions
$m \mapsto C_{0,n}(m,\epsilon)$, for $n \geq 0$ and $m \mapsto F_{n}(m,\epsilon)$, for $n \geq 1$, that belong to the Banach space
$E_{(\beta,\mu)}$ for some $\beta > 0$, $\mu > \max( \mathrm{deg}(Q_{1})+1, \mathrm{deg}(Q_{2})+1)$ and which
depend holomorphically on $\epsilon \in D(0,\epsilon_{0})$. We assume that there exist constants $K_{0},T_{0}>0$
such that (\ref{norm_beta_mu_C0_n}) hold for all $n \geq 1$, for all $\epsilon \in D(0,\epsilon_{0})$.
We deduce that the function
$$ \mathfrak{C}_{0}(T,z,\epsilon) = c_{0,0}(\epsilon)\mathcal{F}^{-1}(m \mapsto C_{0,0}(m,\epsilon))(z) +
\sum_{n \geq 1} c_{0}(\epsilon)\mathcal{F}^{-1}(m \mapsto C_{0,n}(m,\epsilon))(z) T^{n}  $$
represents a bounded holomorphic function on $D(0,T_{0}/2) \times H_{\beta'} \times D(0,\epsilon_{0})$ for any
$0 < \beta' < \beta$ (where
$\mathcal{F}^{-1}$ denotes the inverse Fourier transform defined in Proposition 9). We define
the coefficient $c_{0}(t,z,\epsilon)$ as
\begin{equation}
c_{0}(t,z,\epsilon) = \mathfrak{C}_{0}(\epsilon t,z,\epsilon)
\label{defin_c_0}
\end{equation}
The function $c_{0}$ is holomorphic and bounded on $D(0,r) \times H_{\beta'} \times D(0,\epsilon_{0})$ where
$r \epsilon_{0} < T_{0}/2$.

We make the assumption that the formal $m_{k_1}-$Borel transform
$$\psi_{k_1}(\tau,m,\epsilon) = \sum_{n \geq 1} F_{n}(m,\epsilon) \frac{\tau^n}{\Gamma(\frac{n}{k_1})}$$
is convergent on the disc $D(0,\rho)$ given in Definition 8 and can be analytically continued w.r.t $\tau$ as a function
$\tau \mapsto \psi_{k_1}^{\mathfrak{d}_{p}}(\tau,m,\epsilon)$ on the domain $U_{\mathfrak{d}_{p}} \cup D(0,\rho)$, where
$U_{\mathfrak{d}_{p}}$ is the unbounded sector given in Definition 8, with
$\psi_{k_1}^{\mathfrak{d}_{p}}(\tau,m,\epsilon) \in F_{(\nu,\beta,\mu,k_{1},k_{1})}^{\mathfrak{d}_{p}}$ and such that
there exists a constant $\zeta_{\psi_{k_1}}>0$ such that
\begin{equation}
||\psi_{k_1}^{\mathfrak{d}_{p}}(\tau,m,\epsilon) ||_{(\nu,\beta,\mu,k_{1},k_{1})} \leq \zeta_{\psi_{k_1}}
\label{psi_k1_bounded_norm_k1_k1_main_CP}
\end{equation}
for all $\epsilon \in D(0,\epsilon_{0})$.

From Lemma 4, we know that the accelerated function
$$ \psi_{k_2}^{\mathfrak{d}_{p}}(\tau,m,\epsilon) :=
\mathcal{A}_{m_{k_2},m_{k_1}}^{\mathfrak{d}_{p}}( h \mapsto \psi_{k_1}^{\mathfrak{d}_{p}}(h,m,\epsilon)) (\tau) $$
defines a function that belongs to the space $F_{(\nu',\beta,\mu,k_{2})}^{\mathfrak{d}_{p}}$ for the
unbounded sector $S_{\mathfrak{d}_{p}}$ and the bounded sector $S_{\mathfrak{d}_{p}}^{b}$ given in Definition 8. Moreover,
we get a constant $\zeta_{\psi_{k_2}}>0$ with 
\begin{equation}
||\psi_{k_2}^{\mathfrak{d}_{p}}(\tau,m,\epsilon)||_{(\nu',\beta,\mu,k_{2})} \leq \zeta_{\psi_{k_2}} \label{psi_p_k2_exp_growth}
\end{equation}
for all $\epsilon \in D(0,\epsilon_{0})$. We take the $m_{k_2}-$Laplace transform
\begin{equation}
F^{\mathfrak{d}_{p}}(T,m,\epsilon) := k_{2}
\int_{L_{\mathfrak{d}_{p}}} \psi_{k_2}^{\mathfrak{d}_{p}}(u,m,\epsilon) e^{-(\frac{u}{T})^{k_2}} \frac{du}{u}
\label{defin_F_frak_d_p}
\end{equation}
which exists for all $T \in S_{\mathfrak{d}_{p},\theta,h'}$, $m \in \mathbb{R}$,
$\epsilon \in D(0,\epsilon_{0})$,
where $S_{\mathfrak{d}_{p},\theta,h'}$ is a sector with bisecting direction $\mathfrak{d}_{p}$,
aperture $\frac{\pi}{k_2} < \theta < \frac{\pi}{k_2} +
\mathrm{ap}(S_{\mathfrak{d}_{p}})$ and radius $h'$, where $h'>0$ is some positive real number, for all
$\epsilon \in D(0,\epsilon_{0})$.

We define the forcing term $f^{\mathfrak{d}_{p}}(t,z,\epsilon)$ as
\begin{equation}
f^{\mathfrak{d}_{p}}(t,z,\epsilon) := \mathcal{F}^{-1}( m \mapsto F^{\mathfrak{d}_{p}}(\epsilon t,m,\epsilon) )(z)
\label{forcing_term_frak_d_p}
\end{equation}
By construction, $f^{\mathfrak{d}_{p}}(t,z,\epsilon)$ represents a bounded holomorphic function on
$\mathcal{T} \times H_{\beta'} \times \mathcal{E}_{p}$ (provided that the radius $r_{\mathcal{T}}$ of $\mathcal{T}$
satisfies the inequality $\epsilon_{0}r_{\mathcal{T}} \leq h'$ which will be assumed in the sequel).\bigskip

In the next first main result, we construct a family of actual holomorphic solutions to the equation (\ref{ICP_main_p}) for given
initial data at $t=0$ being identically equal to zero, defined on the sectors $\mathcal{E}_{p}$ with respect to the
complex parameter $\epsilon$. We can also control the difference between any two neighboring solutions
on the intersection of sectors $\mathcal{E}_{p} \cap \mathcal{E}_{p+1}$.

\begin{theo} We consider the equation (\ref{ICP_main_p}) and we assume that the constraints
(\ref{delta_constraints_main_CP}), (\ref{constrain_d_l_delta_l_Delta_l_main_CP}), (\ref{first_constraints_polynomials_Q_R_main_CP}),
(\ref{quotient_Q_Rl_larger_than_radius_main_CP}) and (\ref{quotient_Q_RD_in_S_main_CP}) hold. We also make the additional assumption
that
\begin{multline}
d_{l}+k_{1}+1 = \delta_{l}(k_{1}+1) + d_{l,k_{1}}^{1} \ \ , \ \ d_{l,k_{1}}^{1}>0 \ \ , \ \
\frac{1}{\kappa} = \frac{1}{k_1} - \frac{1}{k_2},\\
\frac{k_2}{k_{2} - k_{1}} \geq \frac{d_{l} + (1 - \delta_{l})}{d_{l} + (1 - \delta_{l})(k_{1}+1) } \ \ , \ \
d_{l,k_{1}}^{1} > (\delta_{l}-1)(k_{2}-k_{1}) \ \ , \ \ \delta_{D} \geq \delta_{l} + \frac{1}{k_2},
\label{constraints_k1_k2_Borel_equation_for_u_p}
\end{multline}
for $1 \leq l \leq D-1$. Let the coefficient $c_{0}(t,z,\epsilon)$ and
the forcing terms $f^{\mathfrak{d}_{p}}(t,z,\epsilon)$ be constructed as in (\ref{defin_c_0}),
(\ref{forcing_term_frak_d_p}). Let a good covering $\{ \mathcal{E}_{p} \}_{0 \leq p \leq \varsigma - 1}$ in $\mathbb{C}^{\ast}$
be given, for which a family of sectors\\
$\{ (S_{\mathfrak{d}_{p},\theta,\epsilon_{0}r_{\mathcal{T}}})_{0 \leq p \leq \varsigma-1},\mathcal{T} \}$
associated to this good covering can be considered.

Then, there exist radii $r_{Q,R_{l}}>0$ large enough, for $1 \leq l \leq D$ and constants
$\zeta_{1,2},\zeta_{0,0},\zeta_{1,0},\zeta_{F}>0$ small enough such that if
\begin{multline}
\sup_{\epsilon \in D(0,\epsilon_{0})}|\frac{c_{1,2}(\epsilon)}{\epsilon}| \leq \zeta_{1,2} \ \ , \ \
\sup_{\epsilon \in D(0,\epsilon_{0})}|\frac{c_{0}(\epsilon)}{\epsilon}| \leq \zeta_{1,0} \ \ , \ \
\sup_{\epsilon \in D(0,\epsilon_{0})}|\frac{c_{0,0}(\epsilon)}{\epsilon}| \leq \zeta_{0,0},\\
\sup_{\epsilon \in D(0,\epsilon_{0})}|\frac{c_{F}(\epsilon)}{\epsilon}| \leq \zeta_{F},
\label{constraints_coeff_epsilon_main_CP}
\end{multline}
thereafter for every $0 \leq p \leq \varsigma -1$, one can construct a solution $u^{\mathfrak{d}_{p}}(t,z,\epsilon)$ of the equation
(\ref{ICP_main_p}) with $u^{\mathfrak{d}_{p}}(0,z,\epsilon) \equiv 0$ which
defines a bounded holomorphic function on the domain $\mathcal{T} \times H_{\beta'} \times
\mathcal{E}_{p}$ for any given $0< \beta'< \beta$.

Moreover, the next estimates hold for the solution $u^{\mathfrak{d}_{p}}$
and the forcing term $f^{\mathfrak{d}_{p}}$ : there exist constants $0 < h'' \leq r_{\mathcal{T}}$,
$K_{p},M_{p}>0$ (independent of $\epsilon$) with the following properties:\\
1) Assume that the unbounded sectors $U_{\mathfrak{d}_{p}}$ and $U_{\mathfrak{d}_{p+1}}$ have sufficiently large aperture
in such a way that $U_{\mathfrak{d}_{p}} \cap U_{\mathfrak{d}_{p+1}}$ contains the sector
$U_{\mathfrak{d}_{p},\mathfrak{d}_{p+1}} = \{ \tau \in \mathbb{C}^{\ast} /
\mathrm{arg}(\tau) \in [\mathfrak{d}_{p},\mathfrak{d}_{p+1}] \}$, then
\begin{multline}
\sup_{t \in \mathcal{T} \cap D(0,h''), z \in H_{\beta'}}
|u^{\mathfrak{d}_{p+1}}(t,z,\epsilon) - u^{\mathfrak{d}_{p}}(t,z,\epsilon)| \leq K_{p}e^{-\frac{M_p}{|\epsilon|^{k_2}}},\\
\sup_{t \in \mathcal{T} \cap D(0,h''), z \in H_{\beta'}}
|f^{\mathfrak{d}_{p+1}}(t,z,\epsilon) - f^{\mathfrak{d}_{p}}(t,z,\epsilon)| \leq K_{p}e^{-\frac{M_p}{|\epsilon|^{k_2}}}
\label{exp_small_difference_u_p_k2}
\end{multline}
for all $\epsilon \in \mathcal{E}_{p+1} \cap \mathcal{E}_{p}$.\\
2) Assume that the unbounded sectors $U_{\mathfrak{d}_{p}}$ and $U_{\mathfrak{d}_{p+1}}$ have empty intersection, then 
\begin{multline}
\sup_{t \in \mathcal{T} \cap D(0,h''), z \in H_{\beta'}}
|u^{\mathfrak{d}_{p+1}}(t,z,\epsilon) - u^{\mathfrak{d}_{p}}(t,z,\epsilon)| \leq K_{p}e^{-\frac{M_p}{|\epsilon|^{k_1}}},\\
\sup_{t \in \mathcal{T} \cap D(0,h''), z \in H_{\beta'}}
|f^{\mathfrak{d}_{p+1}}(t,z,\epsilon) - f^{\mathfrak{d}_{p}}(t,z,\epsilon)| \leq K_{p}e^{-\frac{M_p}{|\epsilon|^{k_1}}}
\label{exp_small_difference_u_p_k1}
\end{multline}
for all $\epsilon \in \mathcal{E}_{p+1} \cap \mathcal{E}_{p}$.
\end{theo}
\begin{proof} Let $0 \leq p \leq \varsigma -1$. Under the assumptions of Theorem 1, using Proposition 16, one can construct a function
$U^{\mathfrak{d}_{p}}(T,m,\epsilon)$ which satisfies $U^{\mathfrak{d}_{p}}(0,m,\epsilon) \equiv 0$ and solves the equation
\begin{multline}
Q(im)(\partial_{T}U^{\mathfrak{d}_{p}}(T,m,\epsilon) ) - T^{(\delta_{D}-1)(k_{2}+1)}\partial_{T}^{\delta_{D}}R_{D}(im)
U^{\mathfrak{d}_{p}}(T,m,\epsilon) \\
= \epsilon^{-1}\frac{c_{1,2}(\epsilon)}{(2\pi)^{1/2}}
\int_{-\infty}^{+\infty}Q_{1}(i(m-m_{1}))U^{\mathfrak{d}_{p}}(T,m-m_{1},\epsilon)
Q_{2}(im_{1})U^{\mathfrak{d}_{p}}(T,m_{1},\epsilon) dm_{1}\\
+ \sum_{l=1}^{D-1} R_{l}(im) \epsilon^{\Delta_{l} - d_{l} + \delta_{l} - 1} T^{d_{l}} \partial_{T}^{\delta_l}
U^{\mathfrak{d}_{p}}(T,m,\epsilon)\\
+ \epsilon^{-1}\frac{c_{0}(\epsilon)}{(2\pi)^{1/2}}\int_{-\infty}^{+\infty}C_{0}(T,m-m_{1},\epsilon)R_{0}(im_{1})
U^{\mathfrak{d}_{p}}(T,m_{1},\epsilon) dm_{1}\\
+  \epsilon^{-1}\frac{c_{0,0}(\epsilon)}{(2\pi)^{1/2}}\int_{-\infty}^{+\infty}C_{0,0}(m-m_{1},\epsilon)R_{0}(im_{1})
U^{\mathfrak{d}_{p}}(T,m_{1},\epsilon) dm_{1}
+ \epsilon^{-1}c_{F}(\epsilon)F^{\mathfrak{d}_{p}}(T,m,\epsilon)
\label{SCP_analytic_frak_d_p}
\end{multline}
where $C_{0}(T,m,\epsilon) = \sum_{n \geq 1} C_{0,n}(m,\epsilon) T^{n}$ is a convergent series on
$D(0,T_{0}/2)$ with values in $E_{(\beta,\mu)}$ and $F^{\mathfrak{d}_{p}}(T,m,\epsilon)$ is given by the formula
(\ref{defin_F_frak_d_p}), for all $\epsilon \in D(0,\epsilon_{0})$. The function
$(T,m) \mapsto U^{\mathfrak{d}_{p}}(T,m,\epsilon)$ is well defined on the domain
$S_{\mathfrak{d}_{p},\theta,h'} \times \mathbb{R}$.

Moreover, $U^{\mathfrak{d}_{p}}(T,m,\epsilon)$ can be written as $m_{k_2}-$Laplace transform
\begin{equation}
U^{\mathfrak{d}_{p}}(T,m,\epsilon) = k_{2}\int_{L_{\gamma_{p}}} \omega_{k_2}^{\mathfrak{d}_{p}}(u,m,\epsilon)
\exp( -(\frac{u}{T})^{k_2} ) \frac{du}{u}
\end{equation}
along a halfline $L_{\gamma_{p}} = \mathbb{R}_{+}e^{\sqrt{-1}\gamma_{p}} \subset S_{\mathfrak{d}_{p}} \cup \{ 0 \}$
(the direction $\gamma_{p}$ may depend on $T$), where $\omega_{k_2}^{\mathfrak{d}_{p}}(\tau,m,\epsilon)$ defines a continuous
function on $(S_{\mathfrak{d}_{p}}^{b} \cup S_{\mathfrak{d}_{p}}) \times \mathbb{R} \times D(0,\epsilon_{0})$, which
is holomorphic with respect to $(\tau,\epsilon)$ on
$(S_{\mathfrak{d}_{p}}^{b} \cup S_{\mathfrak{d}_{p}}) \times D(0,\epsilon_{0})$ for any $m \in \mathbb{R}$
and satisfies the estimates: there exists a constant $C_{\omega_{k_2}^{\mathfrak{d}_{p}}}>0$ with
\begin{equation}
|\omega_{k_2}^{\mathfrak{d}_{p}}(\tau,m,\epsilon)| \leq C_{\omega_{k_2}^{\mathfrak{d}_{p}}} (1+|m|)^{-\mu} e^{-\beta |m|}
\frac{ |\tau| }{1 + |\tau|^{2k_{2}}} \exp( \nu' |\tau|^{k_2} )
\label{omega_k2_frak_d_p_exp_growth}
\end{equation}
for all $\tau \in S_{\mathfrak{d}_{p}}^{b} \cup S_{\mathfrak{d}_{p}}$, all $m \in \mathbb{R}$, all
$\epsilon \in D(0,\epsilon_{0})$. Besides, the function
$\omega_{k_2}^{\mathfrak{d}_{p}}(\tau,m,\epsilon)$ is the analytic continuation w.r.t $\tau$ of the function
\begin{equation}
\tau \mapsto \mathrm{Acc}_{k_{2},k_{1}}^{\mathfrak{d}_{p}}(\omega_{k_1}^{\mathfrak{d}_{p}})(\tau,m,\epsilon) =
\int_{L_{\gamma_{p}^{1}}} \omega_{k_1}^{\mathfrak{d}_{p}}(h,m,\epsilon) G(\tau,h) \frac{dh}{h}
\label{acc_k_2_k_1_omega_k_1_integral}
\end{equation}
where the path of integration is a halfline
$L_{\gamma_{p}^{1}} = \mathbb{R}_{+}e^{\sqrt{-1}\gamma_{p}^{1}} \subset U_{\mathfrak{d}_{p}}$ (the direction $\gamma_{p}^{1}$ may
depend on $\tau$), which defines an analytic function on
$S_{\mathfrak{d}_{p},\kappa,\delta_{p},(c_{2}/\nu)^{1/\kappa}/2} \subset S_{\mathfrak{d}_{p}}^{b}$ which is a
sector with bisecting direction $\mathfrak{d}_{p}$, aperture $\frac{\pi}{\kappa} + \delta_{p}$ and radius
$(c_{2}/\nu)^{1/\kappa}/2$. We recall that
$\omega_{k_1}^{\mathfrak{d}_{p}}(h,m,\epsilon)$ defines a continuous function on
$(U_{\mathfrak{d}_{p}} \cup D(0,\rho)) \times \mathbb{R} \times D(0,\epsilon_{0})$, which is holomorphic
w.r.t $(\tau,\epsilon)$ on $(U_{\mathfrak{d}_{p}} \cup D(0,\rho)) \times D(0,\epsilon_{0})$, for any
$m \in \mathbb{R}$ and satisfies the estimates: there exists a constant $C_{\omega_{k_1}^{\mathfrak{d}_{p}}}>0$ with
\begin{equation}
|\omega_{k_1}^{\mathfrak{d}_{p}}(\tau,m,\epsilon)| \leq C_{\omega_{k_1}^{\mathfrak{d}_{p}}} (1+|m|)^{-\mu} e^{-\beta |m|}
\frac{ |\tau| }{1 + |\tau|^{2k_{1}}} \exp( \nu |\tau|^{\kappa} )
\label{omega_k1_frak_d_p_exp_growth}
\end{equation}
for all $\tau \in U_{\mathfrak{d}_{p}} \cup D(0,\rho)$, all $m \in \mathbb{R}$, all
$\epsilon \in D(0,\epsilon_{0})$.

Using the estimates (\ref{omega_k2_frak_d_p_exp_growth}), we get that the function
$$ (T,z) \mapsto \mathbf{U}^{\mathfrak{d}_{p}}(T,z,\epsilon) = \mathcal{F}^{-1}( m \mapsto U^{\mathfrak{d}_{p}}(T,m,\epsilon) )(z) $$
defines a bounded holomorphic function on $S_{\mathfrak{d}_{p},\theta,h'} \times H_{\beta'}$, for all
$\epsilon \in D(0,\epsilon_{0})$ and any $0 < \beta' < \beta$. For all $0 \leq p \leq \varsigma - 1$, we define
$$ u^{\mathfrak{d}_{p}}(t,z,\epsilon) = \mathbf{U}^{\mathfrak{d}_{p}}(\epsilon t,z,\epsilon) =
\frac{k_2}{(2 \pi)^{1/2}} \int_{-\infty}^{+\infty} \int_{L_{\gamma_{p}}} \omega_{k_2}^{\mathfrak{d}_{p}}(u,m,\epsilon)
\exp( - (\frac{u}{\epsilon t})^{k_2} ) e^{izm} \frac{du}{u} dm. $$
Taking into account the construction provided in 4) from Definition 8, the function $u^{\mathfrak{d}_{p}}(t,z,\epsilon)$ defines
a bounded holomorphic
function on the domain $\mathcal{T} \times H_{\beta'} \times \mathcal{E}_{p}$. Moreover, we have
$u^{\mathfrak{d}_{p}}(0,z,\epsilon) \equiv 0$ and using the properties of the Fourier inverse transform from Proposition 9, we
deduce that $u^{\mathfrak{d}_{p}}(t,z,\epsilon)$ solves the main equation (\ref{ICP_main_p}) on
$\mathcal{T} \times H_{\beta'} \times \mathcal{E}_{p}$.\medskip

Now, we proceed to the proof of the estimates (\ref{exp_small_difference_u_p_k2}). We detail only the arguments for the
functions $u^{\mathfrak{d}_{p}}$ since the estimates for the forcing terms $f^{\mathfrak{d}_{p}}$ follow the same line of
discourse as below with the help of the estimates (\ref{psi_p_k2_exp_growth}) instead of (\ref{omega_k2_frak_d_p_exp_growth}).

Let $0 \leq p \leq \varsigma - 1$ such that
$U_{\mathfrak{d}_{p}} \cap U_{\mathfrak{d}_{p+1}}$ contains the sector $U_{\mathfrak{d}_{p},\mathfrak{d}_{p+1}}$. First of all,
from the integral representation (\ref{acc_k_2_k_1_omega_k_1_integral}) by using a path deformation between $L_{\gamma_{p}^{1}}$
and $L_{\gamma_{p+1}^{1}}$, we observe that the functions
$\mathrm{Acc}_{k_{2},k_{1}}^{\mathfrak{d}_{p}}(\omega_{k_1}^{\mathfrak{d}_{p}})(\tau,m,\epsilon)$
and $\mathrm{Acc}_{k_{2},k_{1}}^{\mathfrak{d}_{p+1}}(\omega_{k_1}^{\mathfrak{d}_{p+1}})(\tau,m,\epsilon)$ must coincide
on the domain $(S_{\mathfrak{d}_{p},\kappa,\delta_{p},(c_{2}/\nu)^{1/\kappa}/2} \cap
S_{\mathfrak{d}_{p+1},\kappa,\delta_{p+1},(c_{2}/\nu)^{1/\kappa}/2}) \times \mathbb{R} \times D(0,\epsilon_{0})$. Hence, there
exists a function that we denote
$\omega_{k_2}^{\mathfrak{d}_{p},\mathfrak{d}_{p+1}}(\tau,m,\epsilon)$ which is holomorphic w.r.t $\tau$ on
$S_{\mathfrak{d}_{p},\kappa,\delta_{p},(c_{2}/\nu)^{1/\kappa}/2} \cup
S_{\mathfrak{d}_{p+1},\kappa,\delta_{p+1},(c_{2}/\nu)^{1/\kappa}/2}$, continuous w.r.t $m$ on $\mathbb{R}$, holomorphic w.r.t
$\epsilon$ on $D(0,\epsilon_{0})$ which coincides with
$\mathrm{Acc}_{k_{2},k_{1}}^{\mathfrak{d}_{p}}(\omega_{k_1}^{\mathfrak{d}_{p}})(\tau,m,\epsilon)$ on
$S_{\mathfrak{d}_{p},\kappa,\delta_{p},(c_{2}/\nu)^{1/\kappa}/2} \times \mathbb{R} \times D(0,\epsilon_{0})$ and with
$\mathrm{Acc}_{k_{2},k_{1}}^{\mathfrak{d}_{p+1}}(\omega_{k_1}^{\mathfrak{d}_{p+1}})(\tau,m,\epsilon)$ on
$S_{\mathfrak{d}_{p+1},\kappa,\delta_{p+1},(c_{2}/\nu)^{1/\kappa}/2} \times \mathbb{R} \times D(0,\epsilon_{0})$.

Now, we put $\rho_{\nu,\kappa}=(c_{2}/\nu)^{1/\kappa}/2$. Using the fact that
the function
$$ u \mapsto \omega_{k_2}^{\mathfrak{d}_{p},\mathfrak{d}_{p+1}}(u,m,\epsilon)
\exp( -(\frac{u}{\epsilon t})^{k_2} )/u $$
is holomorphic on
$S_{\mathfrak{d}_{p},\kappa,\delta_{p},\rho_{\nu,\kappa}} \cup
S_{\mathfrak{d}_{p+1},\kappa,\delta_{p+1},\rho_{\nu,\kappa}}$ for all
$(m,\epsilon) \in \mathbb{R} \times D(0,\epsilon_{0})$, its integral along the union of a segment starting from
0 to $(\rho_{\nu,\kappa}/2)e^{i\gamma_{p+1}}$, an arc of circle with radius $\rho_{\nu,\kappa}/2$ which connects
$(\rho_{\nu,\kappa}/2)e^{i\gamma_{p+1}}$ and $(\rho_{\nu,\kappa}/2)e^{i\gamma_{p}}$ and a segment starting from
$(\rho_{\nu,\kappa}/2)e^{i\gamma_{p}}$ to 0, is equal to zero. Therefore, we can write the difference
$u^{\mathfrak{d}_{p+1}} - u^{\mathfrak{d}_{p}}$ as a sum of three
integrals,
\begin{multline}
u^{\mathfrak{d}_{p+1}}(t,z,\epsilon) - u^{\mathfrak{d}_{p}}(t,z,\epsilon) = \frac{k_2}{(2\pi)^{1/2}}\int_{-\infty}^{+\infty}
\int_{L_{\rho_{\nu,\kappa}/2,\gamma_{p+1}}}
\omega_{k_2}^{\mathfrak{d}_{p+1}}(u,m,\epsilon) e^{-(\frac{u}{\epsilon t})^{k_2}} e^{izm} \frac{du}{u} dm\\ -
\frac{k_2}{(2\pi)^{1/2}}\int_{-\infty}^{+\infty}
\int_{L_{\rho_{\nu,\kappa}/2,\gamma_{p}}}
\omega_{k_2}^{\mathfrak{d}_p}(u,m,\epsilon) e^{-(\frac{u}{\epsilon t})^{k_2}} e^{izm} \frac{du}{u} dm\\
+ \frac{k_2}{(2\pi)^{1/2}}\int_{-\infty}^{+\infty}
\int_{C_{\rho_{\nu,\kappa}/2,\gamma_{p},\gamma_{p+1}}}
\omega_{k_2}^{\mathfrak{d}_{p},\mathfrak{d}_{p+1}}(u,m,\epsilon) e^{-(\frac{u}{\epsilon t})^{k_2}} e^{izm}
\frac{du}{u} dm \label{difference_u_p_decomposition_k2}
\end{multline}
where $L_{\rho_{\nu,\kappa}/2,\gamma_{p+1}} = [\rho_{\nu,\kappa}/2,+\infty)e^{i\gamma_{p+1}}$,
$L_{\rho_{\nu,\kappa}/2,\gamma_{p}} = [\rho_{\nu,\kappa}/2,+\infty)e^{i\gamma_{p}}$ and
$C_{\rho_{\nu,\kappa}/2,\gamma_{p},\gamma_{p+1}}$ is an arc of circle with radius $\rho_{\nu,\kappa}/2$ connecting
$(\rho_{\nu,\kappa}/2)e^{i\gamma_{p}}$ and $(\rho_{\nu,\kappa}/2)e^{i\gamma_{p+1}}$ with a well chosen orientation.\medskip

We give estimates for the quantity
$$ I_{1} = \left| \frac{k_2}{(2\pi)^{1/2}}\int_{-\infty}^{+\infty}
\int_{L_{\rho_{\nu,\kappa}/2,\gamma_{p+1}}}
\omega_{k_2}^{\mathfrak{d}_{p+1}}(u,m,\epsilon) e^{-(\frac{u}{\epsilon t})^{k_2}} e^{izm} \frac{du}{u} dm \right|.
$$
By construction, the direction $\gamma_{p+1}$ (which depends on $\epsilon t$) is chosen in such a way that
$\cos( k_{2}( \gamma_{p+1} - \mathrm{arg}(\epsilon t) )) \geq \delta_{1}$, for all
$\epsilon \in \mathcal{E}_{p} \cap \mathcal{E}_{p+1}$, all $t \in \mathcal{T}$, for some fixed $\delta_{1} > 0$.
From the estimates (\ref{omega_k2_frak_d_p_exp_growth}), we get that
\begin{multline}
I_{1} \leq \frac{k_{2}}{(2\pi)^{1/2}} \int_{-\infty}^{+\infty} \int_{\rho_{\nu,\kappa}/2}^{+\infty}
C_{\omega_{k_2}^{\mathfrak{d}_{p+1}}}(1+|m|)^{-\mu} e^{-\beta|m|}\frac{r}{1 + r^{2k_{2}} } \\
\times \exp( \nu' r^{k_2} )
\exp(-\frac{\cos(k_{2}(\gamma_{p+1} - \mathrm{arg}(\epsilon t)))}{|\epsilon t|^{k_2}}r^{k_2})
e^{-m\mathrm{Im}(z)} \frac{dr}{r} dm\\
\leq \frac{k_{2}C_{\omega_{k_2}^{\mathfrak{d}_{p+1}}}}{(2\pi)^{1/2}}
\int_{-\infty}^{+\infty} e^{-(\beta - \beta')|m|} dm
\int_{\rho_{\nu,\kappa}/2}^{+\infty} \exp( -(\frac{\delta_{1}}{|t|^{k_2}} - \nu'|\epsilon|^{k_2})(\frac{r}{|\epsilon|})^{k_2} )
dr\\
\leq  \frac{2k_{2}C_{\omega_{k_2}^{\mathfrak{d}_{p+1}}}}{(2\pi)^{1/2}} \int_{0}^{+\infty} e^{-(\beta - \beta')m} dm
\int_{\rho_{\nu,\kappa}/2}^{+\infty} \frac{|\epsilon|^{k_2}}{(\frac{\delta_1}{|t|^{k_2}} - \nu'|\epsilon|^{k_2})
k_{2}(\frac{\rho_{\nu,\kappa}}{2})^{k_{2}-1}}
\times \frac{ (\frac{\delta_1}{|t|^{k_2}} - \nu'|\epsilon|^{k_2})k_{2} r^{k_{2}-1} }{|\epsilon|^{k_2}}\\
\times \exp( -(\frac{\delta_{1}}{|t|^{k_2}} - \nu'|\epsilon|^{k_2})(\frac{r}{|\epsilon|})^{k_2} ) dr\\
\leq
\frac{2k_{2}C_{\omega_{k_2}^{\mathfrak{d}_{p+1}}}}{(2\pi)^{1/2}} \frac{|\epsilon|^{k_2}}{(\beta - \beta')
(\frac{\delta_{1}}{|t|^{k_2}} - \nu'|\epsilon|^{k_2})k_{2}(\frac{\rho_{\nu,\kappa}}{2})^{k_{2}-1}}
\exp( -(\frac{\delta_1}{|t|^{k_2}} - \nu'|\epsilon|^{k_2}) \frac{(\rho_{\nu,\kappa}/2)^{k_2}}{|\epsilon|^{k_2}} )\\
\leq \frac{2k_{2}C_{\omega_{k_2}^{\mathfrak{d}_{p+1}}}}{(2\pi)^{1/2}} \frac{|\epsilon|^{k_2}}{(\beta - \beta')
\delta_{2}k_{2}(\frac{\rho_{\nu,\kappa}}{2})^{k_{2}-1}}
\exp( -\delta_{2} \frac{(\rho_{\nu,\kappa}/2)^{k_2}}{|\epsilon|^{k_2}} ) \label{I_1_exp_small_order_k_2}
\end{multline}
for all $t \in \mathcal{T}$ and $|\mathrm{Im}(z)| \leq \beta'$ with
$|t| < (\frac{\delta_{1}}{\delta_{2} + \nu'\epsilon_{0}^{k_2}})^{1/k_{2}}$, for some $\delta_{2}>0$, for all
$\epsilon \in \mathcal{E}_{p} \cap \mathcal{E}_{p+1}$.\medskip

In the same way, we also give estimates for the integral
$$ I_{2} = \left| \frac{k_2}{(2\pi)^{1/2}}\int_{-\infty}^{+\infty}
\int_{L_{\rho_{\nu,\kappa}/2,\gamma_{p}}}
\omega_{k_2}^{\mathfrak{d}_{p}}(u,m,\epsilon) e^{-(\frac{u}{\epsilon t})^{k_2}} e^{izm} \frac{du}{u} dm \right|.
$$
Namely, the direction $\gamma_{p}$ (which depends on $\epsilon t$) is chosen in such a way that
$\cos( k_{2}( \gamma_{p} - \mathrm{arg}(\epsilon t) )) \geq \delta_{1}$, for all
$\epsilon \in \mathcal{E}_{p} \cap \mathcal{E}_{p+1}$, all $t \in \mathcal{T}$, for some fixed $\delta_{1} > 0$.
Again from the estimates (\ref{omega_k2_frak_d_p_exp_growth}) and following the same steps as in (\ref{I_1_exp_small_order_k_2}),
we get that
\begin{equation}
I_{2} \leq \frac{2k_{2}C_{\omega_{k_2}^{\mathfrak{d}_{p}}}}{(2\pi)^{1/2}} \frac{|\epsilon|^{k_2}}{(\beta - \beta')
\delta_{2}k_{2}(\frac{\rho_{\nu,\kappa}}{2})^{k_{2}-1}}
\exp( -\delta_{2} \frac{(\rho_{\nu,\kappa}/2)^{k_2}}{|\epsilon|^{k_2}} ) \label{I_2_exp_small_order_k_2}
\end{equation}
for all $t \in \mathcal{T}$ and $|\mathrm{Im}(z)| \leq \beta'$ with
$|t| < (\frac{\delta_{1}}{\delta_{2} + \nu'\epsilon_{0}^{k_2}})^{1/k_{2}}$, for some $\delta_{2}>0$, for all
$\epsilon \in \mathcal{E}_{p} \cap \mathcal{E}_{p+1}$.\medskip

Finally, we give upper bound estimates for the integral
$$
I_{3} = \left| \frac{k_2}{(2\pi)^{1/2}}\int_{-\infty}^{+\infty}
\int_{C_{\rho_{\nu,\kappa}/2,\gamma_{p},\gamma_{p+1}}}
\omega_{k_2}^{\mathfrak{d}_{p},\mathfrak{d}_{p+1}}(u,m,\epsilon) e^{-(\frac{u}{\epsilon t})^{k_2}} e^{izm} \frac{du}{u} dm \right|.
$$
By construction, the arc of circle $C_{\rho_{\nu,\kappa}/2,\gamma_{p},\gamma_{p+1}}$ is chosen in such a way that
$\cos(k_{2}(\theta - \mathrm{arg}(\epsilon t))) \geq \delta_{1}$, for all $\theta \in [\gamma_{p},\gamma_{p+1}]$ (if
$\gamma_{p} < \gamma_{p+1}$), $\theta \in [\gamma_{p+1},\gamma_{p}]$ (if
$\gamma_{p+1} < \gamma_{p}$), for all $t \in \mathcal{T}$, all $\epsilon \in \mathcal{E}_{p} \cap \mathcal{E}_{p+1}$, for some
fixed $\delta_{1}>0$. Bearing in mind (\ref{omega_k2_frak_d_p_exp_growth}), we get that
\begin{multline}
I_{3} \leq \frac{k_2}{(2\pi)^{1/2}} \int_{-\infty}^{+\infty}  \left| \int_{\gamma_{p}}^{\gamma_{p+1}} \right.
\max \{ C_{\omega_{k_2}^{\mathfrak{d}_{p}}},  C_{\omega_{k_2}^{\mathfrak{d}_{p+1}}} \}
(1+|m|)^{-\mu} e^{-\beta|m|}
\frac{\rho_{\nu,\kappa}/2}{1 + (\rho_{\nu,\kappa}/2)^{2k_{2}} } \\
\times \exp( \nu' (\rho_{\nu,\kappa}/2)^{k_2} )
\exp(-\frac{\cos(k_{2}(\theta - \mathrm{arg}(\epsilon t)))}{|\epsilon t|^{k_2}}(\frac{\rho_{\nu,\kappa}}{2})^{k_2})
\left. e^{-m\mathrm{Im}(z)} d\theta \right| dm\\
\leq \frac{k_{2}( \max \{ C_{\omega_{k_2}^{\mathfrak{d}_{p}}},  C_{\omega_{k_2}^{\mathfrak{d}_{p+1}}} \})}{(2\pi)^{1/2}}
\int_{-\infty}^{+\infty}
e^{-(\beta - \beta')|m|} dm \times
|\gamma_{p} - \gamma_{p+1}| \frac{\rho_{\nu,\kappa}}{2}
\exp( -( \frac{\delta_1}{|t|^{k_2}} - \nu'|\epsilon|^{k_2}) (\frac{\rho_{\nu,\kappa}/2}{|\epsilon|})^{k_2}) \\
\leq \frac{2k_{2}( \max \{ C_{\omega_{k_2}^{\mathfrak{d}_{p}}},
C_{\omega_{k_2}^{\mathfrak{d}_{p+1}}} \}) }{(2\pi)^{1/2}(\beta - \beta')} |\gamma_{p} - \gamma_{p+1}|
\frac{\rho_{\nu,\kappa}}{2} \exp( -\delta_{2} (\frac{\rho_{\nu,\kappa}/2}{|\epsilon|})^{k_2}) \label{I_3_exp_small_order_k_2}
\end{multline}
for all $t \in \mathcal{T}$ and $|\mathrm{Im}(z)| \leq \beta'$ with
$|t| < (\frac{\delta_{1}}{\delta_{2} + \nu' \epsilon_{0}^{k_2}})^{1/k_{2}}$, for some $\delta_{2}>0$, for all
$\epsilon \in \mathcal{E}_{p} \cap \mathcal{E}_{p+1}$.\medskip

Finally, gathering the three above inequalities (\ref{I_1_exp_small_order_k_2}), (\ref{I_2_exp_small_order_k_2}) and
(\ref{I_3_exp_small_order_k_2}), we deduce from the decomposition (\ref{difference_u_p_decomposition_k2}) that
\begin{multline*}
|u^{\mathfrak{d}_{p+1}}(t,z,\epsilon) - u^{\mathfrak{d}_{p}}(t,z,\epsilon)| \leq
\frac{2k_{2}(C_{\omega_{k_2}^{\mathfrak{d}_{p}}} + C_{\omega_{k_2}^{\mathfrak{d}_{p+1}}} )}{(2\pi)^{1/2}}
\frac{|\epsilon|^{k_2}}{(\beta - \beta')
\delta_{2}k_{2}(\frac{\rho_{\nu,\kappa}}{2})^{k_{2}-1}}
\exp( -\delta_{2} \frac{(\rho_{\nu,\kappa}/2)^{k_2}}{|\epsilon|^{k_2}} )+\\
\frac{2k_{2}( \max \{ C_{\omega_{k_2}^{\mathfrak{d}_{p}}},
C_{\omega_{k_2}^{\mathfrak{d}_{p+1}}} \}) }{(2\pi)^{1/2}(\beta - \beta')} |\gamma_{p} - \gamma_{p+1}|
\frac{\rho_{\nu,\kappa}}{2} \exp( -\delta_{2} (\frac{\rho_{\nu,\kappa}/2}{|\epsilon|})^{k_2})
\end{multline*}
for all $t \in \mathcal{T}$ and $|\mathrm{Im}(z)| \leq \beta'$ with
$|t| < (\frac{\delta_{1}}{\delta_{2} + \nu' \epsilon_{0}^{k_2}})^{1/k_{2}}$, for some $\delta_{2}>0$, for all
$\epsilon \in \mathcal{E}_{p} \cap \mathcal{E}_{p+1}$. Therefore, the inequality
(\ref{exp_small_difference_u_p_k2}) holds.\bigskip

In the last part of the proof, we show the estimates (\ref{exp_small_difference_u_p_k1}). Again, we only describe the
arguments for the functions $u^{\mathfrak{d}_{p}}$ since exactly the same analysis can be made for the forcing term
$f^{\mathfrak{d}_{p}}$ using the estimates (\ref{psi_k1_bounded_norm_k1_k1_main_CP}) and
(\ref{psi_p_k2_exp_growth}) instead of (\ref{omega_k2_frak_d_p_exp_growth}) and (\ref{omega_k1_frak_d_p_exp_growth}).

Let $0 \leq p \leq \varsigma-1$ such
that $U_{\mathfrak{d}_{p}} \cap U_{\mathfrak{d}_{p+1}} = \emptyset$. We first consider the following
\begin{lemma} There exist two constants $K_{p}^{\mathcal{A}},M_{p}^{\mathcal{A}}>0$ such that
\begin{equation}
|\mathrm{Acc}_{k_{2},k_{1}}^{\mathfrak{d}_{p+1}}(\omega_{k_1}^{\mathfrak{d}_{p+1}})(\tau,m,\epsilon)
- \mathrm{Acc}_{k_{2},k_{1}}^{\mathfrak{d}_{p}}(\omega_{k_1}^{\mathfrak{d}_{p}})(\tau,m,\epsilon) | \leq
K_{p}^{\mathcal{A}} \exp( - \frac{M_{p}^{\mathcal{A}}}{|\tau|^{\kappa}} ) (1+|m|)^{-\mu} e^{-\beta |m| }
\label{diff_acc_k1k2_exp_small_order_kappa}
\end{equation}
for all $\epsilon \in \mathcal{E}_{p+1} \cap \mathcal{E}_{p}$, all $\tau \in
S_{\mathfrak{d}_{p+1},\kappa,\delta_{p+1},\rho_{\nu,\kappa}} \cap S_{\mathfrak{d}_{p},\kappa,\delta_{p},\rho_{\nu,\kappa}}$,
all $m \in \mathbb{R}$.
\end{lemma}
\begin{proof} We first notice that the functions $\tau \mapsto \omega_{k_1}^{\mathfrak{d}_p}(\tau,m,\epsilon)$ and
$\tau \mapsto \omega_{k_1}^{\mathfrak{d}_{p+1}}(\tau,m,\epsilon)$ are analytic continuations of the common
$m_{k_1}-$Borel transform $\omega_{k_1}(\tau,m,\epsilon) = \sum_{n \geq 1} U_{n}(m,\epsilon) \tau^{n}/\Gamma(n/k_{1})$ which
defines a continuous function on $D(0,\rho) \times \mathbb{R} \times D(0,\epsilon_{0})$, holomorphic w.r.t
$(\tau,\epsilon)$ on $D(0,\rho) \times D(0,\epsilon_{0})$ for any $m \in \mathbb{R}$ with estimates : there exists a constant
$C_{\omega_{k_1}}>0$ with
\begin{equation}
|\omega_{k_1}(\tau,m,\epsilon)| \leq C_{\omega_{k_1}}(1+|m|)^{-\mu}e^{-\beta|m|} \frac{|\tau|}{1 + |\tau|^{2k_1}}
e^{\nu |\tau|^{\kappa}} \label{omega_k_1_exp_growth_decay}
\end{equation}
for all $\tau \in D(0,\rho)$, all $m \in \mathbb{R}$, all $\epsilon \in D(0,\epsilon_{0})$. From the
proof of Proposition 13, we know that
the function $G(\tau,h)$ is holomorphic w.r.t $(\tau,h) \in \mathbb{C}^{2}$ whenever $\tau/h$ belongs to an open unbounded
sector with direction $d=0$ and aperture $\pi/\kappa$. As a result, the integral of the function
$h \mapsto \omega_{k_1}(h,m,\epsilon)G(\tau,h)/h$, for all
$(m,\epsilon) \in \mathbb{R} \times D(0,\epsilon_{0})$, all $\tau \in S_{\mathfrak{d}_{p+1},\kappa,\delta_{p+1},\rho_{\nu,\kappa}}
\cap S_{\mathfrak{d}_{p},\kappa,\delta_{p},\rho_{\nu,\kappa}}$,
along the union of a segment starting from
0 to $(\rho/2)e^{i\gamma_{p+1}^{1}}$, an arc of circle with radius $\rho/2$ which connects
$(\rho/2)e^{i\gamma_{p+1}^{1}}$ and $(\rho/2)e^{i\gamma_{p}^{1}}$ and a segment starting from
$(\rho/2)e^{i\gamma_{p}^{1}}$ to 0, is equal to zero. Therefore, we can write the difference
$\mathrm{Acc}_{k_{2},k_{1}}^{\mathfrak{d}_{p+1}}(\omega_{k_1}^{\mathfrak{d}_{p+1}}) -
\mathrm{Acc}_{k_{2},k_{1}}^{\mathfrak{d}_{p}}(\omega_{k_1}^{\mathfrak{d}_{p}})$ as a sum of three integrals
\begin{multline}
 \mathrm{Acc}_{k_{2},k_{1}}^{\mathfrak{d}_{p+1}}(\omega_{k_1}^{\mathfrak{d}_{p+1}})(\tau,m,\epsilon)
- \mathrm{Acc}_{k_{2},k_{1}}^{\mathfrak{d}_{p}}(\omega_{k_1}^{\mathfrak{d}_{p}})(\tau,m,\epsilon)\\
= \int_{L_{\rho/2,\gamma_{p+1}^{1}}} \omega_{k_1}^{\mathfrak{d}_{p+1}}(h,m,\epsilon) G(\tau,h) \frac{dh}{h} -
\int_{L_{\rho/2,\gamma_{p}^{1}}} \omega_{k_1}^{\mathfrak{d}_{p}}(h,m,\epsilon) G(\tau,h) \frac{dh}{h}\\
+ \int_{C_{\rho/2,\gamma_{p}^{1},\gamma_{p+1}^{1}}} \omega_{k_1}(h,m,\epsilon) G(\tau,h) \frac{dh}{h}
\label{acc_k_2_k_1_omega_k_1_sum_three_integrals}
\end{multline}
where $L_{\rho/2,\gamma_{p+1}^{1}} = [\rho/2,+\infty)e^{i\gamma_{p+1}^{1}}$,
$L_{\rho/2,\gamma_{p}^{1}} = [\rho/2,+\infty)e^{i\gamma_{p}^{1}}$ and
$C_{\rho/2,\gamma_{p}^{1},\gamma_{p+1}^{1}}$ is an arc of circle with radius $\rho/2$ connecting
$(\rho/2)e^{i\gamma_{p}^{1}}$ and $(\rho/2)e^{i\gamma_{p+1}^{1}}$ with a well chosen orientation.\medskip

We give estimates for the quantity
$$ I_{1}^{\mathcal{A}} = |\int_{L_{\rho/2,\gamma_{p+1}^{1}}} \omega_{k_1}^{\mathfrak{d}_{p+1}}(h,m,\epsilon)
G(\tau,h) \frac{dh}{h} |. $$
From the estimates (\ref{G_xi_h_exp_growth_order_kappa}) and (\ref{omega_k1_frak_d_p_exp_growth}), we get that
\begin{multline}
I_{1}^{\mathcal{A}} \leq \int_{\rho/2}^{+\infty} C_{\omega_{k_1}}^{\mathfrak{d}_{p+1}} (1+|m|)^{-\mu}e^{-\beta |m|}
\frac{r}{1 + r^{2k_1}} e^{\nu r^{\kappa}} c_{1} \exp( -c_{2} (\frac{r}{|\tau|})^{\kappa} ) \frac{dr}{r}\\
\leq c_{1}C_{\omega_{k_1}}^{\mathfrak{d}_{p+1}} (1+|m|)^{-\mu}e^{-\beta |m|}
\int_{\rho/2}^{+\infty} \frac{|\tau|^{\kappa}}{(c_{2}- |\tau|^{\kappa}\nu)\kappa (\rho/2)^{\kappa-1} }
\times \frac{(c_{2} - |\tau|^{\kappa} \nu) \kappa r^{\kappa - 1}}{|\tau|^{\kappa}}\\
\times \exp( -(c_{2} - |\tau|^{\kappa}\nu)(\frac{r}{|\tau|})^{\kappa} ) dr\\
\leq c_{1}C_{\omega_{k_1}}^{\mathfrak{d}_{p+1}} (1+|m|)^{-\mu}e^{-\beta |m|}
\frac{|\tau|^{\kappa}}{(c_{2}- |\tau|^{\kappa}\nu)\kappa (\rho/2)^{\kappa-1} }
\exp( -(c_{2} - |\tau|^{\kappa}\nu)(\frac{\rho/2}{|\tau|})^{\kappa} )\\
\leq c_{1}C_{\omega_{k_1}}^{\mathfrak{d}_{p+1}} (1+|m|)^{-\mu}e^{-\beta |m|}
\frac{|\tau|^{\kappa}}{c_{2}(1 - \frac{1}{2^{\kappa}})\kappa (\rho/2)^{\kappa-1}}
\exp( -(c_{2}(1 - \frac{1}{2^{\kappa}})(\frac{\rho/2}{|\tau|})^{\kappa} ) ) \label{I_1_mathcal_A_expo_small}
\end{multline}
for all $\epsilon \in \mathcal{E}_{p+1} \cap \mathcal{E}_{p}$, all $\tau \in
S_{\mathfrak{d}_{p+1},\kappa,\delta_{p+1},\rho_{\nu,\kappa}} \cap S_{\mathfrak{d}_{p},\kappa,\delta_{p},\rho_{\nu,\kappa}}$,
all $m \in \mathbb{R}$.\medskip

In the same way, we also give estimates for the integral
$$ I_{2}^{\mathcal{A}} = |\int_{L_{\rho/2,\gamma_{p}^{1}}} \omega_{k_1}^{\mathfrak{d}_{p}}(h,m,\epsilon)
G(\tau,h) \frac{dh}{h} |. $$
Namely, from the estimates (\ref{G_xi_h_exp_growth_order_kappa}) and (\ref{omega_k1_frak_d_p_exp_growth}),
following the same steps as above in (\ref{I_1_mathcal_A_expo_small}), we get that
\begin{equation}
I_{2}^{\mathcal{A}} \leq c_{1}C_{\omega_{k_1}}^{\mathfrak{d}_{p}} (1+|m|)^{-\mu}e^{-\beta |m|}
\frac{|\tau|^{\kappa}}{c_{2}(1 - \frac{1}{2^{\kappa}})\kappa (\rho/2)^{\kappa-1}}
\exp( -(c_{2}(1 - \frac{1}{2^{\kappa}})(\frac{\rho/2}{|\tau|})^{\kappa} ) ) \label{I_2_mathcal_A_expo_small}
\end{equation}
for all $\epsilon \in \mathcal{E}_{p+1} \cap \mathcal{E}_{p}$, all $\tau \in
S_{\mathfrak{d}_{p+1},\kappa,\delta_{p+1},\rho_{\nu,\kappa}} \cap S_{\mathfrak{d}_{p},\kappa,\delta_{p},\rho_{\nu,\kappa}}$,
all $m \in \mathbb{R}$.\medskip

Finally, we give upper bound estimates for the integral
$$ I_{3}^{\mathcal{A}} = |\int_{C_{\rho/2,\gamma_{p}^{1},\gamma_{p+1}^{1}}} \omega_{k_1}(h,m,\epsilon)
G(\tau,h) \frac{dh}{h} |. $$
Bearing in mind (\ref{G_xi_h_exp_growth_order_kappa}) and (\ref{omega_k_1_exp_growth_decay}), we get that
\begin{multline}
I_{3}^{\mathcal{A}} \leq | \int_{\gamma_{p}^{1}}^{\gamma_{p+1}^{1}} C_{\omega_{k_1}} (1+|m|)^{-\mu}e^{-\beta |m|}
\frac{\rho/2}{1 + (\rho/2)^{2k_{1}}} e^{\nu (\rho/2)^{\kappa} } c_{1} \exp( -c_{2}(\frac{\rho/2}{|\tau|})^{\kappa} ) d\theta |\\
\leq c_{1}C_{\omega_{k_1}} \frac{\rho}{2} |\gamma_{p}^{1} - \gamma_{p+1}^{1}|
(1+|m|)^{-\mu}e^{-\beta |m|} \exp( -(c_{2} - |\tau|^{\kappa} \nu)(\frac{\rho/2}{|\tau|})^{\kappa} )\\
\leq c_{1}C_{\omega_{k_1}} \frac{\rho}{2} |\gamma_{p}^{1} - \gamma_{p+1}^{1}|(1+|m|)^{-\mu}e^{-\beta |m|}
\exp( -(c_{2}(1 - \frac{1}{2^{\kappa}}))(\frac{\rho/2}{|\tau|})^{\kappa} ) \label{I_3_mathcal_A_expo_small}
\end{multline}
for all $\epsilon \in \mathcal{E}_{p+1} \cap \mathcal{E}_{p}$, all $\tau \in
S_{\mathfrak{d}_{p+1},\kappa,\delta_{p+1},\rho_{\nu,\kappa}} \cap S_{\mathfrak{d}_{p},\kappa,\delta_{p},\rho_{\nu,\kappa}}$,
all $m \in \mathbb{R}$.\medskip

Finally, gathering the above inequalities (\ref{I_1_mathcal_A_expo_small}), (\ref{I_2_mathcal_A_expo_small}),
(\ref{I_3_mathcal_A_expo_small}), we deduce from the decomposition (\ref{acc_k_2_k_1_omega_k_1_sum_three_integrals}) that
\begin{multline}
|\mathrm{Acc}_{k_{2},k_{1}}^{\mathfrak{d}_{p+1}}(\omega_{k_1}^{\mathfrak{d}_{p+1}})(\tau,m,\epsilon)
- \mathrm{Acc}_{k_{2},k_{1}}^{\mathfrak{d}_{p}}(\omega_{k_1}^{\mathfrak{d}_{p}})(\tau,m,\epsilon)|\\
\leq c_{1}(C_{\omega_{k_1}}^{\mathfrak{d}_{p+1}} + C_{\omega_{k_1}}^{\mathfrak{d}_{p}})
(1+|m|)^{-\mu}e^{-\beta |m|}
\frac{\rho_{\nu,\kappa}^{\kappa}}{c_{2}(1 - \frac{1}{2^{\kappa}})\kappa (\rho/2)^{\kappa-1}}
\exp( -(c_{2}(1 - \frac{1}{2^{\kappa}})(\frac{\rho/2}{|\tau|})^{\kappa} ) )\\
+ c_{1}C_{\omega_{k_1}} \frac{\rho}{2} |\gamma_{p}^{1} - \gamma_{p+1}^{1}|(1+|m|)^{-\mu}e^{-\beta |m|}
\exp( -(c_{2}(1 - \frac{1}{2^{\kappa}}))(\frac{\rho/2}{|\tau|})^{\kappa} )
\end{multline}
for all $\epsilon \in \mathcal{E}_{p+1} \cap \mathcal{E}_{p}$, all $\tau \in
S_{\mathfrak{d}_{p+1},\kappa,\delta_{p+1},\rho_{\nu,\kappa}} \cap S_{\mathfrak{d}_{p},\kappa,\delta_{p},\rho_{\nu,\kappa}}$,
all $m \in \mathbb{R}$. We conclude that the inequality (\ref{diff_acc_k1k2_exp_small_order_kappa}) holds.
\end{proof}
Using the analytic continuation property (\ref{acc_k_2_k_1_omega_k_1_integral}) and the fact that the functions\\
$u \mapsto \omega_{k_2}^{\mathfrak{d}_{p}}(u,m,\epsilon)\exp( -(\frac{u}{\epsilon t})^{k_2} )/u$ (resp.
$u \mapsto \omega_{k_2}^{\mathfrak{d}_{p+1}}(u,m,\epsilon)\exp( -(\frac{u}{\epsilon t})^{k_2} )/u$ ) are holomorphic
on $S_{\mathfrak{d}_p}^{b} \cup S_{\mathfrak{d}_{p}}$ (resp. on
$S_{\mathfrak{d}_{p+1}}^{b} \cup S_{\mathfrak{d}_{p}}$), we can deform the straight lines of integration $L_{\gamma_{p}}$
(resp. $L_{\gamma_{p+1}}$) in such a way that
\begin{multline}
u^{\mathfrak{d}_{p+1}}(t,z,\epsilon) - u^{\mathfrak{d}_{p}}(t,z,\epsilon)\\
= \frac{k_2}{(2 \pi)^{1/2}} \int_{-\infty}^{+\infty}
\int_{L_{\rho_{\nu,\kappa}/2,\gamma_{p+1}}} \omega_{k_2}^{\mathfrak{d}_{p+1}}(u,m,\epsilon)
\exp( -(\frac{u}{\epsilon t})^{k_2} ) e^{izm} \frac{du}{u} dm\\
- \frac{k_2}{(2 \pi)^{1/2}} \int_{-\infty}^{+\infty}
\int_{L_{\rho_{\nu,\kappa}/2,\gamma_{p}}} \omega_{k_2}^{\mathfrak{d}_{p}}(u,m,\epsilon)
\exp( -(\frac{u}{\epsilon t})^{k_2} ) e^{izm} \frac{du}{u} dm\\
+ \frac{k_2}{(2 \pi)^{1/2}} \int_{-\infty}^{+\infty}
\int_{C_{\rho_{\nu,\kappa}/2,\theta_{p,p+1},\gamma_{p+1}}} \omega_{k_2}^{\mathfrak{d}_{p+1}}(u,m,\epsilon)
\exp( -(\frac{u}{\epsilon t})^{k_2} ) e^{izm} \frac{du}{u} dm\\
- \frac{k_2}{(2 \pi)^{1/2}} \int_{-\infty}^{+\infty}
\int_{C_{\rho_{\nu,\kappa}/2,\theta_{p,p+1},\gamma_{p}}} \omega_{k_2}^{\mathfrak{d}_{p}}(u,m,\epsilon)
\exp( -(\frac{u}{\epsilon t})^{k_2} ) e^{izm} \frac{du}{u} dm\\
+ \frac{k_2}{(2 \pi)^{1/2}} \int_{-\infty}^{+\infty}
\int_{L_{0,\rho_{\nu,\kappa}/2,\theta_{p,p+1}}} \left(
\mathrm{Acc}_{k_{2},k_{1}}^{\mathfrak{d}_{p+1}}(\omega_{k_1}^{\mathfrak{d}_{p+1}})(u,m,\epsilon)
- \mathrm{Acc}_{k_{2},k_{1}}^{\mathfrak{d}_{p}}(\omega_{k_1}^{\mathfrak{d}_{p}})(u,m,\epsilon) \right)\\
\times \exp( -(\frac{u}{\epsilon t})^{k_2} ) e^{izm} \frac{du}{u} dm \label{diff_u_p_sum_5_integrals}
\end{multline}
where $L_{\rho_{\nu,\kappa}/2,\gamma_{p+1}}=[\rho_{\nu,\kappa}/2,+\infty)e^{\sqrt{-1} \gamma_{p+1}}$,
$L_{\rho_{\nu,\kappa}/2,\gamma_{p}}=[\rho_{\nu,\kappa}/2,+\infty)e^{\sqrt{-1} \gamma_{p}}$, 
$C_{\rho_{\nu,\kappa}/2,\theta_{p,p+1},\gamma_{p+1}}$ is an arc of circle with radius
$\rho_{\nu,\kappa}/2$, connecting $(\rho_{\nu,\kappa}/2)e^{\sqrt{-1} \theta_{p,p+1}}$ and
$(\rho_{\nu,\kappa}/2)e^{\sqrt{-1} \gamma_{p+1}}$ with a well chosen orientation, where
$\theta_{p,p+1}$ denotes the bisecting direction of the sector
$S_{\mathfrak{d}_{p+1},\kappa,\delta_{p+1},\rho_{\nu,\kappa}} \cap
S_{\mathfrak{d}_{p},\kappa,\delta_{p},\rho_{\nu,\kappa}}$ and likewise
$C_{\rho_{\nu,\kappa}/2,\theta_{p,p+1},\gamma_{p}}$ is an arc of circle with radius
$\rho_{\nu,\kappa}/2$, connecting the points $(\rho_{\nu,\kappa}/2)e^{\sqrt{-1} \theta_{p,p+1}}$ and
$(\rho_{\nu,\kappa}/2)e^{\sqrt{-1} \gamma_{p}}$ with a well chosen orientation and finally
$L_{0,\rho_{\nu,\kappa}/2,\theta_{p,p+1}} = [0,\rho_{\nu,\kappa}/2]e^{\sqrt{-1}\theta_{p,p+1}}$.\medskip

Following the same lines of arguments as in the estimates (\ref{I_1_exp_small_order_k_2}) and
(\ref{I_3_exp_small_order_k_2}), we get the next inequalities
\begin{multline}
J_{1} = |\frac{k_2}{(2 \pi)^{1/2}} \int_{-\infty}^{+\infty}
\int_{L_{\rho_{\nu,\kappa}/2,\gamma_{p+1}}} \omega_{k_2}^{\mathfrak{d}_{p+1}}(u,m,\epsilon)
\exp( -(\frac{u}{\epsilon t})^{k_2} ) e^{izm} \frac{du}{u} dm|\\
\leq \frac{2k_{2}C_{\omega_{k_2}^{\mathfrak{d}_{p+1}}}}{(2\pi)^{1/2}} \frac{|\epsilon|^{k_2}}{(\beta - \beta')
\delta_{2}k_{2}(\frac{\rho_{\nu,\kappa}}{2})^{k_{2}-1}}
\exp( -\delta_{2} \frac{(\rho_{\nu,\kappa}/2)^{k_2}}{|\epsilon|^{k_2}} ),\\
J_{2} = |\frac{k_2}{(2 \pi)^{1/2}} \int_{-\infty}^{+\infty}
\int_{L_{\rho_{\nu,\kappa}/2,\gamma_{p}}} \omega_{k_2}^{\mathfrak{d}_{p}}(u,m,\epsilon)
\exp( -(\frac{u}{\epsilon t})^{k_2} ) e^{izm} \frac{du}{u} dm|\\
\leq \frac{2k_{2}C_{\omega_{k_2}^{\mathfrak{d}_{p}}}}{(2\pi)^{1/2}} \frac{|\epsilon|^{k_2}}{(\beta - \beta')
\delta_{2}k_{2}(\frac{\rho_{\nu,\kappa}}{2})^{k_{2}-1}}
\exp( -\delta_{2} \frac{(\rho_{\nu,\kappa}/2)^{k_2}}{|\epsilon|^{k_2}} ),\\
J_{3} = |\frac{k_2}{(2 \pi)^{1/2}} \int_{-\infty}^{+\infty}
\int_{C_{\rho_{\nu,\kappa}/2,\theta_{p,p+1},\gamma_{p+1}}} \omega_{k_2}^{\mathfrak{d}_{p+1}}(u,m,\epsilon)
\exp( -(\frac{u}{\epsilon t})^{k_2} ) e^{izm} \frac{du}{u} dm|\\
\leq \frac{2k_{2} C_{\omega_{k_2}^{\mathfrak{d}_{p+1}}} }{(2\pi)^{1/2}(\beta - \beta')}
|\gamma_{p+1} - \theta_{p,p+1}|
\frac{\rho_{\nu,\kappa}}{2} \exp( -\delta_{2} (\frac{\rho_{\nu,\kappa}/2}{|\epsilon|})^{k_2}),\\
J_{4} = |\frac{k_2}{(2 \pi)^{1/2}} \int_{-\infty}^{+\infty}
\int_{C_{\rho_{\nu,\kappa}/2,\theta_{p,p+1},\gamma_{p}}} \omega_{k_2}^{\mathfrak{d}_{p}}(u,m,\epsilon)
\exp( -(\frac{u}{\epsilon t})^{k_2} ) e^{izm} \frac{du}{u} dm|\\
\leq \frac{2k_{2} C_{\omega_{k_2}^{\mathfrak{d}_{p}}}}{(2\pi)^{1/2}(\beta - \beta')}
|\gamma_{p} - \theta_{p,p+1}|
\frac{\rho_{\nu,\kappa}}{2} \exp( -\delta_{2} (\frac{\rho_{\nu,\kappa}/2}{|\epsilon|})^{k_2})
\label{J_1_2_3_4_exp_small_order_k_2}
\end{multline}
for all $t \in \mathcal{T}$ and $|\mathrm{Im}(z)| \leq \beta'$ with
$|t| < (\frac{\delta_{1}}{\delta_{2} + \nu' \epsilon_{0}^{k_2}})^{1/k_{2}}$, for some $\delta_{1},\delta_{2}>0$, for all
$\epsilon \in \mathcal{E}_{p} \cap \mathcal{E}_{p+1}$.\medskip

In the last part of the proof, it remains to give upper bounds for the integral
\begin{multline*}
J_{5} = |\frac{k_2}{(2 \pi)^{1/2}} \int_{-\infty}^{+\infty}
\int_{L_{0,\rho_{\nu,\kappa}/2,\theta_{p,p+1}}} \left(
\mathrm{Acc}_{k_{2},k_{1}}^{\mathfrak{d}_{p+1}}(\omega_{k_1}^{\mathfrak{d}_{p+1}})(u,m,\epsilon)
- \mathrm{Acc}_{k_{2},k_{1}}^{\mathfrak{d}_{p}}(\omega_{k_1}^{\mathfrak{d}_{p}})(u,m,\epsilon) \right)\\
\times \exp( -(\frac{u}{\epsilon t})^{k_2} ) e^{izm} \frac{du}{u} dm|.
\end{multline*}
By construction, there exists $\delta_{1}>0$ such that
$\cos( k_{2}(\theta_{p,p+1} - \mathrm{arg}(\epsilon t)) ) \geq \delta_{1}$ for all
$\epsilon \in \mathcal{E}_{p} \cap \mathcal{E}_{p+1}$, all $t \in \mathcal{T}$. From Lemma 7, we get that
\begin{multline}
J_{5} \leq \frac{k_2}{(2 \pi)^{1/2}} \int_{-\infty}^{+\infty}
\int_{0}^{\rho_{\nu,\kappa}/2} K_{p}^{\mathcal{A}} (1 + |m|)^{-\mu} e^{-\beta |m|}
\exp( - \frac{M_{p}^{\mathcal{A}}}{r^{\kappa}} )\\
\times \exp( - \frac{\cos( k_{2}( \theta_{p,p+1} - \mathrm{arg}(\epsilon t) ) ) }{ |\epsilon t|^{k_2} } r^{k_2} )
e^{-m \mathrm{Im}(z) } \frac{dr}{r} dm\\
\leq \frac{k_{2}K_{p}^{\mathcal{A}}}{(2 \pi)^{1/2}} \int_{-\infty}^{+\infty} e^{-(\beta - \beta')|m|} dm \times J_{5}(\epsilon t)
\label{J_5_leq_J_5_epsilon_t}
\end{multline}
where
\begin{equation}
J_{5}(\epsilon t) = \int_{0}^{\rho_{\nu,\kappa}/2}
\exp( - \frac{M_{p}^{\mathcal{A}}}{r^{\kappa}} ) \exp( -\frac{\delta_{1}}{|\epsilon t|^{k_2}} r^{k_2} ) \frac{dr}{r}.
\label{J_5_epsilon_t}
\end{equation}
The study of estimates for $J_{5}( \epsilon t)$ as $\epsilon$ tends to zero rests on the following two lemmas.

\begin{lemma}[Watson's Lemma. Exercise 4, page 16 in~\cite{ba}]

Let $b>0$ and $f:[0,b] \rightarrow \mathbb{C} $ be a continuous function having the formal expansion
$\sum_{n \geq 0}a_{n}t^n \in \mathbb{C}[[t]]$ as its asymptotic expansion of Gevrey order $\kappa>0$ at 0, meaning there exist
$C,M>0$ such that
$$\left |f(t)-\sum_{n=0}^{N-1}a_{n}t^n \right| \leq CM^{N}N!^{\kappa}|t|^{N},$$
for every $N \geq 1$ and $t\in [0,\delta]$, for some $0<\delta<b$. Then, the function
$$I(x)=\int_{0}^{b}f(s)e^{-\frac{s}{x}}ds$$
admits the formal power series $\sum_{n \geq 0} a_{n}n!x^{n+1} \in \mathbb{C}[[x]]$ as its asymptotic expansion
of Gevrey order $\kappa+1$ at 0, it is to say, there exist $\tilde{C},\tilde{K}>0$ such that
$$\left |I(x)-\sum_{n=0}^{N-1}a_{n}n!x^{n+1}\right| \leq \tilde{C}\tilde{K}^{N+1}(N+1)!^{1+\kappa}|x|^{N+1},$$
for every $N \geq 0$ and $x \in [0,\delta']$ for some $0<\delta'<b$.
\end{lemma}
\begin{lemma}[Exercise 3, page 18 in~\cite{ba}]
Let $\delta,q>0$, and $\psi:[0,\delta] \rightarrow \mathbb{C} $ be a continuous function. The following assertions are equivalent:
\begin{enumerate}
\item There exist $C,M>0$ such that $|\psi(x)|\le CM^{n}n!^{q}|x|^{n},$ for every $n \in \mathbb{N}$, $n\ge 0$ and
$x \in [0,\delta]$.
\item There exist $C',M'>0$ such that $|\psi(x)|\le C'e^{-M'/x^{\frac{1}{q}}}$, for every $x \in (0,\delta]$.
\end{enumerate}
\end{lemma}
We make the change of variable $r^{k_2}=s$ in the integral (\ref{J_5_epsilon_t}) and we get
$$ J_{5}(\epsilon t) = \frac{1}{k_2} \int_{0}^{(\rho_{\nu,\kappa}/2)^{k_2}}
\exp( - \frac{M_{p}^{\mathcal{A}}}{s^{\kappa/k_{2}}} ) \exp( -\frac{\delta_{1}}{|\epsilon t|^{k_2}} s ) \frac{ds}{s}. $$
We put $\psi_{\mathcal{A},p}(s) = \exp( - \frac{M_{p}^{\mathcal{A}}}{s^{\kappa/k_{2}}} )/s$. From Lemma 9, there exist
constants $C,M>0$ such that
$$ |\psi_{\mathcal{A},p}(s)| \leq CM^{n}(n!)^{\frac{k_{2}}{\kappa}} |s|^{n} $$
for all $n \geq 0$, all $s \in [0,(\rho_{\nu,\kappa}/2)^{k_2}]$. In other words,
$\psi_{\mathcal{A},p}(s)$ admits the null formal series $\hat{0} \in \mathbb{C}[[s]]$ as asymptotic expansion of Gevrey order
$k_{2}/\kappa$ on $[0,(\rho_{\nu,\kappa}/2)^{k_2}]$. By Lemma 8, we deduce that the function
$$ I_{\mathcal{A},p}(x) = \int_{0}^{(\rho_{\nu,\kappa}/2)^{k_2}} \psi_{\mathcal{A},p}(s) e^{-\frac{s}{x}} ds $$
has the formal series $\hat{0} \in \mathbb{C}[[x]]$ as asymptotic expansion of Gevrey order
$\frac{k_2}{\kappa} + 1 = \frac{k_2}{k_1}$ on some segment $[0,\delta']$ with $0 < \delta' < (\rho_{\nu,\kappa}/2)^{k_2}$.
Hence, using again Lemma 9, we get two constants $C',M'>0$ with
$$ I_{\mathcal{A},p}(x) \leq C' \exp( -\frac{M'}{x^{k_{1}/k_{2}}} ) $$
for $x \in [0,\delta']$. We deduce the existence of two constants $C_{J_5}>0$, $M_{J_5}>0$ with
\begin{equation}
J_{5}(\epsilon t) \leq C_{J_5} \exp( -\frac{M_{J_5}}{|\epsilon t|^{k_1}} ) \label{J_5_epsilon_t_exp_small_order_k_1} 
\end{equation}
for all $\epsilon \in \mathcal{E}_{p} \cap \mathcal{E}_{p+1}$, all $t \in \mathcal{T} \cap D(0,h_{\mathcal{A},p})$, for some
$h_{\mathcal{A},p}>0$. Gathering the last inequality (\ref{J_5_epsilon_t_exp_small_order_k_1}) and
(\ref{J_5_leq_J_5_epsilon_t}) yields
\begin{equation}
J_{5} \leq \frac{2C_{J_5}k_{2}K_{p}^{\mathcal{A}}}{(2 \pi)^{1/2}(\beta - \beta')}
\exp( -\frac{M_{J_5}}{h_{\mathcal{A},p}^{k_1}|\epsilon|^{k_1}} ) \label{J_5_exp_small_order_k_1} 
\end{equation}
for all $\epsilon \in \mathcal{E}_{p} \cap \mathcal{E}_{p+1}$, all $t \in \mathcal{T} \cap D(0,h_{\mathcal{A},p})$.\medskip

In conclusion, taking into account the above inequalities (\ref{J_1_2_3_4_exp_small_order_k_2}) and
(\ref{J_5_exp_small_order_k_1}), we deduce from the decomposition (\ref{diff_u_p_sum_5_integrals}) that
\begin{multline*}
|u^{\mathfrak{d}_{p+1}}(t,z,\epsilon) - u^{\mathfrak{d}_{p}}(t,z,\epsilon)|
\leq \frac{2k_{2}(C_{\omega_{k_2}^{\mathfrak{d}_{p+1}}} + C_{\omega_{k_2}^{\mathfrak{d}_{p}}})}{(2\pi)^{1/2}}
\frac{|\epsilon|^{k_2}}{(\beta - \beta')
\delta_{2}k_{2}(\frac{\rho_{\nu,\kappa}}{2})^{k_{2}-1}}
\exp( -\delta_{2} \frac{(\rho_{\nu,\kappa}/2)^{k_2}}{|\epsilon|^{k_2}} )\\
+ \frac{2k_{2}}{(2\pi)^{1/2}(\beta - \beta')}\left( C_{\omega_{k_2}^{\mathfrak{d}_{p+1}}}
|\gamma_{p+1} - \theta_{p,p+1}| + C_{\omega_{k_2}^{\mathfrak{d}_{p}}}|\gamma_{p} - \theta_{p,p+1}| \right)
\frac{\rho_{\nu,\kappa}}{2} \exp( -\delta_{2} (\frac{\rho_{\nu,\kappa}/2}{|\epsilon|})^{k_2})\\
+ \frac{2C_{J_5}k_{2}K_{p}^{\mathcal{A}}}{(2 \pi)^{1/2}(\beta - \beta')}
\exp( -\frac{M_{J_5}}{h_{\mathcal{A},p}^{k_1}|\epsilon|^{k_1}} )
\end{multline*}
for all $t \in \mathcal{T}$ with $|t| < (\frac{\delta_{1}}{\delta_{2} + \nu' \epsilon_{0}^{k_2}})^{1/k_{2}}$ and
$|t| \leq h_{\mathcal{A},p}$ for some constants $\delta_{1},\delta_{2},h_{\mathcal{A},p}>0$,
$|\mathrm{Im}(z)| \leq \beta'$, for all $\epsilon \in \mathcal{E}_{p} \cap \mathcal{E}_{p+1}$. Therefore the
inequality (\ref{exp_small_difference_u_p_k1}) holds.
\end{proof}

\section{Existence of formal series solutions in the complex parameter and asymptotic expansion in two levels}

\subsection{Summable and multisummable formal series and a Ramis-Sibuya theorem with two levels}

In the next definitions we recall the meaning of Gevrey asymptotic expansions for holomorphic functions and
$k-$summability. We also give the signification of $(k_{2},k_{1})-$summability for power series in a Banach space,
as described in \cite{ba}.

\begin{defin} Let $(\mathbb{E},||.||_{\mathbb{E}})$ be a complex Banach space and let $\mathcal{E}$ be a bounded open
sector centered at 0.
Let $k>0$ be a positive real number. We say that a holomorphic function $f : \mathcal{E} \rightarrow \mathbb{E}$ admits
a formal power series $\hat{f}(\epsilon)=\sum_{n \geq 0} a_{n} \epsilon^{n} \in \mathbb{E}[[\epsilon]]$ as its asymptotic expansion
of Gevrey order $1/k$ if, for any closed proper subsector $\mathcal{W} \subset \mathcal{E}$ centered at 0, there exist
$C,M>0$ with
\begin{equation}
|| f(\epsilon) - \sum_{n=0}^{N-1} a_{n} \epsilon^{n} ||_{\mathbb{E}} \leq CM^{N}(N!)^{1/k}|\epsilon|^{N} \label{f_asympt_expans}
\end{equation}
for all $N \geq 1$, all $\epsilon \in \mathcal{W}$.

If moreover the aperture of $\mathcal{E}$ is larger than $\frac{\pi}{k} + \delta$ for some $\delta>0$, then the function
$f$ is the unique holomorphic function on $\mathcal{E}$ satisfying (\ref{f_asympt_expans}). In that case, we say that $\hat{f}$ is
$k-$summable on $\mathcal{E}$ and that $f$ defines its $k-$sum on $\mathcal{E}$. In addition, the function $f$ can be reconstructed
from the analytic continuation of the $k_{1}-$Borel transform
$$ \hat{\mathcal{B}}_{k_1}\hat{f}(\tau) = \sum_{n \geq 0} a_{n} \frac{\tau^n}{\Gamma(1 + \frac{n}{k_{1}})} $$
on an unbounded sector and by applying a $k_{1}-$Laplace transform to it, see Section 3.2 from \cite{ba}.
\end{defin}

\begin{defin} Let $(\mathbb{E},||.||_{\mathbb{E}})$ be a complex Banach space and let $0 < k_{1} < k_{2}$ be two positive real
numbers. Let $\mathcal{E}$ be a bounded open sector centered at 0 with aperture $\frac{\pi}{k_{2}} + \delta_{2}$ for some
$\delta_{2}>0$ and let $\mathcal{F}$ be a bounded open sector centered at 0 with aperture $\frac{\pi}{k_{1}} + \delta_{1}$ for some
$\delta_{1}>0$ such that the inclusion $\mathcal{E} \subset \mathcal{F}$ holds.

A formal power series $\hat{f}(\epsilon) = \sum_{n \geq 0} a_{n} \epsilon^{n} \in \mathbb{E}[[\epsilon]]$ is said to be
$(k_{2},k_{1})-$summable on $\mathcal{E}$ if there exist a formal series
$\hat{f}_{2}(\epsilon) \in \mathbb{E}[[\epsilon]]$ which is $k_{2}-$summable on $\mathcal{E}$ with
$k_{2}-$sum $f_{2}: \mathcal{E} \rightarrow \mathbb{E}$ and a second formal series
$\hat{f}_{1}(\epsilon) \in \mathbb{E}[[\epsilon]]$ which is $k_{1}-$summable on $\mathcal{F}$ with
$k_{1}-$sum $f_{1}: \mathcal{F} \rightarrow \mathbb{E}$ such that
$\hat{f} = \hat{f}_{1} + \hat{f}_{2}$. Furthermore, the holomorphic function
$f(\epsilon) = f_{1}(\epsilon) + f_{2}(\epsilon)$ defined on $\mathcal{E}$ is called the
$(k_{2},k_{1})-$sum of $\hat{f}$ on $\mathcal{E}$. In that case, the function $f(\epsilon)$ can be reconstructed
from the analytic continuation of the $k_{1}-$Borel transform of $\hat{f}$ by applying successively some acceleration operator and
Laplace transform of order $k_{2}$, see Section 6.1 from \cite{ba}.
\end{defin}

In this section, we state a version of the classical Ramis-Sibuya theorem (see \cite{hssi}, Theorem XI-2-3) with two different
Gevrey levels which describes also the case when multisummability holds on some sector. We mention that a similar multi-level
version of the Ramis-Sibuya theorem has already been stated in the manuscript \cite{tak} and also in a former work of the
authors, see \cite{lama2}.\medskip

\noindent {\bf Theorem (RS)} Let $0<k_{1}<k_{2}$ be positive real numbers. Let $(\mathbb{E},||.||_{\mathbb{E}})$ be a
Banach space over $\mathbb{C}$ and
$\{ \mathcal{E}_{i} \}_{0 \leq i \leq \nu-1}$ be a good covering in $\mathbb{C}^{\ast}$, see Definition 7. For all
$0 \leq i \leq \nu-1$, let
$G_{i}$ be a holomorphic function from $\mathcal{E}_{i}$ into
the Banach space $(\mathbb{E},||.||_{\mathbb{E}})$ and let the cocycle
$\Delta_{i}(\epsilon) = G_{i+1}(\epsilon) - G_{i}(\epsilon)$ be a
holomorphic function from the sector $Z_{i} = \mathcal{E}_{i+1} \cap \mathcal{E}_{i}$ into $\mathbb{E}$
(with the convention that $\mathcal{E}_{\nu} = \mathcal{E}_{0}$ and $G_{\nu} = G_{0}$). We make the following assumptions.\medskip

\noindent {\bf 1)} The functions $G_{i}(\epsilon)$ are bounded as $\epsilon \in \mathcal{E}_{i}$ tends to the origin in $\mathbb{C}$, for
all $0 \leq i \leq \nu - 1$.\medskip

\noindent {\bf 2)} For some finite subset $I_{1} \subset \{ 0,\ldots\nu-1 \}$ and for all $i \in I_{1}$, the functions
$\Delta_{i}(\epsilon)$ are exponentially flat on $Z_{i}$ of order $k_{1}$, for all $0 \leq i \leq \nu-1$. This means that
there exist constants $K_{i},M_{i}>0$ such that
\begin{equation}
||\Delta_{i}(\epsilon)||_{\mathbb{E}} \leq K_{i}\exp(-\frac{M_{i}}{|\epsilon|^{k_1}}) \label{Delta_i_exp_small_k1}
\end{equation}
for all $\epsilon \in Z_{i}$.\medskip

\noindent {\bf 3)} For all $i \in I_{2}=\{0,\ldots,\nu-1\} \setminus I_{1}$, the functions
$\Delta_{i}(\epsilon)$ are exponentially flat of order $k_{2}$ on $Z_{i}$, for all $0 \leq i \leq \nu-1$. This means that
there exist constants $K_{i},M_{i}>0$ such that
\begin{equation}
 ||\Delta_{i}(\epsilon)||_{\mathbb{E}} \leq K_{i}\exp(-\frac{M_{i}}{|\epsilon|^{k_{2}}} )
\label{Delta_i_exp_small_k2}
\end{equation}
for all $\epsilon \in Z_{i}$.\medskip

Then, there exist a convergent power series $a(\epsilon) \in \mathbb{E}\{ \epsilon \}$ near $\epsilon=0$ and two formal series
$\hat{G}^{1}(\epsilon),\hat{G}^{2}(\epsilon) \in \mathbb{E}[[\epsilon]]$ such that $G_{i}(\epsilon)$ owns the
following decomposition
\begin{equation}
G_{i}(\epsilon) = a(\epsilon) + G_{i}^{1}(\epsilon) + G_{i}^{2}(\epsilon) \label{G_i_equal_G_i_1_plus_G_i_2}
\end{equation}
where $G_{i}^{1}(\epsilon)$ is holomorphic on $\mathcal{E}_{i}$ and has $\hat{G}^{1}(\epsilon)$ as
asymptotic expansion of Gevrey order $1/k_{1}$
on $\mathcal{E}_{i}$, $G_{i}^{2}(\epsilon)$ is holomorphic on $\mathcal{E}_{i}$ and carries
$\hat{G}^{2}(\epsilon)$ as asymptotic expansion of Gevrey order $1/k_{2}$ on $\mathcal{E}_{i}$,
for all $0 \leq i \leq \nu-1$.\medskip

Assume moreover that some integer $i_{0} \in I_{2}$ is such that
$I_{\delta_{1},i_{0},\delta_{2}} = \{ i_{0} - \delta_{1},\ldots,i_{0},\ldots,i_{0}+\delta_{2} \} \subset I_{2}$ for some
integers $\delta_{1},\delta_{2} \geq 0$ and with the property that
\begin{equation}
\mathcal{E}_{i_{0}} \subset S_{\pi/k_{1}} \subset \bigcup_{h \in I_{\delta_{1},i_{0},\delta_{2}}} \mathcal{E}_{h}
\label{E_i0_subset_large_sector_subset_union_E_i}
\end{equation}
where $S_{\pi/k_{1}}$ is a sector centered at 0 with aperture a bit larger than $\pi/k_{1}$. Then, the formal series
$\hat{G}(\epsilon)$ is $(k_{2},k_{1})-$summable on $\mathcal{E}_{i_0}$ and its $(k_{2},k_{1})-$sum is $G_{i_0}(\epsilon)$ on
$\mathcal{E}_{i_0}$.
\begin{proof}
We consider two holomorphic cocycles $\Delta_{i}^{1}(\epsilon)$ and $\Delta_{i}^{2}(\epsilon)$ defined on the sectors $Z_{i}$ in the
following way:
$$ \Delta_{i}^{1}(\epsilon) =
\begin{cases}
\Delta_{i}(\epsilon) & \text{if } i \in I_{1} \\
0 & \text{if } i \in I_{2} 
\end{cases} \ \ , \ \ \Delta_{i}^{2}(\epsilon) =
\begin{cases}
0 & \text{if } i \in I_{1} \\
\Delta_{i}(\epsilon) & \text{if } i \in I_{2} 
\end{cases} $$
for all $\epsilon \in Z_{i}$, all $0 \leq i \leq \nu-1$. We need the following lemma.

\begin{lemma} 1) For all $0 \leq i \leq \nu-1$, there exist bounded holomorphic functions
$\Psi_{i}^{1} : \mathcal{E}_{i} \rightarrow \mathbb{C}$ such that
\begin{equation}
\Delta_{i}^{1}(\epsilon) = \Psi_{i+1}^{1}(\epsilon) - \Psi_{i}^{1}(\epsilon)
\end{equation}
for all $\epsilon \in Z_{i}$, where by convention $\Psi_{\nu}^{1}(\epsilon) = \Psi_{0}^{1}(\epsilon)$. Moreover, there exist
coefficients $\varphi_{m}^{1} \in \mathbb{E}$, $m \geq 0$, such that for each $0 \leq l \leq \nu-1$ and any closed proper subsector
$\mathcal{W} \subset \mathcal{E}_{l}$, centered at 0, there exist two constants $\breve{K}_{l},\breve{M}_{l}>0$ with
\begin{equation}
 || \Psi_{l}^{1}(\epsilon) - \sum_{m=0}^{M-1} \varphi_{m}^{1} \epsilon^{m} ||_{\mathbb{E}} \leq
\breve{K}_{l}(\breve{M}_{l})^{M}(M!)^{1/k_{1}} |\epsilon|^{M} \label{psi_l_1_gevrey_expansion_k1}
\end{equation}
for all $\epsilon \in \mathcal{W}$, all $M \geq 1$.\medskip

\noindent 2) For all $0 \leq i \leq \nu-1$, there exist bounded holomorphic functions
$\Psi_{i}^{2} : \mathcal{E}_{i} \rightarrow \mathbb{C}$ such that
\begin{equation}
\Delta_{i}^{2}(\epsilon) = \Psi_{i+1}^{2}(\epsilon) - \Psi_{i}^{2}(\epsilon)
\end{equation}
for all $\epsilon \in Z_{i}$, where by convention $\Psi_{\nu}^{2}(\epsilon) = \Psi_{0}^{2}(\epsilon)$. Moreover, there exist
coefficients $\varphi_{m}^{2} \in \mathbb{E}$, $m \geq 0$, such that for each $0 \leq l \leq \nu-1$ and any closed proper subsector
$\mathcal{W} \subset \mathcal{E}_{l}$, centered at 0, there exist two constants $\hat{K}_{l},\hat{M}_{l}>0$ with
\begin{equation}
 || \Psi_{l}^{2}(\epsilon) - \sum_{m=0}^{M-1} \varphi_{m}^{2} \epsilon^{m} ||_{\mathbb{E}} \leq
\hat{K}_{l}(\hat{M}_{l})^{M}(M!)^{1/k_{2}} |\epsilon|^{M} \label{psi_l_2_gevrey_expansion_k2}
\end{equation}
for all $\epsilon \in \mathcal{W}$, all $M \geq 1$.
\end{lemma}
\begin{proof} The proof is a consequence of Lemma XI-2-6 from \cite{hssi} which provides the so-called classical
Ramis-Sibuya theorem in Gevrey classes.
 \end{proof}
 
We consider now the bounded holomorphic functions
$$ a_{i}(\epsilon) = G_{i}(\epsilon) - \Psi_{i}^{1}(\epsilon) - \Psi_{i}^{2}(\epsilon) $$
for all $0 \leq i \leq \nu-1$, all $\epsilon \in \mathcal{E}_{i}$. By definition, for $i \in I_{1}$ or $i \in I_{2}$, we have that
$$ a_{i+1}(\epsilon) - a_{i}(\epsilon) = G_{i+1}(\epsilon) - G_{i}(\epsilon) - \Delta_{i}^{1}(\epsilon) - \Delta_{i}^{2}(\epsilon) =
G_{i+1}(\epsilon) - G_{i}(\epsilon) - \Delta_{i}(\epsilon) = 0 $$
for all $\epsilon \in Z_{i}$. Therefore, each $a_{i}(\epsilon)$ is the restriction on $\mathcal{E}_{i}$ of a holomorphic function
$a(\epsilon)$ on $D(0,r) \setminus \{ 0 \}$. Since $a(\epsilon)$ is moreover bounded on $D(0,r) \setminus \{ 0 \}$, the origin turns out
to be a removable singularity for $a(\epsilon)$ which, as a consequence, defines a convergent power series on $D(0,r)$.

Finally, one can write the following decomposition
$$ G_{i}(\epsilon) = a(\epsilon) + \Psi_{i}^{1}(\epsilon) + \Psi_{i}^{2}(\epsilon) $$
for all $\epsilon \in \mathcal{E}_{i}$, all $0 \leq i \leq \nu-1$. Moreover, $a(\epsilon)$ is a convergent power series and from
(\ref{psi_l_1_gevrey_expansion_k1}) we know that $\Psi_{i}^{1}(\epsilon)$ has the series
$\hat{G}^{1}(\epsilon) = \sum_{m \geq 0} \varphi_{m}^{1} \epsilon^{m}$ as
asymptotic expansion of Gevrey order $1/k_{1}$ on $\mathcal{E}_{i}$ and due to
(\ref{psi_l_2_gevrey_expansion_k2}) $\Psi_{i}^{2}(\epsilon)$ carries the series
$\hat{G}^{2}(\epsilon) = \sum_{m \geq 0} \varphi_{m}^{2} \epsilon^{m}$ as
asymptotic expansion of Gevrey order $1/k_{2}$ on $\mathcal{E}_{i}$, for all $0 \leq i \leq \nu-1$. Therefore, the decomposition
(\ref{G_i_equal_G_i_1_plus_G_i_2}) holds.\medskip

Assume now that some integer $i_{0} \in I_{2}$ is such that
$I_{\delta_{1},i_{0},\delta_{2}} = \{ i_{0} - \delta_{1},\ldots,i_{0},\ldots,i_{0}+\delta_{2} \} \subset I_{2}$ for some
integers $\delta_{1},\delta_{2} \geq 0$ and with the property (\ref{E_i0_subset_large_sector_subset_union_E_i}). Then, in the
decomposition (\ref{G_i_equal_G_i_1_plus_G_i_2}), we observe from the construction above that the function
$G_{i_0}^{1}(\epsilon)$ can be analytically continued on the sector $S_{\pi/k_{1}}$ and has the formal series
$\hat{G}^{1}(\epsilon)$ as asymptotic expansion of Gevrey order $1/k_{1}$ on $S_{\pi/k_{1}}$ (this is the consequence of the
fact that $\Delta_{h}^{1}(\epsilon)=0$ for $h \in I_{\delta_{1},i_{0},\delta_{2}}$). Hence,
$G_{i_0}^{1}(\epsilon)$ is the $k_{1}-$sum of $\hat{G}^{1}(\epsilon)$ on $S_{\pi/k_{1}}$ in the sense of Definition 9.
Moreover, we already know that the function $G_{i_0}^{2}(\epsilon)$ has $\hat{G}^{2}(\epsilon)$ as asymptotic expansion
of Gevrey order $1/k_{2}$ on $\mathcal{E}_{i_0}$, meaning that $G_{i_0}^{2}(\epsilon)$ is the $k_{2}-$sum of
$\hat{G}^{2}(\epsilon)$ on $\mathcal{E}_{i_0}$. In other words, by Definition 10, the formal series
$\hat{G}(\epsilon)$ is $(k_{2},k_{1})-$summable on $\mathcal{E}_{i_0}$ and its
$(k_{2},k_{1})-$sum is the function $G_{i_0}(\epsilon) = a(\epsilon) +
G_{i_0}^{1}(\epsilon) + G_{i_0}^{2}(\epsilon)$ on $\mathcal{E}_{i_0}$.
\end{proof}

\subsection{Construction of formal power series solutions in the complex parameter with two levels of asymptotics}

In this subsection, we establish the second main result of our work, namely the existence of a formal power series
$\hat{u}(t,z,\epsilon)$ in the parameter $\epsilon$ whose coefficients are bounded holomorphic
functions on the product of a sector with small radius centered at 0 and a strip in $\mathbb{C}^2$, that is a solution
of the equation (\ref{ICP_main_formal_epsilon}) and which is the common Gevrey asymptotic expansion of order $1/k_{1}$ of the
actual solutions $u^{\mathfrak{d}_{p}}(t,z,\epsilon)$ of (\ref{ICP_main_p}) constructed in Theorem 1. Furthermore,
this formal series $\hat{u}$ and the corresponding functions $u^{\mathfrak{d}_{p}}$ own a fine structure which involves
two levels of Gevrey asymptotics.\medskip

\noindent We first start by showing that the forcing terms $f^{\mathfrak{d}_{p}}(t,z,\epsilon)$ share a common formal power
series $\hat{f}(t,z,\epsilon)$ in $\epsilon$ as asymptotic expansion of Gevrey order $1/k_{1}$ on $\mathcal{E}_{p}$.

\begin{lemma} Let us assume that the hypotheses of Theorem 1 hold. Then, there exists a formal power series 
$$ \hat{f}(t,z,\epsilon) = \sum_{m \geq 0} f_{m}(t,z) \epsilon^{m}/m! $$
whose coefficients $f_{m}(t,z)$ belong to the
Banach space $\mathbb{F}$ of bounded holomorphic functions on
$(\mathcal{T} \cap D(0,h'')) \times H_{\beta'}$ equipped with supremum norm, where $h''>0$ is constructed in Theorem 1,
which is the common asymptotic expansion of Gevrey order $1/k_{1}$ on $\mathcal{E}_{p}$ of the functions
$f^{\mathfrak{d}_{p}}$, seen as holomorphic functions from $\mathcal{E}_{p}$ into
$\mathbb{F}$, for all $0 \leq p \leq \varsigma - 1$.
\end{lemma}
\begin{proof} We consider the family of functions $f^{\mathfrak{d}_{p}}(t,z,\epsilon)$, $0 \leq p \leq \varsigma-1$ constructed
in (\ref{forcing_term_frak_d_p}). For all $0 \leq p \leq \varsigma-1$, we define
$G_{p}^{f}(\epsilon) := (t,z) \mapsto f^{\mathfrak{d}_{p}}(t,z,\epsilon)$,
which is by construction a
holomorphic and bounded function from $\mathcal{E}_{p}$ into the Banach space $\mathbb{F}$ of bounded holomorphic functions on
$(\mathcal{T} \cap D(0,h'')) \times H_{\beta'}$ equipped with the supremum norm, where $\mathcal{T}$ is introduced in
Definition 8 and $h''>0$ is set in Theorem 1.

Bearing in mind the estimates (\ref{exp_small_difference_u_p_k2}) and
(\ref{exp_small_difference_u_p_k1}) and from the fact that $k_{2} > k_{1}$, we see in particular that the cocycle
$\Theta_{p}^{f}(\epsilon) = G_{p+1}^{f}(\epsilon) - G_{p}^{f}(\epsilon)$ is exponentially flat of order $k_{1}$ on
$Z_{p} = \mathcal{E}_{p} \cap \mathcal{E}_{p+1}$, for all $0 \leq p \leq \varsigma - 1$.

From the Theorem (RS) stated above in Section 6.1, we deduce the existence of a convergent power series
$a^{f}(\epsilon) \in \mathbb{F}\{ \epsilon \}$ and a formal series
$\hat{G}^{1,f}(\epsilon) \in \mathbb{F}[[\epsilon]]$ such that $G_{p}^{f}(\epsilon)$ owns the
following decomposition
$$
G_{p}^{f}(\epsilon) = a^{f}(\epsilon) + G_{p}^{1,f}(\epsilon)
$$
where $G_{p}^{1,f}(\epsilon)$ is holomorphic on $\mathcal{E}_{p}$ and has $\hat{G}^{1,f}(\epsilon)$ as
its asymptotic expansion of Gevrey order $1/k_{1}$ on $\mathcal{E}_{p}$,
We define
$$ \hat{f}(t,z,\epsilon) = \sum_{m \geq 0} f_{m}(t,z) \epsilon^{m}/m! := a^{f}(\epsilon) +
\hat{G}^{1,f}(\epsilon).$$
\end{proof}

\noindent The second main result of this work can be stated as follows.

\begin{theo} a) Let us assume that the hypotheses of Theorem 1 hold. Then, there exists a formal power series 
$$ \hat{u}(t,z,\epsilon) = \sum_{m \geq 0} h_{m}(t,z) \epsilon^{m}/m! $$
solution of the equation 
\begin{multline}
Q(\partial_{z})(\partial_{t}\hat{u}(t,z,\epsilon)) =
c_{1,2}(\epsilon)(Q_{1}(\partial_{z})\hat{u}(t,z,\epsilon))(Q_{2}(\partial_{z})\hat{u}(t,z,\epsilon))\\
+ \epsilon^{(\delta_{D}-1)(k_{2}+1) - \delta_{D} + 1}t^{(\delta_{D}-1)(k_{2}+1)}
\partial_{t}^{\delta_D}R_{D}(\partial_{z})\hat{u}(t,z,\epsilon)
+ \sum_{l=1}^{D-1} \epsilon^{\Delta_{l}}t^{d_l}\partial_{t}^{\delta_l}R_{l}(\partial_{z})\hat{u}(t,z,\epsilon)\\
+ c_{0}(t,z,\epsilon)R_{0}(\partial_{z})\hat{u}(t,z,\epsilon) +
c_{F}(\epsilon)\hat{f}(t,z,\epsilon) \label{ICP_main_formal_epsilon}
\end{multline}
whose coefficients $h_{m}(t,z)$ belong to the
Banach space $\mathbb{F}$ of bounded holomorphic functions on
$(\mathcal{T} \cap D(0,h'')) \times H_{\beta'}$ equipped with supremum norm, where $h''>0$ is constructed in Theorem 1,
which is the common asymptotic expansion of Gevrey order $1/k_{1}$ on $\mathcal{E}_{p}$ of the functions
$u^{\mathfrak{d}_{p}}$, seen as holomorphic functions from $\mathcal{E}_{p}$ into
$\mathbb{F}$, for all $0 \leq p \leq \varsigma - 1$. Additionally, the formal series can be decomposed into a sum of three
terms
$$ \hat{u}(t,z,\epsilon) = a(t,z,\epsilon) + \hat{u}_{1}(t,z,\epsilon) + \hat{u}_{2}(t,z,\epsilon) $$
where $a(t,z,\epsilon) \in \mathbb{F}\{ \epsilon \}$ is a convergent series near $\epsilon=0$ and
$\hat{u}_{1}(t,z,\epsilon)$, $\hat{u}_{2}(t,z,\epsilon)$ belong to $\mathbb{F}[[\epsilon]]$ with the property that,
accordingly, the function $u^{\mathfrak{d}_{p}}$ shares a similar decomposition
$$ u^{\mathfrak{d}_{p}}(t,z,\epsilon) = a(t,z,\epsilon) + u_{1}^{\mathfrak{d}_{p}}(t,z,\epsilon) +
u_{2}^{\mathfrak{d}_{p}}(t,z,\epsilon) $$
where $\epsilon \mapsto u_{1}^{\mathfrak{d}_{p}}(t,z,\epsilon)$ is a $\mathbb{F}-$valued function owning
$\hat{u}_{1}(t,z,\epsilon)$ as asymptotic expansion of Gevrey order $1/k_{1}$ on $\mathcal{E}_{p}$ and
where $\epsilon \mapsto u_{2}^{\mathfrak{d}_{p}}(t,z,\epsilon)$ is a $\mathbb{F}-$valued function owning
$\hat{u}_{2}(t,z,\epsilon)$ as asymptotic expansion of Gevrey order $1/k_{2}$ on $\mathcal{E}_{p}$, for all
$0 \leq p \leq \varsigma - 1$.\medskip

\noindent b) We make now the further assumption completing the four properties described in Definition 8 that the
good covering $\{ \mathcal{E}_{p} \}_{0 \leq p \leq \varsigma - 1}$ and that the family of unbounded sectors
$\{ U_{\mathfrak{d}_{p}} \}_{0 \leq p \leq \varsigma - 1}$ satisfy the following property:\medskip

\noindent 5) There exist $0 \leq p_{0} \leq \varsigma - 1$ and two integers $\delta_{1},\delta_{2} \geq 0$ such that for all
$p \in I_{\delta_{1},p_{0},\delta_{2}} = \{ p_{0}- \delta_{1}, \ldots, p_{0}, \ldots, p_{0} + \delta_{2} \}$, the unbounded
sectors $U_{\mathfrak{d}_{p}}$ are such that the intersection $U_{\mathfrak{d}_{p}} \cap U_{\mathfrak{d}_{p+1}}$ contains the sector
$U_{\mathfrak{d}_{p},\mathfrak{d}_{p+1}} = \{ \tau \in \mathbb{C}^{\ast} / \mathrm{arg}(\tau) \in
[\mathfrak{d}_{p},\mathfrak{d}_{p+1}] \}$ and such that
$$ \mathcal{E}_{p_0} \subset S_{\pi/k_{1}} \subset \bigcup_{h \in I_{\delta_{1},p_{0},\delta_{2}}} \mathcal{E}_{h} $$
where $S_{\pi/k_{1}}$ is a sector centered at 0 with aperture slightly larger than $\pi/k_{1}$.\medskip

\noindent Then, the formal series
$\hat{u}(t,z,\epsilon)$ is $(k_{2},k_{1})-$summable on $\mathcal{E}_{p_0}$ and its $(k_{2},k_{1})-$sum is given by
$u^{\mathfrak{d}_{p_0}}(t,z,\epsilon)$.
\end{theo}
\begin{proof} We consider the family of functions $u^{\mathfrak{d}_{p}}(t,z,\epsilon)$, $0 \leq p \leq \varsigma-1$ constructed
in Theorem 1.
For all $0 \leq p \leq \varsigma-1$, we define $G_{p}(\epsilon) := (t,z) \mapsto u^{\mathfrak{d}_{p}}(t,z,\epsilon)$,
which is by construction a
holomorphic and bounded function from $\mathcal{E}_{p}$ into the Banach space $\mathbb{F}$ of bounded holomorphic functions on
$(\mathcal{T} \cap D(0,h'')) \times H_{\beta'}$ equipped with the supremum norm, where $\mathcal{T}$ is introduced in
Definition 8, $h''>0$ is set in Theorem 1 and
$\beta'>0$ is the width of the strip $H_{\beta'}$ on which the coefficient $c_{0}(t,z,\epsilon)$ and the
forcing term $f^{\mathfrak{d}_{p}}(t,z,\epsilon)$ are defined with respect to
$z$, see (\ref{defin_c_0}) and (\ref{forcing_term_frak_d_p}).

Bearing in mind the estimates (\ref{exp_small_difference_u_p_k2}) and
(\ref{exp_small_difference_u_p_k1}) we see that the cocycle
$\Theta_{p}(\epsilon) = G_{p+1}(\epsilon) - G_{p}(\epsilon)$ is exponentially flat of order $k_{2}$ on
$Z_{p} = \mathcal{E}_{p} \cap \mathcal{E}_{p+1}$, for all $p \in I_{2} \subset \{ 0, \ldots, \varsigma-1 \}$ such that
the intersection $U_{\mathfrak{d}_{p}} \cap U_{\mathfrak{d}_{p+1}}$ contains the sector $U_{\mathfrak{d}_{p},\mathfrak{d}_{p+1}}$
and is exponentially flat of order $k_{1}$ on $Z_{p} = \mathcal{E}_{p} \cap \mathcal{E}_{p+1}$, for all
$p \in I_{1} \subset \{ 0, \ldots, \varsigma-1 \}$ such that the intersection
$U_{\mathfrak{d}_{p}} \cap U_{\mathfrak{d}_{p+1}}$ is empty.

From the Theorem (RS) stated above in Section 6.1, we deduce the existence of a convergent power series
$a(\epsilon) \in \mathbb{F}\{ \epsilon \}$ and two formal series
$\hat{G}^{1}(\epsilon),\hat{G}^{2}(\epsilon) \in \mathbb{F}[[\epsilon]]$ such that $G_{p}(\epsilon)$ owns the
following decomposition
$$
G_{p}(\epsilon) = a(\epsilon) + G_{p}^{1}(\epsilon) + G_{p}^{2}(\epsilon)
$$
where $G_{p}^{1}(\epsilon)$ is holomorphic on $\mathcal{E}_{p}$ and has $\hat{G}^{1}(\epsilon)$ as
its asymptotic expansion of Gevrey order $1/k_{1}$ on $\mathcal{E}_{p}$,
$G_{p}^{2}(\epsilon)$ is holomorphic on $\mathcal{E}_{p}$ and carries
$\hat{G}^{2}(\epsilon)$ as its asymptotic expansion of Gevrey order $1/k_{2}$
on $\mathcal{E}_{p}$, for all $0 \leq p \leq \nu-1$.
We set
$$ \hat{u}(t,z,\epsilon) = \sum_{m \geq 0} h_{m}(t,z) \epsilon^{m}/m! := a(\epsilon) +
\hat{G}^{1}(\epsilon) + \hat{G}^{2}(\epsilon).$$
This yields the first part a) of Theorem 2.

Furthermore, under the assumption b) 5) described above, the Theorem (RS) claims that the formal series
$\hat{G}(\epsilon) = a(\epsilon) + \hat{G}^{1}(\epsilon) + \hat{G}^{2}(\epsilon)$ is $(k_{2},k_{1})-$summable on
$\mathcal{E}_{p_0}$ and that its $(k_{2},k_{1})-$sum is given by $G_{p_0}(\epsilon)$.\medskip

It remains to show that the formal series $\hat{u}(t,z,\epsilon)$ solves the main equation (\ref{ICP_main_formal_epsilon}). Since
$u^{\mathfrak{d}_{p}}(t,z,\epsilon)$ (resp. $f^{\mathfrak{d}_{p}}(t,z,\epsilon)$ ) has $\hat{u}(t,z,\epsilon)$
(resp. $\hat{f}(t,z,\epsilon)$) as its asymptotic expansion of Gevrey order $1/k_{1}$ on $\mathcal{E}_{p}$, we have in particular
that
\begin{multline}
\lim_{\epsilon \rightarrow 0, \epsilon \in \mathcal{E}_{p}}
\sup_{t \in \mathcal{T} \cap D(0,h''),z \in H_{\beta'}}
|\partial_{\epsilon}^{m}u^{\mathfrak{d}_{p}}(t,z,\epsilon) - h_{m}(t,z)| = 0,\\
\lim_{\epsilon \rightarrow 0, \epsilon \in \mathcal{E}_{p}}
\sup_{t \in \mathcal{T} \cap D(0,h''),z \in H_{\beta'}}
|\partial_{\epsilon}^{m}f^{\mathfrak{d}_{p}}(t,z,\epsilon) - f_{m}(t,z)| = 0, \label{limit_deriv_order_m_of_up_fp_epsilon}
\end{multline}
for all $0 \leq p \leq \varsigma-1$, all $m \geq 0$. Now, we choose some $p \in \{ 0, \ldots, \varsigma-1 \}$.
By construction, the function
$u^{\mathfrak{d}_{p}}(t,z,\epsilon)$ is a solution of (\ref{ICP_main_p}). We take the derivative of order $m \geq 0$ w.r.t
$\epsilon$ on the
left and right handside of the equation (\ref{ICP_main_p}). From the Leibniz rule, we deduce that
$\partial_{\epsilon}^{m}u^{\mathfrak{d}_{p}}(t,z,\epsilon)$ verifies the following equation
\begin{multline}
Q(\partial_{z})\partial_{t} \partial_{\epsilon}^{m}u^{\mathfrak{d}_{p}}(t,z,\epsilon) =
\sum_{m_{1}+m_{2}+m_{3} = m} \frac{m!}{m_{1}!m_{2}!m_{3}!}
\partial_{\epsilon}^{m_1}c_{1,2}(\epsilon)
\left( Q_{1}(\partial_{z}) \partial_{\epsilon}^{m_2}u^{\mathfrak{d}_{p}}(t,z,\epsilon) \right)\\
\times \left( Q_{2}(\partial_{z}) \partial_{\epsilon}^{m_3}u^{\mathfrak{d}_{p}}(t,z,\epsilon) \right) +
\sum_{m_{1}+m_{2}=m} \frac{m!}{m_{1}!m_{2}!} \partial_{\epsilon}^{m_1}( \epsilon^{(\delta_{D}-1)(k_{2}+1) - \delta_{D}+1})
t^{(\delta_{D}-1)(k_{2}+1)}\\
\times \partial_{t}^{\delta_{D}}R_{D}(\partial_{z})\partial_{\epsilon}^{m_2}u^{\mathfrak{d}_{p}}(t,z,\epsilon)
+ \sum_{l=1}^{D-1}( \sum_{m_{1}+m_{2}=m} \frac{m!}{m_{1}!m_{2}!} \partial_{\epsilon}^{m_1}(\epsilon^{\Delta_l})
t^{d_l}\partial_{t}^{\delta_l}R_{l}(\partial_{z})\partial_{\epsilon}^{m_2}u^{\mathfrak{d}_{p}}(t,z,\epsilon) )\\
+ \sum_{m_{1}+m_{2}=m} \frac{m!}{m_{1}!m_{2}!} \partial_{\epsilon}^{m_1}c_{0}(t,z,\epsilon)R_{0}(\partial_{z})
\partial_{\epsilon}^{m_2}u^{\mathfrak{d}_{p}}(t,z,\epsilon)\\
+ \sum_{m_{1}+m_{2}=m} \frac{m!}{m_{1}!m_{2}!} \partial_{\epsilon}^{m_1}c_{F}(\epsilon)
\partial_{\epsilon}^{m_2}f^{\mathfrak{d}_{p}}(t,z,\epsilon) \label{ICP_main_p_derivative_epsilon}
\end{multline}
for all $m \geq 0$, all $(t,z,\epsilon) \in (\mathcal{T} \cap D(0,h'')) \times H_{\beta'} \times \mathcal{E}_{p}$. If we
let $\epsilon$ tend to zero in (\ref{ICP_main_p_derivative_epsilon}) and if we use
(\ref{limit_deriv_order_m_of_up_fp_epsilon}), we get the recursion
\begin{multline}
Q(\partial_{z})\partial_{t}h_{m}(t,z) = \sum_{m_{1}+m_{2}+m_{3}=m}
\frac{m!}{m_{1}!m_{2}!m_{3}!}(\partial_{\epsilon}^{m_1}c_{1,2})(0) \left( Q_{1}(\partial_{z})h_{m_2}(t,z) \right)
\left( Q_{2}(\partial_{z})h_{m_3}(t,z) \right)\\
+ \frac{m!}{(m - ((\delta_{D}-1)(k_{2}+1)-\delta_{D}+1))!}t^{(\delta_{D}-1)(k_{2}+1)}\partial_{t}^{\delta_D}
R_{D}(\partial_{z})h_{m-((\delta_{D}-1)(k_{2}+1)-\delta_{D}+1)}(t,z)\\
+ \sum_{l=1}^{D-1} \frac{m!}{(m - \Delta_{l})!}t^{d_l}\partial_{t}^{\delta_l}R_{l}(\partial_{z})h_{m - \Delta_{l}}(t,z)
+ \sum_{m_{1}+m_{2}=m} \frac{m!}{m_{1}!m_{2}!}(\partial_{\epsilon}^{m_1}c_{0})(t,z,0)R_{0}(\partial_{z})h_{m_2}(t,z)\\
+ \sum_{m_{1}+m_{2}=m} \frac{m!}{m_{1}!m_{2}!}(\partial_{\epsilon}^{m_1}c_{F})(0)f_{m_2}(t,z) \label{rec_h_m}
\end{multline}
for all $m \geq \max_{1 \leq l \leq D-1} \{ \Delta_{l}, (\delta_{D}-1)(k_{2}+1) - \delta_{D}+1 \}$,
all $(t,z) \in (\mathcal{T} \cap D(0,h'')) \times H_{\beta'}$. Since the functions $c_{1,2}(\epsilon)$,
$c_{0}(t,z,\epsilon)$ and $c_{F}(\epsilon)$ are analytic w.r.t $\epsilon$ at 0, we know that
\begin{multline}
c_{1,2}(\epsilon) = \sum_{m \geq 0} \frac{ (\partial_{\epsilon}^{m}c_{1,2})(0) }{m!} \epsilon^{m} \ \ , \ \
c_{0}(t,z,\epsilon) = \sum_{m \geq 0} \frac{(\partial_{\epsilon}^{m}c_{0})(t,z,0)}{m!}\epsilon^{m},\\
c_{F}(\epsilon) = \sum_{m \geq 0} \frac{ (\partial_{\epsilon}^{m}c_{F})(0) }{m!} \epsilon^{m} \label{Taylor_c12_c0_cF}
\end{multline}
for all $\epsilon \in D(0,\epsilon_{0})$, all $z \in H_{\beta'}$. On other hand, one can check by direct
inspection from the recursion
(\ref{rec_h_m}) and the expansions (\ref{Taylor_c12_c0_cF}) that the series
$\hat{u}(t,z,\epsilon) = \sum_{m \geq 0} h_{m}(t,z)\epsilon^{m}/m!$ formally solves the equation (\ref{ICP_main_formal_epsilon}).
\end{proof}

\section{Application. Construction of analytic and formal solutions in a complex parameter of a nonlinear initial value Cauchy
problem with analytic coefficients and forcing term near the origin in $\mathbb{C}^{3}$}

In this section, we give sufficient conditions on the forcing term $F(T,m,\epsilon)$ for the functions
$u^{\mathfrak{d}_{p}}(t,z,\epsilon)$ and its corresponding formal power series expansion
$\hat{u}(t,z,\epsilon)$ w.r.t $\epsilon$ constructed in Theorem 1 and Theorem 2
to solve a nonlinear problem with holomorphic coefficients and forcing term near the origin given by
(\ref{nlpde_holcoef_near_origin_u_dp}). 

\subsection{A linear convolution initial value problem satisfied by the formal forcing term $F(T,m,\epsilon)$}

Let $k_{1} \geq 1$ be the integer defined above in Section 5 and let $\mathbf{D} \geq 2$ be an integer. For
$1 \leq l \leq \mathbf{D}$, let
$\mathbf{d}_{l}$,$\boldsymbol{\delta}_{l}$,$\mathbf{\Delta}_{l} \geq 0$ be nonnegative integers.
We assume that 
\begin{equation}
1 = \boldsymbol{\delta}_{1} \ \ , \ \ \boldsymbol{\delta}_{l} < \boldsymbol{\delta}_{l+1},  \label{assum_delta_l_bf}
\end{equation}
for all $1 \leq l \leq \mathbf{D}-1$. We make also the assumption that
\begin{equation}
\mathbf{d}_{\mathbf{D}} = (\boldsymbol{\delta}_{\mathbf{D}}-1)(k_{1}+1) \ \ , \ \
\mathbf{d}_{l} > (\boldsymbol{\delta}_{l}-1)(k_{1}+1) \ \ , \ \
\mathbf{\Delta}_{l} - \mathbf{d}_{l} + \boldsymbol{\delta}_{l} - 1 \geq 0 \ \ , \ \
\mathbf{\Delta}_{\mathbf{D}} = \mathbf{d}_{\mathbf{D}} - \boldsymbol{\delta}_{\mathbf{D}} + 1
\label{assum_dl_delta_l_Delta_l_bf}
\end{equation}
for all $1 \leq l \leq \mathbf{D}-1$. Let $\mathbf{Q}(X),\mathbf{R}_{l}(X) \in \mathbb{C}[X]$,
$0 \leq l \leq \mathbf{D}$, be polynomials such that
\begin{equation}
\mathrm{deg}(\mathbf{Q}) \geq \mathrm{deg}(\mathbf{R}_{\mathbf{D}}) \geq \mathrm{deg}(\mathbf{R}_{l}) \ \ , \ \
\mathbf{Q}(im) \neq 0 \ \ , \ \ \mathbf{R}_{\mathbf{D}}(im) \neq 0 \label{assum_Q_Rl_bf}
\end{equation}
for all $m \in \mathbb{R}$, all $0 \leq l \leq \mathbf{D}-1$. Let $\beta,\mu>0$ be the integers defined above in Section 5.
We consider sequences of functions
$m \mapsto \mathbf{C}_{0,n}(m,\epsilon)$, for all $n \geq 0$ and $m \mapsto \mathbf{F}_{n}(m,\epsilon)$, for all $n \geq 1$,
that belong to the Banach space $E_{(\beta,\mu)}$ and which
depend holomorphically on $\epsilon \in D(0,\epsilon_{0})$. We assume that there exist constants $\mathbf{K}_{0},\mathbf{T}_{0}>0$
such that
\begin{equation}
||\mathbf{C}_{0,n}(m,\epsilon)|| _{(\beta,\mu)} \leq \mathbf{K}_{0} (\frac{1}{\mathbf{T}_{0}})^{n} \ \ , \ \
||\mathbf{F}_{n}(m,\epsilon)||_{(\beta,\mu)} \leq \mathbf{K}_{0} (\frac{1}{\mathbf{T}_{0}})^{n} \label{norm_beta_mu_F_n_bf}
\end{equation}
for all $n \geq 1$, for all $\epsilon \in D(0,\epsilon_{0})$. We define
$$ \mathbf{C}_{0}(T,m,\epsilon) = \sum_{n \geq 1} \mathbf{C}_{0,n}(m,\epsilon) T^{n} \ \ , \ \
\mathbf{F}(T,m,\epsilon) = \sum_{n \geq 1} \mathbf{F}_{n}(m,\epsilon) T^{n} $$
which are convergent series on $D(0,\mathbf{T}_{0}/2)$ with values in $E_{(\beta,\mu)}$.
Let $\mathbf{c}_{0}(\epsilon),\mathbf{c}_{0,0}(\epsilon)$ and
$\mathbf{c}_{\mathbf{F}}(\epsilon)$ be bounded holomorphic functions on $D(0,\epsilon_{0})$ which vanish at the origin
$\epsilon=0$.

We make the assumption that the
formal series $F(T,m,\epsilon)=\sum_{n \geq 1} F_{n}(m,\epsilon)T^{n}$, where the coefficients $F_{n}(m,\epsilon)$ are defined
after the problem (\ref{ICP_main_p}) in Section 5 satisfies the next linear initial value problem
\begin{multline}
\mathbf{Q}(im)(\partial_{T}F(T,m,\epsilon) ) =
\sum_{l=1}^{\mathbf{D}} \mathbf{R}_{l}(im) \epsilon^{\mathbf{\Delta}_{l} - \mathbf{d}_{l} + \boldsymbol{\delta}_{l} - 1}
T^{\mathbf{d}_{l}} \partial_{T}^{\boldsymbol{\delta}_l}F(T,m,\epsilon)\\
+ \epsilon^{-1}\frac{\mathbf{c}_{0}(\epsilon)}{(2\pi)^{1/2}}\int_{-\infty}^{+\infty}
\mathbf{C}_{0}(T,m-m_{1},\epsilon)\mathbf{R}_{0}(im_{1})F(T,m_{1},\epsilon) dm_{1}\\
+ \epsilon^{-1}\frac{\mathbf{c}_{0,0}(\epsilon)}{(2\pi)^{1/2}}\int_{-\infty}^{+\infty}\mathbf{C}_{0,0}(m-m_{1},\epsilon)
\mathbf{R}_{0}(im_{1})F(T,m_{1},\epsilon) dm_{1}
+ \epsilon^{-1}\mathbf{c}_{\mathbf{F}}(\epsilon)\mathbf{F}(T,m,\epsilon)
\label{SCP_bf}
\end{multline}
for given initial data $F(0,m,\epsilon) = 0$.\medskip

The existence and uniqueness of the formal power series solution of (\ref{SCP_bf}) is ensured by the following

\begin{prop} There exists a unique formal series
$$ F(T,m,\epsilon) = \sum_{n \geq 1} F_{n}(m,\epsilon) T^{n} $$
solution of (\ref{SCP_bf}) with initial data $F(0,m,\epsilon) \equiv 0$, where the coefficients
$m \mapsto F_{n}(m,\epsilon)$ belong to $E_{(\beta,\mu)}$ for $\beta,\mu>0$ given above and depend
holomorphically on $\epsilon$ in $D(0,\epsilon_{0})$. 
\end{prop}
\begin{proof} From Proposition 4, we get that the coefficients
$F_{n}(m,\epsilon)$ of $F(T,m,\epsilon)$ are well defined, belong to $E_{(\beta,\mu)}$ for all
$\epsilon \in D(0,\epsilon_{0})$, all $n \geq 1$ and satisfy the following recursion relation
\begin{multline}
(n+1)F_{n+1}(m,\epsilon)\\
= \sum_{l=1}^{\mathbf{D}} \frac{\mathbf{R}_{l}(im)}{\mathbf{Q}(im)}\left( \epsilon^{\mathbf{\Delta}_{l} - \mathbf{d}_{l}
+ \boldsymbol{\delta}_{l} - 1}
\Pi_{j=0}^{\boldsymbol{\delta}_{l}-1} (n+\boldsymbol{\delta}_{l}-\mathbf{d}_{l}-j) \right)
F_{n+\boldsymbol{\delta}_{l}-\mathbf{d}_{l}}(m,\epsilon)\\
+ \frac{\epsilon^{-1}\mathbf{c}_{0}(\epsilon)}{\mathbf{Q}(im)}
\sum_{n_{1}+n_{2}=n,n_{1} \geq 1,n_{2} \geq 1} \frac{1}{(2\pi)^{1/2}} \int_{-\infty}^{+\infty}
\mathbf{C}_{0,n_{1}}(m-m_{1},\epsilon)\mathbf{R}_{0}(im_{1})F_{n_2}(m_{1},\epsilon) dm_{1}\\
+ \frac{\epsilon^{-1}\mathbf{c}_{0,0}(\epsilon)}{(2\pi)^{1/2}\mathbf{Q}(im)} \int_{-\infty}^{+\infty}
\mathbf{C}_{0,0}(m-m_{1},\epsilon)\mathbf{R}_{0}(im_{1})F_{n}(m_{1},\epsilon) dm_{1}
+ \frac{\epsilon^{-1}\mathbf{c}_{\mathbf{F}}(\epsilon)}{\mathbf{Q}(im)}\mathbf{F}_{n}(m,\epsilon)
\end{multline}
for all $n \geq \max_{1 \leq l \leq \mathbf{D}}\mathbf{d}_{l}$.
\end{proof}

\subsection{Analytic solutions for an auxiliary linear convolution problem resulting from a $m_{k_1}-$Borel transform
applied to the linear initial value convolution problem}

Using the formula (8.7) from \cite{taya}, p. 3630, we can expand the operators
$T^{\boldsymbol{\delta}_{l}(k_{1}+1)} \partial_{T}^{\boldsymbol{\delta}_l}$ in the form
\begin{equation}
T^{\boldsymbol{\delta}_{l}(k_{1}+1)} \partial_{T}^{\boldsymbol{\delta}_l} = (T^{k_{1}+1}\partial_{T})^{\boldsymbol{\delta}_l} +
\sum_{1 \leq p \leq \boldsymbol{\delta}_{l}-1} A_{\boldsymbol{\delta}_{l},p} T^{k_{1}(\boldsymbol{\delta}_{l}-p)}
(T^{k_{1}+1}\partial_{T})^{p}
\label{expand_op_diff_bf}
\end{equation}
where $A_{\boldsymbol{\delta}_{l},p}$, $p=1,\ldots,\boldsymbol{\delta}_{l}-1$ are real numbers, for all
$1 \leq l \leq \mathbf{D}$. We define integers $\mathbf{d}_{l,k_{1}} \geq 0$ to satisfy
\begin{equation}
\mathbf{d}_{l} + k_{1} + 1 = \boldsymbol{\delta}_{l}(k_{1}+1) + \mathbf{d}_{l,k_{1}} \label{defin_d_l_k1_bold}
\end{equation}
for all $1 \leq l \leq \mathbf{D}$. Multiplying the equation (\ref{SCP_bf}) by $T^{k_{1}+1}$ and using
(\ref{expand_op_diff_bf}), (\ref{defin_d_l_k1_bold}) we can rewrite the equation (\ref{SCP_bf}) in the form
\begin{multline}
\mathbf{Q}(im)( T^{k_{1}+1}\partial_{T}F(T,m,\epsilon) ) \\
= \sum_{l=1}^{\mathbf{D}} \mathbf{R}_{l}(im)\left( \epsilon^{\mathbf{\Delta}_{l} - \mathbf{d}_{l} +
\boldsymbol{\delta}_{l} - 1}T^{\mathbf{d}_{l,k_{1}}}(T^{k_{1}+1}\partial_{T})^{\boldsymbol{\delta}_l}
F(T,m,\epsilon) \right.
\\+ \sum_{1 \leq p \leq \boldsymbol{\delta}_{l}-1} A_{\boldsymbol{\delta}_{l},p}
\left. \epsilon^{\mathbf{\Delta}_{l}-\mathbf{d}_{l}+\boldsymbol{\delta}_{l}-1} T^{k_{1}(\boldsymbol{\delta}_{l}-p)
+ \mathbf{d}_{l,k_{1}}}(T^{k_{1}+1}\partial_{T})^{p}F(T,m,\epsilon) \right)\\
+ \epsilon^{-1}T^{k_{1}+1}
\frac{\mathbf{c}_{0}(\epsilon)}{(2\pi)^{1/2}}\int_{-\infty}^{+\infty}
\mathbf{C}_{0}(T,m-m_{1},\epsilon) \mathbf{R}_{0}(im_{1})F(T,m_{1},\epsilon) dm_{1}\\
+ \epsilon^{-1}T^{k_{1}+1}
\frac{\mathbf{c}_{0,0}(\epsilon)}{(2\pi)^{1/2}}\int_{-\infty}^{+\infty} \mathbf{C}_{0,0}(m-m_{1},\epsilon)
\mathbf{R}_{0}(im_{1})F(T,m_{1},\epsilon) dm_{1}
+ \epsilon^{-1}\mathbf{c}_{\mathbf{F}}(\epsilon)T^{k_{1}+1}\mathbf{F}(T,m,\epsilon)
\label{SCP_irregular_bf}
\end{multline}
As above, we denote $\psi_{k_1}(\tau,m,\epsilon)$ the formal $m_{k_1}-$Borel transform of $F(T,m,\epsilon)$ w.r.t $T$ and
$\boldsymbol{\varphi}_{k_1}(\tau,m,\epsilon)$ the formal $m_{k_1}-$Borel transform of
$\mathbf{C}_{0}(T,m,\epsilon)$ with respect to $T$ and
$\boldsymbol{\psi}_{k_1}(\tau,m,\epsilon)$ the formal $m_{k_1}-$Borel transform of $\mathbf{F}(T,m,\epsilon)$ w.r.t $T$,
\begin{multline*}
\psi_{k_1}(\tau,m,\epsilon) = \sum_{n \geq 1} F_{n}(m,\epsilon) \frac{\tau^n}{\Gamma(\frac{n}{k_{1}})} \ \ , \ \
\boldsymbol{\varphi}_{k_1}(\tau,m,\epsilon) = \sum_{n \geq 1} \mathbf{C}_{0,n}(m,\epsilon) \frac{\tau^n}{\Gamma(\frac{n}{k_1})},\\
\boldsymbol{\psi}_{k_1}(\tau,m,\epsilon) = \sum_{n \geq 1} \mathbf{F}_{n}(m,\epsilon) \frac{\tau^n}{\Gamma(\frac{n}{k_1})}
\end{multline*}
Using (\ref{norm_beta_mu_F_n_bf}) we get that
$\boldsymbol{\varphi}_{k_1}(\tau,m,\epsilon) \in F_{(\nu,\beta,\mu,k_{1},k_{1})}^{\mathfrak{d}_{p}}$ and
$\boldsymbol{\psi}_{k_1}(\tau,m,\epsilon) \in F_{(\nu,\beta,\mu,k_{1},k_{1})}^{\mathfrak{d}_{p}}$, for
all $\epsilon \in D(0,\epsilon_{0})$, for all the unbounded sectors $U_{\mathfrak{d}_{p}}$ centered at 0 and bisecting direction
$\mathfrak{d}_{p} \in \mathbb{R}$ introduced in Definition 8, for some $\nu>0$. Indeed, we have that
\begin{multline}
||\boldsymbol{\varphi}_{k_1}(\tau,m,\epsilon)||_{(\nu,\beta,\mu,k_{1},k_{1})} \leq \sum_{n \geq 1}
||\mathbf{C}_{0,n}(m,\epsilon)||_{(\beta,\mu)} (\sup_{\tau \in \bar{D}(0,\rho) \cup U_{\mathfrak{d}_{p}}}
\frac{1 + |\tau|^{2k_{1}}}{|\tau|} \exp(-\nu |\tau|^{k_{1}})
\frac{|\tau|^n}{\Gamma(\frac{n}{k_{1}})}),\\
||\boldsymbol{\psi}_{k_1}(\tau,m,\epsilon)||_{(\nu,\beta,\mu,k_{1},k_{1})} \leq \sum_{n \geq 1}
||\mathbf{F}_{n}(m,\epsilon)||_{(\beta,\mu)} (\sup_{\tau \in \bar{D}(0,\rho) \cup U_{\mathfrak{d}_{p}}}
\frac{1 + |\tau|^{2k_{1}}}{|\tau|} \exp(-\nu |\tau|^{k_1})
\frac{|\tau|^n}{\Gamma(\frac{n}{k_{1}})}) \label{maj_norm_psi_k_1_bf}
\end{multline}
By using the classical estimates (\ref{x_m_exp_x<}) and the Stirling formula
$\Gamma(n/k_{1}) \sim (2\pi)^{1/2}(n/k_{1})^{\frac{n}{k_1}-\frac{1}{2}}e^{-n/k_{1}}$ as $n$ tends to $+\infty$, we get two constants
$\mathbf{A}_{1},\mathbf{A}_{2}>0$ depending on $\nu,k_{1}$ such that
\begin{multline}
\sup_{\tau \in \bar{D}(0,\rho) \cup U_{\mathfrak{d}_{p}}}
\frac{1 + |\tau|^{2k_{1}}}{|\tau|} \exp(-\nu |\tau|^{k_1})
\frac{|\tau|^n}{\Gamma(\frac{n}{k_{1}})} 
\leq \sup_{x \geq 0} (1+x^{2})x^{\frac{n-1}{k_{1}}}
\frac{e^{-\nu x}}{\Gamma(\frac{n}{k_{1}})}\\
\leq \left( (\frac{n-1}{\nu k_{1}})^{\frac{n-1}{k_{1}}}
e^{-\frac{n-1}{k_{1}}} +
( \frac{n-1}{\nu k_{1}} + \frac{2}{\nu})^{\frac{n-1}{k_{1}}+2} e^{-(\frac{n-1}{k_{1}}+2)} \right)/ \Gamma(n/k_{1})
\leq \mathbf{A}_{1}(\mathbf{A}_{2})^{n} \label{sup_Stirling_bf}
\end{multline}
for all $n \geq 1$, all $\epsilon \in D(0,\epsilon_{0})$. Therefore if $\mathbf{A}_{2} < \mathbf{T}_{0}$ holds, we get the estimates
\begin{multline}
||\boldsymbol{\varphi}_{k_1}(\tau,m,\epsilon)||_{(\nu,\beta,\mu,k_{1},k_{1})} \leq
\mathbf{A}_{1} \sum_{n \geq 1} ||\mathbf{C}_{0,n}(m,\epsilon)||_{(\beta,\mu)}
(\mathbf{A}_{2})^{n} \leq \frac{\mathbf{A}_{1}\mathbf{A}_{2}\mathbf{K}_{0}}{\mathbf{T}_0}
\frac{1}{1 - \frac{\mathbf{A}_{2}}{\mathbf{T}_0}},\\
||\boldsymbol{\psi}_{k_1}(\tau,m,\epsilon)||_{(\nu,\beta,\mu,k_{1},k_{1})} \leq \mathbf{A}_{1} \sum_{n \geq 1}
||\mathbf{F}_{n}(m,\epsilon)||_{(\beta,\mu)}
(\mathbf{A}_{2})^{n} \leq \frac{\mathbf{A}_{1}\mathbf{A}_{2}\mathbf{K}_{0}}{\mathbf{T}_0}
\frac{1}{ 1 - \frac{\mathbf{A}_{2}}{\mathbf{T}_0}}
\label{norm_F_varphi_k_psi_k_epsilon_0_bf}
\end{multline}
for all $\epsilon \in D(0,\epsilon_{0})$.

Observe that $\mathbf{d}_{\mathbf{D},k_{1}}=0$. Using the computation rules for the formal $m_{k_1}-$Borel transform in
Proposition 8, we deduce the following equation satisfied by $\psi_{k_1}(\tau,m,\epsilon)$,
\begin{multline}
\mathbf{Q}(im)( k_{1} \tau^{k_1} \psi_{k_1}(\tau,m,\epsilon) )
= \mathbf{R}_{\mathbf{D}}(im) \left( k_{1}^{\boldsymbol{\delta}_{\mathbf{D}}}
\tau^{\boldsymbol{\delta}_{\mathbf{D}}k_{1}} \psi_{k_1}(\tau,m,\epsilon) \right. \\
+ \sum_{1 \leq p \leq \boldsymbol{\delta}_{\mathbf{D}}-1} A_{\boldsymbol{\delta}_{\mathbf{D}},p}
\frac{\tau^{k_{1}}}{\Gamma(\boldsymbol{\delta}_{\mathbf{D}}-p)}\int_{0}^{\tau^{k_1}}
(\tau^{k_1}-s)^{\boldsymbol{\delta}_{\mathbf{D}}-p-1}
\left. (k_{1}^{p} s^{p} \psi_{k_1}(s^{1/k_{1}},m,\epsilon)) \frac{ds}{s} \right)\\
+ \sum_{l=1}^{\mathbf{D}-1} \mathbf{R}_{l}(im)
\left( \epsilon^{\mathbf{\Delta}_{l}-\mathbf{d}_{l}+\boldsymbol{\delta}_{l}-1}
\frac{\tau^{k_1}}{\Gamma( \frac{\mathbf{d}_{l,k_{1}}}{k_1} )} \right.
\int_{0}^{\tau^{k_1}} (\tau^{k_1}-s)^{\frac{\mathbf{d}_{l,k_{1}}}{k_1}-1}(k_{1}^{\boldsymbol{\delta}_l}s^{\boldsymbol{\delta}_l}
\psi_{k_1}(s^{1/k_{1}},m,\epsilon)) \frac{ds}{s}\\
+ \sum_{1 \leq p \leq \boldsymbol{\delta}_{l}-1} A_{\boldsymbol{\delta}_{l},p}
\epsilon^{\mathbf{\Delta}_{l}-\mathbf{d}_{l}+\boldsymbol{\delta}_{l}-1}
\frac{\tau^{k_{1}}}{\Gamma( \frac{\mathbf{d}_{l,k_{1}}}{k_1} + \boldsymbol{\delta}_{l}-p)} \int_{0}^{\tau^{k_1}}
\left. (\tau^{k_1}-s)^{\frac{\mathbf{d}_{l,k_{1}}}{k_1}+\boldsymbol{\delta}_{l}-p-1}(k_{1}^{p}s^{p}\psi_{k_1}(s^{1/k_1},m,\epsilon))
\frac{ds}{s} \right)\\
 + \epsilon^{-1}
\frac{\tau^{k_1}}{\Gamma(1 + \frac{1}{k_1})} \int_{0}^{\tau^{k_1}}
(\tau^{k_1}-s)^{1/k_1}\\
\times \left( \frac{\mathbf{c}_{0}(\epsilon)}{(2\pi)^{1/2}} s\int_{0}^{s} \int_{-\infty}^{+\infty} \right.
\left. \boldsymbol{\varphi}_{k_1}((s-x)^{1/k_1},m-m_{1},\epsilon) \mathbf{R}_{0}(im_{1})
\psi_{k_1}(x^{1/k_{1}},m_{1},\epsilon) \frac{1}{(s-x)x}
dxdm_{1} \right) \frac{ds}{s}\\
+ \epsilon^{-1}\frac{\tau^{k_1}}{\Gamma(1 + \frac{1}{k_1})} \int_{0}^{\tau^{k_1}}
(\tau^{k_1}-s)^{1/k_1} \frac{\mathbf{c}_{0,0}(\epsilon)}{(2\pi)^{1/2}}
( \int_{-\infty}^{+\infty} \mathbf{C}_{0,0}(m-m_{1},\epsilon)\mathbf{R}_{0}(im_{1})
\psi_{k_1}(s^{1/k_1},m_{1},\epsilon) dm_{1} )\frac{ds}{s}\\
+ \epsilon^{-1} \mathbf{c}_{\mathbf{F}}(\epsilon) \frac{\tau^{k_1}}{\Gamma(1 + \frac{1}{k_1})}\int_{0}^{\tau^{k_{1}}}
(\tau^{k_1}-s)^{1/k_1} \boldsymbol{\psi}_{k_1}(s^{1/k_{1}},m,\epsilon) \frac{ds}{s}. \label{k1_Borel_equation_analytic_bf}
\end{multline}
We make the additional assumption that there exists an unbounded sector
$$ S_{\mathbf{Q},\mathbf{R}_{\mathbf{D}}} = \{ z \in \mathbb{C} /
|z| \geq r_{\mathbf{Q},\mathbf{R}_{\mathbf{D}}} \ \ , \ \ |\mathrm{arg}(z) - d_{\mathbf{Q},\mathbf{R}_{\mathbf{D}}}|
\leq \boldsymbol{\eta}_{\mathbf{Q},\mathbf{R}_{\mathbf{D}}} \} $$
with direction $d_{\mathbf{Q},\mathbf{R}_{\mathbf{D}}} \in \mathbb{R}$, aperture
$\eta_{\mathbf{Q},\mathbf{R}_{\mathbf{D}}}>0$ for some radius $r_{\mathbf{Q},\mathbf{R}_{\mathbf{D}}}>0$
such that
\begin{equation}
\frac{\mathbf{Q}(im)}{\mathbf{R}_{\mathbf{D}}(im)} \in S_{\mathbf{Q},\mathbf{R}_{\mathbf{D}}} \label{quotient_Q_RD_in_S_bf}
\end{equation} 
for all $m \in \mathbb{R}$. We factorize the polynomial $\mathbf{P}_{m}(\tau) = \mathbf{Q}(im)k_{1} -
\mathbf{R}_{\mathbf{D}}(im)k_{1}^{\boldsymbol{\delta}_{\mathbf{D}}}
\tau^{(\boldsymbol{\delta}_{\mathbf{D}}-1)k_{1}}$ in the form
\begin{equation}
\mathbf{P}_{m}(\tau) = -\mathbf{R}_{\mathbf{D}}(im)k_{1}^{\boldsymbol{\delta}_{\mathbf{D}}}
\Pi_{l=0}^{(\boldsymbol{\delta}_{\mathbf{D}}-1)k_{1}-1} (\tau - \mathbf{q}_{l}(m)) \label{factor_P_m_bf}
\end{equation}
where
\begin{multline}
\mathbf{q}_{l}(m) = (\frac{|\mathbf{Q}(im)|}{|\mathbf{R}_{\mathbf{D}}(im)|
k_{1}^{\boldsymbol{\delta}_{\mathbf{D}}-1}})^{\frac{1}{(\boldsymbol{\delta}_{\mathbf{D}}-1)k_{1}}}\\
\times \exp( \sqrt{-1}( \mathrm{arg}( \frac{\mathbf{Q}(im)}{\mathbf{R}_{\mathbf{D}}(im)
k_{1}^{\boldsymbol{\delta}_{\mathbf{D}}-1}}) \frac{1}{(\boldsymbol{\delta}_{\mathbf{D}}-1)k_{1}} +
\frac{2\pi l}{(\boldsymbol{\delta}_{\mathbf{D}}-1)k_{1}} ) ) \label{defin_roots_bf}
\end{multline}
for all $0 \leq l \leq (\boldsymbol{\delta}_{\mathbf{D}}-1)k_{1}-1$, all $m \in \mathbb{R}$.

We choose the family of unbounded sectors $U_{\mathfrak{d}_{p}}$ centered at 0, a small closed disc $\bar{D}(0,\rho)$
(introduced in Definition 8) and we prescribe the sector
$S_{\mathbf{Q},\mathbf{R}_{\mathbf{D}}}$ in such a way that the following conditions hold.\medskip

\noindent 1) There exists a constant $\mathbf{M}_{1}>0$ such that
\begin{equation}
|\tau - \mathbf{q}_{l}(m)| \geq \mathbf{M}_{1}(1 + |\tau|) \label{root_cond_1_bf}
\end{equation}
for all $0 \leq l \leq (\boldsymbol{\delta}_{\mathbf{D}}-1)k_{1}-1$, all $m \in \mathbb{R}$,
all $\tau \in U_{\mathfrak{d}_{p}} \cup \bar{D}(0,\rho)$, for all $0 \leq p \leq \varsigma -1$.\medskip

\noindent 2) There exists a constant $\mathbf{M}_{2}>0$ such that
\begin{equation}
|\tau - \mathbf{q}_{l_0}(m)| \geq \mathbf{M}_{2}|\mathbf{q}_{l_0}(m)| \label{root_cond_2_bf}
\end{equation}
for some $l_{0} \in \{0,\ldots,(\boldsymbol{\delta}_{\mathbf{D}}-1)k_{1}-1 \}$, all $m \in \mathbb{R}$, all
$\tau \in U_{\mathfrak{d}_{p}} \cup \bar{D}(0,\rho)$, for all $0 \leq p \leq \varsigma -1$.\medskip

By construction
of the roots (\ref{defin_roots_bf}) in the factorization (\ref{factor_P_m_bf}) and using the lower bound estimates
(\ref{root_cond_1_bf}), (\ref{root_cond_2_bf}), we get a constant $C_{\mathbf{P}}>0$ such that
\begin{multline}
|\mathbf{P}_{m}(\tau)| \geq \mathbf{M}_{1}^{(\boldsymbol{\delta}_{\mathbf{D}}-1)k_{1}-1}\mathbf{M}_{2}
|\mathbf{R}_{\mathbf{D}}(im)|k_{1}^{\boldsymbol{\delta}_{\mathbf{D}}}
(\frac{|\mathbf{Q}(im)|}{|\mathbf{R}_{\mathbf{D}}(im)|k_{1}^{\boldsymbol{\delta}_{\mathbf{D}}-1}})^{
\frac{1}{(\boldsymbol{\delta}_{\mathbf{D}}-1)k_{1}}} (1+|\tau|)^{(\boldsymbol{\delta}_{\mathbf{D}}-1)k_{1}-1}\\
\geq \mathbf{M}_{1}^{(\boldsymbol{\delta}_{\mathbf{D}}-1)k_{1}-1}\mathbf{M}_{2}
\frac{k_{1}^{\boldsymbol{\delta}_{\mathbf{D}}}}{(k_{1}^{\boldsymbol{\delta}_{\mathbf{D}}-1})^{
\frac{1}{(\boldsymbol{\delta}_{\mathbf{D}}-1)k_{1}}} }
(r_{\mathbf{Q},\mathbf{R}_{\mathbf{D}}})^{\frac{1}{(\boldsymbol{\delta}_{\mathbf{D}}-1)k_{1}}}
|\mathbf{R}_{\mathbf{D}}(im)| \\
\times (\min_{x \geq 0}
\frac{(1+x)^{(\boldsymbol{\delta}_{\mathbf{D}}-1)k_{1}-1}}{(1+x^{k_1})^{(\boldsymbol{\delta}_{\mathbf{D}}-1)
- \frac{1}{k_1}}}) (1 + |\tau|^{k_1})^{(\boldsymbol{\delta}_{\mathbf{D}}-1)-\frac{1}{k_1}}\\
= C_{\mathbf{P}} (r_{\mathbf{Q},\mathbf{R}_{\mathbf{D}}})^{
\frac{1}{(\boldsymbol{\delta}_{\mathbf{D}}-1)k_{1}}} |\mathbf{R}_{\mathbf{D}}(im)|
(1+|\tau|^{k_1})^{(\boldsymbol{\delta}_{\mathbf{D}}-1)-\frac{1}{k_1}}
\label{low_bounds_P_m_bf}
\end{multline}
for all $\tau \in U_{\mathfrak{d}_{p}} \cup \bar{D}(0,\rho)$, all $m \in \mathbb{R}$, all $0 \leq p \leq \varsigma - 1$.

In the next proposition, we give sufficient conditions under which the equation (\ref{k1_Borel_equation_analytic_bf}) has a solution
$\psi_{k_1}^{\mathfrak{d}_{p}}(\tau,m,\epsilon)$ in the Banach space
$F_{(\nu,\beta,\mu,k_{1},k_{1})}^{\mathfrak{d}_{p}}$ where $\beta,\mu$ are defined above.
\begin{prop} Under the assumption that
\begin{equation}
\boldsymbol{\delta}_{\mathbf{D}} \geq \boldsymbol{\delta}_{l} + \frac{1}{k_{1}}
\label{constraints_k1_Borel_equation_bf}
\end{equation}
for all $1 \leq l \leq \mathbf{D}-1$, there exist a radius $r_{\mathbf{Q},\mathbf{R}_{\mathbf{D}}}>0$,
a constant $\boldsymbol{\upsilon}>0$ and constants
$\boldsymbol{\varsigma}_{0,0},\boldsymbol{\varsigma}_{0},\boldsymbol{\varsigma}_{1},\boldsymbol{\varsigma}_{1,0},
\boldsymbol{\varsigma}_{F},\boldsymbol{\varsigma}_{2}>0$ (depending on
$k_{1},C_{\mathbf{P}},\mu,\nu,\epsilon_{0},\mathbf{R}_{l},\mathbf{\Delta}_{l},\boldsymbol{\delta}_{l},\mathbf{d}_{l}$ for
$0 \leq l \leq \mathbf{D}$) such that if
\begin{multline}
\sup_{\epsilon \in D(0,\epsilon_{0})}|\frac{\mathbf{c}_{0}(\epsilon)}{\epsilon}| \leq \boldsymbol{\varsigma}_{1,0} \ \ , \ \
||\boldsymbol{\varphi}_{k_1}(\tau,m,\epsilon)||_{(\nu,\beta,\mu,k_{1},k_{1})} \leq \boldsymbol{\varsigma}_{1},\\
\sup_{\epsilon \in D(0,\epsilon_{0})}|\frac{\mathbf{c}_{0,0}(\epsilon)}{\epsilon}| \leq \boldsymbol{\varsigma}_{0,0} \ \ , \ \
||\mathbf{C}_{0,0}(m,\epsilon)||_{(\beta,\mu)} \leq \boldsymbol{\varsigma}_{0},\\
\sup_{\epsilon \in D(0,\epsilon_{0})}|\frac{\mathbf{c}_{\mathbf{F}}(\epsilon)}{\epsilon}| \leq \boldsymbol{\varsigma}_{F} \ \ , \ \
||\boldsymbol{\psi}_{k_1}(\tau,m,\epsilon)||_{(\nu,\beta,\mu,k_{1},k_{1})} \leq \boldsymbol{\varsigma}_{2}
\label{norm_F_varphi_k1_psi_k_1_small_bf}
\end{multline}
for all $\epsilon \in D(0,\epsilon_{0})$, the equation (\ref{k1_Borel_equation_analytic_bf}) has a unique solution
$\psi_{k_1}^{\mathfrak{d}_{p}}(\tau,m,\epsilon)$ in the space $F_{(\nu,\beta,\mu,k_{1},k_{1})}^{\mathfrak{d}_{p}}$
with the property that
$||\psi_{k_1}^{\mathfrak{d}_{p}}(\tau,m,\epsilon)||_{(\nu,\beta,\mu,k_{1},k_{1})} \leq \boldsymbol{\upsilon}$, for all
$\epsilon \in D(0,\epsilon_{0})$, where $\beta,\mu>0$ are defined above,
for any unbounded sector $U_{\mathfrak{d}_{p}}$ and disc $\bar{D}(0,\rho)$ that satisfy the constraints
(\ref{root_cond_1_bf}), (\ref{root_cond_2_bf}), for all $0 \leq p \leq \varsigma - 1$.
\end{prop}

\begin{proof} The proof will follow the same lines of arguments as in Proposition 14. We give a thorough treatment of it
only for the sake of completeness.

We begin with a lemma which provides appropriate conditions in order to apply a fixed point theorem.
\begin{lemma} One can choose the constant $r_{\mathbf{Q},\mathbf{R}_{\mathbf{D}}}>0$, a constant
$\boldsymbol{\upsilon}$ small enough and constants
$\boldsymbol{\varsigma}_{0,0},\boldsymbol{\varsigma}_{0},\boldsymbol{\varsigma}_{1},\boldsymbol{\varsigma}_{1,0},
\boldsymbol{\varsigma}_{F},\boldsymbol{\varsigma}_{2}>0$
(depending on
$k_{1},C_{\mathbf{P}},\mu,\nu,\epsilon_{0},\mathbf{R}_{l},\mathbf{\Delta}_{l},\boldsymbol{\delta}_{l},\mathbf{d}_{l}$ for
$0 \leq l \leq \mathbf{D}$) such that if (\ref{norm_F_varphi_k1_psi_k_1_small_bf}) holds
for all $\epsilon \in D(0,\epsilon_{0})$, the map $\mathbf{H}_{\epsilon}^{k_1}$ defined by
\begin{multline}
\mathbf{H}_{\epsilon}^{k_1}(w(\tau,m))\\ =
 \frac{\mathbf{R}_{\mathbf{D}}(im)}{\mathbf{P}_{m}(\tau)} \sum_{1 \leq p \leq \boldsymbol{\delta}_{\mathbf{D}}-1}
 A_{\boldsymbol{\delta}_{\mathbf{D}},p}
\frac{1}{\Gamma(\boldsymbol{\delta}_{\mathbf{D}}-p)}\int_{0}^{\tau^{k_1}} (\tau^{k_1}-s)^{\boldsymbol{\delta}_{\mathbf{D}}-p-1}
 (k_{1}^{p} s^{p} w(s^{1/k_1},m)) \frac{ds}{s}\\
+ \sum_{l=1}^{\mathbf{D}-1} \frac{\mathbf{R}_{l}(im)}{\mathbf{P}_{m}(\tau)}
\left( \epsilon^{\mathbf{\Delta}_{l}-\mathbf{d}_{l}+\boldsymbol{\delta}_{l}-1}
\frac{1}{\Gamma( \frac{\mathbf{d}_{l,k_{1}}}{k_1} )} \right.
\int_{0}^{\tau^{k_1}} (\tau^{k_1}-s)^{\frac{\mathbf{d}_{l,k_{1}}}{k_1}-1}({k_1}^{\boldsymbol{\delta}_l}s^{\boldsymbol{\delta}_l}
w(s^{1/k_1},m)) \frac{ds}{s}\\
+ \sum_{1 \leq p \leq \boldsymbol{\delta}_{l}-1} A_{\boldsymbol{\delta}_{l},p}\epsilon^{\mathbf{\Delta}_{l}-\mathbf{d}_{l}
+\boldsymbol{\delta}_{l}-1}
\frac{1}{\Gamma( \frac{\mathbf{d}_{l,k_{1}}}{k_1} + \boldsymbol{\delta}_{l}-p)} \int_{0}^{\tau^{k_1}}
\left. (\tau^{k_1}-s)^{\frac{\mathbf{d}_{l,k_{1}}}{k_1}+\boldsymbol{\delta}_{l}-p-1}({k_1}^{p}s^{p}
w(s^{1/k_1},m)) \frac{ds}{s} \right)\\
+ \epsilon^{-1}
\frac{1}{\mathbf{P}_{m}(\tau)\Gamma(1 + \frac{1}{k_1})} \int_{0}^{\tau^{k_1}}
(\tau^{k_1}-s)^{1/k_1}\\
\times \left( \frac{\mathbf{c}_{0}(\epsilon)}{(2\pi)^{1/2}} s\int_{0}^{s} \int_{-\infty}^{+\infty} \right.
\left. \boldsymbol{\varphi}_{k_1}((s-x)^{1/k_1},m-m_{1},\epsilon) \mathbf{R}_{0}(im_{1}) w(x^{1/k_1},m_{1}) \frac{1}{(s-x)x}
dxdm_{1} \right) \frac{ds}{s}\\
+ \epsilon^{-1}\frac{1}{\mathbf{P}_{m}(\tau)\Gamma(1 + \frac{1}{k_1})} \int_{0}^{\tau^{k_1}}
(\tau^{k_1}-s)^{1/k_1} \frac{\mathbf{c}_{0,0}(\epsilon)}{(2\pi)^{1/2}} ( \int_{-\infty}^{+\infty}
\mathbf{C}_{0,0}(m-m_{1},\epsilon)\mathbf{R}_{0}(im_{1})
w(s^{1/k_1},m_{1}) dm_{1} )\frac{ds}{s}\\
+ \epsilon^{-1}\mathbf{c}_{\mathbf{F}}(\epsilon) \frac{1}{\mathbf{P}_{m}(\tau)\Gamma(1 + \frac{1}{k_1})}\int_{0}^{\tau^{k_1}}
(\tau^{k_1}-s)^{1/k_1} \boldsymbol{\psi}_{k_1}(s^{1/k_1},m,\epsilon) \frac{ds}{s}
\end{multline}
satisfies the next properties.\\
{\bf i)} The following inclusion holds
\begin{equation}
\mathbf{H}_{\epsilon}^{k_1}(\bar{B}(0,\boldsymbol{\upsilon})) \subset \bar{B}(0,\boldsymbol{\upsilon}) \label{H_k2_inclusion_bf}
\end{equation}
where $\bar{B}(0,\boldsymbol{\upsilon})$ is the closed ball of radius $\boldsymbol{\upsilon}>0$ centered at 0 in
$F_{(\nu,\beta,\mu,k_{1},k_{1})}^{\mathfrak{d}_{p}}$,
for all $\epsilon \in D(0,\epsilon_{0})$.\\
{\bf ii)} We have
\begin{equation}
|| \mathbf{H}_{\epsilon}^{k_1}(w_{1}) - \mathbf{H}_{\epsilon}^{k_1}(w_{2})||_{(\nu,\beta,\mu,k_{1},k_{1})}
\leq \frac{1}{2} ||w_{1} - w_{2}||_{(\nu,\beta,\mu,k_{1},k_{1})}
\label{H_k2_shrink_bf}
\end{equation}
for all $w_{1},w_{2} \in \bar{B}(0,\boldsymbol{\upsilon})$, for all $\epsilon \in D(0,\epsilon_{0})$.
\end{lemma}
\begin{proof} We first check the property (\ref{H_k2_inclusion_bf}). Let $\epsilon \in D(0,\epsilon_{0})$ and
$w(\tau,m)$ be in $F_{(\nu,\beta,\mu,k_{1},k_{1})}^{\mathfrak{d}_{p}}$. We take
$\boldsymbol{\varsigma}_{0,0},\boldsymbol{\varsigma}_{0},\boldsymbol{\varsigma}_{1},\boldsymbol{\varsigma}_{1,0},
\boldsymbol{\varsigma}_{F},\boldsymbol{\varsigma}_{2}>0$ and
$\boldsymbol{\upsilon} > 0$ such that $||w(\tau,m)||_{(\nu,\beta,\mu,k_{1},k_{1})} \leq \boldsymbol{\upsilon}$ and
(\ref{norm_F_varphi_k1_psi_k_1_small_bf}) hold for all $\epsilon \in D(0,\epsilon_{0})$.

Due to (\ref{low_bounds_P_m_bf}), for $1 \leq p \leq \boldsymbol{\delta}_{\mathbf{D}}-1$ and by means of Proposition 5
(take the particular case
$S_{d}^{b} = D(0,\rho)$, $S_{d}=U_{\mathfrak{d}_{p}}$ and $k=k_{1}$), we deduce
\begin{multline}
||\frac{\mathbf{R}_{\mathbf{D}}(im)}{\mathbf{P}_{m}(\tau)}A_{\boldsymbol{\delta}_{\mathbf{D}},p}
\frac{1}{\Gamma(\boldsymbol{\delta}_{\mathbf{D}}-p)}\int_{0}^{\tau^{k_1}} (\tau^{k_1}-s)^{\boldsymbol{\delta}_{\mathbf{D}}-p-1}
 (k_{1}^{p} s^{p} w(s^{1/k_1},m)) \frac{ds}{s}||_{(\nu,\beta,\mu,k_{1},k_{1})}\\
 \leq \frac{|A_{\boldsymbol{\delta}_{\mathbf{D}},p}|k_{1}^{p}C_{6}}{\Gamma(\boldsymbol{\delta}_{\mathbf{D}}-p)
 C_{\mathbf{P}}(r_{\mathbf{Q},\mathbf{R}_{\mathbf{D}}})^{\frac{1}{(\boldsymbol{\delta}_{\mathbf{D}}-1)k_{1}}}}
 ||w(\tau,m)||_{(\nu,\beta,\mu,k_{1},k_{1})}\\
 \leq \frac{|A_{\boldsymbol{\delta}_{\mathbf{D}},p}|k_{1}^{p}C_{6}}{\Gamma(\boldsymbol{\delta}_{\mathbf{D}}-p)
 C_{\mathbf{P}}(r_{\mathbf{Q},\mathbf{R}_{\mathbf{D}}})^{\frac{1}{(\boldsymbol{\delta}_{\mathbf{D}}-1)k_{1}}}}
\boldsymbol{\upsilon} \label{fix_point_norm_2_estim_2_bf}
\end{multline}
For all $1 \leq l \leq \mathbf{D}-1$, by means of Proposition 5 as above and under the constraint
(\ref{constraints_k1_Borel_equation_bf}), we have
\begin{multline}
||\frac{\mathbf{R}_{l}(im)}{\mathbf{P}_{m}(\tau)} \epsilon^{\mathbf{\Delta}_{l}-\mathbf{d}_{l}+\boldsymbol{\delta}_{l}-1}
\frac{1}{\Gamma( \frac{\mathbf{d}_{l,k_{1}}}{k_1} )}
\int_{0}^{\tau^{k_1}} (\tau^{k_1}-s)^{\frac{\mathbf{d}_{l,k_{1}}}{k_1}-1}({k_1}^{\boldsymbol{\delta}_l}s^{\boldsymbol{\delta}_l}
w(s^{1/k_1},m)) \frac{ds}{s}||_{(\nu,\beta,\mu,k_{1},k_{1})}\\
\leq \frac{k_{1}^{\boldsymbol{\delta}_l}C_{6}}{\Gamma( \frac{\mathbf{d}_{l,k_{1}}}{k_1}) C_{\mathbf{P}}
(r_{\mathbf{Q},\mathbf{R}_{\mathbf{D}}})^{\frac{1}{(\boldsymbol{\delta}_{\mathbf{D}}-1)k_{1}}} }
|\epsilon|^{\mathbf{\Delta}_{l} - \mathbf{d}_{l} + \boldsymbol{\delta}_{l} -1}
\sup_{m \in \mathbb{R}} \frac{|\mathbf{R}_{l}(im)|}{|\mathbf{R}_{\mathbf{D}}(im)|} ||w(\tau,m)||_{(\nu,\beta,\mu,k_{1},k_{1})}\\
\leq \frac{k_{1}^{\boldsymbol{\delta}_l}C_{6}}{\Gamma( \frac{\mathbf{d}_{l,k_{1}}}{k_1})
C_{\mathbf{P}}(r_{\mathbf{Q},\mathbf{R}_{\mathbf{D}}})^{\frac{1}{(\boldsymbol{\delta}_{\mathbf{D}}-1)k_{1}}} }
\epsilon_{0}^{\mathbf{\Delta}_{l} - \mathbf{d}_{l} + \boldsymbol{\delta}_{l} -1 }
\sup_{m \in \mathbb{R}} \frac{|\mathbf{R}_{l}(im)|}{|\mathbf{R}_{\mathbf{D}}(im)|} \boldsymbol{\upsilon}
\label{fix_point_norm_2_estim_3_bf}
\end{multline}
and
\begin{multline}
|| \frac{\mathbf{R}_{l}(im)}{\mathbf{P}_{m}(\tau)} \frac{ A_{\boldsymbol{\delta}_{l},p}
\epsilon^{\mathbf{\Delta}_{l}-\mathbf{d}_{l}+\boldsymbol{\delta}_{l}-1} }
{\Gamma( \frac{\mathbf{d}_{l,k_{1}}}{k_1} + \boldsymbol{\delta}_{l}-p)} \int_{0}^{\tau^{k_1}}
(\tau^{k_1}-s)^{\frac{\mathbf{d}_{l,k_{1}}}{k_1}+\boldsymbol{\delta}_{l}-p-1}(k_{1}^{p}s^{p}
w(s^{1/k_1},m)) \frac{ds}{s} ||_{(\nu,\beta,\mu,k_{1},k_{1})}\\
\leq \frac{|A_{\boldsymbol{\delta}_{l},p}|k_{1}^{p}C_{6}}{\Gamma( \frac{\mathbf{d}_{l,k_{1}}}{k_{1}} + \boldsymbol{\delta}_{l} - p)
C_{\mathbf{P}}
(r_{\mathbf{Q},\mathbf{R}_{\mathbf{D}}})^{\frac{1}{(\boldsymbol{\delta}_{\mathbf{D}}-1)k_{1}}} }
|\epsilon|^{\mathbf{\Delta}_{l}-\mathbf{d}_{l}+\boldsymbol{\delta}_{l} -1}\\
\times \sup_{m \in \mathbb{R}} \frac{|\mathbf{R}_{l}(im)|}{|\mathbf{R}_{\mathbf{D}}(im)|}
||w(\tau,m)||_{(\nu,\beta,\mu,k_{1},k_{1})}\\
\leq \frac{|A_{\boldsymbol{\delta}_{l},p}|k_{1}^{p}C_{6}}{\Gamma(
\frac{\mathbf{d}_{l,k_{1}}}{k_{1}} + \boldsymbol{\delta}_{l} - p)
C_{\mathbf{P}}(r_{\mathbf{Q},\mathbf{R}_{\mathbf{D}}})^{
\frac{1}{(\boldsymbol{\delta}_{\mathbf{D}}-1)k_{1}}} } \epsilon_{0}^{\mathbf{\Delta}_{l}-\mathbf{d}_{l}+\boldsymbol{\delta}_{l}-1}
\sup_{m \in \mathbb{R}} \frac{|\mathbf{R}_{l}(im)|}{|\mathbf{R}_{\mathbf{D}}(im)|} \boldsymbol{\upsilon},
\label{fix_point_norm_2_estim_4_bf}
\end{multline}
for all $1 \leq p \leq \boldsymbol{\delta}_{l}-1$. Taking into account Lemma 2 and Proposition 6, we also get that
\begin{multline}
|| \epsilon^{-1}
\frac{\mathbf{c}_{0}(\epsilon)}{\mathbf{P}_{m}(\tau) \Gamma(1 + \frac{1}{k_1})} \int_{0}^{\tau^{k_1}}
(\tau^{k_1}-s)^{1/k_1}\\
\times \left( \frac{1}{(2\pi)^{1/2}} s\int_{0}^{s} \int_{-\infty}^{+\infty} \right.
\boldsymbol{\varphi}_{k_1}((s-x)^{1/k_1},m-m_{1},\epsilon) \\
\times \left. \mathbf{R}_{0}(im_{1}) w(x^{1/k_1},m_{1}) \frac{1}{(s-x)x}
dxdm_{1} \right) \frac{ds}{s} ||_{(\nu,\beta,\mu,k_{1},k_{1})}\\
\leq \frac{\boldsymbol{\varsigma}_{1,0}}{\Gamma(1 + \frac{1}{k_1})(2\pi)^{1/2}} \frac{C_{7}
||\boldsymbol{\varphi}_{k_1}(\tau,m,\epsilon)||_{(\nu,\beta,\mu,k_{1},k_{1})} ||w(\tau,m)||_{(\nu,\beta,\mu,k_{1},k_{1})} }{
C_{\mathbf{P}}(r_{\mathbf{Q},\mathbf{R}_{\mathbf{D}}})^{\frac{1}{(\boldsymbol{\delta}_{\mathbf{D}}-1)k_{1}}} }\\
\leq \frac{\boldsymbol{\varsigma}_{1,0}}{\Gamma(1 + \frac{1}{k_1})(2\pi)^{1/2}} \frac{C_{7}
\boldsymbol{\varsigma}_{1} \boldsymbol{\upsilon}}{
C_{\mathbf{P}}(r_{\mathbf{Q},\mathbf{R}_{\mathbf{D}}})^{\frac{1}{(\boldsymbol{\delta}_{\mathbf{D}}-1)k_{1}}} }
\label{fix_point_norm_2_estim_5_bf}
\end{multline}
Due to Lemma 2 and Proposition 7,
\begin{multline}
|| \epsilon^{-1}\frac{\mathbf{c}_{0,0}(\epsilon)}{\mathbf{P}_{m}(\tau)\Gamma(1 + \frac{1}{k_1})} \int_{0}^{\tau^{k_1}}
(\tau^{k_1}-s)^{1/k_1} \frac{1}{(2\pi)^{1/2}} ( \int_{-\infty}^{+\infty} \mathbf{C}_{0,0}(m-m_{1},\epsilon)\\
\times \mathbf{R}_{0}(im_{1})
w(s^{1/k_1},m_{1}) dm_{1} )\frac{ds}{s} ||_{(\nu,\beta,\mu,k_{1},k_{1})}\\
\leq \frac{\boldsymbol{\varsigma}_{0,0}}{\Gamma(1 + \frac{1}{k_1})(2\pi)^{1/2}} \frac{C_{8} \boldsymbol{\varsigma}_{0}
\boldsymbol{\upsilon} }{
C_{\mathbf{P}}(r_{\mathbf{Q},\mathbf{R}_{\mathbf{D}}})^{\frac{1}{(\boldsymbol{\delta}_{\mathbf{D}}-1)k_{1}}} }
\label{fix_point_norm_2_estim_6_bf}
\end{multline}
holds. Finally, from Lemma 2 and Proposition 5, we get
\begin{multline}
|| \epsilon^{-1} \frac{\mathbf{c}_{\mathbf{F}}(\epsilon)}{\mathbf{P}_{m}(\tau)\Gamma(1 + \frac{1}{k_1})}\int_{0}^{\tau^{k_1}}
(\tau^{k_1}-s)^{1/k_1} \boldsymbol{\psi}_{k_1}(s^{1/k_1},m,\epsilon) \frac{ds}{s} ||_{(\nu,\beta,\mu,k_{1},k_{1})}\\
\leq \frac{\boldsymbol{\varsigma}_{\mathbf{F}}C_{6}}{\Gamma(1 + \frac{1}{k_2})C_{\mathbf{P}}
(r_{\mathbf{Q},\mathbf{R}_{\mathbf{D}}})^{\frac{1}{(\boldsymbol{\delta}_{\mathbf{D}}-1)k_{1}}}
\min_{m \in \mathbb{R}} |\mathbf{R}_{\mathbf{D}}(im)| }
||\boldsymbol{\psi}_{k_1}(\tau,m,\epsilon)||_{(\nu,\beta,\mu,k_{1},k_{1})}\\
\leq \frac{\boldsymbol{\varsigma}_{\mathbf{F}}C_{6}}{\Gamma(1 + \frac{1}{k_1})C_{\mathbf{P}}
(r_{\mathbf{Q},\mathbf{R}_{\mathbf{D}}})^{\frac{1}{(\boldsymbol{\delta}_{\mathbf{D}}-1)k_{1}}}
\min_{m \in \mathbb{R}} |\mathbf{R}_{\mathbf{D}}(im)| } \boldsymbol{\varsigma}_{2}. \label{fix_point_norm_2_estim_7_bf}
\end{multline}
Now, we choose $\boldsymbol{\upsilon}$,
$\boldsymbol{\varsigma}_{0,0},\boldsymbol{\varsigma}_{0},\boldsymbol{\varsigma}_{1},
\boldsymbol{\varsigma}_{1,0},\boldsymbol{\varsigma}_{\mathbf{F}},\boldsymbol{\varsigma}_{2}>0$
and $r_{\mathbf{Q},\mathbf{R}_{\mathbf{D}}}>0$ such that
\begin{multline}
\sum_{1 \leq p \leq \boldsymbol{\delta}_{\mathbf{D}}-1}
\frac{|A_{\boldsymbol{\delta}_{\mathbf{D}},p}|k_{1}^{p}C_{6}}{\Gamma(\boldsymbol{\delta}_{\mathbf{D}}-p)
C_{\mathbf{P}}(r_{\mathbf{Q},\mathbf{R}_{\mathbf{D}}})^{\frac{1}{(\boldsymbol{\delta}_{\mathbf{D}}-1)k_{1}}}}
\boldsymbol{\upsilon}\\
+ \sum_{1 \leq l \leq \mathbf{D}-1}\frac{k_{1}^{\boldsymbol{\delta}_l}C_{6}}{\Gamma( \frac{\mathbf{d}_{l,k_{1}}}{k_1})
C_{\mathbf{P}}(r_{\mathbf{Q},\mathbf{R}_{\mathbf{D}}})^{\frac{1}{(\boldsymbol{\delta}_{\mathbf{D}}-1)k_{1}}} }
\epsilon_{0}^{\mathbf{\Delta}_{l} - \mathbf{d}_{l} + \boldsymbol{\delta}_{l} -1 }
\sup_{m \in \mathbb{R}} \frac{|\mathbf{R}_{l}(im)|}{|\mathbf{R}_{\mathbf{D}}(im)|} \boldsymbol{\upsilon}\\
+ \sum_{1 \leq p \leq \boldsymbol{\delta}_{l}-1}
\frac{|A_{\boldsymbol{\delta}_{l},p}|k_{1}^{p}C_{6}}{
\Gamma( \frac{\mathbf{d}_{l,k_{1}}}{k_{1}} + \boldsymbol{\delta}_{l} - p)
C_{\mathbf{P}}(r_{\mathbf{Q},\mathbf{R}_{\mathbf{D}}})^{\frac{1}{(\boldsymbol{\delta}_{\mathbf{D}}-1)k_{1}}} }
\epsilon_{0}^{\mathbf{\Delta}_{l}-\mathbf{d}_{l}+\boldsymbol{\delta}_{l}-1}
\sup_{m \in \mathbb{R}} \frac{|\mathbf{R}_{l}(im)|}{|\mathbf{R}_{\mathbf{D}}(im)|} \boldsymbol{\upsilon}\\
+ \frac{\boldsymbol{\varsigma}_{1,0}}{\Gamma(1 + \frac{1}{k_1})(2\pi)^{1/2}} \frac{C_{7}
\boldsymbol{\varsigma}_{1} \boldsymbol{\upsilon}}{ C_{\mathbf{P}}(r_{\mathbf{Q},\mathbf{R}_{\mathbf{D}}})^{
\frac{1}{(\boldsymbol{\delta}_{\mathbf{D}}-1)k_{1}}} } +
\frac{\boldsymbol{\varsigma}_{0,0}}{\Gamma(1 + \frac{1}{k_1})(2\pi)^{1/2}} \frac{C_{8} \boldsymbol{\varsigma}_{0}
\boldsymbol{\upsilon} }{
C_{\mathbf{P}}(r_{\mathbf{Q},\mathbf{R}_{\mathbf{D}}})^{\frac{1}{(\boldsymbol{\delta}_{\mathbf{D}}-1)k_{1}}} }\\
+ \frac{\boldsymbol{\varsigma}_{\mathbf{F}}C_{6}}{\Gamma(1 + \frac{1}{k_1})
C_{\mathbf{P}}(r_{\mathbf{Q},\mathbf{R}_{\mathbf{D}}})^{\frac{1}{(\boldsymbol{\delta}_{\mathbf{D}}-1)k_{1}}}
\min_{m \in \mathbb{R}} |\mathbf{R}_{\mathbf{D}}(im)| } \boldsymbol{\varsigma}_{2} \leq \boldsymbol{\upsilon}
\label{fix_point_sum<upsilon_bf}
\end{multline}
Gathering all the norm estimates (\ref{fix_point_norm_2_estim_2_bf}),
(\ref{fix_point_norm_2_estim_3_bf}), (\ref{fix_point_norm_2_estim_4_bf}), (\ref{fix_point_norm_2_estim_5_bf}),
(\ref{fix_point_norm_2_estim_6_bf}), (\ref{fix_point_norm_2_estim_7_bf}) under the constraint
(\ref{fix_point_sum<upsilon_bf}), one gets
(\ref{H_k2_inclusion_bf}).\bigskip

Now, we check the second property (\ref{H_k2_shrink_bf}). Let $w_{1}(\tau,m),w_{2}(\tau,m)$ be in
$F^{\mathfrak{d}_{p}}_{(\nu,\beta,\mu,k_{1},k_{1})}$.
We take $\boldsymbol{\upsilon} > 0$ such that
$$ ||w_{l}(\tau,m)||_{(\nu,\beta,\mu,k_{1},k_{1})} \leq \boldsymbol{\upsilon},$$
for $l=1,2$. 

From the estimates (\ref{fix_point_norm_2_estim_2_bf}), (\ref{fix_point_norm_2_estim_3_bf}), 
(\ref{fix_point_norm_2_estim_4_bf}), (\ref{fix_point_norm_2_estim_5_bf}), (\ref{fix_point_norm_2_estim_6_bf}) and
under the constraint (\ref{constraints_k1_Borel_equation_bf}), we find that for $1 \leq p \leq \boldsymbol{\delta}_{\mathbf{D}}-1$, 
\begin{multline}
||\frac{\mathbf{R}_{\mathbf{D}}(im)}{\mathbf{P}_{m}(\tau)}A_{\boldsymbol{\delta}_{\mathbf{D}},p}
\frac{1}{\Gamma(\boldsymbol{\delta}_{\mathbf{D}}-p)}\int_{0}^{\tau^{k_1}} (\tau^{k_1}-s)^{\boldsymbol{\delta}_{\mathbf{D}}-p-1}\\
\times 
 (k_{1}^{p} s^{p} (w_{1}(s^{1/k_1},m) - w_{2}(s^{1/k_1},m))) \frac{ds}{s}||_{(\nu,\beta,\mu,k_{1},k_{1})}\\
 \leq \frac{|A_{\boldsymbol{\delta}_{\mathbf{D}},p}|k_{1}^{p}C_{6}}{
 \Gamma(\boldsymbol{\delta}_{\mathbf{D}}-p)
 C_{\mathbf{P}}(r_{\mathbf{Q},\mathbf{R}_{\mathbf{D}}})^{\frac{1}{(\boldsymbol{\delta}_{\mathbf{D}}-1)k_{1}}}}
 ||w_{1}(\tau,m) - w_{2}(\tau,m)||_{(\nu,\beta,\mu,k_{1},k_{1})} \label{fix_point_norm_2_estim_2_shrink_bf}
\end{multline}
holds and that for $1 \leq l \leq \mathbf{D}-1$, 
\begin{multline}
||\frac{\mathbf{R}_{l}(im)}{\mathbf{P}_{m}(\tau)} \epsilon^{\mathbf{\Delta}_{l}-\mathbf{d}_{l}+\boldsymbol{\delta}_{l}-1}
\frac{1}{\Gamma( \frac{\mathbf{d}_{l,k_{1}}}{k_1} )}
\int_{0}^{\tau^{k_1}} (\tau^{k_1}-s)^{\frac{\mathbf{d}_{l,k_{1}}}{k_1}-1}\\
\times (k_{1}^{\boldsymbol{\delta}_l}s^{\boldsymbol{\delta}_l}
(w_{1}(s^{1/k_1},m) - w_{2}(s^{1/k_1},m))) \frac{ds}{s}||_{(\nu,\beta,\mu,k_{1},k_{1})}\\
\leq \frac{k_{1}^{\boldsymbol{\delta}_l}C_{6}}{\Gamma( \frac{\mathbf{d}_{l,k_{1}}}{k_1})
C_{\mathbf{P}}(r_{\mathbf{Q},\mathbf{R}_{\mathbf{D}}})^{\frac{1}{(\boldsymbol{\delta}_{\mathbf{D}}-1)k_{1}}} }
\epsilon_{0}^{\mathbf{\Delta}_{l} - \mathbf{d}_{l} + \boldsymbol{\delta}_{l} - 1}\\
\times \sup_{m \in \mathbb{R}} \frac{|\mathbf{R}_{l}(im)|}{|\mathbf{R}_{\mathbf{D}}(im)|}
||w_{1}(\tau,m) - w_{2}(\tau,m)||_{(\nu,\beta,\mu,k_{1},k_{1})}
\label{fix_point_norm_2_estim_3_shrink_bf}
\end{multline}
and
\begin{multline}
|| \frac{\mathbf{R}_{l}(im)}{\mathbf{P}_{m}(\tau)} \frac{ A_{\boldsymbol{\delta}_{l},p}
\epsilon^{\mathbf{\Delta}_{l}-\mathbf{d}_{l}+\boldsymbol{\delta}_{l}-1} }
{\Gamma( \frac{\mathbf{d}_{l,k_{1}}}{k_1} + \boldsymbol{\delta}_{l}-p)} \int_{0}^{\tau^{k_1}}
(\tau^{k_1}-s)^{\frac{\mathbf{d}_{l,k_{1}}}{k_1}+ \boldsymbol{\delta}_{l}-p-1}\\
\times (k_{1}^{p}s^{p}(w_{1}(s^{1/k_1},m) - w_{2}(s^{1/k_1},m))) \frac{ds}{s} ||_{(\nu,\beta,\mu,k_{1},k_{1})}\\
\leq \frac{|A_{\boldsymbol{\delta}_{l},p}|k_{1}^{p}C_{6}}{\Gamma( \frac{\mathbf{d}_{l,k_{1}}}{k_{1}} +
\boldsymbol{\delta}_{l} - p)
C_{\mathbf{P}}(r_{\mathbf{Q},\mathbf{R}_{\mathbf{D}}})^{\frac{1}{(\boldsymbol{\delta}_{\mathbf{D}}-1)k_{1}}} }
\epsilon_{0}^{\mathbf{\Delta}_{l}-\mathbf{d}_{l}+\boldsymbol{\delta}_{l} - 1} \\
\times \sup_{m \in \mathbb{R}} \frac{|\mathbf{R}_{l}(im)|}{|\mathbf{R}_{\mathbf{D}}(im)|}
||w_{1}(\tau,m) - w_{2}(\tau,m)||_{(\nu,\beta,\mu,k_{1},k_{1})}
\label{fix_point_norm_2_estim_4_shrink_bf}
\end{multline}
arise for all $1 \leq p \leq \boldsymbol{\delta}_{l}-1$, together with
\begin{multline}
|| \epsilon^{-1}
\frac{\mathbf{c}_{0}(\epsilon)}{\mathbf{P}_{m}(\tau) \Gamma(1 + \frac{1}{k_1})} \int_{0}^{\tau^{k_1}}
(\tau^{k_1}-s)^{1/k_1}\\
\times \left( \frac{1}{(2\pi)^{1/2}} s\int_{0}^{s} \int_{-\infty}^{+\infty} \right.
\boldsymbol{\varphi}_{k_1}((s-x)^{1/k_1},m-m_{1},\epsilon) \\
\times \left. \mathbf{R}_{0}(im_{1}) (w_{1}(x^{1/k_1},m_{1}) - w_{2}(x^{1/k_1},m_{1})) \frac{1}{(s-x)x}
dxdm_{1} \right) \frac{ds}{s} ||_{(\nu,\beta,\mu,k_{1},k_{1})}\\
\leq \frac{\boldsymbol{\varsigma}_{1,0}}{\Gamma(1 + \frac{1}{k_1})(2\pi)^{1/2}} \frac{C_{7}
\boldsymbol{\varsigma}_{1}
||w_{1}(\tau,m) - w_{2}(\tau,m)||_{(\nu,\beta,\mu,k_{1},k_{1})}}{
C_{\mathbf{P}}(r_{\mathbf{Q},\mathbf{R}_{\mathbf{D}}})^{\frac{1}{(\boldsymbol{\delta}_{\mathbf{D}}-1)k_{1}}} }
\label{fix_point_norm_2_estim_5_shrink_bf}
\end{multline}
and finally
\begin{multline}
|| \epsilon^{-1}\frac{\mathbf{c}_{0,0}(\epsilon)}{\mathbf{P}_{m}(\tau)\Gamma(1 + \frac{1}{k_1})} \int_{0}^{\tau^{k_1}}
(\tau^{k_1}-s)^{1/k_1} \frac{1}{(2\pi)^{1/2}} ( \int_{-\infty}^{+\infty} \mathbf{C}_{0,0}(m-m_{1},\epsilon)\\
\times \mathbf{R}_{0}(im_{1})
(w_{1}(s^{1/k_1},m_{1}) - w_{2}(s^{1/k_1},m_{1})) dm_{1} )\frac{ds}{s} ||_{(\nu,\beta,\mu,k_{1},k_{1})}\\
\leq \frac{\boldsymbol{\varsigma}_{0,0}}{\Gamma(1 + \frac{1}{k_1})(2\pi)^{1/2}}
\frac{C_{8} \boldsymbol{\varsigma}_{0} ||w_{1}(\tau,m) - w_{2}(\tau,m)||_{(\nu,\beta,\mu,k_{1},k_{1})} }{
C_{\mathbf{P}}(r_{\mathbf{Q},\mathbf{R}_{\mathbf{D}}})^{\frac{1}{(\boldsymbol{\delta}_{\mathbf{D}}-1)k_{1}}} }.
\label{fix_point_norm_2_estim_6_shrink_bf}
\end{multline}
Now, we take $\boldsymbol{\varsigma}_{0,0}, \boldsymbol{\varsigma}_{0}, \boldsymbol{\varsigma}_{1}, \boldsymbol{\varsigma}_{1,0}$
and $r_{\mathbf{Q},\mathbf{R}_{\mathbf{D}}}$ such that
\begin{multline}
\sum_{1 \leq p \leq \boldsymbol{\delta}_{\mathbf{D}}-1}
\frac{|A_{\boldsymbol{\delta}_{\mathbf{D}},p}|k_{1}^{p}C_{6}}{\Gamma(\boldsymbol{\delta}_{\mathbf{D}}-p)
C_{\mathbf{P}}(r_{\mathbf{Q},\mathbf{R}_{\mathbf{D}}})^{\frac{1}{(\boldsymbol{\delta}_{\mathbf{D}}-1)k_{1}}}}\\
+ \sum_{1 \leq l \leq \mathbf{D}-1} \frac{k_{1}^{\boldsymbol{\delta}_l}C_{6}}{\Gamma( \frac{\mathbf{d}_{l,k_{1}}}{k_1})
C_{\mathbf{P}}(r_{\mathbf{Q},\mathbf{R}_{\mathbf{D}}})^{\frac{1}{(\boldsymbol{\delta}_{\mathbf{D}}-1)k_{1}}} }
\epsilon_{0}^{\mathbf{\Delta}_{l} - \mathbf{d}_{l} + \boldsymbol{\delta}_{l} -1}
\sup_{m \in \mathbb{R}} \frac{|\mathbf{R}_{l}(im)|}{|\mathbf{R}_{\mathbf{D}}(im)|}\\
+ \sum_{1 \leq p \leq \boldsymbol{\delta}_{l}-1} \frac{|A_{\boldsymbol{\delta}_{l},p}
|k_{1}^{p}C_{6}}{\Gamma( \frac{\mathbf{d}_{l,k_{1}}}{k_{1}} + \boldsymbol{\delta}_{l} - p)
C_{\mathbf{P}}(r_{\mathbf{Q},\mathbf{R}_{\mathbf{D}}})^{\frac{1}{(\boldsymbol{\delta}_{\mathbf{D}}-1)k_{1}}} }
\epsilon_{0}^{\boldsymbol{\Delta}_{l}-\mathbf{d}_{l}+\boldsymbol{\delta}_{l}-1}
\sup_{m \in \mathbb{R}} \frac{|\mathbf{R}_{l}(im)|}{|\mathbf{R}_{\mathbf{D}}(im)|}\\
+ \frac{1}{\Gamma(1 + \frac{1}{k_1})(2\pi)^{1/2}} \frac{\boldsymbol{\varsigma}_{1,0}C_{7}
\boldsymbol{\varsigma}_{1} + \boldsymbol{\varsigma}_{0,0}C_{8} \boldsymbol{\varsigma}_{0}}{
C_{\mathbf{P}}(r_{\mathbf{Q},\mathbf{R}_{\mathbf{D}}})^{\frac{1}{(\boldsymbol{\delta}_{\mathbf{D}}-1)k_{1}}} } \leq \frac{1}{2}
\label{fix_point_sum<1/2_bf}
\end{multline}
Bearing in mind the estimates
(\ref{fix_point_norm_2_estim_2_shrink_bf}), (\ref{fix_point_norm_2_estim_3_shrink_bf}), (\ref{fix_point_norm_2_estim_4_shrink_bf}),
(\ref{fix_point_norm_2_estim_5_shrink_bf}), (\ref{fix_point_norm_2_estim_6_shrink_bf}) under the constraint
(\ref{fix_point_sum<1/2_bf}), one gets (\ref{H_k2_shrink_bf}).

Finally, we choose $\boldsymbol{\upsilon}$ and $r_{\mathbf{Q},\mathbf{R}_{\mathbf{D}}}$ such that both
(\ref{fix_point_sum<upsilon_bf}) and (\ref{fix_point_sum<1/2_bf}) are
fulfilled. This yields our lemma.
\end{proof}
We consider the ball $\bar{B}(0,\boldsymbol{\upsilon}) \subset F_{(\nu,\beta,\mu,k_{1},k_{1})}^{\mathfrak{d}_{p}}$
constructed in Lemma 12 which is a complete
metric space for the norm $||.||_{(\nu,\beta,\mu,k_{1},k_{1})}$. From the lemma above, we get that
$\bold{H}_{\epsilon}^{k_1}$ is a
contractive map from $\bar{B}(0,\boldsymbol{\upsilon})$ into itself. Due to the classical contractive mapping theorem,
we deduce that the map $\bold{H}_{\epsilon}^{k_1}$ has a unique fixed point denoted by
$\psi_{k_1}^{\mathfrak{d}_{p}}(\tau,m,\epsilon)$ (i.e
$\bold{H}_{\epsilon}^{k_1}(\psi_{k_1}^{\mathfrak{d}_{p}}(\tau,m,\epsilon))= \psi_{k_1}^{\mathfrak{d}}(\tau,m,\epsilon)$) in
$\bar{B}(0,\boldsymbol{\upsilon})$, for all $\epsilon \in D(0,\epsilon_{0})$. Moreover, the function
$\psi_{k_1}^{\mathfrak{d}_{p}}(\tau,m,\epsilon)$ depends holomorphically on $\epsilon$ in $D(0,\epsilon_{0})$. By construction,
$\psi_{k_1}^{\mathfrak{d}_{p}}(\tau,m,\epsilon)$ defines a solution of the equation
(\ref{k1_Borel_equation_analytic_bf}). This yields our proposition.
\end{proof}

As a summary of Proposition 18, we get that the $m_{k_1}-$Borel transform
$\psi_{k_1}(\tau,m,\epsilon)$ of the formal series $F(T,m,\epsilon)$ solution of the equation (\ref{SCP_bf}) is convergent
w.r.t $\tau$ on $D(0,\rho)$ as series in coefficients in $E_{(\beta,\mu)}$, for all $\epsilon \in D(0,\epsilon_{0})$, and can
be analytically continued on each unbounded sector $U_{\mathfrak{d}_{p}}$ as a function
$\tau \mapsto \psi_{k_1}^{\mathfrak{d}_{p}}(\tau,m,\epsilon)$ which belongs to the space
$F^{\mathfrak{d}_{p}}_{(\nu,\beta,\mu,k_{1},k_{1})}$. In other words, the assumed constraints
(\ref{psi_k1_bounded_norm_k1_k1_main_CP}) are fulfilled.

\subsection{A linear initial value Cauchy problem satisfied by the analytic forcing terms $f^{\mathfrak{d}_{p}}(t,z,\epsilon)$}

We keep the notations and the assumptions made in the previous subsection. From the assumption (\ref{norm_beta_mu_F_n_bf}), we deduce
that the functions
\begin{multline}
\check{\mathbf{C}}_{0}(T,z,\epsilon) = \mathbf{c}_{0,0}(\epsilon) \mathcal{F}^{-1}(m \mapsto \mathbf{C}_{0,0}(m,\epsilon) )(z)
+ \sum_{n \geq 1} \mathbf{c}_{0}(\epsilon) \mathcal{F}^{-1}(m \mapsto \mathbf{C}_{0,n}(m,\epsilon))(z)T^{n},\\
\check{\mathbf{F}}(T,z,\epsilon) = \sum_{n \geq 1} \mathcal{F}^{-1}(m \mapsto \mathbf{F}_{n}(m,\epsilon) )(z)T^{n}
\end{multline}
represent bounded holomorphic functions on $D(0,\mathbf{T}_{0}/2) \times H_{\beta'} \times D(0,\epsilon_{0})$ for any
$0 < \beta' < \beta$ (where $\mathcal{F}^{-1}$ denotes the inverse Fourier transform defined in Proposition 9). We define the
coefficients
\begin{equation}
\mathbf{c}_{0}(t,z,\epsilon) = \check{\mathbf{C}}_{0}(\epsilon t,z,\epsilon) \ \ , \ \ 
\mathbf{f}(t,z,\epsilon) = \check{\mathbf{F}}(\epsilon t,z,\epsilon)
\end{equation}
which are holomorphic and bounded on $D(0,r) \times H_{\beta'} \times D(0,\epsilon_{0})$ where
$r \epsilon_{0} \leq \mathbf{T}_{0}/2$.
\begin{prop} Under the constraints (\ref{assum_delta_l_bf}), (\ref{assum_dl_delta_l_Delta_l_bf}),
(\ref{assum_Q_Rl_bf}), (\ref{norm_beta_mu_F_n_bf}) and the assumptions (\ref{quotient_Q_RD_in_S_bf}),
(\ref{root_cond_1_bf}), (\ref{root_cond_2_bf}), (\ref{constraints_k1_Borel_equation_bf}),
(\ref{norm_F_varphi_k1_psi_k_1_small_bf}), the forcing term
$f^{\mathfrak{d}_{p}}(t,z,\epsilon)$ represented by the formula (\ref{forcing_term_frak_d_p}) solves the following linear
Cauchy problem
\begin{multline}
\mathbf{Q}(\partial_{z})(\partial_{t}f^{\mathfrak{d}_{p}}(t,z,\epsilon)) =
\epsilon^{(\boldsymbol{\delta}_{\mathbf{D}}-1)(k_{1}+1) - \boldsymbol{\delta}_{\mathbf{D}} + 1}
t^{(\boldsymbol{\delta}_{\mathbf{D}}-1)(k_{1}+1)}
\partial_{t}^{\boldsymbol{\delta}_{\mathbf{D}}}\mathbf{R}_{\mathbf{D}}(\partial_{z})f^{\mathfrak{d}_{p}}(t,z,\epsilon)\\
+ \sum_{l=1}^{\mathbf{D}-1} \epsilon^{\mathbf{\Delta}_{l}}t^{\mathbf{d}_l}
\partial_{t}^{\boldsymbol{\delta}_l}\mathbf{R}_{l}(\partial_{z})f^{\mathfrak{d}_{p}}(t,z,\epsilon)
+ \mathbf{c}_{0}(t,z,\epsilon)\mathbf{R}_{0}(\partial_{z})f^{\mathfrak{d}_{p}}(t,z,\epsilon) +
\mathbf{c}_{\mathbf{F}}(\epsilon)\mathbf{f}(t,z,\epsilon) \label{ICP_main_fp_bf}
\end{multline}
for given initial data $f^{\mathfrak{d}_{p}}(0,z,\epsilon) \equiv 0$, for all $t \in \mathcal{T}$, $z \in H_{\beta'}$ and
$\epsilon \in \mathcal{E}_{p}$ (provided that the radius $r_{\mathcal{T}}$ of $\mathcal{T}$ fulfills the restriction
$\epsilon_{0}r_{\mathcal{T}} \leq \min(h',T_{0}/2,\mathbf{T}_{0}/2)$).
\end{prop}
\begin{proof} From Proposition 18, we know that the formal series $F(T,m,\epsilon) = \sum_{n \geq 1} F_{n}(m,\epsilon)T^{n}$
is $m_{k_1}-$summable w.r.t $T$ in all directions $\mathfrak{d}_{p}$, $0 \leq p \leq \varsigma-1$ (in the sense of
Definition 4). Therefore, from the estimates (\ref{psi_k1_bounded_norm_k1_k1_main_CP}), we deduce that the
$m_{k_1}-$Laplace transform
$$ \mathcal{L}_{m_{k_1}}^{\mathfrak{d}_{p}}(\tau \mapsto \psi_{k_1}^{\mathfrak{d}_{p}}(\tau,m,\epsilon))(T)
= k_{1}\int_{L_{\mathfrak{d}_{p}}}
\psi_{k_1}^{\mathfrak{d}_{p}}(u,m,\epsilon) e^{-(\frac{u}{T})^{k_{1}}} \frac{du}{u} $$
defines a bounded and holomorphic function on any sector $S_{\mathfrak{d}_{p},\theta_{k_1},h'_{k_1}}$ w.r.t $T$, for all
$m \in \mathbb{R}$, all $\epsilon \in D(0,\epsilon_{0})$, where
$S_{\mathfrak{d}_{p},\theta_{k_1},h'_{k_1}}$ is a sector with bisecting direction $\mathfrak{d}_{p}$,
aperture $\frac{\pi}{k_1} < \theta_{k_1} < \frac{\pi}{k_1} + \mathrm{ap}(U_{\mathfrak{d}_{p}})$ and some radius
$h'_{k_1}>0$. Moreover, using the algebraic properties of the $m_{k_1}-$sums in the formula (\ref{sum_prod_deriv_m_k_sum}), we
deduce that $\mathcal{L}_{m_{k_1}}^{\mathfrak{d}_{p}}(\tau \mapsto \psi_{k_1}^{\mathfrak{d}_{p}}(\tau,m,\epsilon))(T)$ solves
the equations (\ref{SCP_irregular_bf}) and then (\ref{SCP_bf}) for all $T \in S_{\mathfrak{d}_{p},\theta_{k_1},h'_{k_1}}$,
all $m \in \mathbb{R}$, all
$\epsilon \in D(0,\epsilon_{0})$ and vanishes at $T=0$. Now, let $F^{\mathfrak{d}_{p}}(T,m,\epsilon)$ defined in
(\ref{defin_F_frak_d_p}).
\begin{lemma} The following identity
$$ F^{\mathfrak{d}_{p}}(T,m,\epsilon) =
\mathcal{L}_{m_{k_1}}^{\mathfrak{d}_{p}}(\tau \mapsto \psi_{k_1}^{\mathfrak{d}_{p}}(\tau,m,\epsilon))(T)
$$
holds, for all $T \in S_{\mathfrak{d}_{p},\theta,h'}$, $m \in \mathbb{R}$, $\epsilon \in D(0,\epsilon_{0})$, as defined just
after the definition (\ref{defin_F_frak_d_p}), for
$\frac{\pi}{k_2} < \theta < \frac{\pi}{k_2} + \mathrm{ap}(S_{\mathfrak{d}_{p}})$ and some radius $h'>0$. 
\end{lemma}
\begin{proof} By construction, we can write
\begin{multline*}
F^{\mathfrak{d}_{p}}(T,m,\epsilon) = k_{2} \int_{L_{\mathfrak{d}_{p}}}
\left( \int_{L_{\mathfrak{d}_{p}}} \psi_{k_1}^{\mathfrak{d}_{p}}(h,m,\epsilon) \right. \\
\left. \times ( -\frac{k_{2}k_{1}}{2i\pi} u^{k_2} \int_{V_{\mathfrak{d}_{p},k_{2},\delta'}}
\exp( -(\frac{h}{v})^{k_1} + (\frac{u}{v})^{k_2} ) \frac{dv}{v^{k_{2}+1}} ) \frac{dh}{h} \right) e^{-(\frac{u}{T})^{k_2}}
\frac{du}{u}
\end{multline*}
for some $0 < \delta' < \frac{\pi}{\kappa}$, where $V_{\mathfrak{d}_{p},k_{2},\delta'}$ is defined in Proposition 13. Using
Fubini's theorem yields
\begin{equation}
F^{\mathfrak{d}_{p}}(T,m,\epsilon) = k_{1}\int_{L_{\mathfrak{d}_{p}}}
\psi_{k_1}^{\mathfrak{d}_{p}}(h,m,\epsilon) A(T,h) \frac{dh}{h} \label{F_d_p_integral_A}
\end{equation}
where
\begin{multline}
 A(T,h) = k_{2} \int_{L_{\mathfrak{d}_{p}}} \frac{-k_{2}}{2i \pi} u^{k_2}
\left(\int_{V_{\mathfrak{d}_{p},k_{2},\delta'}} \exp( -(\frac{h}{v})^{k_1} + (\frac{u}{v})^{k_2}) \frac{dv}{v^{k_{2}+1}} \right)
e^{-(\frac{u}{T})^{k_2}} \frac{du}{u}\\
=\mathcal{L}_{m_{k_2}}^{\mathfrak{d}_{p}}(
u \mapsto (\mathcal{B}_{m_{k_2}}^{\mathfrak{d}_{p}}( v \mapsto e^{-(\frac{h}{v})^{k_1}} ))(u) )(T) \label{defin_A}
\end{multline}
for all $T \in S_{\mathfrak{d}_{p},\theta,h'}$, $m \in \mathbb{R}$, $\epsilon \in D(0,\epsilon_{0})$.
But we observe from the inversion formula (\ref{Laplace_Borel_inverse}) that $A(T,h) = \exp(-(h/T)^{k_1} )$.
Gathering (\ref{F_d_p_integral_A}) and (\ref{defin_A}) yields the lemma 13.
\end{proof}
From Lemma 13, we deduce that $F^{\mathfrak{d}_{p}}(T,m,\epsilon)$ solves the equation
(\ref{SCP_bf}) for all $T \in S_{\mathfrak{d}_{p},\theta,h'}$, all $m \in \mathbb{R}$ and all $\epsilon \in D(0,\epsilon_{0})$.
Hence, using the properties of the Fourier inverse transform from Proposition 9, we deduce that the analytic forcing term
$f^{\mathfrak{d}_{p}}(t,z,\epsilon) = \mathcal{F}^{-1}( m \mapsto F^{\mathfrak{d}_{p}}(\epsilon t,m,\epsilon))(z)$ solves
the linear Cauchy problem (\ref{ICP_main_fp_bf}), for all $t \in \mathcal{T}$, all $z \in H_{\beta'}$ and all
$\epsilon \in \mathcal{E}_{p}$.
\end{proof}
We are in position to state the main result of this section
\begin{theo} We take for granted that the assumptions of Theorem 1 hold. We also make the hypothesis that the constraints
(\ref{assum_delta_l_bf}), (\ref{assum_dl_delta_l_Delta_l_bf}),
(\ref{assum_Q_Rl_bf}), (\ref{norm_beta_mu_F_n_bf}) and the assumptions (\ref{quotient_Q_RD_in_S_bf}),
(\ref{root_cond_1_bf}), (\ref{root_cond_2_bf}), (\ref{constraints_k1_Borel_equation_bf}),
(\ref{norm_F_varphi_k1_psi_k_1_small_bf}) hold. We denote $P(t,z,\epsilon,\partial_{t},\partial_{z})$ and
$\mathbf{P}(t,z,\epsilon,\partial_{t},\partial_{z})$ the linear differential operators
\begin{multline}
P(t,z,\epsilon,\partial_{t},\partial_{z}) = Q(\partial_{z})\partial_{t} -
\epsilon^{(\delta_{D}-1)(k_{2}+1) - \delta_{D}+1} t^{(\delta_{D}-1)(k_{2}+1)} \partial_{t}^{\delta_{D}} R_{D}(\partial_{z})\\
- \sum_{l=1}^{D-1} \epsilon^{\Delta_{l}} t^{d_l} \partial_{t}^{\delta_l} R_{l}(\partial_{z}) -
c_{0}(t,z,\epsilon)R_{0}(\partial_{z}),\\
\mathbf{P}(t,z,\epsilon,\partial_{t},\partial_{z}) = \mathbf{Q}(\partial_{z})\partial_{t} -
\epsilon^{(\boldsymbol{\delta}_{\mathbf{D}}-1)(k_{1}+1) - \boldsymbol{\delta}_{\mathbf{D}}+1}
t^{(\boldsymbol{\delta}_{\mathbf{D}}-1)(k_{1}+1)} \partial_{t}^{\boldsymbol{\delta}_{\mathbf{D}}}
\mathbf{R}_{\mathbf{D}}(\partial_{z})\\
- \sum_{l=1}^{\mathbf{D}-1} \epsilon^{\boldsymbol{\Delta}_{l}} t^{\mathbf{d}_l} \partial_{t}^{\boldsymbol{\delta}_l}
\mathbf{R}_{l}(\partial_{z}) -
\mathbf{c}_{0}(t,z,\epsilon)\mathbf{R}_{0}(\partial_{z}).
\end{multline}
Then, the functions $u^{\mathfrak{d}_{p}}(t,z,\epsilon)$ constructed in Theorem 1 solve the following nonlinear PDE
\begin{multline}
\mathbf{P}(t,z,\epsilon,\partial_{t},\partial_{z})P(t,z,\epsilon,\partial_{t},\partial_{z})u^{\mathfrak{d}_{p}}(t,z,\epsilon)\\
= c_{1,2}(\epsilon)\mathbf{P}(t,z,\epsilon,\partial_{t},\partial_{z})
\left( Q_{1}(\partial_{z})u^{\mathfrak{d}_{p}}(t,z,\epsilon) \times Q_{2}(\partial_{z})u^{\mathfrak{d}_{p}}(t,z,\epsilon) \right)
\\
+ c_{F}(\epsilon)\mathbf{c}_{\mathbf{F}}(\epsilon)\mathbf{f}(t,z,\epsilon) \label{nlpde_holcoef_near_origin_u_dp}
\end{multline}
whose coefficients and forcing term $\mathbf{f}$ are analytic functions on
$D(0,r_{\mathcal{T}}) \times H_{\beta'} \times D(0,\epsilon_{0})$,
with vanishing initial data $u^{\mathfrak{d}_{p}}(0,z,\epsilon) \equiv 0$, for all $t \in \mathcal{T}$, all $z \in H_{\beta'}$
and all $\epsilon \in \mathcal{E}_{p}$. Moreover, the formal power series
$\hat{u}(t,z,\epsilon) = \sum_{m \geq 0} h_{m}(t,z) \epsilon^{m}/m!$ constructed in Theorem 2 formally solves the same
equation (\ref{nlpde_holcoef_near_origin_u_dp}).
\end{theo}
\begin{proof} The reason why $u^{\mathfrak{d}_{p}}(t,z,\epsilon)$ solves the equation (\ref{nlpde_holcoef_near_origin_u_dp})
follows directly from the fact that $u^{\mathfrak{d}_{p}}(t,z,\epsilon)$ solves the nonlinear equation
$$
P(t,z,\epsilon,\partial_{t},\partial_{z})u^{\mathfrak{d}_{p}}(t,z,\epsilon) =
c_{1,2}(\epsilon)
(Q_{1}(\partial_{z})u^{\mathfrak{d}_{p}}(t,z,\epsilon) \times Q_{2}(\partial_{z})u^{\mathfrak{d}_{p}}(t,z,\epsilon)) +
c_{F}(\epsilon)f^{\mathfrak{d}_{p}}(t,z,\epsilon)
$$
according to Theorem 1 and from the additional feature that $f^{\mathfrak{d}_{p}}(t,z,\epsilon)$ solves the linear equation
$$
\mathbf{P}(t,z,\epsilon,\partial_{t},\partial_{z})f^{\mathfrak{d}_{p}}(t,z,\epsilon) =
\mathbf{c}_{\mathbf{F}}(\epsilon)\mathbf{f}(t,z,\epsilon)
$$
as shown in Proposition 19. Finally in order to show that $\hat{u}(t,z,\epsilon)$ formally solves
(\ref{nlpde_holcoef_near_origin_u_dp}) we see that with the help of the second equality in
(\ref{limit_deriv_order_m_of_up_fp_epsilon}) and following exactly the same lines of arguments as in the last part
of Theorem 2, one can show that the power series $\hat{f}(t,z,\epsilon) = \sum_{m \geq 0} f_{m}(t,z) \epsilon^{m}/m!$
constructed in Lemma 11 formally solves the linear equation
\begin{equation}
\mathbf{P}(t,z,\epsilon,\partial_{t},\partial_{z})\hat{f}(t,z,\epsilon) =
\mathbf{c}_{\mathbf{F}}(\epsilon)\mathbf{f}(t,z,\epsilon) \label{CP_hat_f}
\end{equation}
Combining the equations (\ref{ICP_main_formal_epsilon}) and (\ref{CP_hat_f}) yields the result.
\end{proof}

\end{document}